%% file: main.tex
\documentclass[a4paper,11pt,reqno]{amsart}
\usepackage[hidelinks]{hyperref}
\usepackage[a4paper,left=25mm,right=25mm,top=30mm,bottom=30mm,marginpar=20mm]{geometry}
\usepackage[utf8x]{inputenc}
\usepackage{amsmath}
\usepackage{amssymb}
\usepackage{amsthm}
\usepackage{amsfonts}
\usepackage{mathtools}
\usepackage{array}
\usepackage{enumerate}
\usepackage{xcolor}
\usepackage{bbm}
\usepackage{esint}
\usepackage{graphicx}
\usepackage{float}
\usepackage{setspace}
\usepackage{esint}
\usepackage{mathrsfs}
\usepackage{stmaryrd}
\allowdisplaybreaks

\newtheorem{theorem}{Theorem}[section]
\newtheorem{definition}[theorem]{Definition}
\newtheorem*{definition*}{Definition}
\newtheorem*{theorem*}{Theorem}
\newtheorem{lemma}[theorem]{Lemma}
\newtheorem*{lemma*}{Lemma}
\newtheorem{corollary}[theorem]{Corollary}
\newtheorem{example}[theorem]{Example}

\newtheorem{remark}[theorem]{Remark}
\newtheorem*{remark*}{Remark}

\newtheorem{proposition}[theorem]{Proposition}

\newtheorem{alphtheorem}{Theorem}

\newcommand{\sbullet}{\,\begin{picture}(1,1)(-0.5,-2)\circle*{2}\end{picture}\,}
\newcommand{\wsc}{\overset{*}{\rightharpoonup}}
\newcommand{\ip}[2]{\left<#1,#2\right>}
\newcommand{\restrict}{	\begin{picture}(10,8)\put(2,0){\line(0,1){7}}\put(1.8,
0){\line(1,0){7}}\end{picture}}
\newcommand{\restrictsmall}{\begin{picture}(7,6)\put(1,0){\line(0,1){6}}\put(0.8
,0){\line(1,0){6}}\end{picture}}
	
\newcommand{\bnu}{\boldsymbol\nu}

\newcommand{\bsigma}{\boldsymbol\sigma}
\newcommand{\bdelta}{\boldsymbol\delta}
\newcommand{\btau}{\boldsymbol\tau}

\newcommand{\pd}{\partial}
\newcommand{\dd}{\mathrm{d}}

\newcommand{\mbR}{\mathbb{R}}
\newcommand{\mbN}{\mathbb{N}}
\newcommand{\mbfM}{\mathbf{M}}

\newcommand{\mbfE}{\mathbf{E}}
\newcommand{\mcF}{\mathcal{F}}
\newcommand{\mbfR}{\mathbf{R}}
\newcommand{\mbfY}{\mathbf{Y}}
\newcommand{\Lp}{\mathrm{L}}
\newcommand{\W}{\mathrm{W}}
\newcommand{\C}{\mathrm{C}} 
\newcommand{\bB}{\mathbb{B}}
\newcommand{\Tan}{\mathrm{Tan}}
\newcommand{\lL}{\mathcal{L}^d}
\newcommand{\Lift}{\mathbf{L}}
\newcommand{\ALift}{\mathbf{AL}}
\newcommand{\YLift}{\mathbf{LY}}	
\newcommand{\AYLift}{\mathbf{ALY}}

\newcommand{\Lrm}{\mathrm{L}}

\newcommand{\Acal}{\mathcal{A}}
\newcommand{\Bcal}{\mathcal{B}}
\newcommand{\Ccal}{\mathcal{C}}
\newcommand{\Dcal}{\mathcal{D}}
\newcommand{\Ecal}{\mathcal{E}}
\newcommand{\Fcal}{\mathcal{F}}
\newcommand{\Gcal}{\mathcal{G}}
\newcommand{\Hcal}{\mathcal{H}}

\newcommand{\Jcal}{\mathcal{J}}

\newcommand{\Lcal}{\mathcal{L}}

\newcommand{\Ncal}{\mathcal{N}}
\newcommand{\Ocal}{\mathcal{O}}

\newcommand{\Kfrak}{\mathfrak{K}}

\newcommand{\Gbf}{\mathbf{G}}
\newcommand{\Hbf}{\mathbf{H}}

\newcommand{\Mbf}{\mathbf{M}}

\newcommand{\Rbf}{\mathbf{R}}

\DeclareMathOperator{\id}{id}

\DeclareMathOperator{\diverg}{div}

\DeclareMathOperator{\dist}{dist}

\DeclareMathOperator{\supp}{supp}

\DeclareMathOperator{\gr}{gr}

\newcommand{\setBB}[2]{\biggl\{\, #1 \ \ \textup{\textbf{:}}\ \ #2 \,\biggr\}}

\newcommand{\norm}[1]{\|#1\|}

\newcommand{\altnorm}[1]{{\left\vert\kern-0.25ex\left\vert\kern-0.25ex\left\vert #1 \right\vert\kern-0.25ex\right\vert\kern-0.25ex\right\vert}}

\newcommand{\ddpr}[1]{\langle\!\langle #1 \rangle\!\rangle}

\newcommand{\ddprb}[1]{\bigl\langle\hspace{-2.5pt}\bigl\langle #1 \bigr\rangle\hspace{-2.5pt}\bigr\rangle}
\newcommand{\ddprB}[1]{\Bigl\langle\!\!\Bigl\langle #1 \Bigr\rangle\!\!\Bigr\rangle}

\newcommand{\N}{\mathbb{N}}
\newcommand{\R}{\mathbb{R}}

\newcommand{\ONE}{\mathbbm{1}}

\newcommand{\todown}{\downarrow}

\newcommand{\frarg}{\,\sbullet\,}
\newcommand{\BV}{\mathrm{BV}}

\newcommand{\toY}{\overset{\mathrm{Y}}{\to}}

\DeclarePairedDelimiter{\asc}{[}{]}

\DeclareMathOperator*{\wstarlim}{w*-lim}

\newcommand{\RL}{\mathbf{R}^\mathrm{L}}
\newcommand{\RBVw}{\mathbf{R}^{w*}}

\newcommand{\Fcalrw}{\mathcal{F}_{**}}
\newcommand{\Fcalro}{\mathcal{F}^{\,1}_{**}}

\def\Xint#1{\mathchoice
{\XXint\displaystyle\textstyle{#1}}%
{\XXint\textstyle\scriptstyle{#1}}%
{\XXint\scriptstyle\scriptscriptstyle{#1}}%
{\XXint\scriptscriptstyle\scriptscriptstyle{#1}}%
\!\int}
\def\XXint#1#2#3{{\setbox0=\hbox{$#1{#2#3}{\int}$ }
\vcenter{\hbox{$#2#3$ }}\kern-.6\wd0}}

\def\dashint{\Xint-}

\title{Liftings, Young measures, and lower semicontinuity}

\author{Filip Rindler}
\address[Filip Rindler]{Mathematics Institute, University of Warwick, Coventry CV4 7AL, United Kingdom}
\email{F.Rindler@warwick.ac.uk}

\author{Giles Shaw}
\address[Giles Shaw]{University of Reading, Department of Mathematics and Statistics, Whiteknights, PO Box 220, Reading RG6 6AX, United Kingdom}
\email[Corresponding author]{giles.shaw@gmail.com}

\begin{document}		

\input{introprelims}

\input{liftings}
\input{youngmeasures}
\input{wstarblowup}
\input{recoveryseqs}

\section{Compliance with ethical standards}
\noindent\textbf{Funding:} G.S.'s contribution to this work forms part of their PhD thesis and was supported by the UK Engineering and Physical Sciences Research Council (EPSRC) grant EP/H023348/1 for the University of Cambridge Centre for Doctoral Training, the Cambridge Centre for Analysis. This project has received funding from the European Research Council (ERC) under the European Union's Horizon 2020 research and innovation programme, grant agreement No 757254 (SINGULARITY). F.~R.\ also acknowledges the support from an EPSRC Research Fellowship on Singularities in Nonlinear PDEs (EP/L018934/1).

\bigskip

\noindent\textbf{Conflict of Interest:} The authors declare that they have no conflict of interest.

\bibliography{math_bib}
\bibliographystyle{plain}
\end{document}

%% file: introprelims.tex

\begin{abstract}
This work introduces liftings and their associated Young measures as new tools to study the asymptotic behaviour of sequences of pairs $(u_j,Du_j)_j$ for $(u_j)_j\subset\BV(\Omega;\R^m)$ under weak* convergence. These tools are then used to prove an integral representation theorem for the relaxation of the functional
\[
\Fcal\colon u\mapsto\int_\Omega f(x,u(x),\nabla u(x))\;\dd x,\quad u\in\W^{1,1}(\Omega;\R^m),\quad\Omega\subset\R^d\text{ open,}
\]
to the space $\BV(\Omega;\R^m)$. Lower semicontinuity results of this type were first obtained by~Fonseca \& M\"uller [\emph{Arch. Ration. Mech. Anal.} 123 (1993), 1--49] and later improved by a number of authors, but our theorem is valid under more natural, essentially optimal, hypotheses than those currently present in the literature, requiring principally that $f$ be Carath\'eodory and quasiconvex in the final variable. The key idea is that liftings provide the right way of localising $\Fcal$ in the $x$ and $u$ variables simultaneously under weak* convergence. As a consequence, we are able to implement an optimal measure-theoretic blow-up procedure.
\end{abstract}

\maketitle
\author

\section{Introduction}\label{secintro}
Finding an integral representation for the relaxation $\Fcal_{**}$ of the functional
\begin{equation}\label{eqoriginalfunctional}
\Fcal[u]:=\int_\Omega f(x,u(x),\nabla u(x))\;\dd x,\quad u\in\W^{1,1}(\Omega;\R^m),\quad\Omega\subset\R^d\text{ open,}\quad d,m\geq 1,
\end{equation}
to the space $\BV(\Omega;\R^m)$ of functions of bounded variation, where $f \colon \Omega \times \R^m \times \R^{m \times d} \to [0,\infty)$ has linear growth in the last variable, is of great importance in the Calculus of Variations. Defined by the formula
\[
\Fcal_{**}[u]:=\inf\left\{\liminf_{j\to\infty}\Fcal[u_j]\colon\; (u_j)_j\subset\W^{1,1}(\Omega;\R^m)\text{ and }u_j\rightsquigarrow u\in\BV(\Omega;\R^m)\right\},
\]
where, for the moment, we do not specify the notion of convergence ``$\rightsquigarrow$'' with respect to which $\Fcal_{**}$ is computed, the relaxation is the greatest functional on $\BV(\Omega;\R^m)$ which is both less than or equal to $\Fcal$ on $\W^{1,1}(\Omega;\R^m)$ and lower semicontinuous (with respect to $\rightsquigarrow$) over $\BV(\Omega;\R^m)$. From a theoretical point of view, identifying $\Fcal_{**}$ is a necessary step in the application of the Direct Method to minimisation problems with linear growth: indeed, if $f$ is coercive, minimising sequences for $\Fcal$ are merely bounded in the non-reflexive space $\W^{1,1}(\Omega;\R^m)$ and can only be expected to converge weakly* in $\BV(\Omega;\R^m)$. Since candidate minimisers are only guaranteed to exist in this larger space, there is a need to find a `faithful' extension of $\Fcal$ to $\BV(\Omega;\R^m)$ to which the Direct Method can be applied. If they exist, minimisers in $\BV(\Omega;\R^m)$ of this extension of $\Fcal$ can then be seen as weak solutions to the original minimisation problem over $\W^{1,1}(\Omega;\R^m)$.

From the perspective of applications, if $f(x,y,A)=g(x,y)h(x,y,A)$ then the Cauchy--Schwarz inequality implies that computing $\Fcal_{**}$ provides a (usually optimal) lower bound for the $\Gamma$-limit of the sequence of singular perturbations
\begin{equation}\label{eqperturbfamily}
\Ecal_\varepsilon[u]:=\varepsilon^{-1}\int_\Omega \left[g(x,u(x))\right]^2\;\dd x+\varepsilon\int_\Omega \left[h(x,u(x),\nabla u(x))\right]^2\;\dd x,
\end{equation}
which arise in a vast number of phase transition problems from the physical sciences~\cite{AviGig87AMPR,Bald90MICP,Fonse89PTEM,FonTar89GTPT,Gurt87SRCG,Modi87TGPT,RuStKe89RDPE}. In this context, minimisers of $\Fcal_{**}$ can be seen as ``physically reasonable'' solutions to the highly non-unique problem of minimising the coarse-grain energy $\int_\Omega [g(x,u(x))]^2 \;\dd x$, which take into account the fact that transitions between phases should have an energetic cost.

The first general solution to the relaxation problem was provided by Fonseca~\&~M\"uller in~\cite{FonMul93RQFB} (see also Ambrosio \& Dal Maso~\cite{AmbDal92RBVQ} for the $u$-independent case). Motivated by problems in the theory of phase transitions, the authors showed that, if $f$ is quasiconvex in the final variable, satisfies
\[
  g(x,y)|A|\leq f(x,y,A)\leq Cg(x,y)(1+|A|)
\]
for some $g\in\C(\Omega\times\R^m;[0,\infty))$, and strong \emph{localisation hypotheses} (see below) hold, then the relaxation of $\mcF$ with respect to the strong $\Lp^1$-convergence in $\BV(\Omega;\R^m)$ is given by
\begin{align}
\begin{split}\label{eqfunctional}
\mcF_{**}[u] &= \int_\Omega f(x,u(x),\nabla u(x))\;\dd x+\int_\Omega f^\#\left(x,u(x),\frac{\dd D^cu}{\dd|D^c u|}(x)\right)\;\dd|D^cu|(x) \\
  &\qquad +\int_{\mathcal{J}_u}K_f[u](x)\;\dd\mathcal{H}^{d-1}(x),  
\end{split}
\end{align}
where $f^\#$ is the recession function of $f$ defined by $f^\#(x,y,A):=\limsup_{t\to\infty}t^{-1}f(x,y,tA)$, and
\[
\begin{aligned}
K_f[u](x):= \inf\, \setBB{\frac{1}{\omega_{d-1}}\int_{\bB^d}f^\infty(x,\varphi(y),\nabla\varphi(y))\;\dd y}{&\varphi\in\C^\infty(\bB^d;\mbR^m), \\
&{\varphi}|_{\pd\bB^d}=u^\pm(x)\text{ if }\ip{y}{n_u(x)} \gtrless 0}
\end{aligned}
\]
is the surface energy density associated with $f$. Here, $d,m\geq 1$, and we have used the usual decomposition
\[
  Du = \nabla u \Lcal^d + D^s u,  \qquad D^s u = (u^+ - u^-) \otimes n_u \, \Hcal^{d-1} \restrict \Jcal_u + D^c u,
\]
for the derivative $Du$ of a function $u\in\BV(\Omega;\R^m)$ (see, for example~\cite{AmFuPa00FBVF}). Fonseca~\&~M\"uller's result, later improved in the subsequent papers~\cite{BoFoMa98GMR,FonLeo01LSC}, makes use of the \emph{blow-up method} to obtain a lower bound for $\Fcal_{**}$: if $(u_j)_j\subset\W^{1,1}(\Omega;\R^m)$ is such that $(\Fcal[u_j])_j$ is bounded then, upon passing to a subsequence,  $(f(x,u_j(x),\nabla u_j(x))\lL\restrict\Omega)_j$ must converge weakly* in $\mbfM^+(\Omega)$ to some Radon measure $\mu$. Next, one computes lower bounds for the Radon--Nikodym Derivatives $\frac{\dd\mu}{\dd\lL}$, $\frac{\dd\mu}{\dd|D^cu|}$, $\frac{\dd\mu}{\dd\Hcal^{d-1}\restrict\Jcal_u}$ via estimates of the form
\begin{equation}\label{eqL1lscineq}
\frac{\dd\mu}{\dd\lL}(x_0)\geq\liminf_{r\to 0}\lim_{j\to\infty}\frac{1}{r^d}\frac{1}{\omega_d}\int_{B(x_0,r)}f(x,u_j(x),\nabla u_j(x))\;\dd x.
\end{equation}
To obtain the inequality ``$\geq$'' in~\eqref{eqfunctional}, it suffices to bound the right hand side of~\eqref{eqL1lscineq} from below by $f(x_0,u(x_0),\nabla u(x_0))$ and to obtain analogous results for $\frac{\dd\mu}{|D^cu|}$ and $\frac{\dd\mu}{\dd\Hcal^{d-1}\restrict\Jcal_u}$. The authors of~\cite{FonMul93RQFB} achieve this by noting that the partial coercivity property $g(x,y)|A|\leq f(x,y,A)\leq Cg(x,y)(1+|A|)$ combined with the rescaling in $r$ of each $u_j$ allows for $(u_j)_j$ to be replaced by a rescaled and truncated sequence $(w_j)_j$, which is crucially weakly* and $\Lp^\infty$ convergent in $\BV(\bB^d;\R^m)$ to a blow-up limit $z\mapsto\nabla u(x_0)z$.


Fonseca~\&~M\"uller's result confirms that, as in the case for problems posed over $\W^{1,p}(\Omega;\R^m)$ for $p>1$, quasiconvexity is still the right qualitative condition to require for variational problems with linear growth. However, it has been an open question as to whether the $(x,y)$-localisation hypotheses which have been used until now (both for the results in~\cite{FonMul93RQFB} and for the substantial later improvements in~\cite{FonLeo01LSC}) are truly necessary for~\eqref{eqfunctional} to hold. Up to an error which grows linearly in $A$, these assumptions state that $f$ must be such that $f(x_0,y_0,A) \approx f(x,y,A)$ whenever $(x,y)$ is sufficiently close to $(x_0,y_0)$ and that $f^\#(x_0,y,A) \approx f^\#(x,y,A)$ uniformly in $y$ for $|x_0-x|$ sufficiently small.

It is known in the case of superlinear growth that for $\Fcal$ to be weakly lower semicontinuous over $\W^{1,p}(\Omega;\R^m)$, $f$ need only be quasiconvex in the final variable, Carath\'eodory, and satisfy the growth bound $0\leq f(x,y,A)\leq C(1+|A|^p)$. For the situation where $f$ has linear growth and the task is to find the relaxation of $\Fcal$ to $\BV(\Omega;\R^m)$, a short overview of available results is as follows: in the scalar valued case where $m=1$, Dal Maso~\cite{DMas79IRBV} obtained the scalar counterpart to~\eqref{eqfunctional} (the term over $\Jcal_u$ admits a simpler form) under the assumption of coercivity for the general case where $f=f(x,y,A)$. If no coercivity is required of $f$, the task of finding the $\Lp^1$-relaxation $\Fcalro$ is highly non-trivial and involves the regularity of $f$ in the $x$ variable, see for instance~\cite{AmCiFu07RRBV}. In the vector-valued case, Aviles and Giga used the theory of currents to obtain~\eqref{eqfunctional} under the assumption that $f$ is continuous, convex in the final variable, coercive, and also satisfies a specific isotropy property, see~\cite{AviGig91VIMB,AviGig92MCGR}. The situation where $f$ is quasiconvex in the final variable is harder: the case where $f$ is independent of both $x$ and $u$ was settled by Ambrosio and Dal Maso~\cite{AmbDal92RBVQ}, and Kristensen together with the first author obtained~\eqref{eqfunctional} under just the assumption that $f=f(x,A)$ is $u$-independent, Carath\'eodory and such that $f^\infty$ exists~\cite{KriRin10RSI}.

For the $u$-dependent case where $f=f(x,y,A)$ and $f$ is quasiconvex in the final variable, the only available results are the original identification in~\cite{FonMul93RQFB} due to Fonseca and M\"uller, later improvements in~\cite{BoFoMa98EBER,BoFoMa98GMR}, and finally the most recent results in Fonseca and Leoni~\cite{FonLeo01LSC}. Roughly, these results require (in addition to quasiconvexity) that $f$ is Borel, such that $f^\#$ exists, and that, for every $(x_0,y_0)\in\Omega\times\R^m$ and $\varepsilon > 0$, there exists $\delta >0$ such that $|x-x_0|+|y-y_0| < \delta$ implies $f(x_0,y_0,A) - f(x,y,A) \leq \varepsilon(1+f(x,y,A))$ for all $A\in\R^{m\times d}$ and that $|x-x_0| \leq\delta$ implies $f^\#(x_0,y,A)-f^\#(x,y,A)\leq \varepsilon(1+f^\#(x,y,A))$ for all $(y,A)\in\R^m\times\R^{m\times d}$.

Reasoning by analogy with the cases where $f$ has superlinear growth over $\W^{1,p}(\Omega;\R^m)$ or where $f$ has linear growth and either $m=1$ (so that $f\colon\Omega\times\R\times\R^d\to\R$) or $f$ is $u$-independent (so that $f=f(x,A)$), the implication is that~\eqref{eqfunctional} should hold under just the assumptions that $f$ be quasiconvex in the final variable and possesses sufficient growth and regularity to ensure that the right hand side of~\eqref{eqfunctional} is well-defined over $\BV(\Omega;\R^m)$. This result is, essentially, our Theorem~\ref{wsclscthm} below.

To avoid the use of extraneous hypotheses, we must pass from $\lim_j f(x,u_j(x),\nabla u_j(x))$ to $\lim_j f(x_0,u(x_0),\nabla u_j(x))$ in~\eqref{eqL1lscineq} using only the behaviour of the sequence $(u_j)_j$ rather than any special properties of the integrand. In order to do this, we must improve our understanding of the behaviour of weakly* convergent sequences in $\BV(\Omega;\R^m)$, particularly under the blow-up rescaling $x\mapsto (x-x_0)/r$. These are much more poorly behaved than weakly convergent sequences in $\W^{1,p}(\Omega;\R^m)$ for $p>1$ thanks to interactions between $(u_j)_j$ and $(Du_j)_j$ in the limit as $j\to\infty$, see Example~\ref{exnonelementarylifting}. In the reflexive Sobolev case, powerful truncation techniques~\cite{FoMuPe98ACOE,Kris99LSSW} mean that these interactions can be neglected in the sequence under consideration, but no such tools are available in $\BV(\Omega;\R^m)$ when $m>1$.

The main contributions of this paper are the development of a new theory for understanding weakly* convergent sequences and blow-up procedures in $\BV(\Omega;\R^m)$, together with the use of this theory to provide a new proof for the integral representation of the weak* relaxation $\Fcalrw$ of $\Fcal$ to $\BV(\Omega;\R^m)$ under weaker hypotheses. This representation is valid for Carath\'eodory integrands and does not require the $(x,y)$-localisation properties of $f$ which have previously been used:

\begin{alphtheorem}\label{wsclscthm}
Let $f\colon\overline{\Omega}\times\R^m\times\R^{m\times d}\to\R$ where $d\geq 2$ and $m\geq 1$ be such that
\begin{enumerate}[(i)]
\item\label{wsthmhypoth1} $f$ is a Carath\'eodory function whose recession function $f^\infty$ exists in the sense of Definition~\ref{defrecessionfunction} and satisfies $f^\infty\geq 0$;
\item\label{wsthmhypoth2} $f$ satisfies a growth bound of the form 
\begin{equation}\label{eqintrogrowth}
-C(1+|y|^p+h(A))\leq f(x,y,A)\leq C(1+|y|^{d/(d-1)}+|A|),
\end{equation}
for some $C>0$, $p\in[1,d/(d-1))$, $h\in\C(\R^{m\times d})$ satisfying $h^\infty\equiv 0$, and for all $(x,y,A)\in\overline{\Omega}\times\R^m\times\R^{m\times d}$;
\item $f(x,y,\frarg)$ is quasiconvex for every $(x,y)\in\overline{\Omega}\times\R^m$.
\end{enumerate}
Then the sequential weak* relaxation of $\Fcal$ to $\BV(\Omega;\R^m)$ is given by
\begin{align*}
\Fcalrw[u]&=\int_\Omega f(x,u(x),\nabla u(x))\;\dd x+\int_\Omega f^\infty\left(x,u(x),\frac{\dd D^c u}{\dd|D^c u|}(x)\right)\;\dd|D^cu|(x)\\
&\qquad+\int_{\Jcal_u} K_f[u](x)\;\dd\Hcal^{d-1}(x).
\end{align*}
\end{alphtheorem}
The hypotheses of Theorem~\ref{wsclscthm} are essentially the optimal conditions under which $\Fcalrw$ is guaranteed to be finite over all of $\BV(\Omega;\R^m)$, and this result consequently identifies the extension of $\Fcal$ to $\BV(\Omega;\R^m)$ for the purposes of the Direct Method whenever such an extension is well-defined. Theorem~\ref{wsclscthm} can therefore be seen as a $\BV$-version of the optimal Sobolev lower semicontinuity theorems obtained by Acerbi \& Fusco~\cite{AceFus84SPCV} and Marcellini~\cite{Marc85AQFL}.

Example~\ref{expositiverecession} below shows that, unlike in the case where $f=f(x,A)$ does not depend explicitly on $u$ (see~\cite{KriRin10RSI}), the positivity assumption $f^\infty \geq 0$ is necessary for $\Fcalrw$ to be finite and for a general integral formula to hold. In fact, assuming $f^\infty\geq 0$ alone without also requiring the lower bound in~\eqref{wsthmhypoth2} is not enough, see Example~\ref{exlscbreaks}. The situation here is subtle, and~\eqref{wsthmhypoth2} can be improved slightly to a bound that is truly sharp in this regard at the expense of a more complicated statement, see Definition~\ref{defrepresentationf} and Theorem~\ref{thmwscrelaxation}.

The hypotheses of Theorem~\ref{wsclscthm} are valid in situations where the $(x,y)$-localisation hypotheses used in~\cite{FonMul93RQFB,FonLeo01LSC} do not hold and also where partial coercivity fails: for example, the function $f(x,y,A)=(1+|y|)^{-1} |y|^{1-|x|}|A|$ defined on $\bB^d\times\R^m\times\R^{m\times d}$, which violates~(H4) and Equation~{(1.15)} in~\cite{FonMul93RQFB}  and~\cite{FonLeo01LSC} respectively, and functions of the form $f(x,y,A)=[B(x,y):A]^+$ where $B(x,y)\in\C(\Omega\times\R^m;\R^{m\times d})$ is bounded ($A:B$ is the usual Frobenius inner product between matrices and $[\frarg]^+$ denotes the positive part).
 

We emphasise that our $\Fcalrw$ is the relaxation of $\Fcal$ to $\BV(\Omega;\R^m)$ with respect to sequential weak* convergence, whereas the relaxation of interest from the perspective of some applications and the one which is the subject of the earlier works works~\cite{FonMul93RQFB,AmbPal93IRRF,BoFoMa98GMR,FonLeo01LSC,AviGig92MCGR} is $\Fcalro$, the relaxation of $\Fcal$ with respect to strong $\Lp^1(\Omega;\R^m)$ convergence in $\BV(\Omega;\R^m)$. In the absence of coercivity $\Fcalro$ might be strictly less than $\Fcalrw$, but it is always the case that $f$ in these applications is partially coercive in the sense that $g(x,y)|A| \leq f(x,y,A)\leq Cg(x,y)(1+|A|)$ for some continuous $g\colon\Omega\times\R^m\to[0,\infty)$ and $C>0$. In these circumstances it is therefore reasonable to expect that the problem of computing $\Fcalro$ reduces to that of computing $\Fcalrw$ `locally' in regions of $\Omega\times\R^m$ where $g(x,y)>0$ (and this is in fact the strategy followed in previous works), and so our work opens the possibility of new progress in this area. Indeed, in the sequel to this paper,~\cite{RinSha18RPCF} we use Theorem~\ref{wsclscthm} to derive an integral representation for $\Fcalro$ valid under improved hypotheses on the integrand $f$.


Theorem~\ref{wsclscthm} assumes that $f^\infty$ exists in a stronger sense than has been classically required in the literature (see Definition~\ref{defrecessionfunction}), where only the upper recession function $f^\#$ is used. In fact, the other properties required of $f$ in~\cite{FonLeo01LSC} imply that their $f^\#$ must exist in the sense of Definition~\ref{defrecessionfunction} at every point of continuity for $f^\#$, that $f^\#$ must be lower semicontinuous, and such that $f^\#(\frarg,y,\frarg)$ is continuous in $(x,A)$ for every $y\in\R^m$.

Our proof of the lower semicontinuity "$\geq$" component of Theorem~\ref{wsclscthm} is based on the idea of understanding joint limits for pairs $(u_j,Du_j)_j$ under weak* convergence as objects in the graph space $\Omega\times\R^m$, rather than solely in $\Omega$. To do this, we develop a theory of \emph{liftings} (ideas of this type were first introduced by Jung \& Jerrard in~\cite{JunJer04SCML}), which replaces functions $u\in\BV(\Omega;\R^m)$ by graph-like measures $\gamma\asc{u}:=\gr^{u}_\#(Du)\in\mbfM(\Omega\times\R^m;\R^{m\times d})$, where $\gr^u\colon x\mapsto (x,u(x))$ is the graph map of $u$. Using a Reshetnyak-type perspective construction, the functional $\Fcal$ can be generalised to one defined on the space of liftings. The key point is that working in this more general setting means that we can think of sequences $(f(x,u_j(x),\nabla u_j(x)))_j$ as converging weakly* to Radon measures in $\mbfM^+(\Omega\times\R^m)$ rather than merely in $\mbfM^+(\Omega)$. This allows us to estimate $\Fcalrw$ more precisely from below by computing Radon--Nikodym derivatives at points $(x,u(x))\in\Omega\times\R^m$ with respect to the total variation of the (elementary) lifting $\gamma\asc{u}$ rather than merely at points $x\in\Omega$ with respect to the derivative $|Du|$. In order to carry out these computations, we lay out a framework of generalized Young measures associated to liftings under weak* convergence that allows us to ``freeze the $u$-variable'' for a wide class of functionals with linear growth by employing robust tools from Geometric Measure Theory. In particular we use a new type of Besicovitch Derivation Theorem which allows us to differentiate in $\Omega\times\R^m$ with respect to graphical measures of the form $\eta=\gr^u_\# \lambda$, $\lambda\in\mbfM^+(\Omega)$ using very general families of sets $B(x,r)\times B(y,R)\subset\Omega\times\R^m$ (see the discussion which precedes Theorem~\ref{thmgeneralisedbesicovitch}).

The idea of understanding the joint weak limits of sequences of pairs $(u_j,\nabla u_j)_j$ by considering instead objects defined over the graph space $\Omega\times\R^m$ has of course been explored before. This strategy is successfully followed in~\cite{DMas79IRBV} to identify $\Fcal_{**}$ in the case $m=1$ and, for $m>1$ with $f$ coercive, convex, and isotropic in the final variable, {currents are used} to identify $\Fcal_{**}$ in~\cite{AviGig91VIMB,AviGig92MCGR}. More generally, `graph-like' objects have been widely used to better understand nonlinear functionals in the Calculus of Variations and Geometric Measure Theory in the context of currents (in particular, Cartesian currents) and varifolds~\cite{Almg66PP,FedFle60NIC,Fede69GMT,Alla72FVV,GiMoSo98CCCV1,GiMoSo98CCCV2,Menn16WDFV}. Liftings seem to be situated at a sweet spot between the usual techniques of the Calculus of Variations and the higher abstractions of geometric analysis, admitting a simple yet surprisingly powerful calculus which appears to be well-suited for functionals of this type. We hope that this tool will prove to be useful in a variety of related problems.

The final step in the proof of Theorem~\ref{wsclscthm} is to show that the lower bound for $\Fcalrw$ obtained via our theory of liftings is optimal. This is equivalent to constructing weak* \emph{approximate recovery sequences} $(u^\varepsilon)_j\subset\C^\infty(\Omega;\R^m)$ for $\Fcalrw$ and $\varepsilon>0$ which are such that $u^\varepsilon\wsc u$ and
\[
\limsup_{j\to\infty}\left|\Fcal[u^\varepsilon_j]- \Fcalrw[u]\right|\leq\varepsilon.
\]
Perhaps surprisingly, Example~\ref{exbadrecoveryseq} demonstrates that it is not always possible to construct genuine weak* recovery sequences which satisfy
\[
u_j\wsc u\quad\text{and}\lim_{j\to\infty}\Fcal[u_j]=\Fcalrw[u],
\]
even for $f$ continuous and convex in the final variable: in contrast to the $u$-independent cases where $f=f(x,A)$ or the scalar valued case where $u\colon\Omega\to\R$, it can occur in the absence of coercivity that $\Fcalrw$ admits a global minimiser for which \emph{no} weakly* convergent minimising sequence exists. Nevertheless, we are able to construct approximate recovery sequences for $\Fcalrw$ using a novel `cut and paste' technique based around the rectifiability of the measure $K_f[u]\Hcal^{d-1}\restrict\Jcal_u$ combined with Young measure techniques.

This paper is organised as follows: after notation is established and preliminary results and concepts are introduced in Section~\ref{secpreliminaries}, liftings are defined and their theory developed in Section~\ref{chapliftings}. We prove a structure theorem, establish the convergence and compactness properties of liftings, and we show how $\Fcal$ can be extended to a functional $\Fcal_\Lrm$ defined on the space of liftings. In Section~\ref{chapyoungmeasures}, we develop a theory of Young measures associated to liftings, including representation and compactness theorems. Section~\ref{chaptangentliftingyms} introduces tangent Young measures and their associated Jensen inequalities, which together suffice to implement an optimal weak* blow-up procedure and deduce the lower semicontinuity component of Theorem~\ref{wsclscthm}. Finally, in Section~\ref{chaprecoveryseqs} we construct approximate recovery sequences for $\Fcalrw$, before combining these with the results Section~\ref{chaptangentliftingyms} to state and prove Theorem~\ref{thmwscrelaxation}, which is a slightly more general version of Theorem~\ref{wsclscthm}.

\section*{Acknowledgements}
The authors would like to thank Irene Fonseca, Jan Kristensen and Neshan Wickramasekera for several helpful discussions related to this paper. 

\section{Preliminaries}\label{secpreliminaries}

Throughout this work, $\Omega\subset\R^d$ will always be assumed to be a bounded open domain with compact Lipschitz boundary $\pd\Omega$ in dimension $d\geq 2$, and $\bB^k$, $\pd\bB^k$ will denote the open unit ball in $\R^k$ and its boundary (the unit sphere) respectively. The open ball of radius $r$ centred at $x\in\R^k$ is $B(x,r)$, although we will sometimes write $B^k(x,r)$ if the dimension of the ambient space needs to be emphasised for clarity. The volume of the unit ball in $\R^k$ will be denoted by $\omega_k:=\Lcal^k(\bB^k)$, where $\Lcal^k$ is the usual $k$-dimensional Lebesgue measure. We will write $\R^{m\times d}$ for the space of $m\times d$ real valued matrices, and $\id_{\R^m}$ for the identity matrix living in $\R^{m\times m}$. The map $\pi\colon\Omega\times\R^m\to\Omega$ denotes the projection $\pi((x,y))=x$, and $T^{(x_0,r)}\colon\R^d\to\R^d$, $T^{(x_0,r),(y_0,s)}\colon\R^d\times\R^m\to\R^d\times\R^m$ represent the homotheties $x\mapsto (x-x_0)/r$ and $(x,y)\mapsto((x-x_0)/r,(y-y_0)/s))$. Tensor products $a\otimes b\in\R^{m\times d}$ and $f\otimes g$ for vectors $a\in\R^m$, $b\in\R^d$, and real valued functions $f$, $g$, are defined componentwise by $(a\otimes b)_{i,j}=a_ib_j$ and $(f\otimes g)(x,y)=f(x)g(y)$ respectively.

The closed subspaces of $\BV(\Omega;\R^m)$ and $\C^\infty(\Omega;\R^m)$ consisting only of the functions satisfying $(u):=\dashint_\Omega u(x)\;\dd x=0$ are denoted by $\BV_\#(\Omega;\R^m)$ and $\C^\infty_\#(\Omega;\R^m)$ respectively. We shall use the notation $(u)_\Omega$ when the domain of integration might not be clear from context, as well as the abbreviation $(u)_{x,r}:=(u)_{B(x,r)}$.  We shall sometimes use subscripts for clarity when taking the gradient with respect to a partial set of variables: that is, if $f=f(x,y)\in\C^1(\Omega\times\R^m)$ then $\nabla_x f=(\pd_{x_1}f,\pd_{x_2}f,\ldots,\pd_{x_d}f)$ and $\nabla_yf=(\pd_{y_1}f,\pd_{y_2}f,\ldots,\pd_{y_m}f)$.

\subsection{Measure theory}\label{subsecmeasuretheory}
For a separable locally convex metric space $X$, the space of vector-valued Radon measures on $X$ taking values in a normed vector space $V$ will be written as $\mbfM(X;V)$ or just $\mbfM(X)$ if $V=\R$. The cone of positive Radon measures on $X$ is $\mbfM^+(X)$, and the set of elements $\mu\in\mbfM(X;V)$ whose total variation $|\mu|$ is a probability measure, is $\mbfM^1(X;V)$. The notation $\mu_j\wsc\mu$ will denote the usual weak* convergence of measures, and we recall that $\mu_j$ is said to converge to $\mu$ \textbf{strictly} if $\mu_j\wsc \mu$ and in addition $|\mu_j|(X)\to|\mu|(X)$. Given a map $T$ from $X$ to another separable, locally convex metric space $Y$, the pushforward operator $T_\#\colon\Mbf(X;V)\to\mbfM(Y;V)$ is defined by
\[
\ip{\varphi}{T_\#\mu}:=\ip{\varphi\circ T}{\mu},\qquad\varphi\in\C_0(Y).
\]
If $T$ is continuous and proper, then $T_\#$ is continuous when $\mbfM(X;V)$ and $\mbfM(Y;V)$ are equipped with their respective weak* or strict topologies. 

We omit the proof of the following simple lemma:
\begin{lemma}\label{lempushforwardderivative}
Let $\mu\in\mbfM(X;V)$, $\nu\in\mbfM^+(X)$ satisfy $\mu\ll\nu$ and let $T\colon X\to Y$ be a continuous injective map. Then it holds that
\[
\frac{\dd T_\#\mu}{\dd T_\#\nu}\circ T=\frac{\dd \mu}{\dd\nu}\quad\text{ and }\quad |T_\#\mu|=T_\#|\mu|.
\]
\end{lemma} 

Given a function $f\colon X\times V\to\R$ which is positively one-homogeneous in the final variable (that is, $f(x,tA)=tf(x,A)$ for all $t \geq 0$ and $A\in V$) and a measure $\mu\in\mbfM(X;V)$, we shall use the abbreviated notation
\begin{equation}\label{eqreshetnyakrepresentation}
\int_X f(x,\mu):=\int_X f\left(x,\frac{\dd\mu}{\dd|\mu|}(x)\right)\;\dd|\mu|(x).
\end{equation}
We note that, if $T\colon X\to Y$ is an injection then applying Lemma~\ref{lempushforwardderivative} to $T_\#\mu$ and $|T_\#\mu|$ lets us deduce
\begin{align}
\begin{split}\label{eqperspectiveobservation}
\int_Y f(y,T_\#\mu)&=\int_Yf\left(y,\frac{\dd T_\#\mu}{\dd|T_\#\mu|}(y)\right)\;\dd|T_\#\mu|(y)\\
&=\int_Y f\left(y,\frac{\dd T_\#\mu}{\dd T_\#|\mu|}(y)\right)\;\dd T_\#|\mu|(y)\\
&=\int_X f\left(T(x),\frac{\dd\mu}{\dd |\mu|}(x)\right)\;\dd|\mu|(x)=\int_X f(T(x),\mu).
\end{split}
\end{align}

If $\mu$ is a measure on $X\times Y$ then we recall that the \textbf{Disintegration of Measures Theorem} (see Theorem~2.28 in~\cite{AmFuPa00FBVF}) allows us to decompose $\mu$ as the (generalised) product $\mu=\pi_\#|\mu|\otimes\rho$, where $\pi_\#|\mu|$ is the pushforward of $|\mu|$ onto $X$ and $\rho$ is a ($\pi_\#|\mu|$-almost everywhere defined) parametrised measure. Here, $\pi_\#|\mu|\otimes\rho$ is defined (uniquely) via
\[
(\pi_\#|\mu|\otimes\rho)(E\times F):=\int_E\rho_x(F)\;\dd[\pi_\#|\mu|](x)\quad\text{ for all Borel subsets } E\subset X\text{ and } F\subset Y.
\]

For $k\in[0,\infty)$, the $k$-dimensional Hausdorff (outer) measure on $\R^d$ is written as $\Hcal^{k}$ and, if $A\in\Bcal(\R^d)$ is a Borel set satisfying $\Hcal^k(A)<\infty$, its restriction $\Hcal^k\restrict A$ to $A$ defined by $[\Hcal^k\restrict A](\frarg):=\Hcal^k(\frarg\cap A)$ is a finite Radon measure. A set $A\subset\R^m$ is said to be \textbf{countably $\boldsymbol{\Hcal^k}$-rectifiable} if there exists a sequence of Lipschitz functions $f_i\colon\R^k\to\R^d$ ($i \in \N$) such that
\[
\Hcal^k\left(A\setminus\bigcup_{i=1}^\infty f_i(\R^k)\right)=0,
\]
and \textbf{$\boldsymbol{\Hcal^k}$-rectifiable} if in addition $\Hcal^k(A)<\infty$.
We say that $\mu\in\mbfM(\R^d;V)$ is a \textbf{$\boldsymbol k$-rectifiable measure} if there exists a countably $\Hcal^k$-rectifiable set $A\subset\R^d$ and a Borel function $f\colon A\to V$ such that $\mu=f\Hcal^k\restrict A$.


With $A$ assumed to be countably $\Hcal^k$-rectifiable, we can define the Radon--Nikodym derivative for any $\mu\in\mbfM(\R^d)$ with respect to $\Hcal^k\restrict A$, given for $\Hcal^k$-almost every $x\in A$, by
\begin{equation}\label{eqgenderivative}
\frac{\dd\mu}{\dd\Hcal^k\restrict A}(x):=\lim_{r\to 0}\frac{\mu(B(x,r))}{\omega_{k}r^k}.
\end{equation}
The function $\frac{\dd\mu}{\dd\Hcal^k\restrict A}$ is a Radon--Nikodym Derivative in the sense that $\frac{\dd\mu}{\dd\Hcal^k\restrict A}\Hcal^k\restrict A$ is a $k$-rectifiable measure and that we can decompose
\[
\mu=\frac{\dd\mu}{\dd\Hcal^k\restrict A}\Hcal^k\restrict A+\mu^s,\quad\text{ where $\mu^s$ satisfies }\quad\frac{\dd\mu}{\dd\Hcal^k\restrict A}\Hcal^k\restrict A\perp \mu^s,
\]
in analogy with the usual Lebesgue--Radon--Nikodym decomposition.

A measure $\mu\in\mbfM(\R^d;V)$ is said to be admit a ($k$-dimensional) \textbf{approximate tangent space} at $x_0$ if there exists an (unoriented) $k$-dimensional hyperplane $\tau\subset\R^d$ and $\theta\in V$ such that
\[
r^{-k} T_\#^{(x_0,r)}\mu\to\theta\Hcal^k\restrict(\overline{\bB^d}\cap\tau) \quad\text{ strictly in }\mbfM(\overline{\bB^d};V)\quad\text{ as }r\to 0.
\]
The existence of approximate tangent spaces characterises the class of rectifiable measures in the sense that $\mu\in\mbfM(\R^d;V)$ possesses a $k$-dimensional approximate tangent space at $|\mu|$-almost every $x_0\in\R^d$ if and only if $\mu$ is $k$-rectifiable (see Theorem~2.83 in~\cite{AmFuPa00FBVF}). 

The (column-wise) divergence of a measure $\mu\in\mbfM(\R^m;\R^{m\times d})$, written as $\diverg\mu$ or $\nabla\cdot\mu$, is an $\R^d$ row vector-valued distribution on $\R^m$  defined by duality via the formula
\[
\int_{\R^m}\varphi(y)\;\dd(\nabla\cdot\mu)(y)=-\int_{\R^m}\nabla\varphi(y)\;\dd\mu(y)\qquad\text{ for all }\varphi\in\C_c^\infty(\R^m).
\]

\begin{lemma}\label{lemdivmunoatoms}
Any $\mu\in\mbfM(\R^m;\R^{m\times d})$ satisfying $\diverg\mu\equiv 0$ (column-wise) is non-atomic. That is, $\mu(\{a\})=0$ for every $a\in\R^m$.
\end{lemma}

\begin{proof}
Let $g\in C_0^1(\R^m)$ be arbitrary. Defining $g_\lambda(y):=\frac{1}{\lambda}g(a+\lambda(y-a))$ for $\lambda\in\R$, we therefore have that
\[
0=-\ip{g_\lambda}{\diverg\mu}:=-\int_{\supp g_\lambda}\nabla g_\lambda(y)\;\dd\mu(y).
\]
Noting that $\nabla g_\lambda(a)=\nabla g(a)$ and $\norm{\nabla g_\lambda}_\infty=\norm{\nabla g}_\infty$, we can let $\lambda\to\infty$ so that $\mathbbm{1}_{\supp g_\lambda}\nabla g_\lambda\to\mathbbm{1}_{\{a\}}\nabla g(a)$ pointwise before using the Dominated Convergence Theorem to deduce
\[
\nabla g(a) \mu\left(\{a\}\right)=0.
\]
By varying $\nabla g(a)$ through $\R^m$, we see that $\mu(\{a\})=0$ as required.
\end{proof}

\subsection{\texorpdfstring{$BV$}{BV} functions}
Given a function $u\in\BV(\Omega;\R^m)$, we recall the mutually singular decomposition $Du=\nabla\lL\restrict\Omega +D^cu +D^ju$ of the derivative $Du$, where $|D^cu|\ll \Hcal^{d-1}$, $D^ju$ is absolutely continuous with respect to $\Hcal^{d-1}\restrict\Jcal_u$, and $\Jcal_u$ is the countably $\Hcal^{d-1}$-rectifiable jump set of $u$. Each $u\in\BV(\Omega;\R^m)$ admits a precise representative $\widetilde{u}\colon\Omega\to\R^m$ which is defined $\Hcal^{d-1}$-almost everywhere in $\Omega\setminus\Jcal_u$. The \textbf{jump interpolant} associated to $u$ is then the function $u^\theta\colon\Omega\times[0,1]\to\R^m$ defined, up to a choice of orientation $n_u$ for the jump set $\Jcal_u$ of $u$, for $\Hcal^{d-1}$-almost every $x\in\Omega$ by
\begin{equation}\label{eqdefjumpinterpolant}
u^\theta(x):=\begin{cases}
\theta u^-(x)+(1-\theta)u^+(x) &\text{if }x\in\Jcal_u,\\
\widetilde{u}(x) &\text{otherwise.}
\end{cases}
\end{equation}
The need to fix a choice of orientation for $\Jcal_u$ in order to properly define $u^\theta$ is obviated by the fact that $u^\theta$ will only appear in expressions of the form
\[
\int_0^1 \varphi(u^\theta(x))\;\dd\theta
\]
which are invariant of our choice of $n_u$.

Given (the precise representative of) a function $u\in\BV(\Omega;\R^m)$, the function associated to its graph is denoted by $\gr^u\colon x\mapsto (x,u(x))$. If $\mu$ is a measure on $\Omega$ satisfying both $|\mu|\ll\Hcal^{d-1}$ and $|\mu|(\Jcal_u)=0$ (we will usually take $\mu=|Du|\restrict(\Omega\setminus\Jcal_u$), its pushforward under $\gr^u$ then still makes sense as the Radon measure $\gr^u_\#\mu$ on $\Omega\times\R^m$.

A sequence $(u_j)_j\subset\BV(\Omega;\R^m)$ is said to converge \textbf{strictly} to $u\in\BV(\Omega;\R^m)$ if $u_j\to u$ in $\Lp^1(\Omega;\R^m)$ and $Du_j\to Du$ strictly in $\mbfM(\Omega;\R^{m\times d})$ as $j\to\infty$. We say that $u_j$ converges \textbf{area-strictly} to $u$ if $u_j\to u$ in $\Lp^1(\Omega;\R^m)$ and in addition
\[
\int_\Omega\sqrt{1+|\nabla u_j(x)|^2}\;\dd x+|D^s u_j|(\Omega)\to\int_\Omega\sqrt{1+|\nabla u(x)|^2}\;\dd x+|D^s u|(\Omega)
\]
as $j\to\infty$. It is the case that area-strict convergence implies strict convergence in $\BV(\Omega;\R^m)$ and that strict convergence implies weak* convergence. That none of these notions of convergence coincide follows from considering the sequence $(u_j)_j\subset\BV((-1,1))$ given by $u_j(x):=x+(a/j)\sin(jx)$ for some $a\neq 0$ fixed. This sequence converges weakly* to the function $x\mapsto x$ for any $a\in\R\setminus\{0\}$, strictly if and only if $|a|\leq 1$, but (since the function $z\mapsto\sqrt{1+|z|^2}$ is strictly convex away from $0$) never area-strictly. Smooth functions are area-strictly (and hence strictly) dense in $\BV(\Omega;\R^m)$: indeed, if $u\in\BV(\Omega;\R^m)$ and $(u_\rho)_{\rho>0}$ is a family of radially symmetric mollifications of $u$ then it holds that $u_\rho\to u$ area-strictly as $\rho \todown 0$.

If $\Omega\subset\R^d$ is such that $\pd\Omega$ is Lipschitz and compact, then the trace onto $\pd\Omega$ of a function $u\in\BV(\Omega;\R^m)$ is denoted by $u|_{\pd\Omega}\in\Lp^1(\pd\Omega;\R^m)$. The trace map $u\mapsto u|_{\pd\Omega}$ is norm-bounded from $\BV(\Omega;\R^m)$ to $\Lp^1(\pd\Omega;\R^m)$ and is continuous with respect to strict convergence (see Theorem~3.88 in~\cite{AmFuPa00FBVF}). If $u,v\in\BV(\Omega;\R^m)$ are such that $u|_{\pd\Omega}=v|_{\pd\Omega}$, then we shall sometimes simply say that ``$u=v$ on $\pd\Omega$''.

The following proposition, a proof for which can be found in the appendix of~\cite{KriRin10CGYM} (or Lemma~B.1 of~\cite{Bild03CVP} in the case of a Lipschitz domain $\Omega$), states that we can even require that smooth area-strictly convergent approximating sequences satisfy the trace equality ${u_j}|_{\pd\Omega}=u|_{\pd\Omega}$:
\begin{proposition}\label{propfixedbdary}
For every $u\in\BV(\Omega;\R^m)$, there exists a sequence $(u_j)_j\subset\C^\infty(\Omega;\R^m)$ with the property that 
\[
u_j\to u\qquad\text{ area-strictly in }\BV(\Omega;\R^m)\text{ as }j\to\infty \qquad\text{ and }\qquad {u_j}|_{\pd\Omega}=u|_{\pd\Omega}.
\]
Moreover, if $u|_{\pd\Omega}\in\Lp^\infty(\pd\Omega;\R^m)$ we can assume that $(u_j)_j\subset(\C^\infty\cap\Lp^\infty)(\Omega;\R^m)$  and, if $u\in\Lp^\infty(\Omega;\R^m)$, then we can also require that $\sup_j\norm{u_j}_{\Lp^\infty}\leq\norm{u}_{\Lp^\infty}$.
\end{proposition}

\begin{theorem}[Blowing up $\BV$ functions]\label{thmbvblowup}
Let $u\in\BV(\Omega;\R^m)$ and write $\Omega$ as the disjoint union
\[
\Omega=\Dcal_u\cup\mathcal{J}_u\cup \Ccal_u\cup\mathcal{N}_u
\]
and $Du$ as the mutually singular sum
\[
Du=\nabla u\lL+D^ju+D^cu,\qquad \nabla u\lL=Du\restrict\Dcal_u,\quad D^ju=Du\restrict\Jcal_u,\quad D^cu=Du\restrict\Ccal_u,
\]
where $\mathcal{D}_u$ denotes the set of points at which $u$ is approximately differentiable, $\mathcal{J}_u$ denotes the set of jump points of $u$, $\mathcal{C}_u$ denotes the set of points where $u$ is approximately continuous but not approximately differentiable, and $\mathcal{N}_u$ satisfies $|Du|(\Ncal_u)=\mathcal{H}^{d-1}(\mathcal{N}_u)=0$. For $r>0$ and $x\in\Dcal_u\cup\Ccal_u\cup\Jcal_u$, define $u^r\in\BV(\bB^d;\R^m)$ by
\[
u^r(z):=c_r\left(\frac{u(x+rz)-(u)_{x,r}}{r}\right),\quad c_r:=\begin{cases}1&\text{ if }x\in\Dcal_u,\\
r&\text{ if }x\in\Jcal_u,\\
\frac{r^d}{|Du|(B(x,r))}&\text{ if }x\in\Ccal_u.
\end{cases}
\]
Then the following trichotomy relative to $\lL\restrict\Omega+|Du|$ holds:
\begin{enumerate}[(i)]
\item For $\lL$-almost every $x\in\Omega$, 
\[
u^r\to \nabla u(x)\frarg \text{ strongly in }\BV(\bB^d;\R^m)\text{ as $r\todown 0$}.
\]
\item\label{bvblowupjump} For $\Hcal^{d-1}$-almost every $x\in\Jcal_u$,
\[
u^r\to \frac{1}{2}\begin{cases}u^+(x)-u^-(x)\quad\text{ if }\ip{z}{n_{u(x)}}\geq 0\\
u^-(x)-u^+(x)\quad\text{ if }\ip{z}{n_{u(x)}}<0,
\end{cases}
\]
strictly in $\BV(\bB^d;\R^m)$ as $r\downarrow 0$.
\item For $|D^c u|$-almost every $x\in\Omega$ and for any sequence $r_n\downarrow 0$, the sequence $(u^{r_n})_{r_n}$ contains a subsequence which converges weakly* in $\BV(\bB^d;\R^m)$ to a non-constant limit function of the form
\begin{equation}\label{eqcantorblowup}
\eta(x)\gamma\left(\ip{\frarg}{\zeta(x)}\right),\qquad\frac{\dd D^c u}{\dd|D^c u|}(x)=\eta(x)\otimes\zeta(x),
\end{equation}
where $\gamma\in\BV((-1,1);\R)$ is non-constant and increasing. Moreover, if $(u^{r_n})_{n}$ is a sequence converging weakly* in this fashion then, for any $\varepsilon>0$, there exists $\tau\in(1-\varepsilon,1)$ such that the sequence $(u^{\tau r_n})_{n}$ converges strictly in $\BV(\bB^d;\R^m)$ to a limit of the form described by~\eqref{eqcantorblowup}.
\end{enumerate}
In all three situations, we denote $\lim_r u^r$ (or $\lim_n u^{r_{n}}$) by $u^0$. If the base (blow-up) point $x$ needs to be specified explicitly to avoid ambiguity, then we shall write $u^r_{x}$, $u^{r_n}_{x}$ and $u^0_{x}$.
\end{theorem}
\begin{proof}
For points $x\in\Dcal_u\cup\Jcal_u$, the conclusion of Theorem~\ref{thmbvblowup} follows directly from the approximate continuity of $\nabla u$ at $\lL$-almost every $x\in\Omega$ and the existence of the jump triple $(u^+,u^-,n_u)$ for $\Hcal^{d-1}$-almost every $x\in\Jcal_u$. For points $x\in\Ccal_u$, the weak* precompactness of sequences $(u^{r_n})_n$ follows from the fact that $(u^{r_n})_{\bB^d}=0$ and $|Du^{r_n}|(\bB^d)=1$ combined with the weak* compactness of bounded sets in $\BV(\bB^d;\R^m)$. The representation of $u^0$ is non-trivial and can only be obtained through the use of Alberti's Rank One Theorem, see Theorem~3.95 in~\cite{Albe93ROPD}. It remains for us to show that, given a weakly* convergent sequence $u^{r_n}\wsc u$ and $\varepsilon>0$, we can always find $\tau\in(1-\varepsilon,1)$ such that $(u^{\tau r_n})_n$ is strictly convergent in $\BV(\bB^d;\R^m)$.

First note that we can assume that $|Du^{r_n}|\wsc|Du^0|$ in $\mbfM^+(\bB^d)$ since (see for instance Theorem~2.44 in~\cite{AmFuPa00FBVF}) for arbitrary $\mu\in\mbfM(\Omega;\R^{m\times d})$ it is always true that
\[
\Tan(\mu,x)=\frac{\dd\mu}{\dd|\mu|}(x)\Tan(|\mu|,x)\qquad\text{ for $|\mu|$-almost every }x\in\Omega.
\]
Now let $\tau\in(1-\varepsilon,1)$ be such that $|Du^0|(\tau\pd\bB^d)=0$ and $|Du^0|(\tau\bB^d)>0$. This implies that $|Du^{r_n}|(\tau\bB^d)\to|Du^0|(\tau\bB^d)$ and hence that $Du^{r_n}\restrict\tau\bB^d\to Du^0\restrict\tau\bB^d$ strictly. Since 
\[
Du^{\tau r_n}=\frac{|Du|(B(x_0,r_n))}{|Du|(B(x_0,\tau r_n))}T_\#^{(0,\tau)}Du^{r_n}\restrict\tau\bB^d,
\]
and
\[
\frac{|Du|(B(x_0,r_n))}{|Du|(B(x_0,\tau r_n))}=\frac{1}{|Du^{r_n}|(\tau\bB^d)}\to\frac{1}{|Du^0|(\tau\bB^d)}\quad\text{ as }n\to\infty,
\]
we therefore see that
\[
Du^{\tau r_n}\to\frac{1}{|Du^0|(\tau\bB^d)}T_\#^{(0,\tau)}Du^0\restrict\tau\bB^d\quad\text{ strictly in }\BV(\bB^d;\R^m)\text{ as }n\to\infty
\]
as required.
\end{proof}

For $x\in\Jcal_u$, the function $u^0$ gives a 'vertically recentered' description of the behaviour of $u$ near $x$. It will be convenient to have a compact notation for also describing this behaviour when $u$ is not recentered.

\begin{definition}\label{defjumpblowup}
For $x\in\Jcal_u$, define $u^\pm\in\BV(\bB^d;\R^m)$ by
\[
u^\pm(z):=\begin{cases}
u^+(x) &\text{ if }\ip{z}{n_u(x)}\geq 0,\\
u^-(x) &\text{ if }\ip{z}{n_u(x)}<0.
\end{cases}
\]
If the choice of base point $x\in\Jcal_u$ needs to be emphasised for clarity, we shall write $u^\pm_x$.
\end{definition}
This definition is independent of the choice of orientation $(u^+,u^-,n_u)$ and, for $\Hcal^{d-1}$-almost every $x\in\Jcal_u$, the rescaled function $u(x+r\frarg)$ converges strictly to $u^\pm$ as $r\downarrow 0$.

The following proposition was first proved in~\cite{RinSha15SCE}:

\begin{proposition}\label{ctsembedding}
Let $\Omega\subset\mbR^d$ be a domain with Lipschitz boundary and assume that $d>1$. Then the embedding
\[\BV(\Omega;\mbR^m)\hookrightarrow \Lp^{\frac{d}{d-1}}(\Omega;\mbR^m)
\]
is continuous when $\BV(\Omega;\mbR^m)$ is equipped with the topology of strict convergence.
\end{proposition}

\begin{theorem}[The chain rule in $\BV$]\label{bvchainrule}
Let $u\in\BV(\Omega;\R^m)$ and let $f\in\C^1(\R^m;\R^k)$ be Lipschitz. It follows that $v:=f\circ u\in\BV(\Omega;\R^k)$ and that
\begin{align*}
Dv&=\left(\int_0^1\nabla f(u^\theta(x))\;\dd\theta\right)Du\\
&=\nabla f(u)\nabla u\lL\restrict\Omega+\nabla f(u)D^cu+(f(u^+)-f(u^-))\otimes n_u\Hcal^{d-1}\restrict\Jcal_u.
\end{align*}
\end{theorem}

\subsection{Integrands and compactified spaces}


\begin{definition}[Recession functions]\label{defrecessionfunction}
For $f\colon\overline{\Omega}\times\R^m\times\R^{m\times d}\to\R$, define the \textbf{recession function} $f^\infty\colon\overline{\Omega}\times\R^m\times\R^{m\times d}\to\R$ of $f$ by
\begin{equation*}
 f^\infty\left(x,y,A\right)=\lim_{\substack{(x_,y_k,A_k)\to (x,y,A)\\t_k\to\infty}}\frac{f\left(x_k,y_k,t_kA_k\right)}{t_k},
\end{equation*}
whenever the right hand side exists for every $(x,y,A)\in\overline{\Omega}\times\R^m\times\R^{m\times d}$ independently of the order in which the limits of the individual sequences $(x_k)_k\subset\overline{\Omega}$, $(y_k)_k\subset\R^m$, $(A_k)_k\subset\R^{m\times d}$, $(t_k)_k\subset(0,\infty)$ are taken and of the sequences used. The definition of $f^\infty$ implies that, whenever it exists, it must be continuous.
\end{definition}

The following example demonstrates that the assumption in Theorem~\ref{wsclscthm} that $f^\infty \geq 0$ cannot be relaxed independently of the other requirements on $f$.
\begin{example}\label{expositiverecession}
Define $f\in\C(\mbR^2\times\mbR^2)$ by $f(y,A)=-\frac{y_2}{1+|y_2|}\cdot A_1$ so that $f\equiv f^\infty$ and let $v_{k,j}\subset\W^{1,1}((-1,1);\mbR^2)$ be given by
\[
v_{j,k}(x):=\begin{cases}
k(1-\cos(jx),\sin(jx)) &\text{ if }x\in(0,2\pi/j),\\
0 & \text{otherwise}.
\end{cases}
\]
For each fixed $k$, it is clear that $v_{j,k}\wsc 0$ in $\BV((-1,1);\R^2)$ as $j\to\infty$. We can also see, however, that,
\begin{align*}
\lim_{j\to\infty}\int_{-1}^1f(v_{j,k},\nabla v_{j,k})\;\dd x&=\lim_{j\to\infty}kj\int_0^{2\pi/j}\frac{-k\sin(jx)}{1+k|\sin(jx)|}\sin(jx)\;\dd x\\
&=k\int_0^{2\pi}\frac{-(\sin z)^2}{1/k+|\sin z|}\;\dd z\to-\infty\text{ as }k\to\infty.
\end{align*}
It therefore follows that $\Fcalrw[0]=-\infty$. Moreover, for $u\in(\C^1\cap\BV)((-1,1);\mbR^2)$, we can repeat the procedure above with $u_{j,k} = u+v_{j,k}$ in place of $v_{j,k}$  to obtain
	\[
\mathcal{F}_{**}\equiv-\infty\quad\text{ on }\BV((-1,1);\R^2).
	\]
Defining $\widetilde{v}_{j,k}\in\BV((-1,1)^2;\R^2)$ by $\widetilde{v}_{j,k}(x,y)=v_{j,k}(x)$ (since Theorem~\ref{wsclscthm} is only stated for $d>1$) and noting both that $\widetilde{v}\wsc 0$ in $\BV((-1,1)^2;\R^2)$ and
\[
\int_{(-1,1)^2}f(\widetilde{v}_{j,k},\nabla\widetilde{v}_{j,k})\;\dd x = \int_{-1}^1f(v_{j,k},\nabla v_{j,k})\;\dd x,
\]
we see that the conclusion of Theorem~\ref{wsclscthm} cannot hold for $f$ and that no general integral formula is possible in this case.
\end{example}

Not every function $f\colon\R^{m\times d}\to\R$ with linear growth possesses a recession function in the sense of Definition~\ref{defrecessionfunction}, as simple examples show (this even holds for quasiconvex functions, see~\cite{Mull92QFHD}). The \textbf{upper recession function} defined by
\[
  f^\#(A):=\limsup_{t\to\infty}\frac{f(tA)}{t},
\]
however, always exist in $\R$, and is often denoted by $f^\infty$ in the literature.
If $f^\infty$ does exist in the sense of Definition~\ref{defrecessionfunction}, then it is clear that $f^\#=f^\infty$.

Define the \textbf{sphere compactification} $\sigma\R^m$ of $\R^m$ to be the locally convex metric space given by the disjoint union
\[
\sigma\R^m:=\R^m\uplus\infty\pd\bB^m,\qquad\infty\pd\bB^m:=\{\infty e\colon e\in\pd\bB^m\},
\]
endowed with the metric induced by the bijection $i\colon\sigma\R^m\to\overline{\bB^m}$,
\[
i(y)=
\begin{cases}
\frac{y}{1+|y|}&\text{ if }y\in\R^m,\\
e&\text{ if }y=\infty e\in\infty\pd\bB^m.
\end{cases}
\]
That is, $d_{\sigma\R^m}(y,w):=|i(y)-i(w)|$ and we have that $f\colon\sigma\R^m\to\R$ is continuous if and only if $f\circ i^{-1}\colon\overline{\bB^m}\to\R$ is continuous.

The space $\mbfM(\overline{\Omega}\times\sigma\R^m)$ is now abstractly defined as the dual of $\C(\overline{\Omega}\times\sigma\R^m)$, and can also be understood in terms of more familiar spaces of measures as follows: $\mu\in\mbfM(\overline{\Omega}\times\sigma\R^m)$ if and only if there exist $\rho\in\mbfM(\overline{\Omega}\times\R^m)$ and $\eta\in\mbfM(\overline{\Omega}\times\pd\bB^m)$ such that
\[
\mu(E)=\rho(E\cap(\overline{\Omega}\times\R^m))+\eta\left(\{(x,e)\in\overline{\Omega}\times\pd\bB^m\colon (x,\infty e)\in E\cap(\overline{\Omega}\times\infty\pd\bB^m)\}\right)
\]
for every Borel subset $E\subset\overline{\Omega}\times\sigma\R^m$.

A function $f$ defined on $\overline{\Omega}\times\R^m\times\R^{m\times d}$ admits a canonical extension $h$ defined on the partially compactified space $\overline{\Omega}\times\sigma\R^m\times\R^{m\times d}$, where $h$ is defined by
\begin{equation}\label{eqextensiondef}
h(x,y,A):=\begin{cases}
f(x,y,A)&\quad\text{for }(x,y,A)\in\overline{\Omega}\times\R^m\times\R^{m\times d},\\
\displaystyle\lim_{(x_j,y_j,A_j)\to(x,y,A)} f(x_j,y_j,A_j)&\quad \text{for }(x,y,A)\in\overline{\Omega}\times\infty\pd\bB^m\times\R^{m\times d},
\end{cases}
\end{equation}
whenever the limit appearing in~\eqref{eqextensiondef} exists independently of our choice of convergent sequence $((x_j,y_j,A_j))_j\subset\overline{\Omega}\times\R^m\times\R^{m\times d}$. This method of extension is canonical in the sense that $f\in\C(\overline{\Omega}\times\R^m\times\R^{m\times d})$ occurs as the restriction $f=g|_{\Omega\times\R^m\times\R^{m\times d}}$ for some $g\in\C(\overline{\Omega}\times\sigma\R^m\times\R^{m\times d})$ if and only if $g=h$.

Given $c\in\R^m$ we extend the addition operator $y\mapsto y+c$ continuously from $\R^m$ to $\sigma\R^m$ in this way by setting
\begin{equation}\label{eqextendedaddition}
y+c:=\begin{cases}
y+c \qquad & \text{ if }y\in\R^m,\\
y \qquad & \text{ if }y\in\infty\pd\bB^m.
\end{cases}
\end{equation}

\begin{definition}\label{defelemintegrands}
A function $f\in \C(\overline{\Omega}\times\R^m\times\R^{m\times d})$ is a member of $\mbfE(\Omega\times\R^m)$ if there exists a function $g_f\in\C(\overline{\Omega}\times\sigma\R^m\times\overline{\bB^{m\times d}})$ which is such that
\[
f(x,y,A)=(1+|A|)g_f\left(x,y,\frac{A}{1+|A|}\right)\quad\text{ for }(x,y,A)\in\overline{\Omega}\times\R^m\times\R^{m\times d}.
\]
\end{definition}
We see that $f\in\mbfE(\Omega\times\R^m)$ implies $|f(x,y,A)|\leq C(1+|A|)$ (with $C=\norm{g_f}_\infty$) for all $(x,y,A)\in\overline{\Omega}\times\sigma\R^m\times\R^{m\times d}$. and that the recession function $f^\infty$ exists. In addition, both $f$ and $f^\infty$ admit extensions to $\C(\overline{\Omega}\times\sigma\R^m\times\R^{m\times d})$ in the sense of~\eqref{eqextensiondef}. Note however that, as the example $f(y,A)=\exp(-(|y|-|A|)^2)|A|$ demonstrates, the existence of continuous extensions for $f$ and $f^\infty$ does not guarantee that $f\in\mbfE(\Omega\times\R^m)$. For this to be the case, we must also require that
\begin{equation}\label{eqextendedffore}
h(x,y,A)=\lim_{j\to\infty}\frac{ f(x_j,y_j,t_jA_j)}{t_j}\qquad\text{ for }(x,y,A)\in\overline{\Omega}\times\infty\pd\bB^m\times\R^{m\times d},
\end{equation}
for any sequences $((x_j,y_j,A_j))_j\subset\overline{\Omega}\times\R^m\times\R^{m\times d}$, $(t_j)_j\subset[0,\infty)$ such that $(x_j,y_j,A_j)\to(x,y,A)$ and $t_j\to\infty$, where $h$ is the extension of $f^\infty$ to $\overline{\Omega}\times\sigma\R^m\times\R^{m\times d}$ given by~\eqref{eqextensiondef}. For $f\in\mbfE(\Omega\times\R^m)$, this limit always exists by virtue of the continuity of $g_f$ at points $(x,y,A)\in\overline{\Omega}\times\sigma\R^m\times\pd\bB^{m\times d}$.

Our interest in integrands which admit extensions to $\overline{\Omega}\times\sigma\R^m\times\R^{m\times d}$ stems from the fact that, in order to compute limits of the form
\begin{equation}\label{eqlimitwewant}
\lim_{j\to\infty}\int_\Omega f(x,u_j(x),\nabla u_j(x))\;\dd x
\end{equation}
for weakly* convergent sequences $(u_j)_j\subset\W^{1,1}(\Omega;\R^m)$, it is necessary that the extension $h$ of $f^\infty$ exists in some sense. This can be seen by considering sequences of the form $u_j(x):=j^{d-1}u(jx)$ (where $u\in\C_c(\R^d;\R^m)$ is fixed), which are such that $|u_j(x)|\to\infty$ on $\supp|\nabla u_j|\lL\restrict\Omega$. It turns out (see Example~\ref{exlscbreaks} below in conjunction with Proposition~\ref{lemextendedrepresentation}) that it is not sufficient to require that this extension exists as the limit~\eqref{eqextensiondef} (with $f=f^\infty$). On the other hand, the requirement that $h$ exist in the sense of the limit~\eqref{eqextendedffore} is clearly stronger than necessary, since it precludes us from considering any integrands which are unbounded in $y$ such as $f(x,y,A)=|y|$, for which we can compute representation. Definition~\ref{defextendedrecessionfn} below provides the optimal existence requirement for an extension of $f^\infty$ from the perspective of computing~\eqref{eqlimitwewant} for as wide a class of integrands $f$ as possible:

\begin{definition}\label{defextendedrecessionfn}
Given $f\colon\overline{\Omega}\times\R^m\times\R^{m\times d}\to\R$, define the \textbf{extended recession function} $\sigma f^\infty\colon\overline{\Omega}\times\sigma\R^m\times\R^{m\times d}\to\R$ of $f$ by
\begin{equation}\label{eqdefweakextendedrecessionfunction}
\sigma f^\infty(x,y,A)=\lim_{\substack{(x_j,y_j,A_j)\to(x,y,A)\\ t_j\to\infty\\ |y_j|^{d/(d-1)}=\Ocal(t_j)}}\frac{ f(x_j,y_j,t_jA_j)}{t_j},
\end{equation}
whenever the right hand side exists for every $(x,y,A)\in\overline{\Omega}\times\sigma\R^m\times\R^{m\times d}$ independently of which sequences $t_j\uparrow\infty$ and $((x_j,y_j,A_j))_j\subset\overline{\Omega}\times\R^m\times\R^{m\times d}$ are used provided that the constraint $|y_j|^{d/(d-1)}=\Ocal(t_j)$ is satisfied, that is,
\[
\limsup_{j\to\infty}\frac{|y_j|^{d/(d-1)}}{t_j}<\infty.
\]
The definition of $\sigma f^\infty$ implies that, whenever it exists, $\sigma f^\infty$ is continuous and also that $f^\infty$ exists and satisfies $f^\infty=\sigma f^\infty\restrict(\overline{\Omega}\times\R^m\times\R^{m\times d})$.
\end{definition}
Thanks to the existence of the limit~\eqref{eqextendedffore}, we see that $\sigma f^\infty$ exists for all $f\in\mbfE(\Omega\times\R^m)$.
\begin{lemma}\label{prelimlemrecessioncontrol}
Let $f\colon\overline{\Omega}\times\R^m\times\R^{m\times d}$ be such that $f^\infty$ exists and $f^\infty\equiv 0$. Then, for any $\varepsilon>0$ and $K\Subset\R^m$, there exists $R>0$ such that $|A|\geq R$ implies
\[
|f(x,y,A)|\leq\varepsilon(1+|A|)\qquad\text{ for all }(x,y)\subset\overline{\Omega}\times K.
\]
If $f$ is such that $\sigma f^\infty$ exists with $\sigma f^\infty\equiv 0$ then, for any $\varepsilon>0$ and $k>0$, there exists $R>0$ such that $|A|\geq R$ implies 
\[
|f(x,y,A)|\leq\varepsilon(1+|A|)\qquad\text{ for all }x\in\overline{\Omega}\text{ and }y\in\R^m\text{ satisfying }|y|^{d/(d-1)}\leq k(1+|A|).
\]
\end{lemma}

\begin{proof}
To prove the first statement, assume for a contradiction that there exists a sequence of points $((x_k,y_k,A_k))_k\subset\overline{\Omega}\times K\times \R^{m\times d}$ such that $|A_k|\to\infty$ and $|f(x_k,y_k,A_k)|>\varepsilon(1+|A_k|)$ for some fixed $\varepsilon>0$. By passing to a subsequence, we can assume that $((x_k,y_k))\to(x,y)$ in $\overline{\Omega}\times K$ and that $A'_k:=A_k/(1+|A_k|)$ converges to some limit $B\in\pd\mathbb{B}^{m\times d}$ and $t_k:=(1+|A_k|)\to\infty$. Taking the limit in $f(x_k,y_k,A_k)/(1+|A_k|)=t_k^{-1}f(x_k,y_k,t_kA'_k)$ we would then have that $|f^\infty(x,y,B)|\geq\varepsilon$, a contradiction.

To prove the second statement, we proceed similarly by assuming that there exists a sequence of points $((x_j,y_j,A_j))_j\subset\overline{\Omega}\times \R^m\times \R^{m\times d}$ such that $|A_j|\to\infty$, $|y_j|^{d/(d-1)}\leq k(1+|A_j|)$, and $|f(x,y,A)|>\varepsilon(1+|A|)$. Letting $t_j:=1+|A_j|$ and passing to a subsequence we can assume that $((x_j,y_j,A_j/(1+|A_j|)))_j$ converges to some limit $(x,y,B)\in\overline{\Omega}\times\sigma\R^m\times\pd\bB^{m\times d}$. Since $|y_j|^{d/(d-1)}\leq k t_j$, the definition of $\sigma f^\infty$ then implies that
\[
0=|\sigma f^\infty(x,y,B)|=\left|\lim_{j\to\infty}\frac{f(x_j,y_j,t_j A_j/(1+|A_j|))}{t_j}\right|\geq\varepsilon,
\]
which gives us the required contradiction.
\end{proof}

\begin{definition}[Representation integrands]\label{defrepresentationf}
A function $f\colon\overline{\Omega}\times\R^m\times\R^{m\times d}\to\R$ is said to be a member of $\mbfR(\Omega\times\R^m)$ if
$f$ is Carath\'eodory and its recession function $f^\infty$ exists. We shall primarily be interested in the following subsets of $\Rbf(\Omega\times\R^m)$, defined by the growth bounds (which are understood to hold uniformly in their respective parameters for all $(x,y,A)\in\overline{\Omega}\times\R^m\times\R^{m\times d}$) satisfied by their members:
\begin{itemize}
\item $f\in\RL(\Omega\times\R^m)$ if there exists an exponent $p\in[1,d/(d-1))$ and a constant $C>0$ such that $|f(x,y,A)|\leq C(1+|y|^p+|A|)$,
\item $f\in\RBVw(\Omega\times\R^m)$ if there exists $C>0$ and a function $h\in\RL(\Omega\times\R^m)$ with $h\geq 0$ for which $\sigma h^\infty$ exists and satisfies $\sigma h^\infty\equiv 0$ such that
\begin{equation}\label{eqBVwcond}
- h(x,y,A)\leq f(x,y,A)\leq C(1+|y|^{d/(d-1)}+|A|).
\end{equation}
\end{itemize}
\end{definition}
The classes $\RL(\Omega\times\R^m)$ and $\RBVw(\Omega\times\R^m)$ are named as such because they represent the largest classes of integrands to which we will refer whilst making statements about liftings and weakly* convergent $\BV$-functions. In particular, the conclusion of Theorem~\ref{wsclscthm} also holds for all $f\in\RBVw(\Omega\times\R^m)$. The following example demonstrates that the lower bound which we require for members of $\RBVw(\Omega\times\R^m)$ is optimal in the sense that Theorem~\ref{wsclscthm} does not hold if we merely assume that $f\colon\overline{\Omega}\times\R^m\times\R^{m\times d}\to\R$ satisfies $|f(x,y,A)|\leq C(1+|A|)$, is such that $f^\infty$ exists with $f^\infty\geq 0$, and is quasiconvex (or even convex) in the final variable.

\begin{example}\label{exlscbreaks}
For $d\geq 2$, fix $e\in\pd\bB^{d+1}$ and let $\overline{u}\in(\W^{1,1}\cap\C^\infty_0)(\bB^d;\pd B^{d+1}(e,1)\setminus\{0\})$ be a homeomorphism which can be extended continuously to a map defined on $\overline{\bB^d}$ by setting $u(\pd\bB^d)=0$. For example, we can take $\overline{u}$ to be the composition $w^{-1}\circ v$ of the dilation map $v\colon x\in\bB^d\to x/(1-|x|)$ with the inverse of the stereographic projection $w\colon \pd B^{d+1}(e,1) \setminus\{0\}\to\R^d$. It follows that $x\mapsto \overline{u}(x)/|\overline{u}(x)|$ is a homeomorphism between $\bB^d$ and the upper hemisphere $\{y\in\pd\bB^{d+1}\colon y\cdot e>0\}$.

As $\overline{u}$ is not constant, there exists $\varphi\in\C_0(\bB^d;\R^{(d+1)\times d})$ such that
\[
\int_{\bB^d}\nabla \overline{u}(x):\varphi(x)\;\dd x>0.
\]
Since $\varphi|_{\pd\bB^d}=0$, we can define the zero-homogeneous function $B\in\C(\R^{d+1};\R^{(d+1)\times d})$ by
\[
{B}(y):=\begin{cases}
\varphi\left(\left(\frac{\overline{u}}{|\overline{u}|}\right)^{-1}\left(\frac{y}{|y|}\right)\right)&\quad\text{ if }y\cdot e>0 ,\\
\varphi\left(\left(\frac{\overline{u}}{|\overline{u}|}\right)^{-1}\left(-\frac{y}{|y|}\right)\right)&\quad\text{ if }y\cdot e<0 ,\\
0 &\quad\text{ if }y\cdot e=0.
\end{cases}
\]
By construction, we have that $B$ satisfies
\[
\int_{\bB^d}\nabla \overline{u}(x):B(\overline{u}(x))\;\dd x=\int_{\bB^d}\nabla \overline{u}(x):\varphi(x)\;\dd x>0.
\]
Since the function $B$ is zero-homogeneous and $u(x)\neq 0$ for $x\in\bB^d$, for every $x\in\bB^d$ it holds that 
\[
s\nabla \overline{u}(x):B(s\overline{u}(x))=s\nabla \overline{u}(x):B(\overline{u}(x))<s^{d/(d-1)}|u(x)|^{d/(d-1)} 
\]
for all $s>$ sufficiently large. We can therefore find $s>1$ such that, if we define $u\in(\W^{1,1}\cap\C_0^\infty)(\bB^d;\R^{d+1})$ by $u(x)=s\overline{u}(x)$, it holds that
\begin{equation}\label{eqexampleineq}
\int_{\{\nabla u(x)<|u(x)|^{d/(d-1)}\}}\nabla {u}(x):B({u}(x))\;\dd x>0.
\end{equation}

Now, define $f\colon\R^{d+1}\times\R^{(d+1)\times d}\to\R$ by
\[
f(x,y,A):=\begin{cases}
-|y|^{d/(d-1)}-\log(A:B(y)-|y|^{d/(d-1)}+1) &\quad\text{ if }A:B(y)\geq|y|^{d/(d-1)},\\
-A:B(y)&\quad\text{ if }A:B(y)<|y|^{d/(d-1)},
\end{cases}
\]
where $A:B=\sum_{i,j}A_{ij}B_{ij}$ denotes the Frobenius product defined on matrices in $\R^{(d+1)\times d}$. Note that $f\in\C(\R^{d+1}\times\R^{(d+1)\times d})$, satisfies $|f(y,A)|\leq 2\norm{B}_\infty|A|$, and is such that $f^\infty$ exists and is given by the formula
\[
f^\infty(y,A)=\begin{cases}
0 &\quad\text{ if }A:B(y)\geq 0,\\
-A:B(y)&\quad\text{ if }A:B(y)< 0.
\end{cases}
\]
In particular, $f^\infty\geq 0$ on $\R^{d+1}\times\R^{(d+1)\times d}$. We also note that, whilst $f^\infty$ extends continuously to a non-negative function defined on all of $\sigma\R^{d+1}\times\R^{(d+1)\times d}$ in the sense of~\eqref{eqextensiondef} (with $f=f^\infty$), $\sigma f^\infty$ does not exist according to Definition~\ref{defextendedrecessionfn}.

Finally, we claim that $f(y,\frarg)$ is convex on $\R^{(d+1)\times d}$ for every $y\in\R^{d+1}$: this follows from the fact that the function
\[
t\mapsto\begin{cases}
-a-\log(t-a+1) &\quad\text{ if }t\geq a,\\
-t&\quad\text{ if }t<a,
\end{cases}
\]
is convex on $\R$ for each fixed $a\geq 0$ and that the map $A\mapsto A:B(y)$ is linear.

Next, define the functional $\Fcal\colon\W^{1,1}(\bB^d;\R^{d+1})\to\R$ by
\[
\Fcal[u]:=\int_{\bB^d}f(u(x),\nabla u(x))\;\dd x,
\]
and for each $r\in(0,1)$ define $u_r\in(\W^{1,1}\cap\C^\infty_0)(\bB^d;\R^{d+1})$ by
\[
u_r(x):=\begin{cases}
r^{1-d}u\left(\frac{x}{r}\right)&\quad\text{ if }|x|\leq r,\\
0 &\quad\text{ if }|x|>r.
\end{cases}
\]
Using the change of variables $z=x/r$, the fact that $\lim_{r\downarrow 0}r\log(r^{-1})=0$, and the zero-homogeneity of $B$, we can compute
\begin{align*}
\lim_{r\to 0}\Fcal[u_r]=&-\int_{\{\nabla u(z):B(u(z))<|u(z)|^{d/(d-1)}\}}\nabla u(z): B(u(z))\;\dd z\\
&\qquad -\int_{\{\nabla u(z):B(u(z))\geq|u(z)|^{d/(d-1)}\}}|u(z)|^{d/(d-1)}\;\dd z.
\end{align*}
By virtue of~\eqref{eqexampleineq}, then, we have that $\lim_{r\downarrow 0}\Fcal[u_r]<0$. It is easy to see that $u_r\wsc 0$ in $\BV(\bB^d;\R^{d+1})$ as $r\to 0$ and so we deduce that $\Fcalrw[0]<0$. In fact, by replacing $u$ with $k\cdot u$, repeating the procedure above and then letting $k\to\infty$, we can even see that $\Fcalrw[0]=-\infty$. On the other hand, since $f(y,0)=0$ for all $y\in\R^{d+1}$ and $f^\infty\geq 0$ on $\R^{d+1}\times\R^{(d+1)\times d}$, the integral functional given in the statement of Theorem~\ref{wsclscthm} must be non-negative at $u\equiv 0$. Hence, the conclusion of Theorem~\ref{wsclscthm} cannot hold for the integrand $f$ defined above. 
\end{example}

\subsection{Functionals and surface energies}\label{subsecfunctionals}
For $f\in\Rbf(\Omega\times\R^m)$, we define the extended functional $\Fcal\colon\BV(\Omega;\R^m)\to\overline{\R}$ by
\begin{equation}\label{eqctsfunctional}
\Fcal[u]:=\int_\Omega f(x,u(x),\nabla u(x))\;\dd x+\int_\Omega\int_0^1 f^\infty\left(x,u^\theta(x),\frac{\dd D^su}{\dd|D^su|}(x)\right)\;\dd\theta\;\dd|D^su|(x),
\end{equation}
where $u^\theta$ is the jump interpolant defined above by~\eqref{eqdefjumpinterpolant}.

This choice of extension for $\Fcal$ to $\BV(\Omega;\R^m)$ is different to the one discussed in Section~\ref{secintro}, where $\Fcal$ is extended to $\BV(\Omega;\R^m)$ by $\Fcal_{**}$, and is used for technical reasons: whilst the method of extension for $\Fcal$ by relaxation is the right choice from the point of view of seeking existence of minimisers, $\Fcalrw$ is not continuous with respect to any convergence with respect to which $\C^\infty(\Omega;\R^m)$ is dense in $\BV(\Omega;\R^m)$ (see Example~\ref{exbadrecoveryseq}) and hence not an ideal functional to work with with respect to analysis in $\BV(\Omega;\R^m)$. By contrast, for integrands $f\in\RBVw(\Omega\times\R^m)$ (which need not be quasiconvex in the final variable), Theorem~\ref{thmareastrictcontinuity} below, states that $\Fcal$ as defined by~\eqref{eqctsfunctional} is the area-strictly continuous extension of $u\mapsto\int_\Omega f(x,u(x),\nabla u(x))\;\dd x$ from $\W^{1,1}(\Omega;\R^m)$ (and even $\C^\infty(\Omega;\R^m)$) to $\BV(\Omega;\R^m)$. Proposition~\ref{propfixedbdary} therefore implies that Theorem~\ref{wsclscthm} can equivalently be seen as identifying the weak* relaxation of this continuously extended $\Fcal$ from $\BV(\Omega;\R^m)$ to $\BV(\Omega;\R^m)$, which is the approach that we take in what follows.

\begin{theorem}\label{thmareastrictcontinuity}
 Let $\Omega\subset\mbR^d$ be a bounded domain with Lipschitz boundary and let $f\in \mbfR(\Omega\times\R^m)$ satisfy the growth bound
\begin{equation}\label{assumption1}
|f(x,y,A)|\leq C(1+|y|^{d/(d-1)}+|A|) \quad\text{ for all } (x,y,A)\in\Omega\times\mbR^m\times\mbR^{m\times d}.
\end{equation}
Then the functional $\Fcal\colon\BV(\Omega;\R^m)\to\R$ is area-strictly continuous.
\end{theorem}
Theorem~\ref{thmareastrictcontinuity} is proved under slightly more general hypotheses in Theorem~5.2 of~\cite{RinSha15SCE}.

Given $u\in\BV(\Omega;\R^m)$ and $x\in\Jcal_u$, define the class of functions $\Acal_u(x)$ by
\begin{equation}\label{eqw*energyclass}
\Acal_u(x):=\left\{\varphi\in\left(\C^\infty\cap\Lp^\infty\right)(\bB^d;\R^m)\colon\;\varphi= u^\pm_{x}\text{ on }\pd\bB^d\right\},
\end{equation}
where $u^\pm_x$ is as given in Definition~\ref{defjumpblowup}. For $f\in\Rbf(\Omega\times\R^m)$ and $u\in\BV(\Omega;\R^m)$, the surface energy density $K_f[u]$ is defined for $x\in\Jcal_u$ via
\begin{equation}\label{eqsurfacedensity}
K_f[u](x):=\inf\left\{\frac{1}{\omega_{d-1}}\int_{\bB^d}f^\infty(x,\varphi(z),\nabla\varphi(z))\;\dd z\colon\;\varphi\in\Acal_u(x)\right\}.
\end{equation}

Lemma~\ref{lemenergymeasurable} below shows that $K_f[u]$ is always $\Hcal^{d-1}$-measurable and hence that the integral
\[
\int_{\Jcal_u}K_f[u[(x)\;\dd\Hcal^{d-1}(x),
\]
is always well-defined for every $u\in\BV(\Omega;\R^m)$.

\begin{lemma}\label{lemenergymeasurable}
If $f\in\Rbf(\Omega\times\R^m)$ then $K_f[u]$ is $\Hcal^{d-1}\restrict\Jcal_u$-measurable and equal $\Hcal^{d-1}\restrict\Jcal_u$-almost everywhere to an upper-semicontinuous function.
\end{lemma}

\begin{proof}
First, fix a triple $(u^+,u^-,n_u)\colon\Jcal_u\to\R^m\times\R^m\times\pd\bB^d$ such that $n_u$ orients $\Jcal_u$ and $u^+$, $u^-$ are the one sided jump limits of $u$ with respect to $n_u$. Fix also $\varepsilon>0$. The triple $(u^+,u^-,n_u)$ is Borel and hence $|D^ju|$-measurable, and so Lusin's Theorem implies that there exists a compact set $K_\varepsilon\Subset\Jcal_u$ such that $|D^ju|(\Jcal_u\setminus K_\varepsilon)\leq\varepsilon$ and $(u^+,u^-,n_u)$ is continuous when restricted to $K_\varepsilon$.

Let $x\in K_\varepsilon$ and $(x_j)_j\subset K_\varepsilon$ be such that $x_j\to x$. For each $j\in\mbN$ let $R_j\colon\bB^d\to\bB^d$ be a rotation mapping $n_u(x)$ to $n_u(x_j)$ such that $R_j\to\id_{\R^d}$ as $j\to\infty$ and let $S_j\colon\R^m\to\R^m$ be a sequence of linear maps mapping $u^+(x)$ to $u^+(x_j)$ and $u^-(x)$ to $u^-(x_j)$ such that $S_j\to\id_{\R^m}$ as $j\to\infty$ (such a choice of $(R_j)_j$, $(S_j)_j$ is possible by the fact that $x,x_j\in K_\varepsilon$). Now for $\delta>0$, let $\varphi\in\Acal_u(x)$ be such that
\[
\frac{1}{\omega_{d-1}}\int_{\bB^d}f^\infty(x,\varphi(z),\nabla \varphi(z))\;\dd z\leq K_f[u](x)+\delta.
\]
Define $\varphi_j\in\C^\infty(\bB^d;\R^m)$ by $\varphi_j(z):=S_j\varphi\left(R_jz\right)$ and note that $\varphi_j\in\Acal_u(x_j)$. By the convergence properties assumed of $(R_j)$ and $(S_j)$, we have that $\varphi_j\to\varphi$ strictly in $\BV(\bB^d;\R^m)$ as $j\to\infty$. Next, define $\mu_j,\mu\in\mbfM(\Omega\times\R^m;\R^{m\times d})$ by
\[	
\mu_j:=\delta_{x_j}\otimes\varphi_\#(\nabla\varphi_j\lL\restrict\bB^d),\quad \mu:=\delta_{x}\otimes\varphi_\#(\nabla\varphi\lL\restrict\bB^d).
\]
It can easily be seen that $\mu_j$ converges strictly in $\mbfM(\bB^d\times\R^m;\R^{m\times d})$ to $\mu$ as $j\to\infty$. Using Reshetnyak's Continuity Theorem and the positive one-homogeneity of $f^\infty$, we therefore deduce
\begin{align*}
\int_{\bB^d}f^\infty(x_j,\varphi_j(z),\nabla\varphi_j(z))\;\dd z&=\int_{\bB^\times\R^m}f^\infty\left(z,y,\frac{\dd\mu_j}{\dd|\mu_j|}(z,y)\right)\;\dd|\mu_j|(z,y)\\
&\to\int_{\bB^\times\R^m}f^\infty\left(z,y,\frac{\dd\mu}{\dd|\mu|}(z,y)\right)\;\dd|\mu|(z,y)\\
&=\int_{\bB^d}f^\infty(x,\varphi(z),\nabla\varphi(z))\;\dd z
\end{align*}
as $j\to\infty$. By our choice of $\varphi$ and the boundary condition satisfied by each $\varphi_j$, we therefore have that
\begin{align*}
K_f[u](x)+\delta&\geq\lim_{j\to\infty}\frac{1}{\omega_{d-1}}\int_{\bB^d}f^\infty(x_j,\varphi_j(z),\nabla\varphi_j(z))\;\dd z\geq\limsup_{j\to\infty}K_f[u](x_j).
\end{align*}
It follows from the arbitrariness of $x\in K_\varepsilon$ and $\delta>0$ that $K_f[u]$ is upper semicontinuous when restricted to $K_\varepsilon$. Finally, define $F_\varepsilon\colon\Jcal_u\to[0,\infty]$ by
\[
F_\varepsilon(x):=\begin{cases}
K_f[u](x)&\text{ if }x\in K_\varepsilon,\\
\infty&\text{ otherwise},
\end{cases}
\]
and note that $F:=\inf_{\varepsilon>0}F_\varepsilon$ is equal to $K_f[u]$ at $|D^ju|$-almost every $x\in\Jcal_u$ and hence $\Hcal^{d-1}\restrict\Jcal_u$-almost every $x\in\Jcal_u$. The conclusion now follows from the fact that the pointwise infimum of a collection of upper semicontinuous functions is upper semicontinuous and hence measurable.
\end{proof}

\begin{corollary}
If $f\in\Rbf(\Omega\times\R^m)$ satisfies $f^\infty \geq 0$ and $u\in\BV(\Omega;\R^m)$ is such that 
\[
\int_{\Jcal_u}K_f[u](x)\;\dd\Hcal^{d-1}(x)<\infty,
\]
then $K_f[u]\Hcal^{d-1}\restrict\Jcal_u\in\mbfM^+(\Omega)$ is a $(d-1)$-rectifiable measure.
\end{corollary}

\begin{proof}
This follows directly from Lemma~\ref{lemenergymeasurable} combined with the discussion about rectifiability in Section~\ref{subsecmeasuretheory}.
\end{proof}

%% file: liftings.tex
\section{Liftings}\label{chapliftings}
In this section, we develop a theory of liftings. In turn we investigate their functional analytic properties, their relationship to $\BV$-functions, a structure theorem, some blow-up results, and a discussion of how integral functionals over $\BV(\Omega;\R^m)$ can be represented in terms of liftings. Two separate spaces are introduced: the space of liftings $\Lift(\Omega\times\R^m)$, and the space of approximable liftings $\ALift(\Omega\times\R^m)\subset\Lift(\Omega\times\R^m)$. The space $\Lift(\Omega\times\R^m)$ is larger than strictly necessary for our purposes, but has good compactness properties and is sufficiently well behaved that working in this setting of extra generality allows for a cleaner presentation of most of the results described here and in Sections~\ref{chapyoungmeasures} and~\ref{chaptangentliftingyms}. The exceptions to this rule is the Jensen inequalities for $\Fcal$, derived in Theorems~\ref{thmregtangentyms} and~\ref{thmjumptangentyms}, Section~\ref{secjensenineq}, for which we must work within the more restrictive class $\ALift(\Omega\times\R^m)$.


The assumption that $\Omega\subset\R^d$ with $d\geq 2$ is used in this paper only in Proposition~\ref{ctsembedding} and the Young measure theory developed in Section~\ref{chapyoungmeasures}. Consequently, the results presented here are also valid for domains $\Omega\subset\R$.

\subsection{Functional Analysis and the Structure Theorem}

\begin{definition}\label{defliftings}
A \textbf{lifting} is a measure $\gamma\in\mbfM(\Omega\times\R^m;\R^{m\times d})$ for which there exists a function $u\in\BV_\#(\Omega;\R^m)$ such that the \textbf{chain rule} formula
\begin{equation}\label{eqchainrule}
\int_{\Omega}\nabla_x\varphi(x,u(x))\;\dd x+\int_{\Omega\times\R^m}\nabla_y\varphi(x,y)\;\dd\gamma(x,y)=0\quad\text{for all }\varphi\in\C_{0}^1(\Omega\times\R^m)
\end{equation}
holds. The space of all liftings is denoted by $\Lift(\Omega\times\R^m)$. Weak* convergence of liftings in $\Lift(\Omega\times\R^m)$ means weak* convergence of the liftings in considered as measures in $\mbfM(\Omega\times\R^m;\R^{m\times d})$.
\end{definition}
Definition~\ref{defliftings} was first given by Jung \& Jerrard in~\cite{JunJer04SCML} where the authors initiated the study of elementary liftings (which they refer to as minimal liftings), introduced below in Definition~\ref{defelementaryliftings}. This paper also contains the first proofs of Lemma~\ref{lemliftingspushforward} and Proposition~\ref{lemmustrict} below. Our proofs for these results are new and, in the case of Proposition~\ref{lemmustrict}, are obtained as a corollary of the Structure Theorem (Theorem~\ref{thmliftingsstructure}), which does not feature in~\cite{JunJer04SCML}.

The following lemma implies that each $\gamma\in\Lift(\Omega\times\R^m)$ is associated to a \textit{unique} $u\in\BV_\#(\Omega;\R^m)$ satisfying~\eqref{eqchainrule}. We shall refer to this $u$ as the \textbf{barycentre} of $\gamma$, writing $\asc{\gamma}=u$.
\begin{lemma}\label{lemliftingspushforward}
If $\gamma\in\Lift(\Omega\times\R^m)$ satisfies~\eqref{eqchainrule}, then it holds that
\[
\pi_\#\gamma=Du\text{ in }\mbfM(\Omega;\R^{m\times d})\text{ and }\pi_\#|\gamma|\geq|Du|\text{ in }\mbfM^+(\Omega).
\]
In particular, if $\gamma\in\Lift(\Omega\times\R^m)$, then the element $u\in\BV_\#(\Omega;\R^m)$ with respect to which $\gamma$ verifies~\eqref{eqchainrule} is unique.
\end{lemma}

\begin{proof}
Let $\psi(x,y)=f(x)y$, where $f\in\C_0(\Omega)$ is arbitrary, so that $\nabla\psi=(\nabla_x\psi,\nabla_y\psi)=(\nabla f \otimes y,f\id_{\R^m})$. Now let $(\chi_R)_{R>0}\subset\C^1_c(\R^m)$ be a family of cut-off functions satisfying $\mathbbm{1}_{B(0,R)}\leq\chi_R\leq \mathbbm{1}_{B(0,3R^2)}$, $\norm{\nabla\chi_R}_\infty\leq 1/R^2$, $\chi_R\uparrow \mathbbm{1}$ as $R\to\infty$. Setting $\varphi_R(x,y):=\chi_R(y)\psi(x,y)$, we see that $\nabla\varphi_R\in\C_0(\Omega\times\R^m)$, $\sup_R\norm{\nabla\varphi_R}_\infty<\infty$, and $\nabla\varphi_R\to\nabla\varphi$ pointwise as $R\to\infty$. The chain rule~\eqref{eqchainrule} then implies that
\[
\int_\Omega\nabla_x\varphi_R(x,u(x))\;\dd x+\int_{\Omega\times\R^m}\nabla_y\varphi_R(x,y)\;\dd\gamma(x,y)=0.
\]
Letting $R\to\infty$ and using the Dominated Convergence Theorem, we therefore obtain
\begin{align*}
0&=\int_\Omega \nabla f(x)u(x)\;\dd x+\int_{\Omega\times\R^m}f(x)\id_{\R^m}\;\dd\gamma(x,y)\\
&=\int_\Omega\nabla f(x)u(x)\;\dd x+\int_\Omega f(x)\;\dd(\pi_\#\gamma)(x).
\end{align*}
Hence, after an integration by parts,
\[
\int_\Omega f(x)\;\dd Du(x)=\int_\Omega f(x)\;\dd(\pi_\#\gamma)(x)\quad\text{ for all }f\in\C^1_0(\Omega),
\]
which implies the first result. For $A\in\mathcal{B}(\Omega)$, we can now compute
\begin{align*}
\pi_\#|\gamma|\left(A\right)&=|\gamma|\left(A\times\R^m\right) \\
& =\sup\left\{\sum_{h=0}^\infty |\gamma\left(B_h\right)|:\left(B_h\right)\subset\mathcal{B}(\R^d\times\R^m)\text{ is a partition of }A\times\R^m.\right\} \\ &
\geq\sup\left\{\sum_{h=0}^\infty|\gamma\left(A_h\times\R^m\right)|:\left(A_h\right)\subset\mathcal{B}(\R^d)\text{ is a partition of }A.\right\} \\
&=\sup\left\{\sum_{h=0}^\infty|Du(A_h)|:\left(A_h\right)\subset\mathcal{B}(\R^d)\text{ is a partition of }A.\right\}\\
&=|Du|\left(A\right),
\end{align*}
from which it follows that $\pi_\#|\gamma|\geq|Du|$ as desired.
\end{proof}

\begin{lemma}[Compactness for $\Lift$]\label{lemliftingcompactness}
Let $(\gamma_j)_j\subset\Lift(\Omega\times\R^m)$ be such that $\sup_j|\gamma_j|(\Omega\times\R^m)<\infty$. Then there exists a subsequence $(\gamma_{j_k})_k\subset(\gamma_j)_j$ and a limit $\gamma\in\Lift(\Omega\times\R^m)$ such that
\[
\gamma_{j_k}\wsc\gamma\text{ in $\mbfM(\Omega\times\R^m;\R^{m\times d})$ and }\asc{\gamma_{j_k}}\wsc\asc{\gamma}\text{ in }\BV_\#(\Omega;\R^m).
\]
Moreover, the map $\gamma\mapsto\asc{\gamma}$ is sequentially weakly* continuous from $\Lift(\Omega\times\R^m)$ to $\BV_\#(\Omega;\R^m)$ and the space $\Lift(\Omega\times\R^m)$ is sequentially weakly* closed.
\end{lemma}

\begin{proof}
Lemma~\ref{lemliftingspushforward} implies that $(\asc{\gamma_j})_j$ is bounded in $\BV_\#(\Omega;\R^m)$ whenever $(\gamma_j)_j$ is bounded in $\mbfM(\Omega\times\R^m;\R^{m\times d})$.  If $(\gamma_j)_j\subset\Lift(\Omega\times\R^m)$ satisfies $\sup_j|\gamma_j|(\Omega\times\R^m)<\infty$, we can therefore use the sequential weak* compactness of bounded sets in $\mbfM(\Omega\times\R^m;\R^{m\times d})$ and $\BV_\#(\Omega;\R^m)$ to pass to a subsequence $(\gamma_{j_k})_k\subset(\gamma_j)_j$ converging weakly* in $\mbfM(\Omega\times\R^m;\R^{m\times d})$ to a limit $\gamma$ and such that $\asc{\gamma_{j_k}}\wsc u$ for some $u\in\BV_\#(\Omega;\R^m)$. Taking the limit in~\eqref{eqchainrule} as $j\to\infty$ and using the Dominated Convergence Theorem, it follows that the pair $(u,\gamma)$ satisfies~\eqref{eqchainrule} and hence that $\gamma\in\Lift(\Omega\times\R^m)$ with $u=\asc{\gamma}$. The space $\Lift(\Omega\times\R^m)$ is therefore sequentially weakly* closed.

To see that $\gamma\mapsto\asc{\gamma}$ is sequentially weak* continuous, note that if $(\gamma_j)_j\subset\Lift(\Omega\times\R^m)$ is such that $\gamma_j\wsc\gamma$ in $\Lift(\Omega\times\R^m)$ as $j\to\infty$ then $(\gamma_j)_j$ must be uniformly bounded in $\mbfM(\Omega\times\R^m)$. The preceding discussion therefore implies that, upon passing to a further subsequence, $\asc{\gamma_j}\wsc u$ for some $u\in\BV_\#(\Omega;\R^m)$. Passing to the limit again in~\eqref{eqchainrule}, we find that $u=\asc{\gamma}$ and so, since this argument can be applied to any subsequence of $(\gamma_j)_j$, we reach the desired conclusion.
\end{proof}

\begin{corollary}\label{corliftingtographcts}
If $\gamma_j\wsc\gamma$ in $\Lift(\Omega\times\R^m)$ then $\gr^{\asc{\gamma_j}}_\#(\lL\restrict\Omega)\to\gr^{\asc{\gamma}}_\#(\lL\restrict\Omega)$ strictly in $\mbfM^1(\Omega\times\R^m)$.
\end{corollary}
\begin{proof}
Since weak* convergence in $\BV(\Omega;\R^m)$ implies strong $\Lp^1(\Omega;\R^m)$ convergence, Lemma~\ref{lemliftingcompactness} implies that $\asc{\gamma_j}\to\asc{\gamma}$ in $\Lp^1(\Omega;\R^m)$ whenever $\gamma_j\wsc\gamma$ in $\Lift(\Omega\times\R^m)$. Using the Dominated Convergence Theorem, we therefore deduce that, for any $\varphi\in\C_b(\Omega\times\R^m)$,
\begin{align*}
\int\varphi(x,y)\;\dd \gr^{\asc{\gamma_j}}_\#(\lL)(x,y)&=\int_\Omega \varphi(x,\asc{\gamma_j}(x))\;\dd x\\
&\to\int_\Omega\varphi(x, \asc{\gamma}(x))\;\dd x =\int\varphi(x,y)\;\dd \gr^{\asc{\gamma}}_\#(\lL)(x,y)\quad\text{ as $j\to\infty$},
\end{align*}
which is what was to be shown.
\end{proof}

\begin{definition}[Elementary Liftings]\label{defelementaryliftings}
Given $u\in\BV_\#(\Omega;\R^m)$, the \textbf{elementary lifting} $\gamma\asc{u}\in\Lift(\Omega\times\R^m)$ associated to $u$ is defined by
\[
\gamma\asc{u}:=|Du|\otimes\left(\frac{\dd Du}{\dd|Du|}\int_0^1\delta_{u^\theta}\;\dd\theta\right),
\]
that is,
\[
\ip{\varphi}{\gamma\asc{u}}=\int_\Omega\int_0^1\varphi(x,u^\theta(x))\;\dd\theta\;\dd Du(x)\quad\text{for all }\varphi\in\C_0(\Omega\times\R^m),
\]
where $u^\theta$ is the jump interpolant defined in Section~\ref{secpreliminaries}.
\end{definition}
That $\gamma\asc{u}\in\Lift(\Omega\times\R^m)$ follows from the chain rule for $\BV$-functions: if $u\in\BV(\Omega;\R^m)$ and $\varphi\in\C^1_0(\Omega\times\R^m)$, then the composition $\varphi\circ\gr^u=\varphi(\frarg,u(\frarg))$ is an element of $\BV(\Omega)$ satisfying $(\varphi\circ\gr^u)|_{\pd\Omega}=0$. By Stokes' Theorem, this implies
\[
\int_\Omega \;\dd D(\varphi\circ\gr^u)(x)=\int_{\pd\Omega}(\varphi\circ\gr^u)|_{\pd\Omega} \; n_{\pd\Omega}(x)\;\dd\Hcal^{d-1}(x)=0,
\]
where $n_{\pd\Omega}$ is the (inwards pointing) normal orientation vector for $\pd\Omega$.  Applying the chain rule, Theorem~\ref{bvchainrule}, to $\varphi\circ\gr^u$ and writing $\nabla\varphi=(\nabla_x\varphi,\nabla_y\varphi)$, we see that
\begin{align*}
D (\varphi\circ\gr^u)&=\left(\int_0^1\nabla\varphi\circ(\gr^u)^\theta\;\dd \theta\right) \cdot D(\gr^u)\\
&=\left(\int_0^1\nabla\varphi(\frarg,u^\theta(\frarg))\;\dd \theta\right) \cdot (\id_{\R^d}\lL\restrict\Omega,Du)\\
&=\nabla\varphi(\frarg,u(\frarg)) \cdot (\id_{\R^d},\nabla u)\lL\restrict\Omega+\nabla\varphi(\frarg,u(\frarg)) \cdot (0,D^cu)\\
&\qquad+\left(\int_0^1\nabla\varphi(\frarg,u^\theta(\frarg))\;\dd \theta\right) \cdot (0,D^ju)\\
&=\nabla_x\varphi(\frarg,u(\frarg)) \lL \restrict \Omega +\nabla_y\varphi(\frarg,u(\frarg))\nabla u(\frarg) \lL \restrict \Omega \\
&\qquad+\nabla_y\varphi(\frarg,u(\frarg))D^cu +\left(\int_0^1\nabla\varphi_y(\frarg,u^\theta(\frarg))\;\dd \theta\right)D^ju.
\end{align*}
Integrating over $\Omega$ with respect to $x$, we deduce
\[
\int_\Omega\nabla_x\varphi(x,u(x))\;\dd x+\int_{\Omega\times\R^m}\nabla_y\varphi(x,y)\;\dd\gamma\asc{u}(x,y)=0,
\]
as required.

\begin{remark}
We note here that~\eqref{eqchainrule} and Definition~\ref{defelementaryliftings} both make sense for $u\in\BV(\Omega;\R^m)$ (rather than just $\BV_\#(\Omega;\R^m)$) and $\gamma\in\mbfM(\Omega\times\R^m;\R^{m\times d})$ and could be used to define liftings associated to arbitrary $\BV$-functions. The downside for this extra generality is that the barycentre $\asc{0}$ of the zero lifting is no longer unique and, more importantly, that the control $\sup_j|\gamma_j|(\Omega\times\R^m)<\infty$ no longer enforces $\sup_j\norm{\asc{\gamma_j}}_\BV<\infty$. As a result, the map $\gamma\mapsto\asc{\gamma}$ is no longer continuous and Lemma~\ref{lemliftingcompactness} is no longer true. Instead, the discussion surrounding~\eqref{eqliftingrep} in Section~\ref{secliftingperspective} shows how the behaviour of arbitrary weakly* convergent sequences in $\BV(\Omega;\R^m)$ can be described in terms of liftings as we have chosen to define them.
\end{remark}

\begin{definition}[Approximable liftings]\label{defapproxliftings}
A lifting $\gamma$ is said to be \textbf{approximable} if it arises as the weak* limit of a sequence of elementary liftings. The space of all approximable liftings is denoted by $\ALift(\Omega\times\R^m)$,
\[
\ALift(\Omega\times\R^m):=\bigl\{\gamma\in\Lift(\Omega\times\R^m) \colon\; \gamma=\wstarlim\gamma\asc{u_j}\text{ for some }(u_j)_j\subset\BV_\#(\Omega;\R^m)\bigr\}.
\]
\end{definition}
Note that, despite being defined as a sequential closure, it is an open question as to whether $\ALift(\Omega\times\R^m)$ is either sequentially weakly* closed or weakly* closed since weak* topologies are not in general metrizable on unbounded sets. Example~\ref{exstrictliftinginclusion} demonstrates that, in the one-dimensional case at least, the inclusion $\ALift(\Omega\times\R^m)\subset\Lift(\Omega\times\R^m)$ is strict:

\begin{example}\label{exstrictliftinginclusion}
There exists $\gamma\in\Lift((-1,1)\times\R^2)$ such that $\gamma\notin\ALift((-1,1)\times\R^2)$. Consequently, $\ALift((-1,1)\times\R^2)$ and $\Lift((-1,1)\times\R^2)$ do not coincide.
\end{example}

\begin{proof}
First, we claim that any $\gamma\in\ALift((-1,1)\times\R^2)$ must satisfy 
\[
\supp\gamma\subset(-1,1)\times B(0,R)\quad\text{ for some $R>0$}.
\]
 To see this, let $(u_j)_j\subset\BV_\#((-1,1);\R^2)$ be such that
\[
\gamma\asc{u_j}\wsc\gamma\quad\text{ in }\mbfM((-1,1)\times\R^2;\R^2).
\]
Since the sequence $(u_j)_j$ is norm-bounded in $\BV((-1,1);\R^2)$, the one-dimensional Sobolev Embedding Theorem implies that $\sup_j\norm{u_j}_\infty<\infty$. Letting $R>0$ be such that $\norm{u_j}_\infty< R$ for all $j\in\mbN$ we see that, for all $\varphi\in\C_0((-1,1)\times\R^2)$ satisfying $\varphi|_{(-1,1)\times B(0,R)}=0$,
\[
\int\varphi(x,y)\;\dd\gamma\asc{u_j}(x,y)=\int_{-1}^1\int_0^1\varphi(x,u^\theta_j(x))\;\dd\theta\;\dd Du(x)=0.
\]
Letting $j\to\infty$, we deduce
\[
\int\varphi(x,y)\;\dd\gamma(x,y)=0\quad\text{ for all }\varphi\in\C_0((-1,1)\times\R^2)\text{ with }\varphi|_{(-1,1)\times B(0,R)}=0
\]
and hence that $|\gamma|((-1,1)\times(\R^2\setminus B(0,R)))=0$, which implies $\supp\gamma\subset(-1,1)\times B(0,R)$ as required.

Next, define $\mu_k\in\mbfM((-1,1)\times\R^2;\R^2)$ for each $k\in\mbN$ by
\[
\int\varphi(x,y)\;\dd\mu_k(x,y)=k\int_0^{2\pi}\varphi\left(0,k\begin{pmatrix}\cos z\\ \sin z\end{pmatrix}\right)\begin{pmatrix}\sin z\\ -\cos z\end{pmatrix}\;\dd z.
\]
Clearly, $\supp\mu_k =\{0\}\times \pd B(0,k)$, $|\mu_k|((-1,1)\times\R^2)=2\pi k$, and, since
\begin{align*}
\int\nabla_y\varphi(x,y)\;\dd\mu_k(x,y)&=k\int_0^{2\pi}\nabla_y\varphi\left(0,k\begin{pmatrix}\cos z\\ \sin z\end{pmatrix}\right)\begin{pmatrix}\sin z\\ -\cos z\end{pmatrix}\;\dd z\\
&=\int_0^{2\pi}\nabla\left(\varphi\left(0,k\begin{pmatrix}\cos\frarg\\ \sin\frarg\end{pmatrix}\right)\right)(z)\;\dd z\\
&=\varphi\left(0,k\begin{pmatrix}\cos 2\pi\\ \sin 2\pi\end{pmatrix}\right)-\varphi\left(0,k\begin{pmatrix}\cos 0\\ \sin 0\end{pmatrix}\right)=0
\end{align*}
for every $\varphi\in\C_0^1((-1,1)\times\R^2)$, we have that $\mu_k\in\Lift((-1,1)\times\R^2)$ with $\asc{\mu_k}=0$ for each $k$. Defining
\begin{equation}\label{eqmeasuresum}
\mu:=\sum_{i=k}^\infty k^{-3}\mu_k,
\end{equation}
we have that $\mu\in\mbfM((-1,1)\times\R^2;\R^2)$ with $|\mu|(\{0\}\times\pd B(0,k))=2\pi/k^2>0$ for each $k$. Moreover, since the sum~\eqref{eqmeasuresum} is strongly convergent in $\mbfM((-1,1)\times\R^2;\R^2)$, we see that
\[
\int\nabla_y\varphi(x,y)\;\dd\mu(x,y)=\lim_{n\to\infty}\sum_{k=1}^n\int\nabla_y\varphi(x,y)\;\dd\mu_k(x,y)=0
\]
for every $\varphi\in\C_0^1((-1,1)\times\R^2)$, which implies that $\mu\in\Lift((-1,1)\times\R^2)$ with $\asc{\mu}=0$. Since $|\mu|(\{0\}\times\pd B(0,k))>0$ for every $k\in\mbN$, however, this implies that $\supp\mu\not\subset(-1,1)\times B(0,R)$ for any $R>0$, and so $\mu\not\in\ALift((-1,1)\times\R^2)$.
\end{proof}

The following is a direct corollary of Lemma~\ref{lemliftingcompactness} combined with Definition~\ref{defapproxliftings}:
\begin{corollary}[Lifting generation from $\BV$]
Let $(u_j)_j\subset\BV_\#(\Omega;\R^m)$ be a bounded sequence with $u_j \wsc u$ in $\BV_\#(\Omega;\R^m)$. Then there exists a (non-relabelled) subsequence and a limit $\gamma\in\ALift(\Omega\times\R^m)$ with $\asc{\gamma}=u$ such that
\[
\gamma\asc{u_j}\wsc\gamma\text{ in }\ALift(\Omega\times\R^m).
\]
\end{corollary}

Example~\ref{exnonelementarylifting} below demonstrates how non-elementary liftings can arise as weak* limits of sequences of elementary liftings, and shows that this phenomenon gives rise to behaviour for integrands which is very different to the $u$-independent and scalar valued cases:

\begin{example}\label{exnonelementarylifting}
Define $(u_j)_j\subset\W_\#^{1,1}((-1,1);\R^2)$ by
\[
u_j(x):=
-\begin{pmatrix}
1\\0
\end{pmatrix}\mathbbm{1}_{(-1,0]}(x)-
\begin{pmatrix}
\cos(jx)\\
\sin(jx)
\end{pmatrix}\mathbbm{1}_{(0,\pi/j]}(x)+
\begin{pmatrix}
1\\0
\end{pmatrix}\mathbbm{1}_{(\pi/j,1)}(x)
+\frac{1}{j}\left[\frac{\pi}{2}\begin{pmatrix}
1\\0
\end{pmatrix}+
\begin{pmatrix}
0\\1
\end{pmatrix}
\right]
\]
and note that 
\[
u_j\wsc u^0:=-\begin{pmatrix}1\\0\end{pmatrix}\mathbbm{1}_{(-1,0]}+\begin{pmatrix}1\\0\end{pmatrix}\mathbbm{1}_{(0,1)}
\]
in $\BV((-1,1);\R^2)$. By passing to a non-relabelled subsequence if necessary, we can assume that $\gamma\asc{u_j}\wsc\gamma$ for some $\gamma\in\ALift((-1,1)\times\R^2)$ with $\asc{\gamma}=u^0$.

Let $\psi\in\C_c((-1,1)\times[0,\infty);[0,1])$ be such that $\psi(t,s)=1$ for $(t,s)\in[-1/2,1/2]\times[0,2]$, and define $\varphi\in\C_0(\R^2)$ by 
\[
\varphi\left(x,y\right)=\psi(x,|y|)(|y|^2-1)^2.
\]	
We can compute for $j\geq 4$,
\[
\int_{(-1,1)\times\R^2}\varphi(x,y)\;\dd\gamma\asc{u_j}(x,y)=j\int_0^{\pi/j}\left(\left|\begin{pmatrix}
\cos(jx)+\frac{\pi}{2j}\\
\sin(jx)+\frac{1}{j}
\end{pmatrix}\right|^2-1\right)^2\begin{pmatrix}-\sin(jx)\\ \cos(jx)\end{pmatrix}\;\dd x.
\]
Thus,
\[
\lim_{j\to\infty}\int_{(-1,1)\times\R^2}\varphi(x,y)\;\dd\gamma\asc{u_j}(x,y)=\int_0^{\pi}\left(\left|\begin{pmatrix}
\cos(z)\\
\sin(z)
\end{pmatrix}\right|^2-1\right)^2\begin{pmatrix}-\sin(z)\\ \cos(z)\end{pmatrix}\;\dd z
\]
and so, by the construction of $\varphi$,
\[
\int_{(-1,1)\times\R^2}\varphi(x,y)\;\dd\gamma=0.
\]
However, we also see that
\begin{align*}
\int_{(-1,1)\times\R^2}\varphi(x,y)\;\dd\gamma\asc{u^0}(x,y)&=\int_0^1\left(\left|\begin{pmatrix}2\theta-1\\0\end{pmatrix}\right|^2-1\right)^2\begin{pmatrix}2\\0\end{pmatrix}\;\dd \theta\\
&=\int_0^1\left(4\theta^2+4\theta\right)^2\begin{pmatrix}2\\0\end{pmatrix}\;\dd \theta\neq 0,
\end{align*}
from which we can conclude that $\gamma\neq\gamma\asc{u^0}$.
\end{example}

\subsection{The Structure Theorem}
We now investigate the structure of liftings $\gamma\in\Lift(\Omega\times\R^m)$ and, in particular, how a general lifting $\gamma$ must relate to the elementary lifting $\gamma\asc{u}$ where $\asc{\gamma}=u$. It is clear that $\mu+\gamma\asc{u}\in\Lift(\Omega\times\R^m)$ with $\asc{\mu+\gamma\asc{u}}=u$ whenever $\mu\in\mbfM(\Omega\times\R^m;\R^{m\times d})$ satisfies $\diverg_y\mu=0$ (column-wise divergence), since $\diverg_y\mu=0$ implies that $\int f(x)\nabla_yg(y)\;\dd\mu(x,y)=0$ for all $f\in\C_0(\Omega)$ and $g\in\C_0^1(\R^m)$. It turns out that \emph{every} $\gamma\in\Lift(\Omega\times\R^m)$ can be written in this form:
\begin{theorem}[Structure Theorem for Liftings]\label{thmliftingsstructure}
If $\gamma\in\Lift(\Omega\times\R^m)$ with $u=\asc{\gamma}$, then $\gamma$ admits the following decomposition into mutually singular measures:
\begin{equation*}
\gamma=\gamma\asc{u}\restrict((\Omega\setminus\mathcal{J}_u)\times\R^m)+\gamma^{\mathrm{gs}}.
\end{equation*}
Moreover, $\gamma^{\mathrm{gs}}\in\mbfM(\Omega\times\R^m;\R^{m\times d})$ satisfies 
\begin{equation*}
\diverg_y\gamma^{\mathrm{gs}}=-|D^ju|\otimes\frac{n_u}{|u^+-u^-|}(\delta_{u^+}-\delta_{u^-}),
\end{equation*}
and it is \textbf{graph-singular} with respect to $u$ in the sense that $\gamma^{\mathrm{gs}}$ is singular with respect to all measures of the form $\gr^u_\#\lambda$ where $\lambda\in\mbfM(\Omega)$ satisfies both $\lambda\ll\Hcal^{d-1}$ and $\lambda(\Jcal_u)=0$.
\end{theorem}
\begin{proof}
Since $\gamma$ and $\gamma\asc{u}$ both satisfy the chain rule~\eqref{eqchainrule}, we can test with functions of the form $\varphi=f\otimes g$ for $f\in\C^1_0(\Omega)$, $g\in\C^1_0(\R^m)$ to obtain
\[
\int_\Omega\nabla_x f(x) g(u(x))\;\dd x+\int_{\Omega\times\R^m}f(x)\nabla_yg(y)\;\dd\gamma\asc{u}(x,y)=0
\]
and
\[
\int_\Omega\nabla_x f(x) g(u(x))\;\dd x+\int_{\Omega\times\R^m}f(x)\nabla_yg(y)\;\dd\gamma(x,y)=0.
\]
Taking the difference of these two equations then leads us to the identity
\begin{equation*}
\int_{\Omega\times\R^m} f(x)\nabla_y g(y)\;\dd(\gamma-\gamma\asc{u})(x,y)=0\qquad\text{ for all }f\in\C^1_0(\Omega),\; g\in\C^1_0(\R^m).
\end{equation*}
Now write $\eta:=\pi_\#|\gamma-\gamma\asc{u}|$ and disintegrate $\gamma-\gamma\asc{u}=\eta\otimes\rho$ for some weakly* $\eta$-measurable parametrised measure $\rho\colon\Omega\mapsto\mbfM^1(\R^m;\R^{m\times d})$ so that
\[
\int_\Omega f(x)\left\{\int_{\R^m}\nabla_y g(y)\;\dd\rho_x(y)\right\}\;\dd\eta(x)=0
\]
for every $f\in\C_0^1(\Omega)$ and $g\in\C_0^1(\R^m)$. Varying $f$ through a countable dense subset of $\C_0(\Omega)$, we deduce
\[
\int_{\R^m}\nabla_y g(y)\;\dd\rho_x(y)=0\qquad\text{ for }\eta\text{-almost every }x\in\Omega.
\]
Since $g\in\C_0^1(\R^m)$ was arbitrary, it follows that $\diverg_y\rho_x=0$ (column-wise divergence) for $\eta$-almost every $x\in\Omega$ and hence that $\diverg_y(\gamma-\gamma\asc{u})=0$. Defining
\[
  \gamma^{\mathrm{gs}}:=\gamma-\gamma\asc{u}\restrict((\Omega\setminus\Jcal_u)\times\R^m),
\]
we therefore immediately obtain from the definition of $\gamma\asc{u}$ that
\begin{align}
\begin{split}\label{eqdivergidentity}
\diverg_y\gamma^{\mathrm{gs}}=\diverg_y\left(\gamma\asc{u}\restrict(\Jcal_u\times\R^m)\right)
&=|D^ju|\otimes \left(\diverg_y\left(\frac{\dd D^ju}{\dd |D^ju|}\int_0^1\delta_{u^\theta}\;\dd\theta\right)\right).
\end{split}
\end{align}
Abbreviating $\mu_x:=\frac{\dd D^ju}{\dd |D^ju|}(x)\int_0^1\delta_{u^\theta(x)}\;\dd\theta$, we can compute, for $\eta$-almost every $x\in\Omega$,
\begin{align*}
-\ip{g}{\diverg_y\mu_x}=&\ip{\nabla g}{\frac{\dd D^ju}{\dd |D^ju|}(x)\int_0^1\delta_{u^\theta(x)}\;\dd\theta}\\
=&\int_0^1\nabla g(u^\theta(x))\frac{\dd D^ju}{\dd|D^ju|}(x)\;\dd\theta\\
=&\left(\int_0^1 \nabla g(u^\theta(x))\;\dd\theta\right)\left[\frac{(u^+(x)-u^-(x))}{|u^+(x)-u^-(x)|}\otimes n_u(x)\right]\\
=&\left[\left(\int_0^1 \nabla g(u^\theta(x))\;\dd\theta\right)\cdot(u^+(x)-u^-(x))\right]\frac{n_u(x)}{|u^+(x)-u^-(x)|}\\
=&\left(\int_0^1\frac{\dd }{\dd\theta}g(u^\theta(x))\;\dd\theta\right)\frac{n_u(x)}{|u^+(x)-u^-(x)|} \\
=&\frac{g(u^+(x))-g(u^-(x))}{|u^+(x)-u^-(x)|}n_u(x),
\end{align*}
which implies
\begin{equation*}
\diverg_y\left[\frac{\dd D^ju}{\dd |D^ju|}(x)\int_0^1\delta_{u^\theta(x)}\;\dd\theta\right]=-\frac{n_u(x)}{|u^+(x)-u^-(x)|}(\delta_{u^+(x)}-\delta_{u^-(x)})
\end{equation*}
and hence that
\[
\diverg_y\gamma^{\mathrm{gs}}=-|D^ju|\otimes\frac{n_u	}{|u^+-u^-|}(\delta_{u^+}-\delta_{u^-})
\]
as required.

To show that $\gamma^{\mathrm{gs}}$ is graph-singular with respect to $u$, we argue as follows: let $\lambda\in\mbfM^+(\Omega)$ satisfy $\lambda\ll\Hcal^{d-1}$, $\lambda(\Jcal_u)=0$, and let $E\subset\Omega$ be a set such that $\lambda(E)=\lambda(\Omega)$, the precise representative $u(x)$ is defined for all $x\in E$, and $E\cap\Jcal_u=\emptyset$. By Lemma~\ref{lemdivmunoatoms}, $\rho_x(\{u(x)\})=0$ for $\eta$-almost every $x\in\Omega$, and so, since $E\cap\Jcal_u=\emptyset$,
\[
|\gamma^{\mathrm{gs}}|(\gr^u(E))=|\gamma-\gamma\asc{u}|(\gr^u(E))=\int_E|\rho_x|(\{u(x)\})\;\dd\eta(x)=0.
\]
Because $(\gr^u_\#\lambda)(\gr^u(E))=(\gr^u_\#\lambda)(\Omega\times\R^m)$, it follows that $\gamma^{\mathrm{gs}}$ and $\gr^u_\#\lambda$ charge disjoint sets, which suffices to prove the claim. Since $\gamma\asc{u}\restrict((\Omega\setminus\Jcal_u)\times\R^m)=\gr^u_\#(\nabla u\lL+D^cu)$ is a $u$-graphical measure, we therefore also deduce that $\gamma\asc{u}((\Omega\setminus\Jcal_u)\times\R^m)\perp\gamma^{\mathrm{gs}}$, as required.
\end{proof}
As we have discussed, no more can be said about the structure of $\gamma\in\Lift(\Omega\times\R^m)$ beyond the conclusion of Theorem~\ref{thmliftingsstructure} in general, but a lot more can be said for the special cases where either $m=1$ or $\pi_\#|\gamma|(\Omega)=|Du|(\Omega)$ for $u=\asc{\gamma}$:

\begin{corollary}\label{cor1dliftings}
Every lifting $\gamma\in\Lift(\Omega\times\R)$ is elementary: $\gamma=\gamma\asc{u}$ for some $u\in\BV_\#(\Omega;\R)$.
\end{corollary}

\begin{proof}
If $m=1$, then the operators $\diverg_y$ and $\nabla_y$ coincide, and so~\eqref{eqdivergidentity} now states
\[
\nabla_y \gamma^{\mathrm{gs}}=\nabla_y\left(\gamma\asc{u}\restrict\left(\Jcal_u\times\R\right)\right).
\]
Disintegrating 
\[
\gamma^{\mathrm{gs}}-\gamma\asc{u}\restrict(\Jcal_u\times\R)=\pi_\#|\gamma^{\mathrm{gs}}-\gamma\asc{u}\restrict(\Jcal_u\times\R)|\otimes\rho,
\]
we therefore have that $\nabla_y\rho_x\equiv 0$ for $\pi_\#|\gamma^{\mathrm{gs}}-\gamma\asc{u}\restrict(\Jcal_u\times\R)|$-almost every $x\in\Omega$. Since any distribution whose gradient vanishes must be constant, it must follow that
\[
\rho_x= c_x\Lcal^1
\]
for some constant $c_x\in\R^d$. As $\rho_x\in\mbfM^1(\R;\R^d)$ is a finite measure, however, we see that $c_x=0$ and hence that $\gamma^{\mathrm{gs}}=\gamma\asc{u}\restrict(\Jcal_u\times\R^m)$. It then follows from Theorem~\ref{thmliftingsstructure} that
\[
\gamma=\gamma\asc{u}\restrict(\Omega\setminus\Jcal_u\times\R)+\gamma^{\mathrm{gs}}=\gamma\asc{u}
\]
as required.
\end{proof}

If $\pi_\#|\gamma|(\Omega)=|Du|(\Omega)$, Lemma~\ref{lemliftingspushforward} implies that $|\gamma|(A\times\R^m)=|Du|(A)=|\gamma\asc{u}|(A\times\R^m)$ for any Borel set $A\subset\Omega$. Theorem~\ref{thmliftingsstructure} then yields
\[
|\gamma\asc{u}|(A\times\R^m)=|\gamma|(A\times\R^m)=|\gamma\asc{u}|((A\setminus\Jcal_u)\times\R^m)+|\gamma^{\mathrm{gs}}|(A\times\R^m).
\]
In particular, if $A\cap\Jcal_u=\emptyset$, then $|\gamma\asc{u}|((A\setminus\Jcal_u)\times\R^m)=|\gamma\asc{u}|(A\times\R^m)$ and we deduce 
\[
|\gamma^{\mathrm{gs}}|(A\times\R^m)=0\qquad\text{ for all Borel sets }A\subset\Omega\setminus\Jcal_u.
\]
Thus $\gamma^{\mathrm{gs}}=\gamma\restrict(\Jcal_u\times\R^m)$, and hence
\[
\gamma\restrict((\Omega\setminus\Jcal_u)\times\R^m)=\gamma\asc{u}\restrict((\Omega\setminus\Jcal_u)\times\R^m),\qquad \pi_\#|\gamma^{\mathrm{gs}}|=|D^ju|.
\]
Disintegrating $\gamma^{\mathrm{gs}}=|D^ju|\otimes\theta$ for some weakly* $|D^ju|$-measurable parametrised measure $\theta\colon\Omega\to\mbfM^1(\R^m;\R^{m\times d})$, we see that the second part of Theorem~\ref{thmliftingsstructure} now reads
\begin{equation}\label{eqdivtheta}
\diverg_y \theta_x=-\frac{n_u(x)}{|u^+(x)-u^-(x)|}(\delta_{u^+(x)}-\delta_{u^-(x)})\qquad\text{ for }|D^ju|\text{-almost every }x\in\Jcal_u.
\end{equation}
Lemmas~\ref{lemdivmassonline} and~\ref{lemdivonlineidentify} below, which are special cases of Theorems~5.3 and~D.1 respectively from~\cite{BrCoLi86HMD} and for which simplified proofs can be found in~\cite{RinSha15SCE}, show that, since $|\theta_x|(\R^m)=1$, the identity~\eqref{eqdivtheta} in fact forces 
\[
\int_{\R^m} g(y)\;\dd\theta_x(y)=\int_0^1g(u^\theta(x))\frac{\dd Du}{\dd|Du|}(x)\;\dd\theta,
\]
and hence that $\gamma=\gamma\asc{u}$.

\begin{lemma}\label{lemdivmassonline}
Let $a,b\in\mbR^m$ with $a\neq b$, $c\in\pd\bB^d$ and $\mu\in\mbfM^1(\mbR^m;\mbR^{m\times d})$ be such that
\[
\diverg\mu=\frac{c}{|b-a|}(\delta_b-\delta_a)\text{ in }\mbfM(\R^m;\R^d).
\]
Then
\[
|\mu|([a,b])=|\mu|(\mbR^m),
\]
where $[a,b]$ denotes the (closed) straight line segment between $a$ and $b$.
\end{lemma}

\begin{lemma}\label{lemdivonlineidentify}
Let $a,b\in\mbR^m$ with $a\neq b$ and let $\mu\in\mbfM(\mbR^m;\mbR^{m\times d})$ be such that $\diverg\mu\equiv 0$ and $|\mu|(\mbR^m\setminus[a,b])=0$. Then $\mu=0$.
\end{lemma}

Applying Lemma~\ref{lemdivmassonline} to $\theta_x$ and then applying Lemma~\ref{lemdivonlineidentify} to the measure
\[
\theta_x-\frac{\dd Du}{\dd|Du|}(x)\int_0^1\delta_{u^\theta(x)}\;\dd\theta,
\]
we therefore arrive at the following proposition:

\begin{proposition}\label{lemmustrict}
Let $\gamma\in\Lift(\Omega\times\R^m)$ with $u=\asc{\gamma}$ be \textbf{minimal} in the sense that $|\gamma|(\Omega\times\R^m)=|Du|(\Omega)$. Then $\gamma$ must be elementary, $\gamma=\gamma\asc{u}$. In particular, if $u_j\to u$ in $\BV_\#(\Omega;\R^m)$ strictly, then $\gamma\asc{u_j}\to\gamma\asc{u}$ strictly in $\Lift(\Omega\times\R^m)$.
\end{proposition}

\begin{proof}
It remains only to show the second statement: this follows from the fact that the strict convergence of $u_j$ to $u$ implies that $\lim_{j}|\gamma\asc{u_j}|(\Omega\times\R^m)=|Du|(\Omega)$. Thus, for any (non-relabelled) subsequence converging to a limit $\gamma$, we can apply Lemma~\ref{lemliftingspushforward} and the lower semicontinuity of the total variation on open sets to deduce that $|Du|(\Omega)\geq|\gamma|(\Omega\times\R^m)\geq|Du|(\Omega)$. It follows that $\gamma=\gamma\asc{u}$ and, since this argument can be applied to any subsequence of $(u_j)_j$, that the entire sequence converges $\gamma\asc{u_j}\to\gamma\asc{u}$.
\end{proof}

\subsection{Rescaled liftings and tangent liftings}\label{secliftinghomothety}
\begin{lemma}[Rescaled Liftings]\label{lemrestrictionsofliftings}
Let $\gamma\in\Lift(\Omega\times\R^m)$ and, for $r>0$ and $x_0\in\Dcal_{\asc{\gamma}}\cup\Ccal_{\asc{\gamma}}\cup\Jcal_{\asc{\gamma}}$, let $T^{(x_0,r,s)}_\gamma\colon B(x_0,r)\times\R^m\to\bB^d\times\R^m$ be the homothety given by 
\[
T^{(x_0,r,s)}_\gamma:=\left(\frac{x-x_0}{r},s\frac{y-(\asc{\gamma})_{x_0,r}}{r}\right),
\]
where
\[
(\asc{\gamma})_{x_0,r}=\dashint_{B(x_0,r)}\asc{\gamma}(x)\;\dd x.
\]
Defining the measure $\gamma^{r,s}\in\mbfM({\bB^d}\times\R^m;\R^{m\times d})$ by
\[
\gamma^{r,s}:=\frac{s}{r^d}\left(T^{(x_0,r,s)}_\gamma\right)_\#\gamma,
\]
we have that $\gamma^{r,s}\in\Lift(\bB^d\times\R^m)$ with
\[
\asc{\gamma^{r,s}}(z)=s\frac{\asc{\gamma}(x_0+rz)-(\asc{\gamma})_{x_0,r}}{r},\quad\gr^{\asc{\gamma^{r,s}}}_\#(\lL\restrict\bB^d)=\frac{1}{r^d}(T^{(x_0,r,s)}_{\gamma})_\#\gr^{\asc{\gamma}}_\#(\lL\restrict B(x_0,r)),
\]
and
\[
|\gamma^{r,s}|(\bB^d\times\R^m)=\frac{s}{r^d}|\gamma|(B(x_0,r)\times\R^m).
\]
In addition $\gamma^{r,s}\in\ALift(\bB^d\times\R^m)$ if $\gamma\in\ALift(\Omega\times\R^m)$ and $(\gamma\asc{u})^{r,s}=\gamma\asc{(s/c_r)u^r}$, where $u^r$ and $c_r$ are defined as in Theorem~\ref{thmbvblowup}.
\end{lemma}

In the special case where $s=c_r$, we shall use the abbreviations
\begin{equation} \label{eqrescaleabbrev}
\gamma^r:=\gamma^{r,c_r},\qquad T^{(x_0,r)}_\gamma:=T^{(x_0,r,c_r)}_\gamma.
\end{equation}

\begin{proof}
That $\gamma^{r,s}\in\Lift(\bB^d\times\R^m)$ follows from the fact that, for $\varphi\in\C^1_0(\bB^d\times\R^m)$, we can compute	
\begin{align*}
\ip{\nabla_y\varphi}{\frac{s}{r^d}{(T^{(x_0,r,s)}_\gamma)}_\#(\gamma)}&=\frac{s}{r^d}\ip{(\nabla_y\varphi)\circ T^{(x_0,r,s)}_\gamma}{\gamma}\\
&=r^{1-d}\ip{\nabla_y\left(\varphi\circ T^{(x_0,r,s)}_\gamma\right)}{\gamma}\\
&=-r^{1-d}\int\nabla_x\bigl(\varphi\circ T^{(x_0,r,s)}_\gamma\bigr)(x,\asc{\gamma}(x))\;\dd x\\
&=-r^{-d}\int\nabla_x\varphi\left(\frac{x-x_0}{r},s\frac{\asc{\gamma}(x)-(\asc{\gamma})_{x_0,r}}{r}\right)\;\dd x\\
&=-\int\nabla_x\varphi\left(z,s\frac{\asc{\gamma}(x_0+rz)-(\asc{\gamma})_{x_0,r}}{r}\right)\;\dd z.
\end{align*}
It is also clear that
\[
\int_{\bB^d}s\frac{\asc{\gamma}(x_0+rz)-(\asc{\gamma})_{x_0,r}}{r}\;\dd z=0,
\]
and so we see that $\gamma^{r,s}\in\Lift(\bB^d\times\R^m)$ with $\asc{\gamma^{r,s}}(z)=s\frac{\asc{\gamma}(x_0+rz)-(\asc{\gamma})_{x_0,r}}{r}$ as required. To identify $\gr^{\asc{\gamma^{r,s}}}_\#(\lL\restrict\bB^d)$, we need simply observe that
\begin{align*}
\int\varphi(z,w)\;\dd \gr^{\asc{\gamma^{r,s}}}_\#(\lL\restrict\bB^d)(z,w)&=\int\varphi\left(z,s\frac{\asc{\gamma}(x_0+rz)-(\asc{\gamma})_{x_0,r}}{r}\right)\;\dd z\\
&=\frac{1}{r^d}\int\varphi\left(\frac{x-x_0}{r},s\frac{\asc{\gamma}(x)-(\asc{\gamma})_{x_0,r}}{r}\right)\;\dd z\\
&=\frac{1}{r^d}\int_{B(x_0,r)\times\R^m}\varphi\left(\frac{x-x_0}{r},s\frac{y-(\asc{\gamma})_{x_0,r}}{r}\right)\;\dd\gr^{\asc{\gamma}}_\#(\lL)(x,y)\\
&=\frac{1}{r^d}\int\varphi(z,w)\;\dd\left[(T^{(x_0,r,s)}_{\gamma})_\#\gr^{\asc{\gamma}}_\#(\lL\restrict B(x_0,r))\right](z,w)
\end{align*}
for all $\varphi\in\C_0(\bB^d\times\R^m)$.

Moreover,
\[
|\gamma^{r,s}|(\bB^d\times\R^m)=\frac{s}{r^d}|\gamma|\left((T^{(x_0,r,s)}_\gamma)^{-1}(\bB^d\times\R^m)\right)=\frac{s}{r^d}|\gamma|(B(x_0,r)\times\R^m).
\]
To show that $\gamma^{r,s}\in\ALift(\bB^d\times\R^m)$ if $\gamma\in\ALift(\Omega\times\R^m)$, it suffices to note that, if $(\gamma\asc{u_j})_j$ is such that $\gamma\asc{u_j}\wsc\gamma$, then the liftings $\gamma_{j,r,s}\in\Lift(\bB^d\times\R^m)$ defined by
\[
\int\varphi(z,w)\;\dd\gamma_{j,r,s}(z,w)={s}\int\varphi\left(\frac{x-x_0}{r},s\frac{y-(u_j)_{x_0,r}}{r}\right)\;\dd\gamma_j(x,y)\quad\text{ for }\varphi\in\C_0(\Omega\times\R^m)
\]
converge weakly* to $\gamma^{r,s}$ as $j\to\infty$.

Finally,
\begin{align*}
\int \varphi(z,w)\;\dd(\gamma\asc{(u})^{r,s}(z,w)&={s}\int\int_0^1\varphi\left(\frac{x-x_0}{r},s\frac{u^\theta(x)-(u)_{x_0,r}}{r}\right)\;\dd\theta\;\dd Du(x)\\
&=\int\int_0^1\varphi\left(z,\frac{s}{c_r}(u^r)^\theta(x)\right)\;\dd\theta\;\dd D \left[\frac{s}{c_r}u^r\right](z)\\
&=\int\varphi(z,w)\;\dd\gamma\left[\frac{s}{c_r}u^r\right](z,w),
\end{align*}
which shows $\gamma\asc{(s/c_r)u^r}=(\gamma\asc{u})^{r,s}$.
\end{proof}

The following result describes how liftings can be blown up around points $x_0\in\Dcal_{\asc{\gamma}}\cup\Ccal_{\asc{\gamma}}$ in a similar fashion to the results presented for $\BV$-functions in Theorem~\ref{thmbvblowup}. Theorem~\ref{thmdiffusetangentlifting} will be an indispensable tool for the localisation arguments carried out in Section~\ref{chaptangentliftingyms}. A similar procedure can be carried out at points $x_0\in\Jcal_u$ but this is not necessary for the rest of this paper and is therefore omitted for brevity.

\begin{theorem}[Tangent Liftings at diffuse points]\label{thmdiffusetangentlifting}
Let $x_0\in\Dcal_u\cup\Ccal_u$ and $u^r$, $u^0\in\BV(\bB^d;\R^m)$ be as defined in Theorem~\ref{thmbvblowup}. Then, for $\lL+|D^c u|$-almost every $x_0\in\Dcal_u\cup\Ccal_u$,
\begin{itemize}
\item if $x_0\in\Dcal_u$ then $u^r\to u^0$ and $\gamma\asc{u^r}\to\gamma\asc{u^0}$ strictly as $r\to 0$,
\item if $x_0\in\Ccal_u$ then for any sequence $r_n\downarrow 0$ and $\varepsilon>0$, there exists $\tau\in(1-\varepsilon,1)$ such that $u^{\tau r_n}\to u^0$ and $\gamma\asc{u^{\tau r_n}}\to\gamma\asc{u^0}$ strictly as $n\to\infty$.
\end{itemize}
\end{theorem}

\begin{proof}
Combine Theorem~\ref{thmbvblowup} with Proposition~\ref{lemmustrict}.
\end{proof}

\subsection{Perspective constructions and integral representations}\label{secliftingperspective}

Given an integrand $f\in\Rbf(\Omega\times\R^m)$, define the corresponding \textbf{perspective integrand}, $Pf\colon\overline{\Omega}\times\R^m\times\R^{m\times d}\times\R\to\R$ by
\[
(Pf)(x,y,(A,t)):=
\begin{cases}
|t|f(x,y,|t|^{-1}A)&\text{ if }|t|>0,\\
f^\infty(x,y,A)&\text { if }t=0.
\end{cases}
\]
Similarly, the \textbf{perspective measure} $P\gamma\in\mbfM(\Omega\times\R^m;\R^{m\times d}\times\R)$ of a lifting $\gamma\in\Lift(\Omega\times\R^m)$ is defined by
\[
P\gamma:=\left(\gamma,\gr^{\asc{\gamma}}_\#(\lL\restrict\Omega)\right).
\]
Note that $Pf$ is positively one-homogeneous in the $(A,t)$-argument. Corollary~\ref{corliftingtographcts} implies that $P\gamma_j\wsc P\gamma$ if $\gamma_j\wsc\gamma$ in $\Lift(\Omega\times\R^m)$ and, letting
\[
\gamma=\nabla\asc{\gamma}\gr^{\asc{\gamma}}_\#(\lL\restrict\Omega)+\gamma^s
\]
be the Radon--Nikodym decomposition of $\gamma$ with respect to $\gr^{\asc{\gamma}}_\#(\lL\restrict\Omega)$ (that $\frac{\dd\gamma}{\dd\gr^{\asc{\gamma}}_\#(\lL\restrict\Omega)}=\nabla\asc{\gamma}$ follows from Theorem~\ref{thmliftingsstructure}), we also see that $P\gamma$ admits the following decomposition with respect to $\gr^{\asc{\gamma}}_\#(\lL\restrict\Omega)$,
\begin{equation*}
P\gamma=(\nabla \asc{\gamma},1)\gr^{\asc{\gamma}}_\#(\lL\restrict\Omega)+(\gamma^s,0).
\end{equation*}
Thus, $P\gamma_j\to P\gamma$ strictly in $\mbfM(\Omega\times\R^m;\R^{m\times d}\times\R)$ if and only if $\gamma_j\wsc\gamma$ and
\[
\int_\Omega\sqrt{|\nabla \asc{\gamma_j}(x)|^2+1}\;\dd x+|\gamma^s_j|(\Omega\times\R^m)\to\int_\Omega\sqrt{|\nabla\asc{\gamma}(x)|^2+1}\;\dd x+|\gamma^s|(\Omega\times\R^m)
\]
as $j\to\infty$. Note that this holds for $(\gamma\asc{u_j})_j$ and $\gamma\asc{u}$ whenever $u_j\to u$ area-strictly in $\BV_\#(\Omega;\R^m)$.

Recalling the notation introduced in~\eqref{eqreshetnyakrepresentation} in Section~\ref{secpreliminaries}, we see that these perspective constructions permit the following computation:
\begin{align*}
\int_{\Omega\times\R^m} Pf(x,y,P\gamma)=&\int_{\Omega\times\R^m}Pf\left(x,y,(\nabla \asc{\gamma},1)\gr^{\asc{\gamma}}_\#(\lL\restrict\Omega)\right)+\int_{\Omega\times\R^m}Pf\left(x,y,(\gamma^s,0)\right)\\
=&\int_{\Omega}Pf\left(x,\asc{\gamma}(x),\frac{(\nabla\asc{\gamma}(x),1)}{\sqrt{|\nabla\asc{\gamma}(x)|^2+1}}\right)\sqrt{|\nabla\asc{\gamma}(x)|^2+1}\;\dd x\\
&\qquad\qquad+\int_{\Omega\times\R^m}f^\infty\left(x,y,\gamma^s\right)\\
=&\int_\Omega f(x,\asc{\gamma}(x),\nabla\asc{\gamma}(x))\;\dd x+\int_{\Omega\times\R^m}f^\infty\left(x,y,\gamma^s\right).
\end{align*}
Defining $\Fcal_\Lrm\colon\Lift(\Omega\times\R^m)\to\R$ by
\[
\Fcal_\Lrm[\gamma]:=\int_{\Omega\times\R^m} Pf(x,y,P\gamma)=\int_\Omega f(x,\asc{\gamma}(x),\nabla\asc{\gamma}(x))\;\dd x+\int_{\Omega\times\R^m}f^\infty\left(x,y,\gamma^s\right),
\]
we therefore recover
\[
\Fcal[u]=\Fcal_\Lrm[\gamma\asc{u}]\qquad\text{ for }u\in\BV_\#(\Omega;\R^m).
\]
For a general element $u\in\BV(\Omega;\R^m)$, we can still write
\begin{equation}\label{eqliftingrep}
\Fcal[u]=\int_{\Omega\times\R^m}Pf(x,y+(u),P\gamma\asc{u-(u)}),\qquad (u):=\dashint_\Omega u(x)\;\dd x.
\end{equation}
Lemma~\ref{lemliftingsperturb} below and Proposition~\ref{lemextendedrepresentation} in the next section imply that, whenever $(u_j)_j$ is a sequence converging weakly* to $u$ in $\BV(\Omega;\R^m)$, we can expect that
\begin{equation*}
\lim_{j\to\infty}\left|\int_{\Omega\times\R^m}Pf(x,y+(u_j),P\gamma\asc{u_j-(u_j)})-\int_{\Omega\times\R^m} Pf(x,y+(u),P\gamma\asc{u_j-(u_j)})\right|=0,
\end{equation*}
for a large class of integrands $f$, and so for these we can gain a complete understanding of the functional $\Fcal$ on all of $\BV(\Omega;\R^m)$ by considering $\Fcal_\Lrm$ on $\Lift(\Omega;\R^m)$.

\begin{lemma}\label{lemliftingsperturb}
If $f\in\mbfE(\Omega\times\R^m)$, $(\gamma_j)_j\subset\Lift(\Omega\times\R^m)$ is a bounded sequence, and $(c_j)_j\subset\R^m$ is such that $c_j\to c$ as $j\to\infty$, then
\[
\lim_{j\to\infty}\left|\int_{\Omega\times\R^m} Pf(x,y+c_j,P\gamma_j)-\int_{\Omega\times\R^m} Pf(x,y+c,P\gamma_j)\right|=0.
\]
\end{lemma}

\begin{proof}
We claim first that $f\in\mbfE(\Omega\times\R^m)$ implies that $Pf\in\C(\overline{\Omega}\times\sigma\R^m\times\R^{m\times d}\times\R)$. The continuity of $Pf$ at points $(x,y,A,t)\in\overline{\Omega}\times\sigma\R^m\times\R^{m\times d}\times(\R\setminus \{0\})$ is clear from the continuity of $f$, whereas continuity at points $(x,y,A,0)$ for $(x,y,A)\in\overline{\Omega}\times\sigma\R^m\times\R^{m\times d}$ follows directly from the definition of $Pf$ combined with the fact that $\sigma f^\infty$ exists in the stronger sense that the limit~\eqref{eqextendedffore} always returns $\sigma f^\infty$ as noted in Section~\ref{secpreliminaries}. Consequently, $Pf$ must be uniformly continuous on the compact set $\overline{\Omega}\times\sigma\R^m\times\pd\bB^{m\times d +1}$ and so we can estimate
\begin{align*}
&\lim_{j\to\infty}\left|\int_{\Omega\times\R^m} Pf(x,y+c_j,P\gamma_j)-\int_{\Omega\times\R^m} Pf(x,y+c,P\gamma_j)\right|\\
&\qquad\leq \lim_{j\to\infty}\int_{\Omega\times\R^m} \left| Pf\left(x,y+c_j,\frac{\dd P\gamma_j}{\dd|P\gamma_j|}(x,y)\right)- Pf\left(x,y+c,\frac{\dd P\gamma_j}{\dd|P\gamma_j|}(x,y)\right)\right|\;\dd|P\gamma_j|(x,y)\\
&\qquad\leq \lim_{j\to\infty} m(|c_j-c|)|P\gamma_j|(\Omega\times\R^m)\\
&\qquad=0,
\end{align*}
where $m\colon [0,\infty)\to[0,\infty)$ is a modulus of continuity for $Pf$ so that 
\[
|Pf(x_1,y_1,A_1,t_1)-Pf(x_2,y_2,A_2,t_2)|\leq m(|(x_1,y_1,A_1,t_1)-(x_2,y_2,A_2,t_2)|)
\]
for all $(x_1,y_1,A_1,t_1),(x_2,y_2,A_2,t_2)\in\overline{\Omega}\times\sigma\R^m\times\pd\bB^{m\times d +1}$ and $m(r)\to 0$ as $r\to 0$.
\end{proof}

%% file: youngmeasures.tex
\section{Young measures for liftings}\label{chapyoungmeasures}
This section considers Young measures associated to liftings and presents the technical machinery required to manipulate them. These Young measures are compactness tools associated to weakly* convergent sequences $(\gamma_j)_j\subset\Lift(\Omega\times\R^m)$ whose purpose is to allow us to compute the limiting behaviour of sequences $(\int Pf(x,y,P\gamma_j))_j$ for as large a class of integrands $f$ as possible simultaneously. Crucially, every bounded sequence $(\gamma_j)_j\subset\Lift(\Omega\times\R^m)$ can (upon passing to a non-relabelled subsequence) be assumed to generate a Young measure $\bnu$ with the key property that
\[
\lim_{j\to\infty}\int Pf(x,y,P\gamma_j)=\ddprb{f,\bnu}\qquad\text{ for all }f\in\mbfE(\Omega\times\R^m),
\]
where the duality product $\ddpr{f,\bnu}$ is defined below in Section~\ref{secyms}. Together with the localisation principles developed in Section~\ref{chaptangentliftingyms}, the theory developed in this section will lead to a proof of the lower semicontinuity component of Theorem~\ref{wsclscthm}.

The plan for this section is as follows: first, in Section~\ref{secyms}, Young measures over $\Omega\times\R^m$ with target space $\R^{m\times d}$ are defined abstractly. The elementary Young measures associated to elements of $\Lift(\Omega\times\R^m)$ are then introduced, as are the concepts of Young measure convergence, generation, and the duality product $\ddpr{\frarg,\frarg}$. Section~\ref{secymscompact} is concerned with the proof of the Young measure Compactness Theorem (Theorem~\ref{thmymsseqcompact}), which shows that any bounded sequence $(\gamma_j)_j\subset\Lift(\Omega\times\R^m)$ possesses a subsequence generating a Young measure $\bnu$. Section~\ref{secymsmanipulate} then introduces several techniques for manipulating Young measures and, finally, Section~\ref{secymsrepresentation} proves that Young measures can be also used to gain insight into $\lim_{j}\int Pf(x,y,P\gamma_j)$ for integrands $f$ from the larger class $\mbfR(\Omega\times\R^m)$.

\subsection{Young measures}\label{secyms}
In what follows, we will make use of the compactified space $\sigma\R^m:=\R^m\uplus\infty\pd\bB^m$ and the associated spaces of continuous functions and measures introduced and discussed in Section~\ref{secpreliminaries}.
\begin{definition}[Generalised Young measures]\label{defyms}
An $(\Omega\times\R^m;\R^{m\times d})$-\textbf{Young measure} is a quadruple $\bnu=\left(\iota_\nu,\nu,\lambda_\nu,\nu^\infty\right)$ such that
\begin{enumerate}
 \item\label{ymdefcond1} $\iota_\nu\in\mbfM^+(\Omega\times\R^m)$ satisfies $\pi_\#(\iota_\nu)=\mathcal{L}^d\restrict\Omega$ and $\lambda_\nu\in\mbfM^+(\overline{\Omega}\times\sigma\R^m)$;
 \item\label{ymdefcond2} $\nu:\Omega\times\R^m\to\mbfM^1(\R^{m\times d})$ and $\nu^\infty:\overline{\Omega}\times\sigma\R^m\to\mbfM^1(\pd\mathbb{B}^{m\times d})$ are parametrised measures which are weakly*-measurable with respect to $\iota_\nu$ and $\lambda_\nu$ respectively. That is, $(x,y)\mapsto\nu_{x,y}(A)$ is $\iota_\nu$-measurable for every Borel set $A\subset\R^{m\times d}$ and $(x,y)\mapsto\nu^\infty_{x,y}(B)$ is $\lambda_\nu$-measurable for every Borel set $B\subset\pd\bB^{m\times d}$;
 \item\label{ymdefcond3} $(x,y)\mapsto\ip{|\frarg|}{\nu_{x,y}}	\in \Lp^1\left(\Omega\times\R^m,\iota_\nu\right)$;
 \item\label{ymdefcond4} $(\iota_\nu,\nu,\lambda_\nu)$ satisfies the \textbf{Poincar\'e condition}
\begin{equation}\label{eqpoincare}
\mkern-18mu \int_{\Omega\times\R^m}|y|^{d/(d-1)}\;\dd\iota_\nu(x,y)\leq M_d\left(\int_{\Omega\times\R^m}\ip{|\frarg|}{\nu_{x,y}}\;\dd\iota_\nu(x,y)+\lambda_\nu(\overline{\Omega}\times\sigma\R^m)\right)^{d/(d-1)}
\end{equation}
where $M_d:=\inf\{|Du|(\Omega)\colon u\in\BV_{\sharp}(\Omega;\R^m)\text{ and }\norm{u}_{\Lp^{d/(d-1)}}= 1\}$.
\end{enumerate}
We denote the set of all $(\Omega\times\R^m;\R^{m\times d})$-Young measures by $\mathbf{Y}(\Omega\times\R^m;\R^{m\times d})$. When there is no risk of confusion, we will usually abbreviate this to $\mathbf{Y}$.
\end{definition}
We shall refer to $\iota_\nu$ and $\lambda_\nu$ as the \textbf{reference} and \textbf{concentration measures} of $\bnu$, respectively, and $\nu$, $\nu^\infty$ as the \textbf{regular} and \textbf{singular oscillation measures} of $\bnu$.

Recall from Definition~\ref{defelemintegrands} that $\mbfE(\Omega\times\R^m)$ is defined to be the set of functions $f\in\C(\Omega\times\R^m\times\R^{m\times d})$ for which there exists $g_f\in\C(\overline{\Omega}\times\sigma\R^m\times\overline{\bB^{m\times d}})$ satisfying
\[
f(x,y,A)=(1+|A|)g_f\left(x,y,\frac{A}{1+|A|}\right).
\]
Under the \textbf{duality product}
\begin{equation}\label{eqdualityproduct}
\ddprb{f,\bnu}:=\int_{\Omega\times\R^m}\ip{f(x,y,\frarg)}{\nu_{x,y}}\;\dd \iota_\nu(x,y)+\int_{\overline{\Omega}\times\sigma\R^m}\ip{\sigma f^\infty\left(x,y,\frarg\right)}{\nu_{x,y}^\infty}\;\dd\lambda_\nu(x,y)
\end{equation}
for $f\in\mbfE(\Omega\times\R^m)$, Young measures can be seen as continuous linear functionals on $\mbfE(\Omega\times\R^m)$.

 Whilst conditions~\eqref{ymdefcond1},~\eqref{ymdefcond2} and~\eqref{ymdefcond3} are natural from the point of ensuring that~\eqref{eqdualityproduct} is well-defined for integrands $f\in\mbfE(\Omega\times\R^m)$,~\eqref{ymdefcond4} is a technical condition (used in the proofs of Corollary~\ref{corymsclosed} and Proposition~\ref{lemextendedrepresentation}) to ensure that the sequence of measures $(\iota_{\nu_j})_{j\in\mbN}$ is tight in $\mbfM^+(\Omega\times\R^m)$ whenever the sequence $(\bnu_j)_{j\in\mbN}\subset\mbfY(\Omega\times\R^m;\R^{m\times d})$ converges in the sense of Definition~\ref{defymconvergence} below.

\begin{definition}[Elementary Young measures]\label{defelemyms}
Given a lifting $\gamma\in\Lift(\Omega\times\R^m)$, the \textbf{elementary Young measure} $\bdelta=\bdelta\asc{\gamma}=(\iota_\delta,\delta,\lambda_\delta,\delta^\infty)\in\mbfY(\Omega\times\R^m;\R^{m\times d})$ associated to $\gamma$ is defined by
\begin{align*}
\iota_{\delta}:={\gr_\#^{\asc{\gamma}}}(\lL\restrict\Omega),\quad& \delta_{x,y}:=\delta_{\nabla \asc{\gamma}(x)},\\
\lambda_{\delta}:=|\gamma^s|,\quad& \delta^\infty_{x,y}:=\delta_{\frac{\dd\gamma^s}{\dd|\gamma^s|}(x,y)},
\end{align*}
where
\begin{equation*}
\gamma^s:=\gamma-\gr_\#^{\asc{\gamma}}(\nabla\asc{\gamma}\,\lL\restrict\Omega)={\gr_\#^{\asc{\gamma}}} D^c \asc{\gamma}+\gamma^{\mathrm{gs}}
\end{equation*}
is the \textbf{Lebesgue singular part} of $\gamma$.
\end{definition}
If $\gamma\in\Lift(\Omega\times\R^m)$, it follows that we have the representation
\[
  \ip{Pf}{P\gamma}=\ddprb{f,\bdelta\asc{\gamma}}
\]
for $f\in\mbfE(\Omega\times\R^m)$. Note that if $\gamma\asc{u}$ is an elementary lifting and $\bdelta\asc{\gamma\asc{u}}$ is the elementary Young measure associated to $\gamma\asc{u}$, then
\[
\lambda_\delta=|D^su|\otimes\left(\int_0^1\delta_{u^\theta}\;\dd\theta\right),\quad \delta^\infty_{x,y}=\delta_{p(x)},\quad p(x)=\frac{\dd D^s u}{\dd|D^s u|}(x),
\]
so that
\[
  \ddprb{f,\bdelta\asc{\gamma\asc{u}}}=\ip{Pf}{P\gamma\asc{u}}=\Fcal[u].
\]

\begin{definition}[Young measure convergence]\label{defymconvergence}
We say that a sequence $(\bnu_j)_j\subset\mbfY(\Omega\times\R^m;\R^{m\times d})$ weakly* converges to $\bnu\in\mbfY(\Omega\times\R^m;\R^{m\times d})$ as $j\to\infty$ (written $\bnu_j\wsc\bnu$) if
\[
 \ddprb{f,\bnu_j}\to\ddprb{f,\bnu}\text{ as }j\to\infty
\]
for every $f\in\mbfE(\Omega\times\R^m)$.
\end{definition}
If $\bnu\in\mbfY$ is such that $\bdelta\asc{\gamma_j}\wsc\bnu$ for some sequence of elementary Young measures associated to $(\gamma_j)_j\subset\Lift(\Omega\times\R^m)$, then we simply write $\gamma_j\toY\bnu$, saying that $(\gamma_j)_j$ \textbf{generates} $\bnu$. Equivalently, $(\gamma_j)_j$ generates $\bnu$ if and only if
\[
\lim_{j\to\infty}\int Pf(x,y,P\gamma_j)=\ddprb{f,\bnu}\quad\text{ for all }f\in\mbfE(\Omega\times\R^m).
\]

\begin{definition}[Young measure mass]
The \textbf{Young measure mass} $\norm{\bnu}_\mbfY$ of an element $\bnu\in\mbfY(\Omega\times\R^m;\R^{m\times d})$ is defined by
\[
 \norm{\bnu}_{\mbfY}:=\ddprb{\mathbbm{1}\otimes|\frarg|,\bnu},
\]
where $\mathbbm{1}\otimes|\frarg|\in\mbfE(\Omega\times\R^m)$ denotes the function $(x,y,A)\mapsto|A|$.
\end{definition}
The set of Young measures $\mbfY(\Omega\times\R^m;\R^{m\times d})$ is not a linear space and so $\norm{\frarg}_\mbfY$ cannot be a norm in the technical sense of being positively homogeneous and subadditive. Nevertheless, $\norm{\frarg}_\mbfY$-bounded sequences in $\mbfY(\Omega\times\R^m;\R^{m\times d})$ share the same compactness properties as norm-bounded sequences in Banach spaces.

We note that the Poincar\'e condition~\eqref{eqpoincare} can now be rephrased as
\[
\left(\int_{\Omega\times\R^m}|y|^{d/(d-1)}\;\dd\iota_\nu(x,y)\right)^{(d-1)/d}\leq M_d\norm{\bnu}_\mbfY.
\]

\begin{definition}[$\YLift$]
The space $\YLift(\Omega\times\R^m)\subset\mbfY(\Omega\times\R^m;\R^{m\times d})$ is defined to be the weak*~sequential closure of the set of elementary Young measures associated to liftings. That is, $\bnu\in\YLift(\Omega\times\R^m)$ if and only if there exists a sequence $(\gamma_j)_j\subset\Lift(\Omega\times\R^m)$ such that $\gamma_j\toY\bnu$.
\end{definition}

\begin{definition}
The space $\AYLift(\Omega\times\R^m)\subset\YLift(\Omega\times\R^m)$ is defined to be the weak* sequential closure of the set of elementary Young measures associated to elementary liftings. That is, $\bnu\in\AYLift(\Omega\times\R^m)$ if and only if there exists a sequence $(\gamma\asc{u_j})_j\subset\Lift(\Omega\times\R^m)$ such that $\gamma\asc{u_j}\toY\bnu$.
\end{definition}

Since the convergence $\toY$ preserves more information than weak* convergence in the space of liftings, we can show that $\AYLift(\Omega\times\R^m)$ is sequentially closed under weak* convergence in $\mbfY(\Omega\times\R^m;\R^{m\times d})$, even though it is not known whether $\ALift(\Omega\times\R^m)$ is sequentially weakly* closed in $\mbfM(\Omega\times\R^m;\R^{m\times d})$:
\begin{lemma}\label{lemliftymsclosed}
The spaces $\YLift(\Omega\times\R^m)$ and $\AYLift(\Omega\times\R^m)$ are closed under weak* Young measure convergence.
\end{lemma}

\begin{proof}
Since $\mathbbm{1}\otimes|\frarg|\in\mbfE(\Omega\times\R^m)$, $\bnu_j\wsc\bnu$ in $\mbfY$ implies that $\norm{\bnu_j}_\mbfY\to\norm{\bnu}_\mbfY$. Thus, if $(\bnu_j)_j\subset\YLift(\Omega\times\R^m)$ is a mass-bounded sequence such that $\bnu_j\wsc\bnu$ in $\mbfY$, we can assume that each $\bnu_j$ is the limit of a sequence $(\bdelta_j^k)_{k\in\mbN}$ of elementary Young measures (associated either to liftings in the case $\YLift(\Omega\times\R^m)$ or elementary liftings in the case $\AYLift(\Omega\times\R^m)$) satisfying $\sup_k\norm{\bdelta_j^k}_\mbfY\leq 1+\norm{\bnu_j}_\mbfY\leq 2+\norm{\bnu}_\mbfY$. We can then use a diagonal argument to extract a sequence $(\bdelta_j^{k_j})_j$ of elementary Young measures which generates $\bnu$.
\end{proof}

Let $\bnu\in\YLift(\Omega\times\R^m)$ and $(\gamma_j)_j\subset\Lift(\Omega\times\R^m)$ be such that $\gamma_j\toY\bnu$. Testing with integrands $f(x,y,A)=\varphi(x,y)A\in\mbfE(\Omega\times\R^m)$ for $\varphi\in\C_0(\Omega\times\R^m)$, we see that $({\gamma_j})_j$ is a weakly* convergent sequence in $\mbfM(\Omega\times\R^m;\R^{m\times d})$, whose limit we denote by $\gamma$. Since
\[
\int_{\Omega\times\R^m}\varphi(x,y)\;\dd{\gamma}(x,y)=\ddprb{\varphi\otimes\id,\bnu},
\]
it follows that $\gamma$ is determined solely by $\bnu$ independently of the choice of generating sequence $(\bdelta_j)_j$ and that the following definition is coherent:
\begin{definition}
The \textbf{barycentre} of a Young measure $\bnu\in\mbfY(\Omega\times\R^m;\R^{m\times d})$ is the measure $\asc{\bnu}\in\mbfM(\Omega\times\R^m;\R^{m\times d})$ defined by
\begin{equation*}
\ip{\varphi}{\asc{\bnu}}=\ddprb{\varphi\otimes\id,\bnu}\quad\text{ for all } \varphi\in\C_{0}(\Omega\times\R^m).
\end{equation*}
For elements $\bnu\in\YLift(\Omega\times\R^m)$, if $\asc{\bnu}=\gamma$ and $\asc{\gamma}=u$, we will use the notation $\llbracket\bnu\rrbracket:=u$.
\end{definition}

Clearly, if $\bdelta$ is the elementary Young measure associated to $\gamma\in\Lift(\Omega\times\R^m)$ then $\asc{\bdelta}=\gamma$. It is also clear that $\asc{\bnu}\in\Lift(\Omega\times\R^m)$ whenever $\bnu\in\YLift(\Omega\times\R^m)$ and that  $\asc{\bnu}\in\ALift(\Omega\times\R^m)$ whenever $\bnu\in\AYLift(\Omega\times\R^m)$.

It follows immediately that $\asc{\bnu_j}\wsc\asc{\bnu}$ in $\Lift(\Omega\times\R^m)$ whenever $\bnu_j\wsc\bnu$ in $\YLift(\Omega\times\R^m)$. Note that it need not be the case that $\bnu\in\AYLift(\Omega\times\R^m)$ if $\asc{\bnu}\in\ALift(\Omega\times\R^m)$.

The following corollary is now a direct consequence of Corollary~\ref{corliftingtographcts} combined with the identity $\ddpr{\varphi\otimes\mathbbm{1},\bnu}=\int\varphi(x,y)\;\dd\iota_\nu(x,y)$ for $\varphi\in\C_0(\Omega\times\R^m)$.
\begin{corollary}\label{lemstructureliftingym}
If $\bnu\in\YLift(\Omega\times\R^m)$, then
\[
\iota_\nu=\gr_\#^{\llbracket\bnu\rrbracket}(\lL\restrict\Omega).
\]
\end{corollary}

\subsection{Duality and compactness}\label{secymscompact}

For $f \in \mbfE(\Omega\times\R^m)$ let $g_f\in\C(\overline{\Omega}\times\sigma\R^m\times\overline{\bB^{m\times d}})$ be such that 
\[
f(x,y,A)=(1+|A|)g_f\left(x,y,\frac{A}{1+|A|}\right)
\]
according to Definition~\ref{defelemintegrands}. Define the linear operator $S\colon \mbfE(\Omega\times\R^m)\to \C(\overline{\Omega}\times\sigma\R^m\times\overline{\mathbb{B}^{m\times d}})$ by
\[
Sf=g_f,
\]
and note that $S$ admits an inverse $S^{-1}\colon\C(\overline{\Omega}\times\sigma\R^m\times\overline{\mathbb{B}^{m\times d}})\to\mbfE(\Omega\times\R^m)$ given by
\begin{equation}\label{eqdefisom}
(S^{-1}g)(x,y,A)=(1+|A|)g\left(x,y,\frac{A}{1+|A|}\right)\qquad\text{ for }(x,y,A)\in\overline{\Omega}\times\R^m\times\R^{m\times d}.
\end{equation}
It follows that $S$ is a linear isomorphism between $\mbfE(\Omega\times\R^m)$ and $\C(\overline{\Omega}\times\sigma\R^m\times\overline{\bB^{m\times d}})$, and so $\mbfE(\Omega\times\R^m)$ can be given the structure of a Banach space by setting 
\[
\norm{f}_{\mbfE}:=\norm{Sf}_{\C(\overline{\Omega}\times\sigma\R^m\times\overline{\mathbb{B}^{m\times d}})}.
\]
 
We therefore have that $(S^*)^{-1}$ (hereafter denoted by $S^{-*}$) is itself an isometric isomorphism between $(\mbfE(\Omega\times\R^m))^*$ and $\mbfM(\overline{\Omega}\times\sigma\R^m\times\overline{\mathbb{B}^{m\times d}})$ (via the Riesz Representation Theorem). In particular,
Note that if $\bnu\in\mbfY(\Omega\times\R^m;\R^{m\times d})$, then
\begin{align}
|\Omega| + \norm{\bnu}_\mbfY=\ddprb{1 + \mathbbm{1}\otimes|\frarg|,\bnu}=\ip{S(1 + \mathbbm{1}\otimes|\frarg|)}{S^{-*}\bnu}&=\int_{\overline{\Omega}\times\sigma\R^m\times\overline{\bB^{m\times d}}}\;\dd S^{-*}\bnu(x,y,B) \notag\\
&=(S^{-*}\bnu)(\overline{\Omega}\times\sigma\R^m\times\overline{\bB^{m\times d}}), \label{eqnumumass}
\end{align}
for all $\nu \in \mbfY$.

With $S$ and $S^{-*}$ so defined, the duality product between Young measures $\bnu\in\mbfY(\Omega\times\R^{m\times d};\R^{m\times d})$ and integrands $f\in\mbfE(\Omega\times\R^m)$ can be understood in terms of a more familiar integral product between measures and continuous functions:
\begin{equation}
\ddprb{f,\bnu}=\int_{\overline{\Omega}\times\sigma\R^m\times\overline{\bB^{m\times d}}} Sf(x,y,B)\;\dd S^{-*}\bnu(x,y,B).
\end{equation}

Given $f\in\RL(\Omega\times\R^m)$ or $f\in\RBVw(\Omega\times\R^m)$, we can define $Sf\colon\overline{\Omega}\times\R^m\times\overline{\bB^{m\times d}}\to\R$ by analogy via
\[
(Sf)(x,y,B):=\begin{cases}(1-|B|)f\left(x,y,\frac{B}{1-|B|}\right)\quad & \text{ if }(x,y,B)\in\overline{\Omega}\times\R^m\times\bB^{m\times d},\\
f^\infty(x,y,B)\quad & \text{ if }(x,y,B)\in\overline{\Omega}\times\R^m\times\pd\bB^{m\times d}.
\end{cases}
\]
So defined, $Sf$ is always Carath\'eodory on $\overline{\Omega}\times\R^m\times\overline{\bB^{m\times d}}$ and is continuous at each point $(x,y,B)\in\overline{\Omega}\times\R^m\times\pd\bB^{m\times d}$. Moreover, if $f$ is continuous and satisfies $|f(x,y,A)|\leq C(1+|A|)$ then it is clear that $Sf\in\C_b(\overline{\Omega}\times\R^m\times\overline{\bB^{m\times d}})$.

\begin{proposition}\label{thmymsasmeasures}
 The set $S^{-*}\mbfY(\Omega\times\R^m;\R^{m\times d})$ consists precisely of those measures $\mu\in\mbfM^+(\overline{\Omega}\times\sigma\R^m\times\overline{\mathbb{B}^{m\times d}})$ that satisfy the conditions
\begin{equation}\label{eqymcond2a}
\mu\bigl(\overline{\Omega}\times\infty\pd\bB^m\times\bB^{m\times d}\bigr)=0,
\end{equation}

\begin{equation}\label{eqymcond2b}
	\left(\int_{\overline{\Omega}\times\R^m\times{\bB^{m\times d}}}(1-|B|)|y|^{d/(d-1)}\;\dd\mu(x,y,B)\right)^{(d-1)/d}\leq M_d\left(\mu(\overline{\Omega}\times\sigma\R^m\times\overline{\bB^{m\times d}}) - |\Omega|\right)
\end{equation}
(where $M_d$ is as given in Definition~\ref{defyms}), and for which there exists $\kappa\in\mbfM^+(\Omega\times\R^m)$ with $\pi_\#\kappa=\lL\restrict\Omega$ such that
\begin{equation}\label{eqymcond1}
\int_{\overline{\Omega}\times\sigma\R^m\times\overline{\bB^{m\times d}}}(1-|B|)\varphi(x,y)\;\dd\mu(x,y,B)=\int_{\Omega\times\R^m}\varphi(x,y)\;\dd \kappa(x,y)
\end{equation}
for every $\varphi\in\C_0(\overline{\Omega}\times\R^m)$.
\end{proposition}

\begin{proof}
First, we will show that for any $\bnu\in\mbfY(\Omega\times\R^m;\R^{m\times d})$, $S^{-*}\bnu$ satisfies the given conditions. Since $Sf\geq 0$ if and only if $f\geq 0$, and $\ddpr{f,\bnu}\geq 0$ for $f\geq 0$, we have that $S^{-*}\bnu\in\mbfM^+(\overline{\Omega}\times\sigma\R^m\times\overline{\bB^{m\times d}})$. To see that $S^{-*}\bnu$ satisfies \eqref{eqymcond1} with $\kappa=\iota_\nu$, note that, for $\varphi\in \C_0(\overline{\Omega}\times\R^m)$, the function $\psi(x,y,A):=\varphi(x,y)$ is contained in $\mbfE(\Omega\times\R^m)$ with $\psi^\infty\equiv 0$ and that $S\psi(x,y,B)=(1-|B|)\varphi(x,y)$. Every $\nu_{x,y}$ satisfies $\nu_{x,y}(\R^{m\times d})=1$ and so it follows that
\begin{align}\label{eqymidentify}
\begin{split}
\int_{\overline{\Omega}\times\sigma\R^m\times\overline{\bB^{m\times d}}}(1-|B|)\varphi(x,y)\;\dd S^{-*}\bnu(x,y,B)=&\int_{\Omega\times\R^m}\ip{\psi(x,y,\frarg)}{\nu_{x,y}}\;\dd \iota_\nu(x,y)\\
=&\int_\Omega\varphi(x,y)\;\dd \iota_\nu.
\end{split}
\end{align}
Since $\pi_\#\iota_\nu=\mathcal{L}^d\restrict\Omega$, this condition is verified.
Conditions~\eqref{eqymcond2a} and~\eqref{eqymcond2b} follow directly from the definition of $\iota_\nu$, $\lambda_\nu$, the construction of $S$, and the Poincar\'e inequality~\eqref{eqpoincare}.

Now let $\mu\in\mbfM^+(\overline{\Omega}\times\sigma\R^m\times\overline{\bB^{m\times d}})$ satisfy conditions~\eqref{eqymcond2a},~\eqref{eqymcond2b}, and~\eqref{eqymcond1}. It remains to show that $S^*\mu\in\mbfY(\Omega\times\R^m;\R^{m\times d})$. By the Disintegration of Measures Theorem, we can write $\mu=\omega\otimes\eta$ where $\omega\in\mbfM^+(\overline{\Omega}\times\sigma\R^m)$ is the pushforward of $\mu$ onto $\overline{\Omega}\times\sigma\R^m$ and $\eta\colon\overline{\Omega}\times\sigma\R^m\to \mbfM^1(\overline{\mathbb{B}^{m\times d}})$ is weakly* $\omega$-measurable. Let $\omega=p\kappa+\omega^s$ be the Radon--Nikodym decomposition of $\omega$ with respect to $\kappa$. Since $\C_0(\overline{\Omega}\times\R^m)^*=\mbfM(\overline{\Omega}\times\R^m)$ and~\eqref{eqymidentify} is true for arbitrary $\varphi\in \C_0(\overline{\Omega}\times\R^m)$, this implies that
\begin{equation}\label{kappanulebrel}
\ip{1-|B|}{\eta_{x,y}}\omega^s=\left[1-p(x,y)\ip{1-|B|}{\eta_{x,y}}\right]\kappa
\end{equation}
in $\mbfM(\overline{\Omega}\times\R^m)$. By construction, however, $\omega^s$ and $\kappa$ charge disjoint sets. This implies that both sides of~\eqref{kappanulebrel} must be the zero measure. It follows that $\ip{1-|B|}{\eta_{x,y}}=0$ for $\omega^s$-almost every $(x,y)\in\overline{\Omega}\times\sigma\R^m$, and hence that, for $\omega^s$-almost every $(x,y)\in\overline{\Omega}\times\sigma\R^m$, 
\[
\eta_{x,y}(\overline{\bB^{m\times d}})=\eta_{x,y}(\pd\mathbb{B}^{m\times d}).
\]
For $f\in \mbfE(\Omega\times\R^m)$, we may therefore write
\begin{align*}
\ip{Sf\left(x,y,\frarg\right)}{\eta_{x,y}}\omega &=p(x,y)\ip{Sf\left(x,y,\frarg\right)}{\eta_{x,y}}\kappa+\ip{Sf\left(x,y,\frarg\right)}{\eta_{x,y}}\omega^s \\ 
&=\left(p(x,y)\int_{\mathbb{B}^{m\times d}}Sf\left(x,y,B\right)\;\dd\eta_{x,y}(B)\right)\kappa \\ 
&\qquad+\left(\dashint_{\pd\mathbb{B}^{m\times d}}Sf(x,y,B)\;\dd\eta_{x,y}(B)\right)\eta_{x,y}(\pd\mathbb{B}^{m\times d})p(x,y)\kappa \\
&\qquad+\left(\int_{\pd\mathbb{B}^{m\times d}}Sf(x,y,B)\;\dd\eta_{x,y}(B)\right)\omega^s.
\end{align*}
The above decomposition combined with the fact that $\eta_{x,y}(\pd\bB^{m\times d})=1$ $\omega^s$-almost everywhere suggests that we construct the relevant Young measure by defining $\left(\iota_\nu,\nu,\lambda_\nu,\nu^\infty\right)$ as follows:
\begin{equation}\label{eqymidentifydecomp}
\begin{array}{ll}
\displaystyle\langle h,\nu_{x,y}\rangle & \displaystyle:=p(x,y)\int_{\mathbb{B}^{m\times d}}Sh(x,y,B)\;\dd\eta_{x,y}(B) \\
\displaystyle\iota_\nu &\displaystyle :=\kappa\\
\displaystyle\langle Sh,\nu^\infty_{x,y}\rangle & \displaystyle:= \dashint_{\pd\mathbb{B}^{m\times d}}h^\infty(x,y,B)\;\dd\eta_{x,y}(B)\\
\displaystyle\lambda_\nu & \displaystyle:=\eta_{x,y}(\pd\mathbb{B}^{m\times d})\omega.
\end{array}
\end{equation}
It is clear that $\nu_{x,y},\nu_{x,y}^\infty$ and $\iota_\nu,\lambda_\nu$ inherit positivity from $\kappa$, $\eta_{x,y}$ and $\omega$. From \eqref{kappanulebrel}, we see that 
\[
p(x,y)\ip{\left(1-|B|\right)}{\eta_{x,y}}=1 \quad\text{ for }\kappa\text{-a.e.\ }(x,y)\in\overline{\Omega}\times\sigma\R^m,
\]
and so, since by definition  $\ip{\ONE}{\nu_{x,y}}=p(x)\ip{\left(1-|B|\right)}{\eta_{x,y}}$, we have that $\nu_{x,y}\in\mbfM^1(\R^{m\times d})$. Since it is defined as an average, $\nu_{x,y}^\infty$ is also a probability measure for $\lambda_\nu$-almost every $(x,y)\in\overline{\Omega}\times\sigma\R^m$. That $(x,y)\to\ip{|\frarg|}{\nu_{x,y}}\in \Lp^1\left(\Omega\times\R^m;\iota_\nu\right)$ follows from the definition of $S$ applied to the integrand $f(x,y,A)=1+|A|$. The desired weak* measurability properties for $\nu$ and $\nu^\infty$ follow from the definition of $p$ and the fact that $\eta$ is weakly* $\omega$-measurable. Finally, the Poincar\'e inequality~\eqref{eqpoincare} follows directly from condition~\eqref{eqymcond2b}.
\end{proof}

\begin{corollary}\label{corymsclosed}
The set $S^{-*}\mbfY(\Omega\times\R^m;\R^{m\times d})$ is sequentially weakly* closed in $\mbfM(\overline{\Omega}\times\sigma\R^m\times\overline{\bB^{m\times d}})$.
\end{corollary}

\begin{proof}
It suffices to show that conditions~\eqref{eqymcond2a}, and~\eqref{eqymcond2b}, and~\eqref{eqymcond1} from Proposition~\ref{thmymsasmeasures} are sequentially weakly* closed. This is immediate for condition~\eqref{eqymcond1}, since the functions $\psi(x,y,A):=\varphi(x,y)$ for $\varphi\in\C_0(\overline{\Omega}\times\R^m)$ are elements of the predual of $\mbfM(\overline{\Omega}\times\sigma\R^m\times\overline{\bB^{m\times d}})$. 

Now let $(\bnu_j)_j\subset\mbfY(\Omega\times\R^m;\R^{m\times d})$ and $\mu\in\mbfM^+(\overline{\Omega}\times\sigma\R^m\times\overline{\bB^{m\times d}})$ be such that $S^{-*}\bnu_j\wsc\mu$ in $\mbfM(\overline{\Omega}\times\sigma\R^m\times\overline{\bB^{m\times d}})$. The condition $\mu(\overline{\Omega}\times\infty\pd\bB^m\times\bB^{m\times d})=0$ is satisfied if and only if the sequence $\left(S^{-*}\bnu_j\right)_j$ is tight in $\mbfM(\overline{\Omega}\times[(\R^m\times\overline{\bB^{m\times d}})\cup(\sigma\R^m\times\pd\bB^{m\times d})])$. This fails to be the case only if there exists $\varepsilon>0$, $\delta\in[0,1)$, and a sequence of radii $R_j\uparrow\infty$ such that 
\[
\limsup_{j\to\infty}S^{-*}\bnu_j\bigl(\overline{\Omega}\times \{|y|\geq R_j\}\times \delta\bB^{m\times d}\bigr)\geq\varepsilon.
\]
The Poincar\'e inequality~\eqref{eqymcond2b} for $\bnu_j$ implies
\begin{align*}
	&M_d \cdot \left(S^{-*}\bnu_j(\overline{\Omega}\times\sigma\R^m\times\overline{\bB^{m\times d}})-|\Omega|\right)\\
	&\quad\geq\left(\int_{\overline{\Omega}\times \{|y|\geq R_j\}\times \delta\bB^{m\times d}}(1-|B|)|y|^{d/(d-1)}\;\dd S^{-*}\bnu_j(x,y,B)\right)^{(d-1)/d}.
\end{align*}
We would therefore obtain
\begin{align*}
M_d \cdot \limsup_{j\to\infty} S^{-*}\bnu_j(\overline{\Omega}\times\sigma\R^m\overline{\times\bB^{m\times d}}) - M_d\cdot|\Omega|&\geq\limsup_{j\to\infty}\left(\varepsilon(1-\delta)R_j^{d/(d-1)}\right)^{(d-1)/d}\\
&=\infty,
\end{align*}
contradicting the fact that $(S^{-*}\bnu_j)_j$ must be a norm-bounded sequence in $\mbfM(\overline{\Omega}\times\sigma\R^m\times\overline{\bB^{m\times d}})$ (by the Uniform Boundedness Theorem). Finally, since the function $\psi\colon\overline{\Omega}\times\sigma\R^m\times\overline{\bB^{m\times d}}\to[0,\infty)$ given by
\[
\psi(x,y,B):=\begin{cases}
(1-|B|)|y|^{d/(d-1)}&\text{ if } (y,B)\in\R^m\times\bB^{m\times d},\\
0 &\text{ otherwise},
\end{cases}
\]
is lower semicontinuous, it follows (see for instance Proposition~1.62 in~\cite{AmFuPa00FBVF}) that
\begin{align*}
\ip{\psi}{\mu}\leq\liminf_{j\to\infty}\ip{\psi}{S^{-*}\bnu_j}&\leq \lim_{j\to\infty}\left(M_d \left(\cdot S^{-*}\bnu_j(\overline{\Omega}\times\sigma\R^m\times\overline{\bB^{m\times d}}) - |\Omega|\right)\right)^{d/(d-1)}\\
&=\left(M_d \cdot \left(\mu(\overline{\Omega}\times\sigma\R^m\times\overline{\bB^{m\times d}}) - |\Omega|\right)\right)^{d/(d-1)}
\end{align*}
and that condition~\eqref{eqymcond2b} is therefore satisfied.
\end{proof}

As a consequence of Proposition~\ref{thmymsasmeasures} combined with the usual density results for tensor products of continuous functions over a product domain, the following lemma is now immediate:
\begin{lemma}\label{remtensorprods}
There exists a countable family of products $\left\{\varphi_k\otimes h_k\right\}_{k\in\mathbb{N}}\subset\mbfE(\Omega\times\R^m)$, where $\varphi_k\in\C(\overline{\Omega}\times\sigma\R^m)$ and $h_k\in\C(\R^{m\times d})$, such that if for a mass-bounded sequence $(\bnu_j)_j\subset\mbfY(\Omega\times\R^m;\R^{m\times d})$ it holds that
\[
\lim_{j\to\infty}\ddprb{\varphi_k\otimes h_k,\bnu_j}=\ddprb{\varphi_k\otimes h_k,\bnu}\quad\text{ for each }k\in\mbN,
\]
then $\bnu_j\wsc\bnu$.
\end{lemma}

\begin{theorem}[Sequential Compactness in $\mbfY$]\label{thmymsseqcompact}
 If a sequence $(\bnu_j)_j\subset\mbfY(\Omega\times\R^m;\R^{m\times d})$ satisfies
\[
 \sup_j\norm{\bnu_j}_{\mbfY}<\infty,
\]
then there exists a Young measure $\bnu\in\mbfY(\Omega\times\R^m;\R^{m\times d})$ and a subsequence $(\bnu_{j_k})_k\subset(\bnu_j)_j$ such that $\bnu_{j_k}\wsc\bnu$ as $k\to\infty$.
\end{theorem}

\begin{proof}
Note that by~\eqref{eqnumumass} for $\bnu\in\mbfY(\Omega\times\R^m;\R^{m\times d})$ it holds that
\[
|\Omega| + \norm{\bnu}_\mbfY = (S^{-*}\bnu)(\overline{\Omega}\times\sigma\R^m\times\overline{\bB^{m\times d}}),
\]
which implies that $(\bnu_j)_j\subset\mbfY(\Omega\times\R^m;\R^{m\times d})$ is uniformly $\norm{\frarg}_\mbfY$-bounded if and only if $(S^{-*}\bnu_j)_j$ is uniformly norm-bounded in $\mbfM(\overline{\Omega}\times\sigma\R^m\times\overline{\bB^{m\times d}})$. The result now follows from the sequential weak* closure of $S^{-*}\mbfY(\Omega\times\R^m;\R^{m\times d})$ in $\mbfM(\overline{\Omega}\times\sigma\R^m\times\overline{\bB^{m\times d}})$ combined with the Banach--Alaoglu Theorem.
\end{proof}
 
\begin{corollary}
 Let $(\gamma_j)_j\subset\Lift(\Omega\times\R^m)$ satisfy $\sup_j|\gamma_j|(\Omega\times\R^m)<\infty$. Then, upon passing to a (non-relabelled) subsequence, there exists a Young measure $\bnu\in\YLift(\Omega\times\R^m)$ such that
\[
\gamma_j\toY\bnu\text{ and }\gamma_j\wsc\asc{\bnu}.
\]
\end{corollary}
\begin{proof}
Observe that, if $\bdelta$ is the elementary Young measure associated to a lifting $\gamma\in\Lift(\Omega\times\R^m)$,
\begin{align*}
\norm{\bdelta}_\mbfY&=\int_{\Omega\times\R^m}\ip{|\frarg|}{{\delta}_{x,y}}\;\dd\iota_{\delta}(x,y)+\int_{\overline{\Omega}\times\sigma\R^m}\ip{|\frarg|}{{\delta}_{x,y}^\infty}\;\dd\lambda_{\delta}(x,y)\\
&=\int_\Omega|\nabla \asc{\gamma}(x)|\;\dd x+|\gamma^s|(\Omega\times\R^m)=|\gamma|(\Omega\times\R^m),
\end{align*}
where $\gamma^s$ is the Lebesgue singular part of $\gamma$ introduced in Definition~\ref{defelemyms}. It follows that the sequence of elementary Young measures $(\bdelta\asc{\gamma_j})_j$ is bounded in $\mbfY(\Omega\times\R^m;\R^{m\times d})$ and so we can combine Theorem~\ref{thmymsseqcompact} with Lemma~\ref{lemliftymsclosed} to achieve the desired result.
\end{proof}

\subsection{Manipulating Young measures}\label{secymsmanipulate}
The following theorem will be of great importance in the computation of tangent Young measures in Section~\ref{chaptangentliftingyms}.
\begin{theorem}[Restriction Theorem for $\YLift$]\label{thm:restrictionsofyms}
Let $\bnu=(\iota_\nu,\nu,\lambda_\nu,\nu^\infty)\in\YLift(\Omega\times\R^m)$ (or $\bnu\in\AYLift(\Omega\times\R^m)$). Then the restricted Young measure $\bnu\restrict(\overline{\Omega}\times\R^m)$ defined by
\[
\bnu\restrict(\overline{\Omega}\times\R^m):=(\iota_\nu,\nu,{\lambda_\nu}{\restrict(\overline{\Omega}\times\R^m)},\nu^\infty)
\]
is also a member of $\YLift(\Omega\times\R^m)$ (or $\bnu\restrict(\overline{\Omega}\times\R^m)\in\AYLift(\Omega\times\R^m))$.
\end{theorem}
Note that the definition of $\bnu\restrict(\overline{\Omega}\times\R^m)$ is equivalent to the statement that for all $f\in\mbfE(\Omega\times\R^m)$,
\[
\ddprb{f,\bnu\restrict(\overline{\Omega}\times\R^m)}=\int_{\Omega\times\R^m}\ip{f(x,y,\frarg)}{\nu_{x,y}}\;\dd\iota_\nu(x,y)+\int_{\overline{\Omega}\times\R^m}\ip{f^\infty(x,y,\frarg)}{\nu_{x,y}^\infty}\;\dd\lambda_\nu(x,y).
\]
\begin{proof}
Let $(\gamma_j)_j\subset\Lift(\Omega\times\R^m)$ be such that $\gamma_j\toY\bnu$. For each $R\in\mbN$, let $\eta_R\in\C^1(\R^m;\R^m)$ be an injective function such that $\norm{\eta_R}_\infty \leq 2R$, $\eta_R(y)=y$ for $y\in B(0,R)$, $\norm{\nabla\eta_R}_\infty$ is bounded independently of $R$, and $|\nabla\eta_R(y)|\to 0$ as $|y|\to\infty$. It suffices, for example, to take
\[
\eta_R(y):=f_R(|y|)\frac{y}{|y|},\quad\text{ where }\quad f_R(t):=\begin{cases} t &\text{ if }0\leq t<R,\\
R\left(2-\exp\left(\frac{R-t}{R}\right)\right) & \text{ if }t\geq R.
\end{cases}
\]
For each $j,R\in\mbN$, define
\[
c_j^R:=\dashint_\Omega\eta_R(\asc{\gamma_j}(x))\;\dd x,\qquad c_0^R:=\dashint_\Omega\eta_R(\llbracket\bnu\rrbracket(x))\;\dd x,
\]
and denote by $T_j^R\colon\Omega\times\R^m\to\Omega\times\R^m$ the map given by
\[
T_j^R(x,y)=(x,\eta_R(y)-c_j^R).
\]
Now define $\gamma_j^R\in\mbfM(\Omega\times\R^m;\R^{m\times d})$ by
\[
  \gamma_j^R:= (T_j^R)_\#(\nabla\eta_R \, \gamma_j),
\]
so that
\[
\int_{\Omega\times\R^m}\varphi(x,y)\;\dd\gamma_j^R(x,y)=\int_{\Omega\times\R^m}(\varphi\circ T_j^R) (x,y)\nabla\eta_R(y) \;\dd\gamma_j(x,y),\quad\varphi\in\C_0(\Omega\times\R^m).
\]
In particular, for every $\varphi\in\C_0(\Omega\times B(0,R))$,
\[
\int_{\Omega\times\R^m}\varphi(x,y)\;\dd\gamma_j^R(x,y)=\int_{\Omega\times\R^m}\varphi(x,y-c^R_j)\;\dd\gamma_j(x,y).
\]
By the chain rule combined with the fact that $\nabla_y T_j^R=\nabla\eta_R$, we have
\[
\nabla_y(\varphi\circ T_j^R)(x,y)=(\nabla_y\varphi\circ T_j^R)(x,y)\nabla \eta_R(y),
\]
which implies
\begin{align*}
\int_{\Omega\times\R^m}\nabla_y\varphi(x,y)\;\dd\gamma_j^R(x,y)&=\int_{\Omega\times\R^m}\nabla_y(\varphi\circ T_j^R)(x,y)\;\dd\gamma_j(x,y)\\
&=-\int_{\Omega}\nabla_x(\varphi\circ T_j^R)(x,\asc{\gamma_j}(x))\;\dd x\\
&=-\int_\Omega\nabla_x\varphi(x,\eta_R(\asc{\gamma_j}(x))-c_j^R)\;\dd x.
\end{align*}
We therefore see that $\gamma_j^R\in\Lift(\Omega\times\R^m)$ with $\asc{\gamma_j^R}=\eta_R(\asc{\gamma_j})-c_j^R$ and $(x, \asc{\gamma_j^R})=T^R_j(x,\asc{\gamma_j})$,
from which we get
\[
\gr^{\asc{\gamma_j^R}}_\#(\lL\restrict\Omega)=(T_j^R)_\#\left(\gr^{\asc{\gamma_j}}_\#(\lL\restrict\Omega)\right)
\]
and hence that
\[
P\gamma_j^R=(T_j^R)_\# \left[A_R P\gamma_j\right],\qquad\text{where}\qquad A_R(y):=\begin{pmatrix}\nabla\eta_R(y) & 0\\0 & 1 \end{pmatrix}.
\]
Since each $T_j^R\colon\Omega\times\R^m\to\Omega\times\R^m$ is injective (because each $\eta_R$ is injective), Lemma~\ref{lempushforwardderivative} implies that
\[
\frac{\dd A_R P\gamma_j}{\dd\left|A_R P\gamma_j\right|}=\frac{\dd P\gamma_j^R}{\dd|P\gamma_j^R|}\circ T_j^R\quad\text{ and }\quad |P\gamma_j^R|=(T_j^R)_\#|A_RP\gamma_j|.
\]
For $f\in\mbfE(\Omega\times\R^m)$, we can therefore compute
\begin{align*}
\int Pf(x,y,P\gamma_j^R)&=\int Pf\left(x,y,\frac{\dd P\gamma_j^R}{\dd|P\gamma_j^R|}(x,y)\right)\;\dd|P\gamma_j^R|(x,y)\\
&=\int Pf\left(T_j^R(x,y),\frac{\dd P\gamma_j^R}{\dd|P\gamma_j^R|}\circ T_j^R (x,y)\right)\;\dd|A_R P\gamma_j|(x,y)\\
&=\int Pf\left(T_j^R(x,y),\frac{\dd A_R P\gamma_j}{\dd\left|A_R P\gamma_j\right|}(x,y)\right)\frac{\dd|A_R P\gamma_j|}{\dd|P\gamma_j|}(x,y)\;\dd|P\gamma_j|(x,y)\\
&=\int Pf\left(T_j^R(x,y),\left(\frac{\dd A_R P\gamma_j}{\dd\left|A_R P\gamma_j\right|}\cdot\frac{\dd|A_R P\gamma_j|}{\dd|P\gamma_j|}\right)(x,y)\right)\;\dd|P\gamma_j|(x,y)\\
&=\int Pf\left(T_j^R(x,y),\frac{\dd A_R P\gamma_j}{\dd\left|P\gamma_j\right|}(x,y)\right)\;\dd|P\gamma_j|(x,y)\\
&=\int Pf\left(T_j^R(x,y),A_R(y) \frac{\dd P\gamma_j}{\dd\left|P\gamma_j\right|}(x,y)\right)\;\dd|P\gamma_j|(x,y).
\end{align*}
Define $f_{\eta_R}\in\C(\overline{\Omega}\times\R^m\times\R^{m\times d})$, $R>0$, by
\[
f_{\eta_R}(x,y,A):=f\left(T_{0}^R(x,y),\nabla \eta_R(y)A\right).
\]
Since $T_0^R$ can be extended continuously to $\overline{\Omega}\times\sigma\R^m$ by setting $T_0^R(x,\infty e):=(x,2Re)$ for $(x,\infty e)\in\overline{\Omega}\times\infty\pd\bB^m$ and $\nabla\eta_R$ can be continuously extended to $\sigma\R^m$ by setting $\nabla\eta_R(\infty e)=0$ for all $\infty e\in\infty\pd\bB^m$, the fact that $f\in\mbfE(\Omega\times\R^m)$ implies that $f_{\eta_R}\in\mbfE(\Omega\times\R^m)$ for each $R>0$ with
\[
\sigma f_{\eta_R}^\infty(x,y,A)=\begin{cases}
f^\infty(T_0^R(x,y),\nabla\eta_R(y)A)&\quad\text{ if }(x,y,A)\in\overline{\Omega}\times\R^m\times\R^{m\times d},\\
0 &\quad\text{ if }(x,y,A)\in\overline{\Omega}\times\infty\pd\bB^m\times\R^{m\times d}.	
\end{cases}
\]

Since $T_j^R(x,y)=T_0^R(x,y) + (0,c_j^R-c_0^R)$ and $c_j^R\to c_0^R$ for each $R>0$ as $j\to\infty$, we can use the fact that $\gamma_j\toY\bnu$ in combination with Lemma~\ref{lemliftingsperturb} and our previous calculation to deduce that
\begin{align*}
\lim_{j\to\infty}\int Pf(x,y,P\gamma_j^R)&=\lim_{j\to\infty}\int Pf_{\eta_R}(x,y+c_j^R-c_0^R,P\gamma_j)\\
&=\lim_{j\to\infty}\int Pf_{\eta_R}(x,y,P\gamma_j) \\
&=\ddprb{f_{\eta_R},\bnu}.
\end{align*}

Since this limit holds for all $f\in\mbfE(\Omega\times\R^m)$ it follows from Definition~\ref{defymconvergence} and Theorem~\ref{thmymsseqcompact} that each sequence $(\gamma_j^R)_j\subset\Lift(\Omega\times\R^m)$ generates a Young measure $\bnu^R\in\YLift(\Omega\times\R^m)$ defined via
\[
\ddprb{f,\bnu^R}=\ddprb{f_{\eta_R},\bnu}\quad\text{ for all }f\in\mbfE(\Omega\times\R^m).
\]
Since $(\sigma f_{\eta_R}^\infty)\restrict(\overline{\Omega}\times\infty\pd\bB^m\times\overline{\bB^{m\times d}})\equiv 0$ for every $R>0$ and $f_{\eta_R}\to f$, $f_{\eta_R}^\infty\to f^\infty$ pointwise in $\overline{\Omega}\times\R^m\times\R^{m\times d}$ as $R\to\infty$, the Dominated Convergence Theorem (note that each $f_{\eta_R}$ is dominated by the $\iota_\nu$-integrable function $\norm{Sf}_\infty(1+|A|)$) lets us deduce
\begin{align*}
\lim_{R\to\infty}\ddprb{f,\bnu^R}&=\lim_{R\to\infty}\ddprb{f_{\eta_R},\bnu}\\
&=\lim_{R\to\infty}\left\{\int_{\Omega\times\R^m}\ip{f_{\eta_R}}{\nu_{x,y}}\;\dd\iota_\nu(x,y)+\int_{\overline{\Omega}\times\R^m}\ip{f_{\eta_R}^\infty}{\nu^\infty_{x,y}}\;\dd\lambda_\nu(x,y)\right\}\\
&=\int_{\Omega\times\R^m}\ip{f}{\nu_{x,y}}\dd\iota_\nu(x,y)+\int_{\overline{\Omega}\times\R^m}\ip{f^\infty}{\nu^\infty_{x,y}}\;\dd\lambda_\nu(x,y).
\end{align*}
It follows then that the family $(\bnu^R)_{R>0}\subset\YLift(\Omega\times\R^m)$ converges as $R\to\infty$ to a Young measure $\bnu_0\in\YLift(\Omega\times\R^m)$ satisfying
\[
\ddprb{f,\bnu_0}=\int_{\Omega\times\R^m}\ip{f}{\nu_{x,y}}\dd\iota_\nu(x,y)+\int_{\overline{\Omega}\times\R^m}\ip{f^\infty}{\nu^\infty_{x,y}}\;\dd\lambda_\nu(x,y)\quad\text{ for all }f\in\mbfE(\Omega\times\R^m),
\]
from which it is apparent that $\bnu_0=\bnu\restrict(\overline{\Omega}\times\R^m)$, as required. That membership of $\AYLift(\Omega\times\R^m)$ is preserved follows from the fact that $(\gamma\asc{u})^R_j=\gamma\asc{\eta_R(u(x))-c_j^R}$.
\end{proof}

\begin{lemma}\label{lemapproxymprescribedboundary}
If $\bnu\in\AYLift(\Omega\times\R^m)$, then $\bnu$ admits a generating sequence $(\gamma\asc{u_j})_j$ with $(u_j)_j\subset\C^\infty_\#(\Omega;\R^m)$.
\end{lemma}

\begin{proof}
Let $(v_j)_j\subset\BV_\#(\Omega;\R^m)$ be such that $\gamma\asc{v_j}\toY\bnu$ and let $\Hbf:=\{f_k\}_k\subset\mbfE(\Omega\times\R^m)$ be a countable collection of functions which determines Young measure convergence as given in Lemma~\ref{remtensorprods}. By the area-strict density of $\C^\infty(\bB^d;\R^m)$ in $\BV(\Omega;\R^m)$ ensured by Proposition~\ref{propfixedbdary} combined with Theorem~\ref{thmareastrictcontinuity}, we can find a $\BV$-bounded sequence $(u_j)_j\subset\C^\infty_\#(\Omega;\R^m)$ which satisfies
\[
\left|\int Pf_k(x,y,P\gamma\asc{v_j})-\int Pf_k(x,y,P\gamma\asc{u_j})\right|\leq\frac{1}{j}\quad\text{ for all }k\leq j.
\]
It follows that
\[
\left|\ddprb{f_k,\bnu}-\lim_{j\to\infty}\int Pf_k(x,y,P\gamma\asc{u_j})\right|=0\quad\text{ for each }k\in\mbN,
\]
which, upon making use of the density property of $\Hbf$, implies that $\gamma\asc{u_j}$ generates $\bnu$, as required.
\end{proof}

\begin{lemma}\label{lemymprescribedboundary}
Let $\bnu\in\AYLift(\Omega\times\R^m)$ satisfy $\lambda_\nu(\pd\Omega\times\sigma\R^m)=0$. Then there exist sequences $(u_j)_j\subset\C^\infty_\#(\Omega;\R^m)$, $(c_j)_j\subset\R^m$ with $\gamma\asc{u_j}\toY\bnu$ and $\lim_j c_j=0$ satisfying
\[
u_j|_{\pd\Omega}=\llbracket\bnu\rrbracket|_{\pd\Omega}+c_j.
\]
\end{lemma}

\begin{proof}
Invoking Lemma~\ref{lemapproxymprescribedboundary}, let $(v_j)_j\subset\C^\infty_\#(\Omega;\R^m)$ be such that $\gamma\asc{v_j}\toY\bnu$. From Proposition~\ref{propfixedbdary}, we know that there exists a sequence $(w_j)_j\subset\C^\infty_\#(\Omega;\R^m)$ such that ${w_j}|_{\pd\Omega}=\llbracket\bnu\rrbracket|_{\pd\Omega}$ for each $j\in\mbN$ and $w_j\to \llbracket\bnu\rrbracket$ area-strictly in $\BV_\#(\Omega;\R^m)$. For $\varepsilon>0$, let $\eta_k\in\C^\infty_c(\Omega;[0,1])$ be such that
\[
\eta_k(x)=
\begin{cases}
1&\text{ if }\dist(x,\pd\Omega)\geq \frac{2}{k},\\
0&\text{ if }\dist(x,\pd\Omega)\leq\frac{1}{k},
\end{cases}
\]
and define $(u_j^k)_{j,k}\subset\C^\infty_\#(\Omega;\R^m)$ by
\[
u_j^k:=(1-\eta_k)w_j+\eta_kv_j+c_j^k
\]
where $c_j^k$ is the constant which ensures that $\int u_j^k(x)\;\dd x=0$. For sufficiently large $j$ we have that
\begin{align*}
|Du_j^k|(\Omega)&\leq|Dw_j|(\Omega)+|Dv_j|(\Omega)+\norm{\nabla\eta_k}_\infty\norm{w_j-v_j}_{\Lp^1}\\
&\leq 2|D\llbracket\bnu\rrbracket|(\Omega)+|Dv_j|(\Omega)+\norm{\nabla\eta_k}_\infty\norm{w_j-v_j}_{\Lp^1},
\end{align*}
and $\norm{w_j-v_j}_{\Lp^1}\leq\norm{w_j-\llbracket\bnu\rrbracket}_{\Lp^1}+\norm{v_j-\llbracket\bnu\rrbracket}_{\Lp^1}\to 0$ as $j\to\infty$. It follows that $(Du_j^k)_{j,k}$ is uniformly bounded in $\mbfM(\Omega\times\R^m;\R^{m\times d})$ in $k$ and $j$ subject to the constraint $j\geq J_k$ for some increasing sequence $(J_k)_k\subset\mbN$. Note that $u_j^k$ satisfies ${u_j^k}|_{\pd\Omega}=\llbracket\bnu\rrbracket|_{\pd\Omega}+c_j^k$, and
\begin{align*}
c_j^k&=\dashint_\Omega \eta_k(x)\left[w_j(x)-v_j(x)\right]-w_j(x)\;\dd x\\
&=\dashint_\Omega \eta_k(x)\left[w_j(x)-v_j(x)\right]\;\dd x,
\end{align*}
which implies $|c_j^k|\leq[\lL(\Omega)]^{-1}\norm{w_j-v_j}_{\Lp^1}$. Hence, $c_j^k\to 0$ as $j\to\infty$ uniformly in $k$.

For $f\in\mbfE(\Omega\times\R^m)$ with $\norm{Sf}_\infty=1$, we can then compute
\begin{align*}
&\left|\int Pf(x,y,P\gamma\asc{u_j^k})-\int Pf(x,y,P\gamma\asc{v_j})\right|\\
&\qquad \leq 2\bigl(\lL+|Dw_j|+|Dv_j|\bigr)\left(\left\{x\in\Omega\colon\dist(x,\pd\Omega)\leq\frac{2}{k}\right\}\right)
\end{align*}
for every $j,k\in\mbN$. Taking the limit as $j\to\infty$, we obtain
\begin{align*}
&\lim_{j\to\infty} \left|\int Pf(x,y,P\gamma\asc{u_j^k})-\int Pf(x,y,P\gamma\asc{v_j})\right|\\
&\qquad \leq 2\bigl(\lL+|D\llbracket\bnu\rrbracket|+\pi_\#\lambda_\nu\bigr)\left(\left\{x\in\overline{\Omega}\colon\dist(x,\pd\Omega)\leq\frac{2}{k}\right\}\right).
\end{align*}
Since $\pi_\#\lambda_\nu(\pd\Omega)=0$, the right hand side of this expression must converge to $0$ as $k\to\infty$ and we can therefore use a diagonal argument to find a sequence $(j_k)_k$ such that $(\gamma\asc{u_{j_k}^k})_k$ generates $\bnu$ and $c_{j_k}^{k}\to 0$. Since ${u_{j_k}^{k}}|_{\pd\Omega}=\llbracket\bnu\rrbracket|_{\pd\Omega}+c_{j_k}^{k}$, this suffices to prove the lemma.
\end{proof}

\subsection{Extended representation}\label{secymsrepresentation}
All of the results obtained so far in this section have helped us to understand $\lim_{j\to\infty}\int Pf(x,y,P\gamma_j)$ when $f\in\mbfE(\Omega\times\R^m)$. However, the integrands $f$ which are of interest for the applications that we have in mind are not members of $\mbfE(\Omega\times\R^m)$, only of $\mbfR(\Omega\times\R^m)$. In particular, the requirement that both $f$ and $f^\infty$ extend continuously to $\overline{\Omega}\times\sigma\R^m\times\R^{m\times d}$ is too strong to be satisfied by any integrand which is unbounded in the middle variable.

If $\bnu\in\YLift(\Omega\times\R^m)$ satisfies $\lambda_\nu(\overline{\Omega}\times\infty\pd\bB^m)=0$ and we assume just that $f\in\RBVw(\Omega\times\R^m)$, then the duality product
\[
\ddprb{f,\bnu}=\int_{\Omega\times\R^m}\ip{f(x,y,\frarg)}{\nu_{x,y}}\;\dd\iota_\nu(x,y)+\int_{\overline{\Omega}\times\R^m}\ip{f^\infty(x,y,\frarg)}{\nu_{x,y}^\infty}\;\dd\lambda_\nu(x,y)
\]
is still well-defined in $\R$. For the singular part this follows from the fact that $\lambda_\nu$ is a positive measure and that $f^\infty$ is positive and continuous. For the regular part this follows from $\iota_\nu=\gr_\#^{\llbracket\bnu\rrbracket}(\lL\restrict\Omega)$ and $\llbracket\bnu\rrbracket\in\BV(\Omega;\R^m)$. Indeed, the fact that $f$ is Carath\'eodory and satisfies $|f(x,y,A)|\leq C(1+|y|^{d/(d-1)}+|A|)$, the $\iota_\nu$-weak* measurability of the parametrised measure $\nu$, and the Poincar\'{e} Condition~\eqref{eqpoincare} together imply that the map $(x,y)\mapsto\ip{f(x,y,\frarg)}{\nu_{x,y}}$ is $\iota_\nu$-measurable and $\iota_\nu$-integrable.

Similarly, if $f\in\RBVw(\Omega\times\R^m)$ is such that $\sigma f^\infty$ exists, then
\[
\ddprb{f,\bnu}=\int_{\Omega\times\R^m}\ip{f(x,y,\frarg)}{\nu_{x,y}}\;\dd\iota_\nu(x,y)+\int_{\overline{\Omega}\times\sigma\R^m}\ip{\sigma f^\infty(x,y,\frarg)}{\nu_{x,y}^\infty}\;\dd\lambda_\nu(x,y)
\]
is also well-defined even if $\lambda_\nu(\overline{\Omega}\times\infty\pd\bB^m)>0$.

Proposition~\ref{lemextendedrepresentation} and Corollary~\ref{corymlsc} below show that knowledge of $\bnu$ allows us to say something meaningful about $\lim_{j}\int Pf(x,y,P\gamma_j)$ for a very large class of integrands $f$ and, at the same time, that this limiting behaviour does not change if $f$ is replaced by a perturbation $f(\frarg,\frarg+c_j,\frarg)$ where $(c_j)_j\subset\R^m$ is a convergent sequence. As a consequence, the limiting behaviour of $(\Fcal[u_j])_j$ for any $f$ and any weakly* convergent sequence in $\BV(\Omega;\R^m)$ can be understood in terms of the Young measures generated by sequences $(\gamma\asc{u_j-(u_j)})_j\subset\ALift(\Omega\times\R^m)$.

\begin{proposition}[Extended Representation]\label{lemextendedrepresentation}
Let $\bnu\in\YLift(\Omega\times\R^m)$, $(\gamma_j)_j\subset\Lift(\Omega\times\R^m)$ be such that $\gamma_j\toY\bnu$, and let $(c_j)_j\subset\R^m$ be such that $c_j\to 0\in\R^m$. Let $f\in\RL(\Omega\times\R^m)$.
\begin{enumerate}[(i)]
\item If $\lambda_\nu(\overline{\Omega}\times\infty\pd\bB^m)=0$ so that $\bnu\equiv\bnu\restrict(\overline{\Omega}\times\R^m)$ (where $\bnu\restrict(\overline{\Omega}\times\R^m)$ is as described in Theorem~\ref{thm:restrictionsofyms}), we have that
\begin{align}
\begin{split}\label{eqymeq}
\lim_{j\to\infty}\int Pf(x,y+c_j,P\gamma_j)&=\ddprb{f,\bnu}\\
&=\int_{\Omega\times\R^m}\ip{f}{\nu}\;\dd\iota_\nu(x,y)+\int_{\overline{\Omega}\times\R^m}\ip{f^\infty}{\nu^\infty}\;\dd\lambda_\nu(x,y).
\end{split}
\end{align}

\item If instead $\lambda_\nu(\overline{\Omega}\times\infty\pd\bB^m)\geq 0$ and we assume in addition that $\sigma f^\infty$ exists according to Definition~\ref{defextendedrecessionfn}, we also have that
\begin{align}
\begin{split}\label{eqymeqextended}
\lim_{j\to\infty}\int Pf(x,y+c_j,P\gamma_j)&=\ddprb{f,\bnu}\\
&=\int_{\Omega\times\R^m}\ip{f}{\nu}\;\dd\iota_\nu(x,y)+\int_{\overline{\Omega}\times\sigma\R^m}\ip{\sigma f^\infty}{\nu^\infty}\;\dd\lambda_\nu(x,y).
\end{split}
\end{align}
\end{enumerate}
\end{proposition}

\begin{proof}
We shall prove~\eqref{eqymeq} and~\eqref{eqymeqextended} simultaneously:

\emph{Step 1:} Assume first that $f\colon\overline{\Omega}\times\sigma\R^m\times\R^{m\times d}\to\R$ is Carath\'eodory and satisfies $|f(x,y,A)|\leq C$ for all $(x,y,A)\in\Omega\times\R^m\times\R^{m\times d}$ and some fixed $C>0$. By the Scorza--Dragoni Theorem, there exists a compact set $K_\varepsilon\Subset\Omega$ such that $\mathcal{L}^d(\Omega\setminus K_\varepsilon)\leq \varepsilon$ and $f\restrict(K_\varepsilon\times\sigma\R^m\times\R^{m\times d}$) is continuous. By the Tietze Extension Theorem, we can find a continuous function $g\in \C(\overline{\Omega}\times\sigma\R^m\times\R^{m\times d})$ which restricts to $f$ on $K_\varepsilon\times\sigma\R^m\times\R^{m\times d}$ and such that $g$ is bounded and that $\norm{g}_\infty=\norm{f}_\infty$. Since $g\in\mbfE(\Omega\times\R^m)$ with $\sigma g^\infty=\sigma f^\infty=0$, Lemma~\ref{lemliftingsperturb} together with Corollary~\ref{lemstructureliftingym} lets us deduce
\[
\lim_{j\to\infty}\int_{\Omega\times\R^m}Pg(x,y+c_j,P\gamma_j)=\ddprb{g,\bnu}=\int_\Omega \ip{g(x,\llbracket\bnu\rrbracket(x),\frarg)}{\nu_{x,\llbracket\bnu\rrbracket(x)}}\;\dd x.
\]
By the construction of $g$, however, we see that
\[
\left|\int_{\Omega\times\R^m}P(g-f)(x,y+c_j,P\gamma_j)\right|\leq 2\int_{\Omega\setminus K_\varepsilon}\norm{f}_\infty\;\dd x\quad\text{ for all }j\in\mbN,
\]
and that the same estimate holds for $\big|\int_{\Omega\times\R^m}P(g-f)(x,y,P\gamma)\big|$. Letting $\varepsilon\to 0$, we obtain
\[
\lim_{j\to\infty}\int_{\Omega\times\R^m}Pf(x,y+c_j,P\gamma_j)=\ddprb{f,\bnu}
\]
and so both~\eqref{eqymeq} and~\eqref{eqymeqextended} hold.

\emph{Step 2:} Now assume that $f\colon\overline{\Omega}\times\R^m\times\R^{m\times d}\to\R$ is Carath\'eodory and satisfies $|f(x,y,A)|\leq C$ for all $(x,y,A)\in\Omega\times\R^m\times\R^{m\times d}$ and some fixed $C>0$. For $K>0$, let $\varphi_K\in\C_c(\R^m;[0,1])$ be such that $\varphi_K(y)=1$ for $|y|\leq K$. Define $f_K\colon\overline{\Omega}\times\sigma\R^m\times\R^{m\times d}\to\R$ by
\[
f_K(x,y,A):=\begin{cases}
\varphi_K(y)f(x,y,A) \quad & \text{ if }(x,y,A)\in\overline{\Omega}\times\R^m\times\R^{m\times d},\\
0 \quad &\text{ if }(x,y,A)\in\overline{\Omega}\times\infty\pd\bB^m\times\R^{m\times d},
\end{cases}
\]
and note that each $f_K$ satisfies the hypotheses of Step~1. Since $f_K\to f$ pointwise in $\overline{\Omega}\times\R^m\times\R^{m\times d}$ as $K\to\infty$ and $\sigma f^\infty\equiv 0$, the Dominated Convergence Theorem lets us deduce
\begin{align*}
\lim_{K\to\infty}\ddprb{f_K,\bnu}&=\lim_{K\to\infty}\int_\Omega \ip{f_K(x,\llbracket\bnu\rrbracket(x),\frarg)}{\nu_{x,\llbracket\bnu\rrbracket(x)}}\;\dd x\\
&=\int_\Omega \ip{f(x,\llbracket\bnu\rrbracket(x),\frarg)}{\nu_{x,\llbracket\bnu\rrbracket(x)}}\;\dd x\\
&= \ddprb{f,\bnu},
\end{align*}
and, since Step~1 implies
 \[
\lim_{j\to\infty}\int Pf_K(x,y+c_j,P\gamma_j)=\ddprb{f_K,\bnu},
\]
we therefore deduce
\begin{align*}
\left|\lim_{j\to\infty}\int Pf(x,y+c_j,P\gamma_j)-\ddprb{f,\bnu}\right|&=\lim_{K\to\infty}\left|\lim_{j\to\infty}\int Pf(x,y+c_j,P\gamma_j)-\ddprb{f_K,\bnu}\right|\\
&\leq\lim_{K\to\infty}\left|\lim_{j\to\infty}\int P(f-f_K)(x,y+c_j,P\gamma_j)\right|.
\end{align*}
Now fix $\varepsilon>0$ and let $R>0$ be large enough that $C/(1+R)<\varepsilon$.  Recalling the map 
\[
S\colon\mbfE(\Omega\times\R^m)\to\C(\overline{\Omega}\times\sigma\R^m\times\overline{\bB^{m\times d}})
\]
introduced in Section~\ref{secymscompact} together with $S^{-*}:=(S^*)^{-1}$, the inverse of its adjoint, we see that Proposition~\ref{thmymsasmeasures} implies $S^{-*}\bnu(\overline{\Omega}\times\infty\pd\bB^m\times\bB^{m\times d})=0$ and so, by the outer regularity of Radon measures, there exists $K>0$ such that
\[
(S^{-*}\bnu)\left(\overline{\Omega}\times[\R^m\setminus B(0,K-1)]\times \overline{B(0,{R}/{(1+R)})}\right)<\frac{\varepsilon}{2}.
\]
Since $S^{-*}\bdelta\asc{\gamma_j}\wsc S^{-*}\bnu$ in $\mbfM(\overline{\Omega}\times\sigma\R^m\times\overline{\bB^{m\times d}})$, we must therefore have that
\[
(S^{-*}\bdelta\asc{\gamma_j})\left(\overline{\Omega}\times A_{K,R}\right)<\varepsilon,\quad A_{K,R}:=[\R^m\setminus B(0,K-1)]\times \overline{B(0,R/(1+R))}
\]
for all $j\in\mbN$ sufficiently large. Since our choice of $R$ implies $C/(1+R)<\varepsilon$, we have that, whenever $(x,y,B)\in\overline{\Omega}\times\R^m\times\overline{\bB^{m\times d}}$ and $|B|\geq R/(1+R)$ so that $1-|B|\leq 1/(1+R)$,
\[
|(Sf)(x,y,B)|=\left|(1-|B|)f\left(x,y,\frac{B}{1-|B|}\right)\right|\leq C(1-|B|)<\varepsilon.
\] 
Abbreviating $\bdelta_j:=\bdelta\asc{\gamma_j}$, we can now use the fact that $(1-\varphi_K(y+c_j))=0$ when $|y|\leq K-1$ and $|c_j|<1$, to estimate
\begin{align*}
\left|\lim_{j\to\infty}\int P(f-f_K)(x,y+c_j,P\gamma_j)\right|&=\left|\lim_{j\to\infty}\int S(f-f_K)(x,y+c_j,B)\;\dd S^{-*}\bdelta_j(x,y,B)\right|\\
&\leq \lim_{j\to\infty}\int_{\overline{\Omega}\times A_{K,R}}\norm{Sf}_\infty\;\dd S^{-*}\bdelta_j(x,y,B)\\
&\qquad+\lim_{j\to\infty}\int_{\overline{\Omega}\times\{|y|>K-1\}\times \{|B|>\frac{R}{1+R}\}}\varepsilon\;\dd S^{-*}\bdelta_j(x,y,B)\\
&\leq\norm{Sf}_\infty\limsup_{j\to\infty}(S^{-*}\bnu)\left(\overline{\Omega}\times A_{K,R}\right)\\
&\qquad +\varepsilon\sup_j\ddprb{\mathbbm{1}\otimes|\frarg|,\bdelta_j}\\
&\leq \varepsilon\left(\norm{Sf}_\infty+\sup_j\ddprb{\mathbbm{1}\otimes|\frarg|,\bdelta_j}\right).
\end{align*}
Since $\sup_j\ddprb{\mathbbm{1}\otimes|\frarg|,\bdelta_j}<\infty$ and $\varepsilon>0$ was arbitrary, we therefore deduce
\[
\lim_{K\to\infty}\left|\lim_{j\to\infty}\int P(f-f_K)(x,y+c_j,P\gamma_j)\right|=0,
\]
and hence that
\[
\lim_{j\to\infty}\int Pf(x,y+c_j,P\gamma_j)=\ddprb{f,\bnu},
\]
as required.

\emph{Step 3:} Now assume that $f\colon\overline{\Omega}\times\R^m\times\R^{m\times d}\to\R$ is Carath\'eodory and satisfies $|f(x,y,A)|\leq C(1+|A|)$ for some $C>0$ and is such that $f^\infty= 0$ if $\lambda_\nu(\overline{\Omega}\times\infty\pd\bB^m)=0$ and $\sigma f^\infty$ exists with $\sigma f^\infty=0$ if $\lambda_\nu(\overline{\Omega}\times\infty\pd\bB^m)>0$. Let $\varphi_R\in\C_c(\R^{m\times d};[0,1])$ be a cut-off function such that $\varphi_R(A)=1$ for $|A|\leq R$. Define $f_R\colon\overline{\Omega}\times\R^m\times\R^{m\times d}\to\R$ by $f_R(x,y,A):=\varphi_R(A)f(x,y,A)$ and note that $f_R$ satisfies the hypotheses of Step~2. Moreover, since $f_R\to f$ pointwise in $\overline{\Omega}\times\R^m\times\R^{m\times d}$ as $R\to \infty$, the Dominated Convergence Theorem implies that
\[
\lim_{R\to\infty}\ddprb{f_R,\bnu}=\int_\Omega \ip{f(x,\llbracket\bnu\rrbracket(x),\frarg)}{\nu_{x,\llbracket\bnu\rrbracket(x)}}\;\dd x= \ddprb{f,\bnu}.
\]
Splitting
\[
f=f_R+(f-f_R)
\]
and using the fact that Step~2 implies
\[
\lim_{j\to\infty}\int Pf_R(x,y+c_j,P\gamma_j)=\ddprb{f_R,\bnu},
\]
we deduce
\begin{align*}
\left|\lim_{j\to\infty}\int Pf(x,y+c_j,P\gamma_j)-\ddprb{f,\bnu}\right|&=\lim_{R\to\infty}\left|\lim_{j\to\infty}\int Pf(x,y+c_j,P\gamma_j)-\ddprb{f_R,\bnu}\right|\\
&\leq\lim_{R\to\infty}\left|\lim_{j\to\infty}\int P(f-f_R)(x,y+c_j,P\gamma_j)\right|.
\end{align*}

First we assume that $\lambda_\nu(\overline{\Omega}\times\infty\pd\bB^m)=0$ and prove~\eqref{eqymeq}. Since the condition $\lambda_\nu(\overline{\Omega}\times\infty\pd\bB^m)=0$ is equivalent to $S^{-*}\bnu(\overline{\Omega}\times\infty\pd\bB^m\times\pd\bB^{m\times d})=0$ and Proposition~\ref{thmymsasmeasures} also forces $S^{-*}\bnu(\overline{\Omega}\times\infty\pd\bB^m\times\bB^{m\times d})=0$ we have that 
\[
S^{-*}\bnu(\overline{\Omega}\times\infty\pd\bB^m\times\overline{\bB^{m\times d}})=0,
\]
which implies that the sequence of measures $(S^{-*}\bdelta\asc{\gamma_j})_j\subset\mbfM^+(\overline{\Omega}\times\sigma\R^m\times\overline{\bB^{m\times d}})$ is tight in $\overline{\Omega}\times\R^m\times\overline{\bB^{m\times d}}$. That is, for every $\varepsilon>0$ there exists $K>0$ such that
\begin{equation}\label{eqtightnessestimate}
\sup_{j\in\mbN}S^{-*}\bdelta\asc{\gamma_j}\left(\overline{\Omega}\times\{y\in\R^m\colon |y|>K\}\times\overline{\bB^{m\times d}}\right)<\varepsilon.
\end{equation}
For $\varepsilon>0$ fixed, let $K$ verify~\eqref{eqtightnessestimate}. By Lemma~\ref{prelimlemrecessioncontrol}, we have that, for all $R>0$ sufficiently large, $|f(x,y,A)|\leq\varepsilon(1+|A|)$ whenever $(x,y)\in\overline{\Omega\times B(0,K+1)}$ and $|A|\geq R$. Since $(f-f_R)(x,y,A)=0$ whenever $|A|<R$, and $|y|\leq K$ implies $|y+c_j|\leq K+1$ once $|c_j|\leq 1$, we therefore have that, once $R>0$ is sufficiently large,
\begin{align*}
\left|\lim_{j\to\infty}\int P(f-f_R)(x,y+c_j,P\gamma_j)\right|&\leq\left|\lim_{j\to\infty}\int_{\Omega\times  \overline{B(0,K)}} P(f-f_R)(x,y+c_j,P\gamma_j)\right|\\
&\qquad+\left|\lim_{j\to\infty}\int_{\Omega\times\R^m\setminus \overline{ B(0,K)}} P(f-f_R)(x,y+c_j,P\gamma_j)\right|\\
&\leq \varepsilon\lL(\Omega)+\varepsilon\limsup_{j\to\infty}|\gamma_j|(\Omega\times\R^m)\\
&\qquad +CS^{-*}\bdelta\asc{\gamma_j}\left(\overline{\Omega}\times\{y\in\R^m\colon |y|>K\}\times\overline{\bB^{m\times d}}\right)\\\
&\leq\varepsilon(\lL(\Omega)+\limsup_{j\to\infty}|\gamma_j|(\Omega\times\R^m)+1).
\end{align*}
Since $(\gamma_j)_j$ is a norm-bounded sequence in $\mbfM(\Omega\times\R^m;\R^{m\times d})$ and $\varepsilon>0$ was arbitrary, we therefore deduce that
\[
\lim_{R\to\infty}\left|\lim_{j\to\infty}\int P(f-f_R)(x,y+c_j,P\gamma_j)\right|=0
\]
and hence that~\eqref{eqymeq} holds.

Next, we prove that~\eqref{eqymeqextended} holds under the assumption that $\sigma f^\infty\equiv 0$. By Proposition~\ref{thmymsasmeasures}, we have that
\[
\sup_{j}\int_{\Omega\times\R^m}|y|^{d/(d-1)}\;\dd\iota_{\delta_j}(x,y)<\infty,\quad\bdelta_j:=\bdelta\asc{\gamma_j}.
\]
This is equivalent to the statement
\begin{equation}\label{eqdtightestimate}
\sup_{j}\int_{\overline{\Omega}\times\R^m\times\overline{\bB^{m\times d}}}|y|^{d/(d-1)}(1-|B|)\;\dd S^{-*}\bdelta\asc{\gamma_j}(x,y,B)<\infty.
\end{equation}
Now, if there were $\varepsilon>0$ such that, for all $K\in\mbN$,
\begin{equation}\label{eqwrongtightnessestimate}
\varepsilon<\limsup_{j\to\infty}S^{-*}\bdelta\asc{\gamma_j}\left(\overline{\Omega}\times\{(y,B)\in\R^m\times\overline{\bB^{m\times d}}\colon |y|^{d/(d-1)}(1-|B|)>K\}\right),
\end{equation}
we would have that
\begin{align*}
K\varepsilon &\leq K \limsup_{j\to\infty}S^{-*}\bdelta\asc{\gamma_j}\left(\overline{\Omega}\times\{(y,B)\in\R^m\times\overline{\bB^{m\times d}}\colon |y|^{d/(d-1)}(1-|B|)>K\}\right)\\
&\leq \sup_{j}\int_{\overline{\Omega}\times\R^m\times\overline{\bB^{m\times d}}}|y|^{d/(d-1)}(1-|B|)\;\dd S^{-*}\bdelta\asc{\gamma_j}(x,y,B).
\end{align*}
Letting $K\to\infty$ we would then obtain
\[
\limsup_{j\to\infty}\int_{\overline{\Omega}\times\R^m\times\overline{\bB^{m\times d}}}|y|^{d/(d-1)}(1-|B|)\;\dd S^{-*}\bdelta\asc{\gamma_j}(x,y,B)=\infty,
\]
contradicting~\eqref{eqdtightestimate}. Taking the converse of~\eqref{eqwrongtightnessestimate}, we therefore see that for any $\varepsilon>0$ it holds that, for all $K\in\mbN$ sufficiently large,
\begin{equation}\label{eqrighttightestimate}
\limsup_{j\to\infty}S^{-*}\bdelta\asc{\gamma_j}\left(\overline{\Omega}\times\{(y,B)\in\R^m\times\overline{\bB^{m\times d}}\colon |y|^{d/(d-1)}(1-|B|)>K\}\right)<\varepsilon.
\end{equation}
Fix $\varepsilon>0$ and let $K\in\mbN$ verify~\eqref{eqrighttightestimate}. By Lemma~\ref{prelimlemrecessioncontrol} we have that, once $R>0$ is sufficiently large, $|A|\geq R$ implies that
\[
|f(x,y,A)|\leq\varepsilon(1+|A|)\quad\text{ for all }x\in\overline{\Omega}\text{ and }y\in\R^m\text{ such that }|y|^{d/(d-1)}\leq (K+1)(1+|A|).
\]
Defining $A_K\subset\R^m\times\overline{\bB^{m\times d}}$ by
\[
A_K:=\left\{(y,B)\in\R^m\times\overline{\bB^{m\times d}}\colon|y|^{d/(d-1)}(1-|B|)\leq K+1\right\},
\]
we therefore have
\[
\limsup_{j\to\infty}S^{-*}\bdelta\asc{\gamma_j}\left(\overline{\Omega}\times[(\R^m\times\overline{\bB^{m\times d}})\setminus A_K]\right)<\varepsilon
\]
and that
\[
|Sf(x,y,B)|\leq\varepsilon\quad\text{ for }(x,y,B)\in\overline{\Omega}\times A_K\text{ whenever }|B|\geq\frac{R}{1+R}.
\]

Since $S(f-f_R)(x,y,B)=0$ whenever $|B|\leq R/(1+R)$, we can now estimate, for $R>0$ sufficiently large,
\begin{align*}
&\left|\lim_{j\to\infty}\int P(f-f_R)(x,y+c_j,P\gamma_j)\right|\\
&\qquad= \left|\lim_{j\to\infty}\int_{\overline{\Omega}\times A_K} S(f-f_R)(x,y+c_j,B)\;\dd S^{-*}\bdelta\asc{\gamma_j}(x,y,B)\right|\\
&\qquad\qquad +\left|\lim_{j\to\infty}\int_{\overline{\Omega}\times[(\R^m\times\bB^{m\times d})\setminus A_K]} S(f-f_R)(x,y+c_j,B)\;\dd S^{-*}\bdelta\asc{\gamma_j}(x,y,B)\right|\\
&\qquad\leq \lim_{j\to\infty}\int_{\overline{\Omega}\times A_K} \varepsilon\;\dd S^{-*}\bdelta\asc{\gamma_j}(x,y,B)+\norm{Sf}_\infty \sup_jS^{-*}\bdelta\asc{\gamma_j}(\overline{\Omega}\times[(\R^m\times\bB^{m\times d})\setminus A_K])\\
&\qquad\leq \varepsilon\left(\sup_j S^{-*}\bdelta\asc{\gamma_j}(\overline{\Omega}\times\R^m\times\overline{\bB^{m\times d}})+\norm{Sf}_\infty\right).
\end{align*}
Since $\varepsilon>0$ was arbitrary, we therefore deduce that
\[
\lim_{R\to 0}\left|\lim_{j\to\infty}\int P(f-f_R)(x,y+c_j,P\gamma_j)\right|=0
\]
and hence that
\[
\lim_{j\to\infty}\int Pf(x,y+c_j,P\gamma_j)=\ddprb{f,\bnu}
\]
as required.

\emph{Step 4:} Assume now that $f\colon\overline{\Omega}\times\R^m\times\R^{m\times d}\to\R$ is Carath\'eodory with $f^\infty\equiv 0$ as in the previous step, but that $f$ now satisfies the bound $|f(x,y,A)|\leq C(1+|y|^p+|A|)$ for some $p\in[1,d/(d-1))$. For $k\in\mbN$, define $f_k\colon\Omega\times\R^m\times\R^{m\times d}\to\R$ by
\[
f_k(x,y,A):=\left\{\mathbbm{1}_{\{|y|\leq k\}}(y)+\mathbbm{1}_{\{|y|>k\}}(y)\frac{1+k^p+|A|}{1+|y|^p+|A|}\right\}f(x,y,A).
\]
Each $f_k$ satisfies a bound of the form $|f_k(x,y,A)|\leq C_k(1+|A|)$ with $f_k^\infty= f^\infty\equiv 0$ and also satisfies $\sigma f_k^\infty=\sigma f^\infty=0$ if $\sigma f^\infty$ exists with $\sigma f^\infty=0$. It follows that $f_k$ falls under the hypotheses of Step~2. From the definition of $f_k$, we can estimate
\[
|f_k(x,y,A)-f(x,y,A)|\leq\mathbbm{1}_{\{|y|>k\}}(y)\frac{|y|^p-k^p}{1+|y|^p+|A|}|f(x,y,A)|\leq C\mathbbm{1}_{\{|y|>k\}}(y)(|y|^p-k^p),
\]
from which it follows that
\[
\left|\int_{\Omega\times\R^m}P(f-f_k)(x,y+c_j,P\gamma_j)\right|\leq C\int_{\{x\in\Omega:|\asc{\gamma_j}(x)+c_j|> k\}}|\asc{\gamma_j}(x)+c_j|^p-k^p\;\dd x
\]
and that an analogous estimate also holds for $\bnu$. Since $p<d/(d-1)$ and $c_j\to 0$, the sequence $(\asc{\gamma_j}+c_j)_j\subset\BV(\Omega;\R^m)$ is $p$-uniformly integrable from which it follows that this estimate is uniform in $j$. We can therefore let $k\to\infty$ to deduce
\begin{equation*}
\lim_{j\to\infty}\int_{\Omega\times\R^m}Pf(x,y+c_j,P\gamma_j)=\ddprb{f,\bnu}.
\end{equation*}

\emph{Step 5:} Now assume that $f=f^\infty$. We first prove~\eqref{eqymeq}. Since the sequence $(S^{-*}\bdelta\asc{\gamma_j})_j$ is tight in $\overline{\Omega}\times\R^m\times\overline{\bB^{m\times d}}$ it follows (see for instance Proposition~1.62 in~\cite{AmFuPa00FBVF}) that
\[
\lim_{j\to\infty}\int g(x,y,B)\;\dd S^{-*}\bdelta\asc{\gamma_j}(x,y,B)=\int g(x,y,B)\;\dd S^{-*}\bnu(x,y,B)=\ddprb{S^{-1}g,\bnu}
\]
for every $g\in\C_b(\overline{\Omega}\times\R^m\times\overline{\bB^{m\times d}})$. Setting $g=Sf^\infty$, we obtain
\[
\lim_{j\to\infty}\int_{\Omega\times\R^m}Pf^\infty(x,y,P\gamma_j)=\lim_{j\to\infty}\ip{Sf^\infty}{S^{-*}\bdelta\asc{\gamma_j}}=\ddprb{f^\infty,\bnu}.
\]
Since, for compact $K\subset\Omega\times\R^m$,
\[
\int_{K\times\overline{\bB^{m\times d}}}|B|\;\dd S^{-*}\bdelta\asc{\gamma_j}(x,y,B)=|\gamma_j|(K),
\]
we see that the tightness of the sequence $(S^{-*}\bdelta\asc{\gamma_j})_j$ implies that $(|\gamma_j|)_j$ is also tight in $\mbfM^+(\overline{\Omega}\times\R^m)$. For $\varepsilon>0$, let $R>0$ be so large that $\limsup_j|\gamma_j|(\Omega\times\{|y|>R\})\leq\varepsilon$ and let $\varphi_\varepsilon\in\C_0(\R^m;[0,1])$ be such that $\varphi_\varepsilon(y)=1$ for $|y|\leq R$. Since $\varphi_\varepsilon f\in\mbfE(\Omega\times\R^m)$ and $(1-\varphi_\varepsilon)$ is supported in $\{|y|>R\}$, Lemma~\ref{lemliftingsperturb} then implies
\begin{align*}
&\lim_{j\to\infty} \left|\int_{\Omega\times\R^m} Pf^\infty(x,y+c_j,P\gamma_j)-\int_{\Omega\times\R^m} Pf^\infty(x,y,P\gamma_j)\right|\\
&\qquad\leq \lim_{j\to\infty}\left|\int_{\Omega\times\R^m} P(\varphi_\varepsilon f^\infty)(x,y+c_j,P\gamma_j)-\int_{\Omega\times\R^m} P(\varphi_\varepsilon f^\infty)(x,y,P\gamma_j)\right|\\
&\qquad\qquad+ \lim_{j\to\infty}\left|\int_{\Omega\times\R^m} P((1-\varphi_\varepsilon) f^\infty)(x,y+c_j,P\gamma_j)-\int_{\Omega\times\R^m} P((1-\varphi_\varepsilon) f^\infty)(x,y,P\gamma_j)\right|\\
&\qquad\leq 0 +2\limsup_{j\to\infty}|\gamma_j|(\Omega\times\{|y|>R\})\\
&\qquad\leq 2\varepsilon.
\end{align*}
Thus, since $S(\varphi_\varepsilon f)\to Sf$ pointwise in $\overline{\Omega}\times\R^m\times\overline{\bB^{m\times d}}$ as $\varepsilon\to 0$,
\[
\lim_{j\to\infty}\int_{\Omega\times\R^m}Pf^\infty(x,y+c_j,P\gamma_j)=\ddprb{f^\infty,\bnu}
\]
as required.

To obtain~\eqref{eqymeqextended}, note that since $\sigma f^\infty\in\C(\overline{\Omega}\times\sigma\R^m\times\R^{m\times d})$ and is positively one-homogeneous in the final variable, it follows that $ f^\infty\in\mbfE(\Omega\times\R^m)$. Lemma~\ref{lemliftingsperturb} then implies that
\[
\lim_{j\to\infty}\int_{\Omega\times\R^m}Pf^\infty(x,y+c_j,P\gamma_j)=\ddprb{f^\infty,\bnu}
\]
as required.

\emph{Step 6:}
By Step~4 applied to the integrand $f-f^\infty$, we have that
\begin{equation}\label{eqextendedrepnorecession}
\lim_{j\to\infty}\int_{\Omega\times\R^m}P(f-f^\infty)(x,y+c_j,P\gamma_j)=\ddprb{f-f^\infty,\bnu}.
\end{equation}
By Step~5 applied to $f^\infty$, we also have
\begin{equation}\label{eqextendedreprecession}
\lim_{j\to\infty}\int_{\Omega\times\R^m}Pf^\infty(x,y+c_j,P\gamma_j)=\ddprb{f^\infty,\bnu}
\end{equation}
and so the result follows from adding~\eqref{eqextendedrepnorecession} and~\eqref{eqextendedreprecession}.
\end{proof}

\begin{corollary}\label{corymlsc}
Let $\bnu\in\YLift(\Omega\times\R^m)$, $(\gamma_j)_j\subset\Lift(\Omega\times\R^m)$ be such that $\gamma_j\toY\bnu$, $(c_j)_j\subset\R^m$ be such that $c_j\to 0\in\R^m$ and let $f\in\RBVw(\Omega\times\R^m)$. Then it holds that
\begin{equation}\label{eqymineq}
\liminf_{j\to\infty}\int_{\Omega\times\R^m}Pf(x,y+c_j,P\gamma_j)\geq \ddprb{f,\bnu\restrict(\overline{\Omega}\times\R^m)}.
\end{equation}
\end{corollary}

\begin{proof}
First assume that $f\geq 0$. For $R>0$ let $\varphi_R\in\C_c(\R^m;[0,1])$ be such that $\varphi_R(y)=1$ if $|y|\leq R$ and define $f_R\colon\overline{\Omega}\times\R^m\times\R^{m\times d}\to[0,\infty)$ by $f_R(x,y,A):=\varphi_R(y)f(x,y,A)$. It follows that $f_R\in\RL(\Omega\times\R^m)$ and that $f_R\uparrow f$, $f^\infty_R\uparrow f^\infty$ in $\overline{\Omega}\times\R^m\times\R^{m\times d}$ as $R\to \infty$. In addition, $\sigma f_R^\infty$ exists for every $R>0$ and satisfies $\sigma f_R^\infty\restrict(\overline{\Omega}\times\infty\pd\bB^m\times\R^{m\times d})\equiv 0$. By Proposition~\ref{lemextendedrepresentation}~(ii) combined with the fact that $\sigma f_R^\infty\restrict(\overline{\Omega}\times\infty\pd\bB^m\times\R^{m\times d})=0$ implies $\ddpr{f_R,\bnu}=\ddprb{f_R,\bnu\restrict(\overline{\Omega}\times\R^m)}$, we have that
\[
\lim_{j\to\infty}\int Pf_R(x,y+c_j,P\gamma_j)=\ddprb{f_R,\bnu}=\ddprb{f_R,\bnu\restrict(\overline{\Omega}\times\R^m)}.
\]
The positivity of $f$ and the Monotone Convergence Theorem then imply that
\begin{align}
\begin{split}
\liminf_{j\to\infty}\int Pf(x,y+c_j,P\gamma_j)&\geq\lim_{R\to\infty}\lim_{j\to\infty}\int Pf_R(x,y+c_j,P\gamma_j)\\
&=\lim_{R\to\infty}\ddprb{f_R,\bnu\restrict(\overline{\Omega}\times\R^m)}\\
&=\ddprb{f,\bnu\restrict(\overline{\Omega}\times\R^m)}.
\end{split}
\end{align}

If we only assume that $f\in\RBVw(\Omega\times\R^m)$ then, letting $h\in\RL(\Omega\times\R^m)$ with $\sigma h^\infty=0$ be such that 
\[
f(x,y,A)\geq - h(x,y,A),
\]
we can combine the previous step applied to $f+h$ with the fact that Proposition~\ref{lemextendedrepresentation}~(ii) applies to $-h$ to deduce that
\begin{align*}
\liminf_{j\to\infty}\int Pf(x,y+c_j,P\gamma_j)&=\liminf_{j\to\infty}\int P(f+h)(x,y+c_j,P\gamma_j)-\lim_{j\to\infty}\int Ph(x,y+c_j,P\gamma_j)\\
&\geq\ddprb{f+h,\bnu\restrict(\overline{\Omega}\times\R^m)}-\ddprb{h,\bnu}\\
&=\ddprb{f+h,\bnu\restrict(\overline{\Omega}\times\R^m)}-\ddprb{h,\bnu\restrict(\overline{\Omega}\times\R^m)}\\
&=\ddprb{f,\bnu\restrict(\overline{\Omega}\times\R^m)},
\end{align*}
as required, where we used the fact that $\sigma h^\infty=0$ implies $\ddpr{h,\bnu}=\ddprb{h,\bnu\restrict(\overline{\Omega}\times\R^m)}$ to justify the penultimate equality.
\end{proof}

%% file: wstarblowup.tex
\section{Tangent Young measures and Jensen inequalities}\label{chaptangentliftingyms}
By Corollary~\ref{corymlsc} with $\gamma_j=\gamma\asc{u_j-(u_j)}$ and $c_j=(u_j)$, we know that, if $f\in\RBVw(\Omega\times\R^m)$ and $(u_j)_j\subset\BV(\Omega;\R^m)$, $u\in\BV_\#(\Omega;\R^m)$ are such that $u_j\wsc u$ then, by passing to a non-relabelled subsequence so that $(\gamma\asc{u_j-(u_j)})_j$ generates a (lifting) Young measure $\bnu\in\ALift(\Omega\times\R^m)$ with $\llbracket\bnu\rrbracket=u$, it holds that
\[
\liminf_{j\to\infty}\Fcal[u_j]=\liminf_{j\to\infty}\int_{\Omega\times\R^m}Pf(x,y+(u_j),P\gamma\asc{u-(u_j)})\geq\ddprb{f,\bnu\restrict(\overline{\Omega}\times\R^m)	}.
\]
Letting $\lambda_\nu^{\mathrm{gs}}:=\lambda_\nu-\frac{\dd\lambda_\nu}{\dd|\gamma\asc{u}|}|\gamma\asc{u}|\restrict(\overline{\Omega}\setminus\Jcal_u\times\R^m)$ denote the graph-singular part of $\lambda_\nu$ and using Corollary~\ref{lemstructureliftingym} together with the fact that $f^\infty\geq 0$, we can write
\begin{align*}
\ddprb{f,\bnu\restrict(\overline{\Omega}\times\R^m)}&=\int_{\Omega\times\R^m}\ip{f(x,y,\frarg)}{\nu_{x,y}}\;\dd\iota_\nu(x,y)\\
&\qquad+\int_{(\Omega\setminus\Jcal_u)\times\R^m}\frac{\dd\lambda_\nu}{\dd|\gamma\asc{u}|}(x,y)\ip{f^\infty(x,y,\frarg)}{\nu_{x,y}^\infty}\;\dd|\gamma\asc{u}|(x,y)\\
&\qquad+\int_{\overline{\Omega}\times\R^m}\ip{f^\infty(x,y,\frarg)}{\nu_{x,y}^\infty}\;\dd\lambda^{\mathrm{gs}}_\nu(x,y)\\
&\geq\int_\Omega\ip{f(x,u(x),\frarg)}{\nu_{x,u(x)}}\;\dd x\\
&\qquad+\int_{{\Omega}\setminus\Jcal_u}\frac{\dd\lambda_\nu}{\dd|\gamma\asc{u}|}(x,u(x))\langle f^\infty(x,u(x),\frarg),\nu_{x,u(x)}^\infty\rangle\;\dd|Du|(x)\\
&\qquad+\int_{\Jcal_u\times\R^m}\ip{f^\infty(x,y,\frarg)}{\nu_{x,y}^\infty}\;\dd\lambda^{\mathrm{gs}}_\nu(x,y).
\end{align*}
Using the equality $\frac{\dd|\gamma\asc{u}|}{\dd\iota_\nu}(x,u(x))=\frac{\dd|Du|}{\dd\lL}(x)=|\nabla u(x)|$, disintegrating
\[
  \lambda_\nu^{\mathrm{gs}} = \lambda_\nu=\pi_\#\lambda_\nu\otimes\rho \quad\text{ on } \Jcal_u \times \R^m
\]
and using again the positivity of $f^\infty$ to neglect the $(\Hcal^{d-1}\restrict\Jcal_u)\otimes\rho$-singular part of $\lambda^{gs}$, we can then deduce
\begin{align}
\begin{split}\label{eqymlscineq}
\liminf_{j\to\infty}\Fcal[u_j]
&\geq\int_{\Omega}\ip{f(x,u(x),\frarg)}{\nu_{x,u(x)}}+\frac{\dd\lambda_\nu}{\dd\iota_\nu}(x,u(x))\langle f^\infty(x,u(x),\frarg),\nu_{x,u(x)}^\infty\rangle\;\dd x\\
&\qquad+\int_{\Omega}\frac{\dd\lambda_\nu}{\dd|\gamma\asc{u}|}(x,u(x))\langle f^\infty(x,u(x),\frarg),\nu_{x,u(x)}^\infty\rangle\;\dd|D^cu|(x)\\
&\qquad+\int_{\Jcal_u}\frac{\dd\pi_\#\lambda_\nu}{\dd\Hcal^{d-1}\restrict\Jcal_u}(x)\int_{\R^m}\ip{f^\infty(x,y,\frarg)}{\nu_{x,y}^\infty}\;\dd\rho_x(y)\;\dd\Hcal^{d-1}(x).
\end{split}
\end{align}
The lower semicontinuity component of Theorem~\ref{wsclscthm} is now reduced to the task of obtaining the three pointwise inequalities
\[
\ip{f(x,u(x),\frarg)}{\nu_{x,u(x)}}+\frac{\dd\lambda_\nu}{\dd\iota_\nu}(x,u(x))\langle f^\infty(x,u(x),\frarg),\nu_{x,u(x)}^\infty\rangle\geq f(x,u(x),\nabla u(x))
\]
for $\lL$-almost every $x\in\Omega$,
\[
\frac{\dd\lambda_\nu}{\dd|\gamma\asc{u}|}(x,u(x))\langle f^\infty(x,u(x),\frarg),\nu_{x,u(x)}^\infty\rangle\geq f^\infty\left(x,u(x),\frac{\dd D^cu}{\dd|D^cu|}(x)\right)
\]
for $|D^cu|$-almost every $x\in\Omega$, and
\[
\frac{\dd\pi_\#\lambda_\nu}{\dd\Hcal^{d-1}\restrict\Jcal_u}(x)\int_{\R^m}\ip{f^\infty(x,y,\frarg)}{\nu_{x,y}^\infty}\;\dd\rho_x(y)\geq K_f[u]
\]
for $\Hcal^{d-1}$-almost every $x\in\Jcal_u$. In Section~\ref{sectangentyms}, we will show that the left hand side of each of these inequalities can be computed as a duality product $\ddpr{h,\bsigma}$ with $\bsigma\in\AYLift(\Omega\times\R^m)$. In the case of the first two inequalities, $h$ depends only on the $A$ variable, and only on the $y$ and $A$ variables in the third case. From there, the theory in Section~\ref{secjensenineq} uses the definitions of quasiconvexity and $K_f[u]$ to obtain the desired inequalities.

\subsection{Tangent Young measures}\label{sectangentyms}
Let $\bnu\in\YLift(\Omega\times\R^m)$ with $\llbracket\bnu\rrbracket=u\in\BV_\#(\Omega;\R^m)$ and, for $x_0\in\Dcal_u\cup\Ccal_u\cup\Jcal_u$, let $u^{r}$ be as defined in Theorem~\ref{thmbvblowup}. Let $(\gamma_j)_j\subset\Lift(\Omega\times\R^m)$ be a sequence which generates $\bnu$. Recalling
\[
c_r:=\begin{cases}1&\text{ if }x_0\in\Dcal_{u},\\
r&\text{ if }x_0\in\Jcal_{u},\\
\frac{r^d}{|Du|(B(x_0,r))}&\text{ if }x_0\in\Ccal_{u},
\end{cases}
\]
for each $r>0$, let
\[
\gamma_{j}^{r,c_r}=\frac{c_r}{r^d}(T_{\gamma_j}^{(x_0,r,c_r)})_\#\gamma_j\in\Lift(\bB^d\times\R^m)
\]
be a rescaling of $\gamma_j$ of the form discussed in Lemma~\ref{lemrestrictionsofliftings} for each $j\in\mbN$, where $T_{\gamma_j}^{(x_0,r,c_r)}\colon B(x_0,r)\times \R^m\to\bB^d\times\R^m$ denotes the homothety
\[
T_{\gamma_j}^{(x_0,r,c_r)}(x,y)=\left(\frac{x-x_0}{r},c_r\frac{y-(\asc{\gamma_j})_{x_0,r}}{r}\right).
\]
(Note that, according to the terminology of Lemma~\ref{lemrestrictionsofliftings}, it is not the case that $\gamma_j^{r,c_r}=\gamma_j^r$ here since $c_r$ in this context does not represent the volume fraction associated to $D\asc{\gamma_j}$ but is fixed as the volume fraction corresponding to $Du$ instead.) Since Lemma~\ref{lemrestrictionsofliftings} states that
\[
\gr^{\asc{\gamma_{j}^{r,c_r}}}_\#(\lL\restrict\bB^d)=\frac{1}{r^d}(T^{(x_0,r,c_r)}_{\gamma_j})_\#\gr^{\asc{\gamma_j}}_\#(\lL\restrict B(x_0,r)),
\]
we can therefore compute
\begin{align*}
P\gamma_{j}^{r,c_r}&=\left(\gamma_{j}^{r,c_r},\gr^{\asc{\gamma_{j}^{r,c_r}}}_\#(\lL\restrict\bB^d)\right)\\
&=\left(\frac{c_r}{r^d}(T_{\gamma_j}^{(x_0,r,c_r)})_\#{\gamma_j},\frac{1}{r^d}(T_{\gamma_j}^{(x_0,r,c_r)})_\#\gr^{\asc{{\gamma_j}}}_\#(\lL\restrict B(x_0,r))\right)\\
&=\frac{1}{r^d}(T^{(x_0,r,c_r)}_{\gamma_j})_\#\left(c_r{\gamma_j},\gr^{\asc{{\gamma_j}}}_\#(\lL\restrict B(x_0,r))\right)\\
&=\frac{1}{r^d}(T^{(x_0,r,c_r)}_{\gamma_j})_\#\left[\begin{pmatrix}c_r & 0\\ 0 & 1\end{pmatrix}P{\gamma_j}\right],
\end{align*}
where $\begin{pmatrix}c_r & 0\\ 0 & 1\end{pmatrix}\colon\R^{m\times d}\times\R\to\R^{m\times d}\times\R$ denotes the linear map $(A,t)\mapsto(c_r A,t)$.

Thanks to the positive one-homogeneity of $Pf$ in $(A,t)$ in conjunction with~\eqref{eqperspectiveobservation} from Section~\ref{secpreliminaries}, these observations imply that, for $f\in\mbfE(\Omega\times\R^m)$,
\begin{align*}
\int Pf(z,w,P\gamma_j^{(r,c_r)})&=\frac{1}{r^d}\int Pf\left(z,w,(T^{(x_0,r,c_r)}_{\gamma_j})_\#\left[\begin{pmatrix}c_r & 0\\ 0 & 1\end{pmatrix}P\gamma_j\right]\right)\\
&=\frac{1}{r^d}\int Pf\left({T^{(x_0,r,c_r)}_{\gamma_j}}(x,y),\begin{pmatrix}c_r & 0\\ 0 & 1\end{pmatrix}P\gamma_j\right).
\end{align*}

Define $f_{c_r}\in\mbfE(\Omega\times\R^m)$ by
\[
f_{c_r}(x,y,A):=f\left(T^{(x_0,r,c_r)}_{[\bnu]}(x,y),c_rA\right)=f\left(\frac{x-x_0}{r},c_r\frac{y-(u)_{x_0,r}}{r},c_rA\right)
\]
so that (recall~\eqref{eqrescaleabbrev})
\[
Pf_{c_r}(x,y,A)=Pf\left({T^{(x_0,r)}_{[\bnu]}}(x,y),\begin{pmatrix}c_r & 0\\ 0 & 1\end{pmatrix}(A,t)\right).
\]
Noting
\[
T^{(x_0,r,c_r)}_{\gamma_j}(x,y)=T^{(x_0,r)}_{[\bnu]}(x,y)+r^{-1}c_r(0,(\asc{\gamma_j})_{x_0,r}-(u)_{x_0,r}),
\]
we can therefore write
\[
\int Pf(z,w,P\gamma_{j}^{r,c_r})=r^{-d}\int Pf_{c_r}(x,y+r^{-1}c_r[(\asc{\gamma_j})_{x_0,r}-(u)_{x_0,r})],P\gamma_j).
\]
Since
\[
(\asc{\gamma_j})_{x_0,r}-(u)_{x_0,r}\to 0\quad\text{ as }j\to 0,
\]
Lemma~\ref{lemliftingsperturb} implies that
\begin{align*}
\lim_{j\to\infty}\int Pf(z,w,P\gamma_{j}^r)&=\frac{1}{r^d}\lim_{j\to\infty}\int Pf_{c_r}(x,y+r^{-1}c_r[(\asc{\gamma_j})_{x_0,r}-(u)_{x_0,r}],P\gamma_j)\\
&=\frac{1}{r^d}\lim_{j\to\infty}\int Pf_{c_r}(x,y,P\gamma_j)\\
&=\frac{1}{r^d}\ddprb{f_{c_r},\bnu}.
\end{align*}

This limit exists for any $f\in\mbfE(\bB^d\times\R^m)$, and so it follows that $\gamma_{j}^r\toY\bsigma_r$ as $j\to\infty$, where, by virtue of the definition of $f_{c_r}$, $\bsigma_r\in\YLift(\bB^d\times\R^m)$ is defined by

\begin{equation}\label{eqrescaledym}
\ddprb{f,\bsigma_r}=\frac{1}{r^d}\ddprB{f\left(\frac{\frarg-x_0}{r},c_r\left(\frac{\frarg-(u)_{x_0,r}}{r}\right),{c_r}\,\frarg\right),\bnu}.
\end{equation}
We will show that, for $\lL+|Du|$-almost every $x_0\in\Dcal_u\cup\Ccal_u\cup\Jcal_u$, the family $(\bsigma_r)_{r>0}$ (or at least a subsequence) converges weakly* to a limit as $r\to 0$. The primary tools in identifying this limit for $x_0\in\Dcal_u\cup\Ccal_u$ are Theorem~\ref{thmdiffusetangentlifting} which guarantees the strict convergence of the rescaled liftings $(\gamma\asc{u^r})_{r>0}$ and the graphical Besicovitch Derivation Theorem, Theorem~\ref{thmgeneralisedbesicovitch} introduced below. We note that the usual version of the Besicovitch Derivation Theorem and its generalisation, the Morse Derivation Theorem (see~\cite{Mors47PB} and also Theorems~2.22 and~5.52 in~\cite{AmFuPa00FBVF}) cannot be used in this situation, since the aspect ratio $(c_rr^{-1})/r^{-1}=c_r$ corresponding to the scaling present in~\eqref{eqrescaledym} (which, due to the need to apply Theorem~\ref{thmdiffusetangentlifting}, cannot be modified) will converge to $0$ for $x\in\Ccal_u$.

The following theorem demonstrates that, provided the denominator is a graphical measure, derivatives of measures can be computed using families of sets which are very different to the usual decreasing cylinders $B(x,r)\times B(y,r)$, so long as the family is sufficiently well behaved according to the differentiating measure.
\begin{theorem}[Generalised Besicovitch Differentiation Theorem for graphical measures]\label{thmgeneralisedbesicovitch}
Given $u\colon\Omega\to\R^m$, let $\eta=\gr^u_\#\mu\in\mbfM^+(\Omega\times\R^m)$ be a $u$-graphical measure. For each $x\in\Omega$, let $(c_r^x)_{r > 0} \subset (0,\infty)$ satisfy $\lim_{r\downarrow 0}c_r^x=0$ and $(y_r^x)_{r > 0} \subset \R^m$ satisfy $\lim_{r\downarrow 0}y_r^x=u(x)$. Then:

\begin{enumerate}[(i)]
\item\label{graphbesicovitchsingular} If $\lambda$ is a (possibly vector-valued) measure on $\Omega\times\R^m$ satisfying $\lambda\perp\eta$, we have that 
\[
0=\frac{\dd\lambda}{\dd\eta}(x,u(x))=\lim_{r\to 0}\frac{\lambda\left(\overline{B(x,r)\times B(y_r^x,c_r^x)}\right)}{\pi_\#\eta(\overline{B(x,r)})}=\lim_{r\to 0}\frac{\lambda\left(\overline{B(x,r)\times B(y_r^x,c_r^x)}\right)}{\mu(\overline{B(x,r)})}
\]
for $\eta$-almost every $(x,u(x))\in\Omega\times\R^m$, where $\pi\colon\Omega\times\R^m\to\Omega$ is the projection map $\pi((x,y)):=x$.
\item\label{cylindricallebesgue} A cylindrical version of the Lebesgue Differentiation Theorem holds in the sense that, for any $f\in\Lp^1(\Omega\times\R^m,\eta)$ (i.e., the $\Lp^1$-space with respect to $\eta$),
\[
\lim_{r\to 0}\frac{1}{\mu(\overline{B(x,r)})}\int_{\overline{B(x,r)}\times \R^m}|f(\overline{x},y)-f(x,u(x))|\;\dd\eta(\overline{x},y)=0
\]
for $\mu$-almost every $x\in\Omega$.
\item\label{graphbesicovitchgeneral} If $\lambda\in\mbfM(\Omega\times\R^m)$, then
\[
\frac{\dd\lambda}{\dd\eta}(x,u(x))=\lim_{r\to 0}\frac{\lambda\left(\overline{B(x,r)\times B(y_r^x,c_r^x)}\right)}{\pi_\#\eta(\overline{B(x,r)})}=\lim_{r\to 0}\frac{\lambda\left(\overline{B(x,r)\times B(y_r^x,c_r^x)}\right)}{\mu(\overline{B(x,r)})}
\]
for $\eta$-almost every $(x,u(x))\in\Omega\times\R^m$ for which $(c_r^x)_r$ and $(y_r^x)_r$ are such that
\[
\lim_{r\to 0}\frac{\eta\left(\overline{B(x,r)\times B(y_r^x,c_r^x)}\right)}{\mu(\overline{B(x,r)})} = 1.
\]
\end{enumerate}
 \end{theorem}

\begin{proof}
We first establish~\eqref{graphbesicovitchsingular}: let $\lambda$ be a measure on $\Omega\times\R^m$ satisfying $\lambda\perp\eta$ and define for $x \in \Omega$ such that $\mu(\overline{B(x,r)}) > 0$ for all $r > 0$ the function
\[
F(x):=\limsup_{r\to 0}\frac{|\lambda|(\overline{B(x,r)\times  B(y_r^x,c_r^x)})}{\mu(\overline{B(x,r)})}.
\]
Let $Z \subset \Omega$ be such that $\mu(Z) = \mu(\Omega)$, $|\lambda|(\gr^u(Z))=0$ and such that at every $x \in Z$ it holds that $\mu(\overline{B(x,r)}) > 0$ for all $r > 0$. Such a set $Z$ exists since $\eta = \gr^u_\# \mu$ is singular to $\lambda$ and $\eta$ is carried by the set $\gr^u(\Omega)$.

Let $E\subset Z$ be a Borel set such that $F(x)>t>0$ for all $x\in E$ and for $\varepsilon\in(0,t)$ arbitrary let $A\supset \gr^u(E)$ be an open set. Define the family
\begin{align*}
\mathcal{F}:=&\Bigl\{\overline{B(x,r)\times B(y_r^x,c_r^x)}\colon\; x\in E,\text{ }\overline{B(x,r)\times B(y_r^x,{c_r^x})}\subset A\text{ and }\\
&\qquad|\lambda|(\overline{B(x,r)\times B(y_r^x,c_r^x)})\geq(t-\varepsilon)\mu(\overline{B(x,r)})\Bigr\}\subset\R^d\times\R^m.
\end{align*}
Since $c^x_r\downarrow 0$ and $y_r^x\to u(x)$ as $r\downarrow 0$ it follows that $\pi_\#\mathcal{F}$ is a fine cover for $E$ and so, by the Vitali--Besicovitch Covering Theorem, there exists a countable disjoint subfamily of $\pi_\#\mathcal{F}$, which we write as $\pi_\#\mathcal{F}'$ for some disjoint subcollection $\mathcal{F}'\subset\mathcal{F}$, that covers $\mu$-almost all of $E$. We therefore have that
\[
(t-\varepsilon)\mu(E)\leq(t-\varepsilon)\sum_{B\in\mathcal{F}'}\mu(\pi_\# B)\leq\sum_{B\in\mathcal{F}'}|\lambda|(B)\leq|\lambda|(A).
\]
First letting $\varepsilon\downarrow 0$ and then using the outer regularity of $|\lambda|$ to approximate $\gr^u(E)$ with a sequence of open sets, we obtain
\[
t\mu(E)\leq|\lambda|(\gr^u(E)).
\]
Since $|\lambda|(\gr^u(Z))=0$, also $|\lambda|(\gr^u(E))=0$ and if $E$ was such that $\eta(\gr^u(E))=\mu(E)>0$, then $t=0$, a contradiction. It follows that $F(x)=0$ for $\mu$-almost every $x\in\Omega$ and hence that
\[
\lim_{r\to 0}\frac{\lambda(\overline{B(x,r)\times B(y_r^x,c_r^x)})}{\mu(\overline{B(x,r)})}=\frac{\dd\lambda}{\dd\eta}(x,u(x))=0 \quad \text{ for }\eta\text{-a.e.\ }(x,u(x))\in\Omega\times\R^m,
\]
as required. 

Next, we prove~\eqref{cylindricallebesgue}: for $f\in\Lp^1(\Omega\times\R^m,\eta)$, the definition of $\eta=\gr^u_\#\mu$ implies
\begin{align*}
&\lim_{r\to 0}\frac{1}{\mu(\overline{B(x,r)})}\int_{\overline{B(x,r)}\times \R^m}|f(\overline{x},y)-f(x,u(x))|\;\dd\eta(\overline{x},y)\\
&\qquad =\lim_{r\to 0}\frac{1}{\mu(\overline{B(x,r)})}\int_{\overline{B(x,r)}}|f(\overline{x},u(\overline{x}))-f(x,u(x))|\;\dd\mu(\overline{x}).
\end{align*}
It can be checked that the conditions $f\in\Lp^1(\Omega\times\R^m,\eta)$ and $f\circ \gr^u\in\Lp^1(\Omega,\mu)$ are equivalent, and so it follows from $f\in\Lp^1(\Omega\times\R^m,\eta)$ that $\mu$-almost every $x\in\Omega$ is a $\mu$-Lebesgue point for $x\mapsto f(x,u(x))$. We therefore deduce that
\begin{equation}\label{eqlebesguepoint}
\lim_{r\to 0}\frac{1}{\mu(\overline{B(x,r)})}\int_{\overline{B(x,r)}\times \R^m}|f(\overline{x},y)-f(x,u(x))|\;\dd\eta(\overline{x},y)=0
\end{equation}
for $\mu$-almost every $x\in\Omega$, as required.

Finally, to prove~\eqref{graphbesicovitchgeneral} we argue as follows: for $\lambda\in\mbfM(\Omega\times\R^m)$ let
\[
\lambda = \frac{\dd\lambda}{\dd\eta}\eta + \lambda^s, \qquad\lambda\perp\eta
\]
be the usual Radon--Nikod\'{y}m decomposition of $\lambda$ with respect to $\eta$. Noting that~\eqref{cylindricallebesgue} implies
\begin{align*}
&\lim_{r\to 0}\frac{1}{\mu(\overline{B(x,r)})}\int_{\overline{B(x,r)\times B(y_r^x,c_r^x)}}\left|\frac{\dd\lambda}{\dd\eta}(\overline{x},y)-\frac{\dd\lambda}{\dd\eta}(x,u(x))\right|\;\dd\eta(\overline{x},y)\\
&\qquad \leq \lim_{r\to 0}\frac{1}{\mu(\overline{B(x,r)})}\int_{\overline{B(x,r)}\times \R^m}\left|\frac{\dd\lambda}{\dd\eta}(\overline{x},y)-\frac{\dd\lambda}{\dd\eta}(x,u(x))\right|\;\dd\eta(\overline{x},y)\\
&\qquad =0,
\end{align*}
we can use~\eqref{graphbesicovitchsingular} and~\eqref{cylindricallebesgue} together to see that
\begin{align*}
\lim_{r\to 0}\frac{\lambda\left(\overline{B(x,r)\times B(y_r^x,c_r^x)}\right)}{\mu(\overline{B(x,r)})} & = \lim_{r\to 0}\frac{1}{\mu(\overline{B(x,r)})}\int_{\overline{B(x,r)\times B(y_r^x,c_r^x)}}\frac{\dd\lambda}{\dd\eta}(\overline{x},y)\;\dd\eta(\overline{x},y) \\
&\qquad+ \lim_{r\to 0}\frac{\lambda^s\left(\overline{B(x,r)\times B(y_r^x,c_r^x)}\right)}{\mu(\overline{B(x,r)})}\\
& = \frac{\dd\lambda}{\dd\eta}(x,u(x)) \cdot \lim_{r\to 0}\frac{\eta\left(\overline{B(x,r)\times B(y_r^x,c_r^x)}\right)}{\mu(\overline{B(x,r)})} + 0
\end{align*}
for $\eta$-almost every $(x,u(x))\in\Omega\times\R^m$, at which point the conclusion follows.
\end{proof}

Using Theorem~\ref{thmgeneralisedbesicovitch}, we can prove that the behaviour of graphical measures under general homotheties is stable under multiplication by integrable functions:

\begin{lemma}\label{lemgraphicallebesgueconverg}
Let $\eta\in\mbfM^+(\Omega\times\R^m)$ be a $u$-graphical measure on $\Omega\times\R^m$ (that is, $\eta=\gr^u_\#\mu$ for some $\mu\in\mbfM^+(\Omega)$ and some $\mu$-measurable function $u\colon\Omega\to\R^m$) and let $x_0\in\Omega$, $r_n\downarrow 0,c_n\downarrow 0,a_n\downarrow 0$ and $(y_n)_n\subset\R^m$ with $y_n\to u(x_0)$ be such that
\[
a_nT_\#^{(x_0,r_n),(y_n,c_n)}\eta\wsc\eta^0\quad\text{ in }\quad\mbfM(\overline{\bB^d}\times\R^m).
\]
If $f\in\Lp^1(\Omega\times\R^m,\mu)$ and $x_0$ is an $\eta$-cylindrical Lebesgue point for $f$ in the sense of Theorem~\ref{thmgeneralisedbesicovitch}, then
\[
a_nT_\#^{(x_0,r_n),(y_n,c_n)}(f\eta)\wsc f(x_0,u(x_0))\eta^0,
\]
where $T^{(x_0,r_n),(y_n,c_n)}\colon B(x_0,r_n)\times\R^m\to\bB^d\times\R^m$ denotes the homothety
\[
T^{(x_0,r_n),(y_n,c_n)}(x,y)=\left(\frac{x-x_0}{r_n},\frac{y-y_n}{c_n}\right).
\]
Moreover, if $a_nT_\#^{(x_0,r_n),(y_n,c_n)}\eta\to\eta^0$ strictly, then $a_nT_\#^{(x_0,r_n),(y_n,c_n)}f\eta\to f(x_0)\eta^0$ strictly as well.
\end{lemma}

\begin{proof}
Since the sequence $(a_nT_\#^{(x_0,r_n),(y_n,c_n)}\eta)_n$ is bounded, it must be the case that 
\[
\sup_na_n\eta(\overline{B(x_0,r_n)}\times\R^m)=\sup_na_n T_\#^{(x_0,r_n),(y_n,c_n)}\eta(\overline{\bB^d}\times\R^m)<\infty.
\]
The Lebesgue point property of $x_0$ then implies
\begin{align*}
&\lim_{n\to \infty}a_n\left|T_\#^{(x_0,r_n),(y_n,c_n)}(f\eta)-T_\#^{(x_0,r_n),(y_n,c_n)}(f(x_0,u(x_0))\eta)\right|(\bB^d\times\R^m)\\
&\qquad \leq \sup_n\left\{a_n\mu(B(x_0,r_n))\right\} \\
&\qquad\qquad \cdot \lim_{n\to \infty}\frac{1}{\mu(\overline{B(x_0,r_n)})}\int_{\overline{B(x_0,r_n)}\times\R^m}|f(x,y)-f(x_0,u(x_0))|\;\dd\eta(x,y)\\
&\qquad =0.
\end{align*}
We therefore see that, for any $\varphi\in\C_0(\bB^d\times\R^m)$,
\begin{align*}
&\lim_{n\to\infty}\left|\ip{\varphi}{a_{n}T_\#^{(x_0,r_n),(y_n,c_n)}(f\eta)}-\ip{\varphi}{a_{n}T_\#^{(x_0,r_n),(y_n,c_n)}(f(x_0,u(x_0))\eta)}\right|\\
&\qquad \leq \lim_{n\to \infty}\norm{\varphi}_\infty a_n\left|T_\#^{(x_0,r_n),(y_n,c_n)}(f\eta)-T_\#^{(x_0,r_n),(y_n,c_n)}(f(x_0,u(x_0))\eta\right|(\overline{\bB^d}\times\R^m)\\
&\qquad=0.
\end{align*}
It also holds that 
\begin{align*}
&\lim_{n\to\infty}\left|\bigl|a_{n}T_\#^{(x_0,r_n),(y_n,c_n)}(f\eta)\bigr|(\overline{\bB^d}\times\R^m)-\bigl|a_{n}T_\#^{(x_0,r_n),(y_n,c_n)}f(x_0,u(x_0))\eta\bigr|(\overline{\bB^d}\times\R^m)\right|\\
&\qquad \leq \lim_{n\to\infty}a_{n}\left|T_\#^{(x_0,r_n),(y_n,c_n)}f\eta-T_\#^{(x_0,r_n),(y_n,c_n)}f(x_0,u(x_0))\eta\right|(\overline{\bB^d}\times\R^m)\\
&\qquad=0.
\end{align*}
These estimates prove the claims.
\end{proof}

\begin{theorem}\label{thmregtangentyms}
Let $\bnu\in\YLift(\Omega\times\R^m)$ and define $u:=\llbracket\bnu\rrbracket$. For $\lL$-almost every $x_0\in\Omega$, there exists a \textbf{regular tangent Young measure} $\bsigma\in\YLift(\bB^d\times\R^m)$ satisfying
\begin{enumerate}[(i)]
	\item $\llbracket\bsigma\rrbracket=u^0$ where $u^0(z):=\nabla u(x_0)z$, and $\iota_\sigma=\gr_\#^{u^0}(\lL\restrict\bB^d)$;
	\item $\sigma_{z,w}=\nu_{x_0,u(x_0)}$ for $\iota_\sigma$-almost every $(z,w)\in\bB^d\times\R^m$;
	\item $\lambda_\sigma=\frac{\dd\lambda_\nu}{\dd\iota_\nu}(x_0,u(x_0))\iota_\sigma$;
	\item $\sigma^\infty_{z,w}=\nu^\infty_{x_0,u(x_0)}$ for $\lambda_\sigma$-almost every $(z,w)\in\bB^d\times\R^m$.
\end{enumerate}
In addition, if $\bnu\in\AYLift(\Omega\times\R^m)$ then $\bsigma\in\AYLift(\bB^d\times\R^m)$.
\end{theorem}

\begin{proof}
Let $\Hbf:=\{\varphi_k\otimes h_k\}_{k\in\mbN}\subset\mbfE(\bB^d\times\R^m)$ be a countable collection of tensor products determining $\mathbf{LY}$-convergence as discussed in Remark~\ref{remtensorprods} and let $x_0\in\Dcal_u$ be such that
\begin{enumerate}[(I)]	
	\item\label{regtymcond1} the family $u^r$ as defined in Theorem~\ref{thmbvblowup} converges strongly to $u^0$ in $\BV_\#(\Omega;\R^m)$ as $r\to 0$;
	\item\label{regtymcond2} it holds that
		\[
	\lim_{r\to 0}\frac{\pi_\#\lambda_{\nu}^s(\overline{B(x_0,r)})}{r^d}<\infty,
	\]
	where $\lambda_\nu^s:=\lambda_\nu-\frac{\dd\lambda_\nu}{\dd\iota_\nu}\iota_\nu$ is the singular part of $\lambda_\nu$ with respect to $\iota_\nu$;
	\item\label{regtymcond3} for each $h_k\in\Hbf$, $(x_0,u(x_0))$ is a cylindrical $\iota_\nu$-Lebesgue point for the function
	\[
	(x,y)\mapsto\ip{h_k}{\nu_{x,y}}+\frac{\dd\lambda_\nu}{\dd\iota_{\nu}}(x,y)\ip{h_k^\infty}{\nu^\infty_{x,y}}
	\]
	in the sense of Theorem~\ref{thmgeneralisedbesicovitch};
	\item\label{regtymcond4} for each $K\in\mbN$,
	\begin{align*}
	\lim_{r\to 0}\frac{\lambda_{{\nu}}^s(\overline{B(x_0,r)\times B((u)_{x_0,r},{Kr})})}{r^d}&=\lim_{r\to 0}\frac{\lambda_{{\nu}}^s(\overline{B(x_0,r)\times B((u)_{x_0,r},{Kr})})}{\omega_d^{-1}\pi_\#\iota_\nu(\overline{B(x_0,r)})}\\
	&=\omega_d\frac{\dd\lambda^s}{\dd\iota_\nu}(x_0,u(x_0))\\
&=0.
	\end{align*}
\end{enumerate}
That~\eqref{regtymcond1} and~\eqref{regtymcond2} are satisfied for $\lL$-almost every $x_0\in\Omega$ follows from Theorem~\ref{thmbvblowup} and the Besicovitch Derivation Theorem. Note that, although $\lambda^s\perp\iota_\nu$ and $\iota_\nu=\gr_\#^u(\lL\restrict\Omega)$, it need not be the case that $\pi_\#\lambda^s\perp\lL\restrict\Omega$ and so we can only guarantee that the limit appearing in~\eqref{regtymcond2} is finite $\lL$-almost everywhere, rather than equal to $0$. For fixed $k,K\in\mbN$, \eqref{regtymcond3} and~\eqref{regtymcond4} are satisfied $\lL$-almost everywhere by Theorem~\ref{thmgeneralisedbesicovitch}.

For $x_0\in\Dcal_u$, it holds that $c_r=1$ for all $r>0$ and~\eqref{eqrescaledym} therefore reads as
\begin{align}
\begin{split}\label{eqregtymidentify}
&\ddprb{\varphi_k\otimes h_k,\bsigma_r}=\frac{1}{r^d}\ddprb{\varphi_k\left(\frac{\frarg-x_0}{r},\frac{\frarg-(u)_{x_0,r}}{r}\right)\otimes h_k,\bnu}\\
&\qquad=\frac{1}{r^d}\int\varphi_k\left(\frac{x-x_0}{r},\frac{y-(u)_{x_0,r}}{r}\right)\cdot\left\{\ip{h_k}{\nu_{x,y}}+\frac{\dd\lambda_{\nu}}{\dd\iota_{\nu}}(x,y)\ip{h_k^\infty}{\nu^\infty_{x,y}}\right\}\;\dd\iota_\nu(x,y)\\
&\qquad\qquad+\frac{1}{r^d}\int\varphi_k\left(\frac{x-x_0}{r},\frac{y-(u)_{x_0,r}}{r}\right)\ip{h_k^\infty}{\nu^\infty_{x,y}}\;\dd\lambda^s_\nu(x,y).
\end{split}
\end{align}
Since 
\[
\left(\frac{1}{r^d}(T^{(x_0,r)}_{\gamma\asc{u}})_\#\iota_\nu\right)\restrict(\bB^d\times\R^m)=\frac{1}{r^d}(T^{(x_0,r)}_{\gamma\asc{u}})_\#\gr_\#^u(\lL\restrict B(x_0,r))=\gr^{u^r}_\#(\lL\restrict\bB^d)
\]
and the convergence $u^r\to u^0$ in $\Lp^1(\bB^d;\R^m)$ implies that $\gr^{u^r}_\#(\lL\restrict\bB^d)\to \gr^{u^0}_\#(\lL\restrict\bB^d)$ strictly in $\mbfM(\overline{\bB^d}\times\R^m)$, we have that
\[
\frac{1}{r^d}(T^{(x_0,r)}_{\gamma\asc{u}})_\#\iota_\nu\to\gr_\#^{u^0}(\lL)\quad\text{ strictly in }\mbfM(\overline{\bB^d}\times\R^m).
\]
Using condition~\eqref{regtymcond3} and applying Lemma~\ref{lemgraphicallebesgueconverg}, we therefore see that for every $k\in\mbN$ the family $(\mu_r)_{r>0}\subset\mbfM(\overline{\bB^d}\times\R^m)$ defined by
\[
\mu_r:=\frac{1}{r^d}(T^{(x_0,r)}_{\gamma\asc{u}})_\#\left\{\ip{h_k}{\nu_{x,y}}+\frac{\dd\lambda_{\nu}}{\dd\iota_{\nu}}(x,y)\ip{h_k^\infty}{\nu^\infty_{x,y}}\right\}\iota_\nu
\]
converges strictly in $\mbfM(\overline{\bB^d}\times\R^m)$ to the limit
\[
\left\{\ip{h_k}{\nu_{x_0,u(x_0)}}+\frac{\dd\lambda_{\nu}}{\dd\iota_{\nu}}(x_0,u(x_0))\ip{h_k^\infty}{\nu^\infty_{x_0,u(x_0)}}\right\}\gr_\#^{u^0}(\lL\restrict\bB^d).
\]
It then follows immediately that	
\begin{align}
&\lim_{r\to 0}\frac{1}{r^d}\int\varphi_k\left(\frac{x-x_0}{r},\frac{y-(u)_{x_0,r}}{r}\right)\left\{\ip{h_k}{\nu_{x,y}}+\frac{\dd\lambda_{\nu}}{\dd\iota_{\nu}}(x,y)\ip{h_k^\infty}{\nu^\infty_{x,y}}\right\}\;\dd\iota_\nu(x,y)\nonumber\\
&\qquad = \lim_{r\to 0}\int\varphi_k\left(z,w\right)\;\dd \mu_r(z,w)\label{eqregtymidentify2}\\
&\qquad =\int\varphi_k\left(z,w\right)\left\{\ip{h_k}{\nu_{x_0,u(x_0)}}+\frac{\dd\lambda_{\nu}}{\dd\iota_{\nu}}(x_0,u(x_0))\ip{h_k^\infty}{\nu^\infty_{x_0,u(x_0)}}\right\}\;\dd \left[\gr^{u^0}_\#(\lL\restrict\bB^d)\right](z,w).\nonumber
\end{align}
Taking $\varphi_k=\mathbbm{1}$, $h_k=|\frarg|$ (note that we can always enlarge $\Hbf$ by a countable number of functions), this gives
\begin{align}
\begin{split}\label{eqregpart}
&\lim_{r\to 0}\frac{1}{r^d}\int\left\{|\nu_{x,y}|(\R^{m\times d})+\frac{\dd\lambda_{\nu}}{\dd\iota_{\nu}}(x,y)\right\}\;\dd\iota_\nu(x,y)\\
&\qquad \leq \gr_\#^{u^0}(\lL)(\bB^d\times\R^m)\left\{|\nu_{x_0,u(x_0)}|(\R^{m\times d})+\frac{\dd\lambda_{\nu}}{\dd\iota_{\nu}}(x_0,u(x_0))\right\}\\
&\qquad <\infty.
\end{split}
\end{align}
On the other hand, by condition~\eqref{regtymcond2} for $\mathbbm{1}\otimes|\frarg|\in\Hbf$, we observe
\[	
\lim_{r\to 0}\frac{1}{r^d}\lambda^s_\nu(\overline{B(x_0,r)}\times\sigma\R^m)\leq\lim_{r\to 0}\frac{\pi_\#\lambda_\nu^s(\overline{B(x_0,r)})}{r^d}<\infty.
\]
Combining this with~\eqref{eqregpart}, we see that $\lim_{r}\norm{\sigma_r}_\mbfY<\infty$. The sequential weak* closedness of $\YLift(\bB^d\times\R^m)$ in $\mbfY(\Omega\times\R^m;\R^{m\times d})$, see Lemma~\ref{lemliftymsclosed}, together with the Young Measure Compactness Theorem~\ref{thmymsseqcompact} then allows us to obtain a subsequence $r_n\downarrow 0$ such that $\bsigma_{r_n}\toY\bsigma$ for some $\bsigma\in\YLift(\bB^d\times\R^m)$ with $\llbracket\bsigma\rrbracket=\lim_{n}u^{r_n}=u^0$. By assuming that $\Hbf$ contains a countable set of functions $\{\varphi_k\otimes h_k\}$ with $h_k^\infty=0$ which are dense in $\C_0(\bB^d\times\R^m)$, we see that~\eqref{eqregtymidentify} in conjunction with with~\eqref{eqregtymidentify2} implies that $\sigma_{z,w}=\nu_{x_0,u(x_0)}$ for $\iota_\sigma$-almost every $(z,w)\in\bB^d\times\R^m$.

Next, we claim that $\lambda_\sigma^s:=\lambda_\sigma-\frac{\dd\lambda_\sigma}{\dd\iota_\sigma}\iota_\sigma$ satisfies $\supp\lambda_\sigma^s\subset\overline{\Omega}\times\infty\pd\bB^m$. By considering the integrand $\mathbbm{1}\otimes|\frarg|\in\mbfE(\bB^d\times\R^m)$, we see that the convergence $\ddpr{\mathbbm{1}\otimes|\frarg|,\bsigma_{r_n}}\to\ddprb{\mathbbm{1}\otimes|\frarg|,\bsigma}$ implies
\[
\frac{1}{r_n^d}(T^{(x_0,r_n)}_{\gamma\asc{u}})_\#\left(|\nu_{x,y}|(\R^{m\times d})\iota_\nu+\lambda_\nu\right)\wsc|\sigma_{z,w}|(\R^{m\times d})\iota_\sigma+\lambda_\sigma
\]
as $n\to\infty$ in $\mbfM^+(\overline{\Omega}\times\sigma\R^m)$. The computation~\eqref{eqregtymidentify2} has already shown that
\[
\frac{1}{r_n^d}(T^{(x_0,r_n)}_{\gamma\asc{u}})_\#\left\{|\nu_{x,y}|(\R^{m\times d})+\frac{\dd\lambda_\nu}{\dd\iota_\nu}(x,y)\right\}\iota_\nu\to\left\{|\nu_{x_0,u(x_0)}|(\R^{m\times d})+\frac{\dd\lambda_\nu}{\dd\iota_\nu}(x_0,u(x_0))\right\}\iota_\sigma
 \]
strictly in $\mbfM^+(\overline{\bB^d}\times\R^m)$ (and hence strictly in $\mbfM^+(\overline{\bB^d}\times\sigma\R^m)$). Since $\sigma_{z,w}=\nu_{x_0,u(x_0)}$ almost everywhere, we see that
\[
\lambda_\sigma =\frac{\dd\lambda_\nu}{\dd\iota_\nu}(x_0,u(x_0))\iota_\sigma+\wstarlim_{n\to\infty}\frac{1}{r_n^d}(T^{(x_0,r)}_{\gamma\asc{u}})_\#\lambda^s_\nu.
\]
By the lower semicontinuity of the total variation,
\[
\left(\lambda_\sigma-\frac{\dd\lambda_\nu}{\dd\iota_\nu}(x_0,u(x_0))\iota_\sigma\right)(\overline{\bB^d}\times K\bB^m)\leq\liminf_{n\to\infty}\frac{1}{r_n^d}(T^{(x_0,r_n)}_{\gamma\asc{u}})_\#\lambda^s_\nu(\overline{\bB^d}\times K\bB^m).
\]
However, condition~\eqref{regtymcond4} implies
\[
\frac{1}{r_n^d}(T^{(x_0,r_n)}_{\gamma\asc{u}})_\#\lambda^s_\nu(\overline{\bB^d\times K\bB^m})=\frac{\lambda_{{\nu}}^s(\overline{B(x_0,r_n)\times B((u)_{x_0,r_n},{Kr_n})})}{r_n^d}\to 0\text{ as }n\to\infty,
\]
and so $(\lambda^s_\sigma-\frac{\dd\lambda_\nu}{\dd\iota_\nu}(x_0,u(x_0))\iota_\sigma)(\overline{\bB^d}\times K\bB^m)=0$ for each $K\in\mbN$. Letting $K\to\infty$, we see that $\frac{\dd\lambda_\sigma}{\dd\iota_\sigma}\equiv\frac{\dd\lambda_\nu}{\dd\iota_\nu}(x_0,u(x_0))$ and $\lambda_\sigma^s$ must be concentrated on $\overline{\bB^d}\times\infty\pd\bB^m$.

Applying Theorem~\ref{thm:restrictionsofyms}, we can therefore find a (not relabelled) Young measure $\bsigma=\bsigma\restrict(\overline{\bB^d}\times\R^m)\in\YLift(\bB^d\times\R^m)$ satisfying 
\begin{equation*}
\ddprb{\varphi_k\otimes h_k,\bsigma}=\int\varphi_k(z,\nabla u(x_0)z)\left\{\ip{h_k}{\nu_{x_0,u(x_0)}}+\frac{\dd\lambda_{\nu}}{\dd\iota_{\nu}}(x_0,u(x_0))\ip{h_k^\infty}{\nu^\infty_{x_0,u(x_0)}}\right\}\;\dd z
\end{equation*}
for every $\varphi_k\otimes h_k\in\Hbf$. Making use of the density property of $\Hbf$ in $\mbfE(\bB^d\times\R^m)$, we obtain the required result. That $\bsigma$ inherits membership of $\AYLift(\bB^d\times\R^m)$ from $\bnu$ follows from the fact that each $\bsigma_r$ is a member of $\AYLift(\bB^d\times\R^m)$, that $\AYLift(\bB^d\times\R^m)$ is closed under weak* convergence in $\mbfY(\bB^d\times\R^m;\R^{m\times d})$, and that Theorem~\ref{thm:restrictionsofyms} preserves membership of $\AYLift(\bB^d;\R^m)$.
\end{proof}

\begin{theorem}\label{thmcantortangentym}
Let $\bnu\in\YLift(\Omega\times\R^m)$ and define $u:=\llbracket\bnu\rrbracket\in\BV_\#(\Omega;\R^m)$. Then for $|D^cu|$-almost every $x_0\in\Omega$, there exists a \textbf{Cantor tangent Young measure} $\bsigma\in\YLift(\bB^d\times\R^m)$ such that
\begin{enumerate}[(i)]
	\item $[\bsigma]=\gamma\asc{u^0}\in\BV_\#(\bB^d;\R^m)$ for some $u^0$ of the form described in Theorem~\ref{thmbvblowup};
	\item $\sigma_{z,w}=\delta_0$ for $\iota_\sigma$-almost every $(z,w)\in\bB^d\times\R^m$;
	\item $\sigma^\infty_{z,w}=\nu^\infty_{x_0,u(x_0)}$ for $\lambda_\sigma$-almost every $(z,w)\in\bB^d\times\R^m$;
	\item $\lambda_\sigma=\frac{\dd\lambda_\nu}{\dd|\gamma\asc{u}|}(x_0,u(x_0))|\gamma\asc{u^0}|$.
\end{enumerate}
In addition, if $\bnu\in\AYLift(\Omega\times\R^m)$ then $\bsigma\in\AYLift(\bB^d\times\R^m)$.
\end{theorem}

\begin{proof}
As before, Let $\Hbf:=\{\varphi_k\otimes h_k\}_{k\in\mbN}\subset\mbfE(\bB^d\times\R^m)$ be a countable collection of tensor products determining $\mathbf{LY}$-convergence as discussed in Remark~\ref{remtensorprods} and let $x_0\in\Ccal_u$ be such that
\begin{enumerate}[(I)]
\item\label{cantortymcondstrict} there exists a sequence $r_n\downarrow 0$ for which $u^{r_n}$ (as defined in Theorem~\ref{thmbvblowup}) converges strictly in $\BV_\#(\bB^d;\R^m)$ to a limit $u^0$, $\gamma\asc{u^{r_n}}\to\gamma\asc{u^0}$ strictly, and 
\[
\lim_{n\to\infty} r_n^{1-d}{|Du|(B(x_0,r_n))}=0;
\]
\item\label{cantortymcondlebesguesing} for each $k\in\mbN$,
\[
\lim_{r\to 0}\frac{1}{|Du|(B(x_0,{r}))}\Bigg|\int_{B(x_0,{r})}\ip{h_k}{\nu_{x,u(x)}}+\frac{\dd\lambda_\nu}{\dd\iota_\nu}(x,u(x))\langle h^\infty_k,\nu^\infty_{x,u(x)}\rangle\;\dd x\Bigg|=0;
\]
\item \label{cantortymcondbesicovitch} it holds that
\[
\lim_{r\to 0}\frac{\pi_\#\lambda_\nu(\overline{B(x_0,r)})}{|Du|(B(x_0,r))}<\infty, \qquad\lim_{r\to 0}\frac{|D^ju|(\overline{B(x_0,r)})}{|Du|({B(x_0,r)})}=0;
\]
\item\label{cantortymcondlebesgue} for each $k\in\mbN$ and every $\varphi_k\otimes h_k\in\Hbf$, $(x_0,u(x_0))$ is a $\gamma\asc{u}\restrict((\Omega\setminus\Jcal_u)\times\R^m)$ cylindrical Lebesgue point for the function
\[
(x,y)\mapsto \frac{\dd\lambda_\nu}{\dd |\gamma\asc{u}|}(x,y)\ip{h^\infty_k}{\nu^\infty_{x,y}}
\]
in the sense of Theorem~\ref{thmgeneralisedbesicovitch};
\item\label{cantortymcondsing} for each $K\in\mbN$,
\begin{align*}
&\lim_{n\to \infty}\frac{\lambda_\nu^{\mathrm{gs}}(\overline{B(x_0,r_n)\times B((u)_{x_0,r_n},Kr_n c_{r_n}^{-1})})}{|{D}u|(B(x_0,r_n))}\\
&\qquad =\lim_{n\to\infty}\frac{\lambda_\nu^{\mathrm{gs}}(\overline{B(x_0,r_n)\times B((u)_{x_0,r_n},Kr_n c_{r_n}^{-1})})}{\pi_\#|\gamma\asc{u}|(\overline{B(x_0,r_n)})}\\
&\qquad=\frac{\dd\lambda_\nu^{\mathrm{gs}}}{\dd |\gamma\asc{u}|}(x_0,u(x_0))=0,
\end{align*}
where $\lambda_\nu^{\mathrm{gs}}:=\lambda_\nu-\frac{\dd\lambda_\nu}{\dd|\gamma\asc{u}|}|\gamma\asc{u}|\restrict((\Omega\setminus\Jcal_u)\times\R^m)$ is the graph-singular part of $\lambda_\nu$.
\end{enumerate}

By Proposition~3.92 in~\cite{AmFuPa00FBVF}, $|D^cu|$-almost every $x_0\in\Omega$ admits a sequence $s_n\downarrow 0$ such that $\lim_n s_n^{1-d}|Du|(B(x_0,s_n))=0$. By Theorem~\ref{thmdiffusetangentlifting}, we can (upon passing to a non-relabelled subsequence) find $\tau\in(0,1)$ such that $(u^{\tau s_n})_{n}$ and $(\gamma\asc{u^{\tau s_n}})_n$ converge strictly in $\BV(\bB^d;\R^m)$ and $\ALift(\bB^d\times\R^m)$ to limits $u^0$ and $\gamma\asc{u^0}$, respectively. Since
\[
\lim_{n\to\infty}(\tau s_n)^{1-d}|Du|(B(x_0,\tau s_n))\leq\tau^{1-d}\lim_{n\to\infty}s_n^{1-d}|Du|(B(x_0,s_n))=0,
\]
we can take $r_n:=\tau s_n$ to see that condition~\eqref{cantortymcondstrict} holds for $|D^cu|$-almost every $x_0\in\Ccal_u$. Condition~\eqref{cantortymcondlebesguesing} follows from the Besicovitch Derivation Theorem. For $|Du|$-almost every $x_0\in\Ccal_u$,
\[
\lim_{r\to 0}\frac{|Du|(B(x_0,r))}{|D^cu|(B(x_0,r))}=\lim_{r\to 0}\frac{|Du|(\overline{B(x_0,r)})}{|D^cu|(\overline{B(x_0,r)})}=1.
\]
Since $|D^cu|(\pd B(x_0,r))=0$, this implies that $|D^cu|(B(x_0,r))=|D^cu|(\overline{B(x_0,r)})$ for every $r>0$ and $x_0$. Thus, 
\[
\lim_{r\downarrow 0}\frac{|Du|(B(x_0,r))}{|Du|(\overline{B(x_0,r)})}=1\quad\text{ for $|Du|$-almost every $x_0\in\Ccal_u$.}
\]
We see then that \eqref{cantortymcondbesicovitch} also follows from the Besicovitch Derivation Theorem (recall that $|D^ju|$ is concentrated on $\Jcal_u$ and that $\Ccal_u\cap\Jcal_u=\emptyset$). conditions~\eqref{cantortymcondlebesgue} and~\eqref{cantortymcondsing} follow from Theorem~\ref{thmgeneralisedbesicovitch} combined with the fact that $\gamma\asc{u}\restrict((\Omega\setminus\Jcal_u)\times\R^m)=\gr_\#^u(\nabla u\lL+D^cu)$ is a $u$-graphical measure, that 
\[
\lim_{r\downarrow 0}\frac{|\nabla u\lL+D^cu|(B(x_0,r))}{|Du|(B(x_0,r))}=1\quad\text{ for $|Du|$-almost every $x_0\in\Omega\setminus\Jcal_u$,}
\]
that $(u)_{x_0,r}\to u(x_0)$ for $|D^cu|$-almost every $x_0\in\Ccal_u$, and that (by virtue of condition~\eqref{cantortymcondstrict}) $Kr_n c_{r_n}^{-1}=Kr_n^{1-d}|Du|(B(x_0,r_n))\to 0$ as $n\to\infty$.

Relabelling $\bsigma_n:=\bsigma_{r_n}$ and $c_n:=c_{r_n}=r_n^d/|Du|(B(x_0,r_n))$, from~\eqref{eqrescaledym} we get for positively one-homogeneous $h_k$,
\begin{equation}\label{eqcantorymeq}
\ddprb{\varphi_k\otimes h_k,\bsigma_n}=\frac{c_n}{r_n^d}\ddprB{\varphi_k\left(\frac{\frarg-x_0}{r},c_n\left(\frac{\frarg-(u)_{x_0,r_n}}{r_n}\right)\right)\otimes h_k,\bnu}.
\end{equation}
Since we can always assume that $\mathbbm{1}\otimes|\frarg|\in\Hbf$, conditions~\eqref{cantortymcondlebesguesing} and~\eqref{cantortymcondbesicovitch} imply
\begin{align*}
\limsup_{n\to\infty}\norm{\bsigma_n}_\mbfY&=\limsup_{n\to\infty}\left\{\frac{1}{|Du|(B(x_0,r_n))}\int_{B(x_0,r_n)}\ip{|\frarg|}{\nu_{x,u(x)}}\;\dd x+\frac{\pi_\#\lambda_\nu(\overline{B(x_0,r_n)})}{|Du|(B(x_0,r_n))}\right\}\\
&<\infty.
\end{align*}
Passing to a (non-relabelled) subsequence in $n$, we are therefore guaranteed the existence of $\bsigma\in\YLift(\bB^d\times\R^m)$ such that $\bsigma_n\wsc\bsigma$ as $n\to\infty$.

First, we show that $\sigma_{z,w}=\delta_0$ for $\iota_\sigma$-almost every $(z,w)\in\bB^d\times\R^m$. By Corollary~\ref{lemstructureliftingym} and the fact that $\llbracket\bsigma_n\rrbracket=u^{r_n}$, we observe first that $\llbracket\bsigma\rrbracket=u^0$ and hence that $\iota_\sigma=\gr^{u^0}_\#(\lL\restrict\bB^d)$. Let $\varphi\in\C_c(\R^{m\times d};[0,1])$. Since $\mathbbm{1}\otimes\varphi(\frarg)|\frarg|\in\mbfE(\bB^d\times\R^m)$ with $\sigma(\mathbbm{1}\otimes\varphi(\frarg)|\frarg|)^\infty\equiv 0$, we must have that
\[
\ddprb{\mathbbm{1}\otimes\varphi(\frarg)|\frarg|,\bsigma}=\int_{\bB^d}\ip{\varphi(\frarg)|\frarg|}{\sigma_{z,u^0(z)}}\;\dd z.
\]
On the other hand, by~\eqref{eqrescaledym} in combination with condition~\eqref{cantortymcondlebesguesing},
\begin{align*}
\lim_{n\to\infty}\ddprb{\mathbbm{1}\otimes\varphi(\frarg)|\frarg|,\bsigma_n}&=\lim_{n\to\infty}\frac{1}{r_n^d}\int_{B(x_0,r_n)}\ip{\varphi(c_{r_n}\frarg)|c_{r_n}\frarg|}{\nu_{x,u(x)}}\;\dd x\\
&=\lim_{n\to\infty}\frac{1}{|Du|(B(x_0,r_n))}\int_{B(x_0,r_n)}\ip{\varphi(c_{r_n}\frarg)|\frarg|}{\nu_{x,u(x)}}\;\dd x\\
&\leq \lim_{n\to\infty}\frac{1}{|Du|(B(x_0,r_n))}\int_{B(x_0,r_n)}\ip{|\frarg|}{\nu_{x,u(x)}}\;\dd x\\
&=0.
\end{align*}
It follows that $\ip{\varphi(\frarg)|\frarg|}{\sigma_{z,u^0(z)}}=0$ for $\lL$-almost every $z\in\bB^d$ for every $\varphi\in\C_c(\R^{m\times d};[0,1])$, which is only possible if $\sigma_{z,u^0(z)}=\delta_0$ for $\lL$-almost every $z\in\bB^d$.

Estimating next the regular part of~\eqref{eqcantorymeq}, we see that condition~\eqref{cantortymcondlebesguesing} forces
\begin{align}
\begin{split}\label{eqregvanishes}
&\Bigg|\frac{c_n}{r_n^d}\int\varphi_k\bigg(\frac{x-x_0}{r_n},c_n\bigg(\frac{y-(u)_{x_0,r_n}}{r_n}\bigg)\bigg)\\
&\qquad\qquad\qquad\qquad\quad\cdot\left\{\ip{h_k}{\nu_{x,y}}+\frac{\dd\lambda_\nu}{\dd\iota_\nu}(x,y)\ip{h^\infty_k}{\nu^\infty_{x,y}}\right\}\;\dd\iota_\nu(x,y)\Bigg|\\
&\qquad=\Bigg|\frac{c_n}{r_n^d}\int\varphi_k\bigg(\frac{x-x_0}{r_n},c_n\bigg(\frac{u(x)-(u)_{x_0,r_n}}{r_n}\bigg)\bigg)\\
&\qquad\qquad\qquad\qquad\quad\cdot \left\{\ip{h_k}{\nu_{x,u(x)}}+\frac{\dd\lambda_\nu}{\dd\iota_\nu}(x,u(x))\langle h^\infty_k,\nu^\infty_{x,u(x)}\rangle\right\}\;\dd x\Bigg|\\
&\qquad\leq  \frac{\norm{\varphi_k}_\infty}{|Du|(B(x_0,r_n))}\int_{B(x_0,r_n)}\Bigl|\ip{h_k}{\nu_{x,u(x)}}+\frac{\dd\lambda_\nu}{\dd\iota_\nu}(x,u(x))\langle h^\infty_k,\nu^\infty_{x,u(x)}\rangle\Bigr|\;\dd x\\
&\qquad \to 0\text{ as }r_n\to 0.
\end{split}
\end{align}

Now note that, by condition~\eqref{cantortymcondbesicovitch} and the fact that $\Jcal_{u^{r}}=\{z\in\bB^d\colon x_0+rz\in\Jcal_u\}=T^{(x_0,r)}(\Jcal_u)$, it holds that
\begin{equation}\label{eqjumpvanishes}
\lim_{r\to 0}|\gamma\asc{u^{r}}|(\Jcal_{u^{r}}\times\R^m)=\lim_{r\to 0}\frac{|D^ju|(B(x_0,r))}{|Du|(B(x_0,r))}=0.
\end{equation}
Abbreviating $u^n:=u^{r_n}$, we see that~\eqref{eqjumpvanishes} combined with condition~\eqref{cantortymcondstrict} implies
\begin{equation}\label{eqstrictconverg}
(\gamma\asc{u^n}\restrict(\bB^d\setminus\Jcal_{u^n}\times\R^m))_n\text{ converges strictly to }\gamma\asc{u^0}\text{ in }\mbfM(\bB^d\times\R^m;\R^{m\times d}).
\end{equation}
Recalling the notation of Lemma~\ref{lemrestrictionsofliftings} to write
\[
\gamma\asc{u^n}\restrict(\bB^d\setminus\Jcal_{u^n}\times\R^m)=\frac{c_{n}}{r_n^d}\left(T_{\gamma\asc{u}}^{(x_0,r_n)}\right)_\#\bigl[\gamma\asc{u}\restrict((\Omega\setminus\Jcal_u)\times\R^m)\bigr],
\]
we therefore deduce from~\eqref{eqstrictconverg} in combination with Lemma~\ref{lemgraphicallebesgueconverg} and condition~\eqref{cantortymcondlebesgue} that the family $(\mu^k_n)_{n\in\mbN}\subset\mbfM(\bB^d\times\R^m)$ defined by
\[
\mu^k_n:=\frac{c_{n}}{r_n^d}\left(T_{\gamma\asc{u}}^{(x_0,r_n)}\right)_\#\left(\frac{\dd\lambda_\nu}{\dd |\gamma\asc{u}|}(x,y)\ip{h^\infty_k}{\nu^\infty_{x,y}}|\gamma\asc{u}|\restrict((\Omega\setminus\Jcal_u)\times\R^m)\right)
\]
converges strictly to the limit
\[
\left(\frac{\dd\lambda_\nu}{\dd |\gamma\asc{u}|}(x_0,u(x_0))\ip{h^\infty_k}{\nu^\infty_{x_0,u(x_0)}}\right)|\gamma\asc{u^0}|\quad\text{ for each fixed }k\in\mbN.
\]
Thus, for $\varphi\in\C(\overline{\bB^d}\times\sigma\R^m)$,
\begin{align}
&\lim_{n\to\infty}\frac{c_n}{r_n^d}\int_{(\Omega\setminus\Jcal_u)\times\R^m}\varphi\left(\frac{x-x_0}{r_n},c_n\left(\frac{y-(u)_{x_0,r_n}}{r_n}\right)\right)\frac{\dd\lambda_\nu}{\dd |\gamma\asc{u}|}(x,y)\ip{h^\infty_k}{\nu^\infty_{x,y}}\;\dd|\gamma\asc{u}|(x,y)\nonumber\\
&\qquad=\lim_{n\to\infty}\int\varphi\left(z,w\right)\;\dd\mu^k_n(z,w)\label{eqidentifylim}\\
&\qquad=\int \varphi(z,w)\frac{\dd\lambda_\nu}{\dd |\gamma\asc{u}|}(x_0,u(x_0))\langle h^\infty_k,\nu^\infty_{x_0,u(x_0)}\rangle\;\dd|\gamma\asc{u^0}|(z,w)\nonumber.
\end{align}
Taking $h_k=|\frarg|$ in~\eqref{eqidentifylim} and~\eqref{eqregvanishes} implies
\[
\lambda_\sigma=\frac{\dd\lambda_\nu}{\dd|\gamma\asc{u}|}(x_0,u(x_0))|\gamma\asc{u^0}|+\underset{n\to\infty}{\wstarlim}\;\frac{c_n}{r_n^d}\left(T_{\gamma\asc{u}}^{(x_0,r_n)}\right)_\#\lambda_\nu^{\mathrm{gs}},
\]
where the weak* limit is to be understood in $\mbfM^+(\overline{\bB^d}\times\sigma\R^m)$ and we recall $\lambda_\nu^{\mathrm{gs}}=\lambda_\nu-\frac{\dd\lambda_\nu}{\dd|\gamma\asc{u}|}|\gamma\asc{u}|\restrict((\Omega\setminus\Jcal_u)\times\R^m)$.
By condition~\eqref{cantortymcondsing}, however, we see that for $K\in\mbN$,
\begin{align*}
\lim_{n\to\infty}\frac{c_n}{r_n^d}\Bigg|\left(T_{\gamma\asc{u}}^{(x_0,r_n)}\right)_\#\lambda^{\mathrm{gs}}_\nu(\overline{\bB^d}\times K\bB^m)\Bigg|&=\lim_{n\to\infty}\frac{\lambda_\nu^{\mathrm{gs}}(\overline{B(x_0,r_n)\times B((u)_{x_0,r_n},Kr_nc_n^{-1})})}{|{D}u|(B(x_0,r))}\\
&=0.
\end{align*}
Furthermore, by the lower semicontinuity of the total variation,
\begin{align*}
&\left|\lambda_\sigma-\frac{\dd\lambda_\nu}{\dd|\gamma\asc{u}|}(x_0,u(x_0))|\gamma\asc{u^0}|\right|(\overline{\bB^d}\times K\bB^m)\\
&\qquad\leq \liminf_{n\to \infty}\frac{c_n}{r_n^d}\left(T_{\gamma\asc{u}}^{(x_0,r_n)}\right)_\#\lambda^{\mathrm{gs}}_\nu(\overline{\bB^d}\times K\bB^m)\\
&\qquad=0\quad\text{ for evey }K\in\mbN,
\end{align*}
which implies simultaneously that $\frac{\dd\lambda_\sigma}{\dd|\gamma\asc{u^0}|}\equiv\mathrm{const}\equiv\frac{\dd\lambda_\nu}{\dd|\gamma\asc{u}|}(x_0,u(x_0))$ and that $\lambda_\sigma^{\mathrm{gs}}$ is concentrated within the set $\overline{\bB^d}\times\infty\pd\bB^m$. An application of Theorem~\ref{thm:restrictionsofyms} now lets us replace $\bsigma$ by a (non-relabelled) Young measure $\bsigma=\bsigma\restrict(\overline{\bB^d}\times\R^m)$ in $\YLift(\bB^d\times\R^m)$ such that
\[
\ddprb{\varphi_k\otimes h_k,\bsigma}=\int_{\bB^d\times\R^m}\varphi_k(z,w)\frac{\dd\lambda_\nu}{\dd |\gamma\asc{u}|}(x_0,u(x_0))\langle h^\infty_k,\nu^\infty_{x_0,u(x_0)}\rangle\;\dd|\gamma\asc{u^0}|(z,w)
\]
whenever $h_k$ is positively one-homogeneous from which the statement of the theorem follows. The fact that $\bsigma\in\AYLift(\bB^d\times\R^m)$ if $\bnu\in\AYLift(\bB^d\times\R^m)$ follows for the same reasons as in Theorem~\ref{thmregtangentyms}.
\end{proof}

\begin{theorem}\label{thmjumptangentyms}
Let $\bnu\in\YLift(\Omega\times\R^m)$ with $\llbracket\bnu\rrbracket=u$. Then for $\mathcal{H}^{d-1}$-almost every $x_0\in\mathcal{J}_u$, there exists a \textbf{jump tangent Young measure} $\bsigma\in\YLift(\bB^d\times\R^m)$ with the following properties:
\begin{enumerate}[(i)]
\item $\llbracket\bsigma\rrbracket=u^0$, where $u^0$ is as defined in case~\eqref{bvblowupjump} of Theorem~\ref{thmbvblowup}.
\item $\sigma_{z,w}=\delta_0$ for $\iota_\sigma$-a.e $(z,w)\in\bB^d\times\R^m$,
\item $\sigma^\infty_{z,w}=\nu^\infty_{x_0,w+(u^\pm)}$ for $\lambda_\sigma$-a.e $(z,w)\in\bB^d\times\R^m$, where $u^\pm$ is given in Definition~\ref{defjumpblowup} and $(u^\pm)=\dashint_{\bB^d}u^\pm(z)\;\dd z=\frac{1}{2}(u^+(x_0)+u^-(x_0))$,
\item $\lambda_\sigma\in\mbfM^+(\Omega\times\R^m)$ is defined via
\[
\int_{\bB^d\times\R^m}\varphi(z,w)\;\dd\lambda_\sigma(z,w)=\frac{\dd\pi_\#\lambda_\nu}{\dd\Hcal^{d-1}\restrict\Jcal_u}(x_0)\int_{n_u(x_0)^\perp\cap\bB^d}\int_{\R^m}\varphi(z,w-(u^\pm))\;\dd\rho_{x_0}(w)\;\dd\Hcal^{d-1}(z)
\]
for every $\varphi\in\C_b(\Omega\times\R^m)$, where $\rho\colon\overline{\Omega}\to\sigma\R^m$ is defined for $\pi_\#\lambda_\nu$-almost every $x_0\in\overline{\Omega}$ by the disintegration $\lambda_\nu=\pi_\#\lambda_\nu\otimes\rho$.
\end{enumerate} 

Moreover, if $\bnu\in\AYLift(\Omega\times\R^m)$, then $\bsigma\in\AYLift(\bB^d\times\R^m)$.
\end{theorem}

\begin{proof}
Let $\Gbf\subset\C(\overline{\bB^d})\times\C(\sigma\R^m)\times\C(\overline{\bB^{m\times d}})$ be a countable collection of tensor products $\{\varphi_i\otimes\psi_j\otimes h_k\}_{(i,j,k)\in\mbN^3}$ whose span is dense in $\C(\overline{\bB^d}\times\sigma\R^m\times\overline{\bB^{m\times d}})$, and which is such that $\{\varphi_i\}_{i\in\mbN}$, $\{\psi_j\}_{j\in\mbN}$ are dense in $\C(\overline{\bB^d})$ and $\C(\sigma\R^m)$, respectively, and $\{h_k\}_{k\in\mbN}$ contains a countable collection of positively one-homogeneous functions whose restriction to $\pd\bB^{m\times d}$ is dense in $\C(\pd\bB^{m\times d})$.

Let $x_0\in\Jcal_u$ be such that
\begin{enumerate}[(I)]
\item\label{jumptymcondlebesgue} $\displaystyle\frac{\dd\pi_\#\lambda_\nu}{\dd\mathcal{H}^{d-1}\restrictsmall\mathcal{J}_u}\mathcal{H}^{d-1}\restrict\mathcal{J}_u$ admits an approximate tangent plane at $x_0$;
\item\label{jumptymcondlebesguesing} for each $k\in\mbN$, 
\[
\lim_{r\to 0}\frac{1}{r^{d-1}}\int_{B(x_0,r)}\bigl|\ip{h_k}{\nu_{x,u(x)}}\bigr|\;\dd x=0;
\]
\item\label{jumptymcondlebesguerho} for each $(j,k)\in\mbN^2$, $x_0$ is a $\frac{\dd\pi_\#\lambda_\nu}{\dd\mathcal{H}^{d-1}\restrictsmall\mathcal{J}_u}\mathcal{H}^{d-1}\restrictsmall\mathcal{J}_u$-Lebesgue point for the function
\[
x\mapsto\int_{\sigma\R^m}\psi_j(y)\ip{h_k}{\nu^\infty_{x,y}}\;\dd\rho_{x}(y);
\]
\item\label{jumptymcondsing} it holds that
\[
\lim_{r\to 0}\frac{\eta^*(\overline{B(x_0,r)})}{r^{d-1}}=0,
\]
where $\eta^*:=\pi_\#\lambda_\nu-\frac{\dd\pi_\#\lambda_\nu}{\dd\mathcal{H}^{d-1}\restrictsmall\mathcal{J}_u}\mathcal{H}^{d-1}\restrictsmall\mathcal{J}_u$ is the singular part of $\pi_\#\lambda_\nu$ with respect to $\mathcal{H}^{d-1}\restrict\mathcal{J}_u$.
\end{enumerate}
That these conditions can be satisfied simultaneously for $\Hcal^{d-1}$-almost every $x_0\in\Jcal_u$ follows from the existence of tangent planes at $\Hcal^{d-1}$-almost every point for rectifiable measures, the Besicovitch Differentiation Theorem, and the Lebesgue Differentiation Theorem.

As in the proof of Theorem~\ref{thmcantortangentym}, conditions~\eqref{jumptymcondlebesgue},~\eqref{jumptymcondlebesguesing} and~\eqref{jumptymcondsing} imply that
\begin{align*}
\limsup_{n\to\infty}\norm{\bsigma_r}_\mbfY&=\limsup_{r\to 0}\left\{\frac{1}{r^{d-1}}\int_{B(x_0,r)}\ip{|\frarg|}{\nu_{x,u(x)}}\;\dd x+\frac{\pi_\#\lambda_\nu(\overline{B(x_0,r)})}{r^{d-1}}\right\}\\
&<\infty,
\end{align*}
from which we can deduce that the family $(\bsigma_r)_{r>0}$ is $\norm{\frarg}_\mbfY$-bounded.

For exactly the same reasons as in the proof of Theorem~\ref{thmcantortangentym}, we obtain that $\sigma_{z,w}=\delta_0$ for $\iota_\sigma$, where $\bsigma$ is any limit point of $(\bsigma_{r_n})_{n}$ for any sequence $r_n\downarrow 0$.

For $x_0\in\Jcal_u$, $c_r=r$ for all $r>0$ and so, for positively one-homogeneous $h_k$,~\eqref{eqrescaledym} reads
\begin{align*}
\ddprb{\varphi_i\otimes\psi_j\otimes h_k,\bsigma_r}&=\frac{1}{r^{d-1}}\ddprB{\varphi_i\left(\frac{\frarg-x_0}{r}\right)\otimes\psi_j\left(\frarg-(u)_{x_0,r}\right)\otimes h_k,\bnu}\\
&=\frac{1}{r^{d-1}}\int\varphi_i\left(\frac{x-x_0}{r}\right)\psi_j\left(y-(u)_{x_0,r}\right)\ip{h_k}{\nu_{x,y}}\;\dd \iota_\nu(x,y)\\
&\qquad+\frac{1}{r^{d-1}}\int\varphi_i\left(\frac{x-x_0}{r}\right)\psi_j\left(y-(u)_{x_0,r}\right)\ip{h_k}{\nu^\infty_{x,y}}\;\dd\lambda_\nu(x,y).
\end{align*}
First, we observe that condition~\eqref{jumptymcondlebesguesing} implies
\begin{align}
\begin{split}\label{eqjumptymregpart}
&\lim_{r\to 0}\Bigg|\frac{1}{r^{d-1}}\int\varphi_i\left(\frac{x-x_0}{r}\right)\psi_j\left(y-(u)_{x_0,r}\right)\ip{h_k}{\nu_{x,y}}\;\dd \iota_\nu(x,y)\Bigg|\\
&\qquad \leq\lim_{r\to 0}\frac{\norm{\varphi_i\otimes\psi_j}_\infty}{r^{d-1}}\int_{B(x_0,r)}\bigl|\ip{h_k}{\nu_{x,u(x)}}\bigr|\;\dd x\\
&\qquad =0,
\end{split}
\end{align}
from which we conclude $\sigma_{z,w}=\delta_0$ for $\iota_\sigma$-almost every $(z,w)\in\bB^d\times\R^m$. 

Now, combining conditions~\eqref{jumptymcondlebesgue} and~\eqref{jumptymcondsing}, we see that 
\[
r^{1-d}T_\#^{(x_0,r)}\pi_\#\lambda_\nu=r^{1-d}T_\#^{(x_0,r)}\frac{\dd\pi_\#\lambda_\nu}{\dd\mathcal{H}^{d-1}\restrictsmall\mathcal{J}_u}\mathcal{H}^{d-1}\restrict\mathcal{J}_u+r^{1-d}T_\#^{(x_0,r)}\eta^*
\]
converges weakly* in $\mbfM^+(\overline{\bB^d})$ as $r \todown 0$ to 
\[
\mu:=\frac{\dd\pi_\#\lambda_\nu}{\dd\Hcal^{d-1}\restrict\Jcal_u}(x_0)\Hcal^{d-1}\restrict\left(n_u(x_0)^\perp\cap\bB^d\right).
\]
By condition~\eqref{jumptymcondlebesguerho}, we can deduce that the family $(\beta_r^{j,k})_{r>0}\subset\mbfM^+(\overline{\bB^d})$ defined by
\[
\beta_r^{j,k}:=r^{1-d}T_\#^{(x_0,r)}\left(\left[\int_{\sigma\R^m}\psi_j(y)\ip{h_k}{\nu^\infty_{x,y}}\;\dd\rho(y)\right]\pi_\#\lambda_\nu\right)
\]
converges weakly* in $\mbfM^+(\overline{\bB^d})$ to the measure
\[
\beta_0^{j,k}:=\left[\int_{\sigma\R^m}\psi_j(y)\ip{h_k}{\nu^\infty_{x_0,y}}\;\dd\rho_{x_0}(y)\right]\mu.
\]
Letting $\theta\colon\overline{\bB^d}\times\sigma\R^m\to\mbfM^1(\pd\bB^{m\times d})$ denote the weakly* $(\mu\otimes\rho_{x_0})$-measurable parametrised measure defined by $\theta_{z,w}:=\nu^\infty_{x_0,w}$ for $(z,w)\in\bB^d\times\R^m$, it then follows that
\begin{align*}
\lim_{r\to 0}\ip{\varphi_i\otimes\psi_j}{r^{1-d}T_\#^{(x_0,r)}\bigl[\ip{h_k}{\nu^\infty}\lambda_\nu\bigr]}=&\lim_{r\to 0}\ip{\varphi_i}{\beta_r^{j,k}}\\
=&\ip{\varphi_i}{\beta_0^{j,k}}\\
=&\ip{\varphi_i\otimes\psi_j}{\ip{h_k}{\theta}\mu\otimes\rho_{x_0}}.
\end{align*}
for each $(i,j,k)\in\mbN^3$. Since the span of $\{\varphi_i\otimes\psi_j\}_{(i,j)\in\mbN^2}$ is a countable dense subset of $\C(\overline{\bB^d}\times\sigma\R^m)$, we deduce that
\begin{equation}\label{eqjumptymequality}
r^{1-d}T_\#^{(x_0,r)}\bigl[\ip{h_k}{\nu^\infty}\lambda_\nu\bigr]\wsc\ip{h_k}{\theta}\mu\otimes\rho_{x_0}\quad\text{ in }\mbfM^+(\overline{\bB^d}\times\sigma\R^m)\text{ as }r\to 0
\end{equation}
for each $k\in\mbN$.

Noting that
\[
\lim_{r\to 0}\, (u)_{x_0,r}=\frac{u^+(x_0)+u^-(x_0)}{2}=(u^\pm),
\]
we see that
\[	
\lim_{r\to 0}\sup_{(z,w)\in\bB^d\times\R^m}|g(z,w-(u)_{x_0,r})-g(z,w-(u^\pm))|=0\quad\text{ for all }g\in\C(\overline{\bB^d}\times\sigma\R^m).
\]
We can therefore use~\eqref{eqjumptymequality} to compute
\begin{align*}
&\lim_{r\to 0}\frac{1}{r^{d-1}}\int\varphi_i\left(\frac{x-x_0}{r}\right)\psi_j\left(y-(u)_{x_0,r}\right)\ip{h_k}{\nu^\infty_{x,y}}\;\dd\lambda_\nu(x,y)\\
&\qquad =\lim_{r\to 0}\frac{1}{r^{d-1}}\int\varphi_i\left(\frac{x-x_0}{r}\right)\psi_j\left(y-(u^\pm)\right)\ip{h_k}{\nu^\infty_{x,y}}\;\dd\lambda_\nu(x,y)\\
&\qquad =\lim_{r\to 0}\int(\varphi_i\otimes\psi_j)\left(z,w-(u)^\pm\right)\;\dd\bigl[r^{1-d}T_\#^{(x_0,r)}[\ip{h_k}{\nu^\infty}\lambda_\nu]\bigr](z,w)\\
&\qquad =\int (\varphi_i\otimes\psi_j)(z,w-(u)^\pm)\;\dd\bigl[\ip{h_k}{\theta}\mu\otimes\rho_{x_0}\bigr](z,w)\\
&\qquad =\int (\varphi_i\otimes\psi_j)(z,w-(u)^\pm)\ip{h_k}{\nu_{x_0,w}^\infty}\;\dd[\mu\otimes\rho_{x_0}](z,w).
\end{align*}
Combining this with~\eqref{eqjumptymregpart}, we have therefore shown
\begin{equation}\label{eqjumptymequality2}
\lim_{r\to 0}\, \ddprb{\varphi_i\otimes\psi_j\otimes h_k,\bsigma_r}=\int (\varphi_i\otimes\psi_j)(z,w-(u)^\pm)\ip{h_k}{\nu_{x_0,w}^\infty}\;\dd\left[\mu\otimes\rho_{x_0}\right](z,w)
\end{equation}
for every $(i,j)\in\mbN^2$ and $k$ such that $h_k$ is positively one-homogeneous. 

Since $(\sigma_r)_{r>0}$ is $\norm{\frarg}_\mbfY$-bounded, we can invoke Theorem~\ref{thmymsseqcompact} to find a sequence $(\sigma_{r_n})_n$ such that $\sigma_n\wsc\sigma$ for some $\sigma\in\YLift(\bB^d\times\R^m)$ as $n\to\infty$. The density properties of the span of $\{\varphi_i\otimes\psi_j\otimes h_k\}_{(i,j,k)\in\mbN^3}$ combined with~\eqref{eqjumptymequality2} let us deduce that that
\[
\ddprb{f,\bsigma}=\int \ip{f(z,w-(u)^\pm),\frarg)}{\nu_{x_0,w}^\infty}\;\dd\left[\mu\otimes\rho_{x_0}\right](z,w)
\]
for every $f\in\mbfE(\bB^d\times\R^m)$ which is positively one-homogeneous in the final variable. Since we have already noted that we must have $\sigma_{z,w}=\delta_0$, an application of Theorem~\ref{thm:restrictionsofyms} to ensure that $\lambda_\sigma(\overline{\bB^d}\times\infty\pd\bB^m)=0$ leaves us with the desired conclusion. That $\bsigma\in\AYLift(\bB^d\times\R^m)$ if $\bnu\in\AYLift(\bB^d\times\R^m)$ follows again from the closedness of $\AYLift(\bB^d\times\R^m)$ (recall Lemma~\ref{lemliftymsclosed}) in $\YLift(\bB^d\times\R^m)$ under weak* Young measure convergence.
\end{proof}

\begin{remark}\label{remrescaledymconvergence}
Note that the proofs of Theorems~\ref{thmregtangentyms} and~\ref{thmcantortangentym} imply the following result: let $\bnu\in\YLift(\Omega\times\R^m)$ with $\llbracket \bnu\rrbracket=u$. Then, for $\lL+|D^cu|$-almost every $x\in\Dcal_u\cup\Ccal_u$, whenever $r_n\downarrow 0$ is a sequence such that
\[
(u^{r_n})_n\text{ converges strictly to $u^0$ in }\BV(\bB^d;\R^m)\text{ and }r_nc_{r_n}\to 0\text{ as $n\to\infty$}
\]
(note that these conditions hold for every sequence $r_n\downarrow 0$ if $x\in\Dcal_u$), the rescaled sequence $(\bsigma_{r_n})_{n\in\mbN}$ defined by~\eqref{eqrescaledym} converges in $\YLift(\bB^d\times\R^m)$ to a limit $\btau\in\YLift(\bB^d\times\R^m)$ satisfying 
\[
\btau\restrict(\overline{\bB^d}\times\R^m)=\bsigma\qquad\text{and}\qquad\llbracket\btau\rrbracket=u^0.
\]
Here, $\bsigma$ is a tangent Young measure of the form described by Theorem~\ref{thmregtangentyms} if $x\in\Dcal_u$, and by Theorem~\ref{thmcantortangentym} if $x\in\Ccal_u$, and $\btau\restrict(\overline{\bB^d}\times\R^m)$ is the restriction of $\btau$ to $\overline{\bB^d}\times\R^m$ introduced in Theorem~\ref{thm:restrictionsofyms}.
\end{remark}

\subsection{Jensen inequalities}\label{secjensenineq}
Let $\bnu\in\AYLift(\Omega\times\R^m)$ and $x\in\Dcal_u\cup\Ccal_u\cup\Jcal_u$ be such that $\bnu$ admits a tangent Young measure $\bsigma\in\AYLift(\Omega\times\R^m)$ at $(x,u(x))$ and let $f\in\mbfR(\Omega\times\R^m)$ be quasiconvex in the final variable. By Lemma~\ref{lemymprescribedboundary}, there exists a sequence $(u_j)_j\subset\C_\#^\infty(\bB^d;\R^m)$ such that $\gamma\asc{u_j}\toY \bsigma$ and for which ${u_j}|_{\pd\bB^d}-\llbracket\bsigma\rrbracket|_{\pd\bB^d}=c_j$ for some sequence $(c_j)_j\subset\R^m$ converging to $0$. For $x\in\Dcal_u\cup\Ccal_u$, define
\[
h:=\begin{cases}f(x,u(x),\frarg)&\text{ if }x\in\Dcal_u,\\
f^\infty(x,u(x),\frarg)&\text{ if }x\in\Ccal_u,
\end{cases}
\]
so that $\mathbbm{1}\otimes h\in\mbfE(\bB^d\times\R^m)$.

\begin{lemma}\label{lemlebesguejensen}
Let $\bnu\in\AYLift(\Omega\times\R^m)$, with $u=\llbracket\bnu\rrbracket$ and let $f\in\mbfR(\Omega\times\R^m)$ be quasiconvex in the final variable. For $\mathcal{L}^d$-almost every $x\in\Omega$, it holds that
\begin{align*}
\ip{f(x,u(x),\frarg)}{\nu_{x,u(x)}}&+\frac{\dd\lambda_\nu}{\dd\iota_{\nu}}(x,u(x))\ip{f^\infty(x,u(x),\frarg)}{\nu^\infty_{x,u(x)}}\\
&\geq f\left(x,u(x),\ip{\id}{\nu_{x,u(x)}}+\frac{\dd\lambda_\nu}{\dd\iota_{\nu}}(x,u(x))\langle\id,\nu^\infty_{x,u(x)}\rangle\right)\\
&=f(x,u(x),\nabla u(x)).
\end{align*}
\end{lemma}

\begin{proof}
By Theorem~\ref{thmregtangentyms}, there exists a regular tangent Young measure $\bsigma\in\YLift(\bB^d\times\R^m)$ with $\llbracket\bsigma\rrbracket(z)=\nabla u(x)z$ for all $z\in\bB^d$. The boundary condition satisfied by each $u_j$ then implies that ${\nabla u_j}|_{\pd\bB^d}(z)=\nabla u(x)z$ for all $z\in\pd\bB^d$ and that we can therefore write $\nabla u_j=\nabla u(x)+\nabla v_j$ for some $\BV$-bounded sequence $(v_j)_j\subset\C^\infty_0(\bB^d;\R^m)$. Thus, using the quasiconvexity of $h$,
\begin{align*}
\frac{1}{\omega_d}\ddprb{\mathbbm{1}\otimes h,\bsigma}&=\lim_{j\to\infty}\dashint_{\bB^d} h(\nabla u_j(z))\;\dd z\\
&=\lim_{j\to\infty}\dashint_{\bB^d} h(\nabla u(x)+\nabla v_j(z))\;\dd z\\
&\geq h(\nabla u(x))\\
&=f(x,u(x),\nabla u(x)).
\end{align*}
Since Theorem~\ref{thmregtangentyms} now states that
\begin{align*}
\frac{1}{\omega_d}\ddprb{\mathbbm{1}\otimes h,\bsigma}&=\dashint_{\bB^d} \left\{\ip{h}{\nu_{x,u(x)}}+\frac{\dd\lambda_\nu}{\dd\iota_\nu}(x,u(x))\langle h^\infty,\nu^\infty_{x,u(x)}\rangle\right\}\;\dd z\\
&= \ip{f(x,u(x),\frarg)}{\nu_{x,u(x)}}+\frac{\dd\lambda_\nu}{\dd\iota_\nu}(x,u(x))\langle f^\infty(x,u(x),\frarg),\nu^\infty_{x,u(x)}\rangle,
\end{align*}
we obtain the desired result.
\end{proof}

\begin{lemma}\label{lemcantorjensen}
Let $\bnu\in\YLift(\Omega\times\R^m)$ with $u=\llbracket\bnu\rrbracket$ and let $f\in\mbfR(\Omega\times\R^m)$ be positively one-homogeneous and quasiconvex in the final variable. For $|D^c u|$-almost every $x\in\Omega$, it holds that
\begin{align*}
\frac{\dd\lambda_\nu}{\dd|\gamma\asc{u}|}(x,u(x))\langle f^\infty(x,u(x),\frarg),\nu^\infty_{x,u(x)}\rangle &\geq f^\infty\left(x,u(x),\frac{\dd\lambda_\nu}{\dd|\gamma\asc{u}|}(x,u(x))\langle\id,\nu^\infty_{x,u(x)}\rangle\right)\\
&=f^\infty\left(x,u(x),\frac{\dd D u}{\dd|D u|}(x)\right).
\end{align*}
\end{lemma}

\begin{proof}
Employing Theorem~\ref{thmcantortangentym} we see that, for $|D^cu|$-almost every $x\in\Omega$, there exists a Cantor Tangent Young measure $\bsigma\in\YLift(\bB^d\times\R^m)$ with $\llbracket\bsigma\rrbracket=u^0$ for some $u^0\in\BV(\bB^d;\R^m)$ satisfying $\frac{\dd Du^0}{\dd|Du^0|}(z)=\frac{\dd Du}{\dd|Du|}(x)$ for all $z\in\bB^d$, $\sigma_{z,w}=\delta_0$ $\iota_\sigma$-almost everywhere, $\lambda_\sigma=\frac{\dd\lambda_\nu}{\dd|\gamma\asc{u}|}(x,u(x))|\gamma\asc{u^0}|$, and $\sigma^\infty_{z,w}=\nu^\infty_{x,u(x)}$ $\lambda_\sigma$-almost everywhere. Letting $(\gamma_j)_j\subset\Lift(\bB^m\times\R^m)$ be such that $\gamma_j\toY\bsigma$ and testing with the integrand $\varphi\otimes\id\in\mbfE(\bB^d\times\R^m)$ for $\varphi\in\C_0(\bB^d)$ arbitrary, we can compute
\begin{align*}
\lim_{j\to\infty}\int_{\bB^d}\varphi(z)\;\dd D\asc{\gamma_j}(z)&=\lim_{j\to\infty}\ddprb{\varphi\otimes\id,\bdelta\asc{\gamma_j}}\\
&=\ddprb{\varphi\otimes\id,\bsigma}\\
&=\int_{\bB^d}\varphi(z)\frac{\dd\lambda_\nu}{\dd|\gamma\asc{u}|}(x,u(x))\langle\id,\nu_{x,u(x)}^\infty\rangle\;\dd \pi_\#|\gamma\asc{u^0}|(z)\\
&=\int_{\bB^d}\varphi(z)\frac{\dd\lambda_\nu}{\dd|\gamma\asc{u}|}(x,u(x))\langle\id,\nu_{x,u(x)}^\infty\rangle\;\dd |D u^0|(z).
\end{align*}
Since $\varphi\in\C_0(\bB^d)$ was arbitrary and $\gamma_j\toY\bsigma$ implies $u^0=\llbracket\bsigma\rrbracket=\wstarlim_j\asc{\gamma_j}$ in $\BV(\bB^d;\R^m)$, we see that it must be the case that
\[
\frac{\dd\lambda_\nu}{\dd|\gamma\asc{u}|}(x,u(x))\langle\id,\nu_{x,u(x)}^\infty\rangle|Du^0|=Du^0\qquad\text{ in }\mbfM(\bB^d;\R^{m\times d}),
\]
and hence that
\begin{equation}\label{eqidentifybarycentre}
\langle\id,\nu_{x,u(x)}^\infty\rangle\frac{\dd\lambda_\nu}{\dd|\gamma\asc{u}|}(x,u(x))=\frac{\dd Du^0}{\dd|D u^0|}(z)=\frac{\dd Du}{\dd|Du|}(x).
\end{equation}
By Alberti's Rank One Theorem~\cite{Albe93ROPD}, $\frac{\dd Du}{\dd|Du|}(x)$ is a rank-one matrix, so we have obtained that the barycentre $\langle\id,\nu^\infty_{x,u(x)}\rangle$ of the probability measure $\nu^\infty_{x,u(x)}$ is a rank-one matrix. As the recession function of a quasiconvex function with linear growth, $h$ is a quasiconvex (and hence rank-one convex) positively one-homogeneous function. A result due to Kirchheim \& Kristensen~\cite{KirKri16ROCF} states that positively one-homogeneous and rank-one convex functions are in fact convex at all rank-one points $A\in\R^{m\times d}$. By Jensen's inequality, the homogeneity of $h$, and~\eqref{eqidentifybarycentre}, we therefore deduce
\begin{align*}
\frac{\dd\lambda_\nu}{\dd|\gamma\asc{u}|}(x,u(x))\langle h,\nu^\infty_{x,u(x)}\rangle&\geq\frac{\dd\lambda_\nu}{\dd|\gamma\asc{u}|}(x,u(x))h\left(\langle\id,\nu^\infty_{x,u(x)}\rangle\right) \\
&=h\left(\frac{\dd\lambda_\nu}{\dd|\gamma\asc{u}|}(x,u(x))\langle\id,\nu^\infty_{x,u(x)}\rangle\right)\\
&=h\left(\frac{\dd Du}{\dd|Du|}(x)\right).
\end{align*}
Recalling the definition of $h$, the desired conclusion follows.	
\end{proof}
It is interesting to note that we need only $\bnu\in\YLift(\Omega\times\R^m)$ rather than $\bnu\in\AYLift(\Omega\times\R^m)$ for Lemma~\ref{lemcantorjensen} to hold.
\begin{lemma}\label{lemjumpjensen}
Let $f\in\RBVw(\Omega\times\R^m)$, and $\bnu\in\AYLift(\Omega\times\R^m)$ with $\llbracket \bnu\rrbracket=u$ and disintegrate $\lambda_\nu=\pi_\#\lambda_\nu\otimes\rho$. For $\mathcal{H}^{d-1}$-almost every $x\in\mathcal{J}_{u}$, it holds that
\[
\frac{\dd\pi_\#\lambda_\nu}{\dd\Hcal^{d-1}\restrict\Jcal_u}(x)\int_{\R^m}\ip{f^\infty(x,y,\frarg)}{\nu^\infty_{x,y}}\;\dd\rho_x(y)\geq K_f[u](x),
\]
where $K_f[u]\colon\Jcal_u\to\R$ is defined by
\[
K_f[u](x):=\inf\left\{\frac{1}{\omega_{d-1}}\int_{\bB^d}f^\infty(x,\varphi(y),\nabla\varphi(y))\;\dd y\colon\varphi\in\Acal_u(x)\right\}
\]
and we recall that
\[
\Acal_u(x):=\left\{\varphi\in\left(\C^\infty\cap\Lp^\infty\right)(\bB^d;\R^m)\colon\;\varphi= u^\pm_{x}\text{ on }\pd\bB^d\right\}
\]
where $u^\pm$ is as given in Definition~\ref{defjumpblowup}.
\end{lemma}

\begin{proof}
Let $\bsigma\in\AYLift(\bB^d\times\R^m)$ be the jump tangent Young measure to $\bnu$ at $x$ whose existence is guaranteed for $\Hcal^{d-1}\restrict\Jcal_u$-almost every $x\in\Omega$ by Theorem~\ref{thmjumptangentyms}. Apply Lemma~\ref{lemymprescribedboundary} to obtain a sequence $(u_j)_j\subset\C_\#^\infty(\Omega;\R^m)$ with $\gamma\asc{u_j}\toY\bsigma$ and $u_j-c_j=\llbracket\bsigma\rrbracket=u^0$ (where $u^0$ is defined in Theorem~\ref{thmbvblowup}) on $\pd\bB^d$ for some sequence $c_j\to 0$. Since $\gamma\asc{u_j}\toY$ $\bsigma$, it holds that $u_j\wsc u^0$ in $\BV(\bB^d;\R^m)$ and hence that $u_j+(u^\pm)\wsc u^\pm$ in $\BV(\bB^d;\R^m)$ with $u_j-c_j+(u^\pm)\in\Acal_u(x)$ for $j$ large enough. Since $f^\infty\in\RL(\Omega\times\R^m)$ and $\lambda_\sigma(\Omega\times\R^m)=0$, Proposition~\ref{lemextendedrepresentation}~(i) lets us deduce that
\begin{align*}
\ddprb{f^\infty(x,\frarg+(u^\pm),\frarg),\bsigma}&=\lim_{j\to\infty}\int Pf^\infty(x,y-c_j+(u^\pm),P\gamma\asc{u_j})\\
&=\lim_{j\to\infty}\int_{\bB^d}f^\infty(x,u_j(z)-c_j+(u^\pm),\nabla u_j(z))\;\dd z\\
&\geq\omega_{d-1} K_f[u](x).
\end{align*}
Since Theorem~\ref{thmjumptangentyms} implies
\begin{align*}
&\ddprb{f^\infty(x,\frarg+(u^\pm),\frarg),\bsigma}\\
&\qquad =\int_{\overline{\bB^d}\times\R^m}\ip{f^\infty(x,y+(u^\pm),\frarg)}{\sigma^\infty_{z,y}}\;\dd\lambda_\sigma(z,w)\\
&\qquad =\int_{\bB^d}\frac{\dd\pi_\#\lambda_\nu}{\dd\Hcal^{d-1}\restrict\Jcal_u}(x)\int_{\R^m}\ip{f^\infty(x,w,\frarg)}{\nu^\infty_{x,w}}\;\dd \rho_x(w)\;	\dd\Hcal^{d-1}\restrict n_u(x_0)^\perp(z)\\
&\qquad =\omega_{d-1}\frac{\dd\pi_\#\lambda_\nu}{\dd\Hcal^{d-1}\restrict\Jcal_u}(x)\int_{\R^m}\ip{f^\infty(x,y,\frarg)}{\nu^\infty_{x,y}}\;\dd\rho_x(y),
\end{align*}
the conclusion follows.
\end{proof}

\begin{theorem}\label{thmwsclsc}
Let $f\in\RBVw(\Omega\times\R^m)$ be such that $f(x,y,\frarg)$ is quasiconvex for every $(x,y)\in\overline{\Omega}\times\R^m$. If $(u_j)_j\subset\BV(\Omega;\R^m)$ and $u\in\BV(\Omega;\R^m)$ are such that $u_j\wsc u$, then
\begin{align*}
\liminf_{j\to\infty}\Fcal[u_j]&\geq\int_\Omega f(x,u(x),\nabla u(x))\;\dd x+\int_\Omega f^\infty\left(x,u(x),\frac{\dd D^cu}{\dd|D^c u|}(x)\right)\;\dd|D^c u|(x)\\
&\qquad+\int_{\mathcal{J}_u}K_{f}[u](x)\;\dd\mathcal{H}^{d-1}(x).
\end{align*}
\end{theorem}

\begin{proof}
First, assume that $u\in\BV_\#(\Omega;\R^m)$. By the discussion at the start of this section resulting in the inequality~\eqref{eqymlscineq}, we have that
\begin{align*}
\liminf_{j\to\infty}\Fcal[u_j]&\geq\int_{\Omega}\ip{f(x,u(x),\frarg)}{\nu_{x,u(x)}}+\frac{\dd\lambda_\nu}{\dd\iota_\nu}(x,u(x))\langle f^\infty(x,u(x),\frarg),\nu_{x,u(x)}^\infty\rangle\;\dd x\\
&\qquad+\int_{\Omega}\frac{\dd\lambda_\nu}{\dd|\gamma\asc{u}|}(x,u(x))\langle f^\infty(x,u(x),\frarg),\nu_{x,u(x)}^\infty\rangle\;\dd|D^cu|(x)\\
&\qquad+\int_{\Jcal_u}\frac{\dd\pi_\#\lambda_\nu}{\dd\Hcal^{d-1}\restrict\Jcal_u}(x)\int_{\R^m}\ip{f^\infty(x,y,\frarg)}{\nu_{x,y}^\infty}\;\dd\rho_x(y)\;\dd\Hcal^{d-1}(x).
\end{align*}
Applying Lemmas~\ref{lemlebesguejensen},~\ref{lemcantorjensen}, and~\ref{lemjumpjensen} respectively to each of the three terms featuring above, we obtain
\begin{align*}
\liminf_{j\to\infty}\Fcal[u_j]&\geq\int_\Omega f(x,u(x),\nabla u(x))\;\dd x+\int_\Omega f^\infty\left(x,u(x),\frac{\dd D^cu}{\dd|D^c u|}(x)\right)\;\dd|D^c u|(x)\\
&\qquad+\int_{\mathcal{J}_u}K_f[u](x)\;\dd\mathcal{H}^{d-1}(x).
\end{align*}

For the general case where $(u)\neq 0$, define $f_u\in\RBVw(\Omega\times\R^m)$ by $f_u(x,y,A)=f(x,y+(u),A)$ and let $\Fcal_u\colon\BV(\Omega;\R^m)\to[0,\infty)$ be the functional given by
\[
\Fcal_u[v]:=\int Pf_u(x,y,P\gamma\asc{v-(v)}).
\]
Defining $\overline{u}:=u-(u)$ and noting that $D\overline{u}=Du$, $\Jcal_u=\Jcal_{\overline{u}}$, $K_{f_u}[\overline{u}]=K_f[u]$, we then see that
\begin{align*}
\liminf_{j\to\infty}\Fcal[u_j]&=\liminf_{j\to\infty}\Fcal_u[u_j-(u)]\\
&\geq \int_\Omega f_u(x,\overline{u}(x),\nabla \overline{u}(x))\;\dd x+\int_\Omega f_u^\infty\left(x,\overline{u}(x),\frac{\dd D^c\overline{u}}{\dd|D^c \overline{u}|}(x)\right)\;\dd|D^c \overline{u}|(x)\\
&\qquad+\int_{\mathcal{J}_{\overline{u}}}K_{f_u}[\overline{u}](x)\;\dd\mathcal{H}^{d-1}(x)\\
&=\int_\Omega f(x,{u}(x),\nabla{u}(x))\;\dd x+\int_\Omega f^\infty\left(x,{u}(x),\frac{\dd D^c{u}}{\dd|D^c{u}|}(x)\right)\;\dd|D^c {u}|(x)\\
&\qquad+\int_{\mathcal{J}_{u}}K_{f}[{u}](x)\;\dd\mathcal{H}^{d-1}(x)
\end{align*}
as required.
\end{proof}

%% file: recoveryseqs.tex
\section{Recovery sequences and relaxation}\label{chaprecoveryseqs}
This section is devoted to the construction of weak* approximate recovery sequences to show that the lower bound obtained in Theorem~\ref{thmwsclsc}  for $\Fcalrw$ is attained. The recovery sequences in this section are constructed globally without making use of the De Giorgi--Letta Theorem to localise around points $x\in\Dcal_u\cup\Ccal_u\cup\Jcal_u$. We therefore do not require the existence of a finite $(|Du|+\lL)$-absolutely continuous measure which dominates $\Fcalrw$.

Example~\ref{exbadrecoveryseq} below demonstrates the existence of a continuous integrand $f\in\RBVw(\Omega\times\R^m)$ which is convex in the final variable and whose associated weak* relaxation $\Fcalrw$ attains a global minimum at $0$, yet which does not possess any weakly* convergent minimising sequences. In the full vector-valued, $u$-dependent case we can thus only expect to find recovery sequences which converge in the strong $\Lp^1(\Omega;\R^m)$-topology. We must therefore find a way of constructing {approximate recovery sequences} $(u_j)_j\subset\C^\infty(\Omega;\R^m)$ such that $u_j\wsc u$ in $\BV(\Omega;\R^m)$ and $\lim_j|\Fcal[u_j]-\Fcalrw[u]|\leq\varepsilon$ for each $\varepsilon>0$.

\begin{example}[Recovery sequences can be badly behaved]\label{exbadrecoveryseq}
Let $u\in\BV((-1,1);\R^2)$ be given by 
\[u(x):=
\begin{cases}
\begin{pmatrix}
0\\
0\end{pmatrix}&\text{ if }x\leq 0,\\
\begin{pmatrix}
0\\
1\end{pmatrix}&\text{ if }x>0,
\end{cases}
\]
and, for $j\geq 4$, define the sequence $u_j\in\W^{1,1}(\R;\R^2)$ by
\begin{align*}
u_j(x):=&\begin{pmatrix}
j^2x\\
0
\end{pmatrix}\mathbbm{1}_{(0,1/j)}(x)+\begin{pmatrix}
j\\
j(x-1/j)
\end{pmatrix}\mathbbm{1}_{[1/j,2/j]}(x)+\begin{pmatrix}
j-j^2(x-2/j)\\
1
\end{pmatrix}\mathbbm{1}_{(2/j,3/j)}(x)\\
&\qquad+\begin{pmatrix}
0\\
1
\end{pmatrix}\mathbbm{1}_{[3/j,\infty)}(x)
\end{align*}
so that $|Du_j|(-1,1)=2j+1$ and $u_j\to u$ pointwise almost everywhere as $j\to\infty$. Let $f\in\RBVw((-1,1)\times\R^2)$ be given by
\[
f(x,y,A):=\Phi(y)|A|,\qquad\Phi\left(\begin{pmatrix}y_1\\y_2\end{pmatrix}\right):=\frac{y_2(1-y_2)}{1+|y_2|^2}\mathrm{e}^{-|y_1|},
\]
and also define $(u_j^k)_k\subset\W^{1,1}((-1,1);\R^2)$ by $u_j^k(x)=u_j(kx)$. We can see that 
\[
|D u_j^k|(-1,1)=|Du_j|(-1,1)\quad\text{and}\quad\wstarlim_{k\to\infty}u^k_j= u\quad\text{ for each fixed }j.
\]
By a change of coordinates, it is also clear that
\[
\lim_{j\to\infty}\lim_{k\to\infty}\int_{-1}^1 f(x,u_j^k(x),\nabla u_j^k(x))\;\dd x=\lim_{j\to\infty}\int_{-1}^{1}\Phi(u_j(x))|\nabla u_j(x)|\;\dd x=0,
\]
which demonstrates that $\Fcalrw[u]=0$ and that $K_f[u]\equiv 0$.

Now assume that $(v_j)_j\subset\W^{1,1}((-1,1);\R^2)$ is such that $v_j\wsc u$. We will show that
\[
\liminf_{j\to\infty}\int_{-1}^1f(x,v_j(x),\nabla v_j(x))\;\dd x>0.
\]
By the Sobolev Embedding Theorem in one dimension it holds that $\sup_j\norm{v_j}_\infty<\infty$, which implies that, for some $\delta>0$,
\begin{align*}
\lim_{j\to\infty}\int_{-1}^{1}\Phi(v_j(x))|\nabla v_j(x)|\;\dd x&\geq\lim_{j\to\infty}\mathrm{e}^{-\norm{v_j}_\infty}\int_{-1}^1\frac{(v_j(x))_2(1-(v_j(x))_2)}{1+|(v_j(x))_2|^2}|\nabla v_j(x)|\;\dd x\\
&\geq\delta\lim_{j\to\infty}\int_{-1}^1\frac{(v_j(x))_2(1-(v_j(x))_2)}{1+|(v_j(x))_2|^2}|\nabla v_j(x)|\;\dd x\\
&\geq\delta\lim_{j\to\infty}\int_{-1}^1\frac{(v_j(x))_2(1-(v_j(x))_2)}{1+|(v_j(x))_2|^2}|\nabla (v_j(x))_2|\;\dd x.
\end{align*}
Define the sequence $(w_j)_j\subset\W^{1,1}((-1,1))$ by $w_j(z)=(v_j(z))_2$ so that $w_j\wsc w:=\mathbbm{1}_{[0,1)}$ in $\BV((-1,1))$. By Lemma~\ref{lemliftingsperturb}, we have that
\[
\lim_{j\to\infty}\int_{-1}^1\frac{w_j(x)(1-w_j(x))}{1+|w_j(x)|^2}|\nabla w_j(x)|\;\dd x=\lim_{j\to\infty}\int \frac{(y+(w))(1-y-(w))}{1+(y+(w))^2}\;\dd|\gamma\asc{w_j-(w_j)}|(x,y).
\]

Since $(\gamma\asc{w_j-(w_j)})_j$ is a norm bounded sequence in $\Lift((-1,1)\times\R)$, Lemma~\ref{lemliftingcompactness} combined with Corollary~\ref{cor1dliftings} lets us deduce that $\gamma\asc{w_j-(w_j)}\wsc\gamma\asc{w-(w)}$. By Reshetnyak's Lower Semicontinuity Theorem then, it holds that 
\begin{align*}
\lim_{j\to\infty}\int_{-1}^1\frac{w_j(x)(1-w_j(x))}{1+|w_j(x)|^2}|\nabla w_j(x)|\;\dd x &\geq \int\frac{(y+(w))(1-y-(w))}{1+(y+(w))^2}\;\dd|\gamma\asc{w-(w)}|(x,y)\\
&=\int_0^1\frac{w^\theta(0)(1-w^\theta(0))}{1+(w^\theta(0))^2}\;\dd\theta\\
&=\int\frac{\theta(1-\theta)}{1+\theta^2}\;\dd\theta.
\end{align*}
Thus,
\begin{align*}
\lim_{j\to\infty}\int_{-1}^{1}\Phi(v_j(x))|\nabla v_j(x)|\;\dd x&\geq\delta\lim_{j\to\infty}\int_{-1}^1\frac{w_j(x)(1-w_j(x))}{1+|w_j(x)|^2}|\nabla w_j(x)|\;\dd x\\
&\geq\delta\int_0^1\frac{\theta(1-\theta)}{1+\theta^2}\;\dd\theta\\
&>0.
\end{align*}
We have therefore shown that
\[
\liminf_{j\to\infty}\int_{-1}^{1}\Phi(v_j(x))|\nabla v_j(x)|\;\dd x>0
\]
for any sequence $(v_j)_j\subset\W^{1,1}((-1,1);\R^2)$ such that $v_j\wsc u$. It follows that $u$ is a global minimiser for $\Fcal_{**}$ possessing no weakly* convergent recovery sequences. This is in stark contrast to the situations where either $f$ is $u$-independent (see~~\cite{KriRin10RSI}), where recovery sequences can always be found which converge area-strictly in $\BV(\Omega;\R^m)$, or $(v_j)_j$ converges weakly in $\W^{1,p}(\Omega;\R^m)$ for some $p>1$, where recovery sequences can be found which converge strongly in $\W^{1,p}$.
\end{example}

\subsection{Surface energies}

In order to construct approximate recovery sequences, we first consider the crucial surface part.

\begin{lemma}\label{lemK_frecoverysequence}
Let $f\in\RBVw(\Omega\times\R^m)$ and $u\in\BV(\Omega;\R^m)$. Then for $\Hcal^{d-1}$-almost every $x_0\in\Jcal_u$ and every $\varepsilon>0$, there exists a sequence $(u_j)_j\subset\Acal_u(x_0)$ such that
\[
u_j\wsc u^\pm_{x_0}\text{ in }\BV(\bB^d;\R^m),\quad u_j\to u^\pm_{x_0}\text{ in }\Lp^{d/(d-1)}(\bB^d;\R^m),
\]
\[
\lim_{r\to 0}\lim_{j\to\infty}\left|\int_{\bB^d}f^\infty(x_0+rz, u_j(z),\nabla u_j(z))\;\dd z-K_f[u](x_0)\right|<\varepsilon,
\]
and
\[
K_f[u](x_0)=\lim_{r\to 0}r^{1-d}\frac{1}{\omega_{d-1}}\int_{\overline{B_r(x_0)}}K_f[u](x)\;\dd\Hcal^{d-1}\restrict\Jcal_u(x).
\]
\end{lemma}

\begin{proof}
We shall assume that $x_0$ is a point at which $K_f[u]\Hcal^{d-1}\restrict\Jcal_u$ admits an approximate tangent plane. Let $\varepsilon>0$ and $v\in\Acal_u(x_0)$ be such that
\[
\left|\frac{1}{\omega_{d-1}}\int_{\bB^d}f^\infty(x_0,v(z),\nabla v(z))\;\dd z- K_f[u](x_0)\right|<\varepsilon.
\]
Now, let $k\in\mbN$ and, using the Vitali--Besicovitch Covering Theorem, let $\{B(z_i,r_i)\}_{i\in\mbN}$ be a countable collection of balls such that each centre $z_i$ ($i\in\mbN$) is contained in the hyperplane 
\[
n_u(x_0)^\perp=\{z \colon\; z\cdot n_u(x_0)=0\},
\]
$B(z_i,r_i)\subset\bB^d$, $0<r_i\leq\frac{1}{k}$, and
\[
  \Hcal^{d-1}(n_u(x_0)^\perp\cap \left(\bB^d\setminus\bigcup_{i\in\N} B(z_i,r_i)) \right)=0.
\]
Define $v_k\in\W^{1,1}(\bB^d;\R^m)$ by
\[
v_k(z):=\begin{cases}
v\left(\frac{z-z_i}{r_i}\right)&\text{ if }z\in B(z_i,r_i)\text{ (}i \in \N\text{)},\\
u^\pm_{x_0}(z)&\text{ otherwise.}
\end{cases}
\]
That $v_k\in\W^{1,1}(\bB^d;\R^m)$ follows from the boundary condition satisfied by $v$ and the fact that each $z_i\in n_u(x_0)^\perp$. It also still clearly holds that $v_k=u^\pm$ on $\pd\bB^d$. The sequence $(v_k)_k$ is uniformly bounded in $\Lp^\infty(\bB^d;\R^m)$ by $\norm{v}_{\infty}$ and converges pointwise $\lL$-almost everywhere to $u^\pm_{x_0}$, which implies that $v_k\to u^\pm_{x_0}$ in $\Lp^q(\bB^d;\R^m)$ as $k\to\infty$ for every $q\in[1,\infty)$. We can also compute that $|Dv_k|(\bB^d)=|Dv|(\bB^d)$ from which we deduce $v_k\wsc v$ in $\BV(\bB^d;\R^m)$.

We can further observe after changing coordinates that, for any $r\leq 1$,
\begin{align*}
\int_{\bB^d}f^\infty(x_0+rz,v_k(z),&\nabla v_k(z))\;\dd z=\sum_{i\in\mbN}r_i^{d-1}\int_{\bB^d}f^\infty(x_0+r(z_i+r_iz),v(z),\nabla v(z))\;\dd z.
\end{align*}
Since $f^\infty$ is uniformly continuous on the compact set $\overline{\Omega}\times(\norm{v}_{\Lp^\infty}\overline{\bB^m})\times\pd\bB^{m\times d}$, there exists a modulus of continuity $m\colon[0,\infty)\to[0,\infty)$ such that
\begin{align*}
\left|f^\infty\left(x_0+r(z_i+r_iz),v(z),\frac{\nabla v(z)}{|\nabla v(z)|}\right)-f^\infty\left(x_0+rz_i,v(z),\frac{\nabla v(z)}{|\nabla v(z)|}\right)\right|
\leq m(rr_i|z|)\leq m\left({r}/{k}\right).
\end{align*}
Thus, using the one-homogeneity of $f^\infty$,
\begin{align*}
&\left|\int_{\bB^d}f^\infty(x_0+r(z_i+r_iz),v(z),\nabla v(z))-f^\infty(x_0+rz_i,v(z),\nabla v(z))\;\dd z\right|\\
&\qquad\leq m\left({r}/{k}\right) \int_{\bB^d}|\nabla v(z)|\;\dd z
\end{align*}
for every $i\in\mbN$. For $k$ large enough such that $ m(r/k)\int_{\bB^d}|\nabla v(z)|\;\dd z\leq\varepsilon$, this implies,
\[
\int_{\bB^d}f^\infty(x_0+rz,v_k(z),\nabla v_k(z))\;\dd z=\sum_{i\in\mbN}r_i^{d-1}\int_{\bB^d}f^\infty(x_0+rz_i,v(z),\nabla v(z))\;\dd z+O(\varepsilon).
\]
For $\overline{z}\in B(z_i,r_i)$ and $z\in\bB^d$, we can similarly estimate
\begin{align*}
&\left|\int_{\bB^d}f^\infty\left(x_0+rz_i,v(z),\nabla v(z)\right)\;\dd z-\int_{\bB^d}f^\infty\left(x_0+r\overline{z},v(z),\nabla v(z)\right)\;\dd z\right|\\
&\qquad \leq  m({r}/{k})\int_{\bB^d}|\nabla v(z)|\;\dd z,
\end{align*}
from which it follows that
\begin{align*}
&\Bigg|\int_{\bB^d} f^\infty(x_0+rz_i,v(z),\nabla v(z))\;\dd z\\
&\qquad\qquad-\dashint_{B(z_i,r_i)\cap n_u(x_0)^\perp}\int_{\bB^d}f^\infty\left(x_0+r\overline{z},v(z),\nabla v(z)\right)\;\dd z\;\dd\Hcal^{d-1}(\overline{z})\Bigg|\\
&\qquad\leq m(r/k)|Dv|(\bB^d).
\end{align*}
Hence, for $k$ sufficiently large,
\begin{align*}
&\int_{\bB^d}f^\infty(x_0+rz,v_k(z),\nabla v_k(z))\;\dd z\\
&\qquad=\sum_{i\in\mbN}r_i^{d-1}\dashint_{B(z_i,r_i)\cap n_u(x_0)^\perp}\int_{\bB^d}f^\infty\left(x_0+r\overline{z},v(z),\nabla v(z)\right)\;\dd z\;\dd\Hcal^{d-1}(\overline{z})+O(\varepsilon)\\
&\qquad=\sum_{i\in\mbN}\frac{1}{\omega_{d-1}}\int_{B(z_i,r_i)\cap n_u(x_0)^\perp}\int_{\bB^d}f^\infty\left(x_0+r\overline{z},v(z),\nabla v(z)\right)\;\dd z\;\dd\Hcal^{d-1}(\overline{z})+O(\varepsilon).
\end{align*}
Since $\Hcal^{d-1}(n_u(x_0)^\perp\cap(\bB^d\setminus\bigcup_i B(z_i,r_i)))=0$, this implies
\begin{align*}
\int_{\bB^d}f^\infty(x_0+rz,v_k(z),\nabla v_k(z))\;\dd z&=\dashint_{\bB^d\cap n_u(x_0)^\perp}\int_{\bB^d}f^\infty\left(x_0+r\overline{z},v(z),\nabla v(z)\right)\;\dd z\;\dd\Hcal^{d-1}(\overline{z})\\
&\qquad+O(\varepsilon).
\end{align*}
Finally noting
\[
\bigl|f^\infty\left(x_0+r\overline{z},v(z),\nabla v(z)\right)-f^\infty\left(x_0,v(z),\nabla v(z)\right)\bigr|\leq m(r)|\nabla v(z)|
\]
for every $z,\overline{z}\in\bB^d$, and integrating first in $z$ with respect to $\lL\restrict\bB^d$ and then in $\overline{z}$ with respect to $\Hcal^{d-1}\restrict n_u(x_0)^\perp$, we deduce
\begin{equation}\label{eqK_festimate}
\lim_{r\to 0}\lim_{k\to\infty}\left|\int_{\bB^d}f^\infty(x_0+rz,v_k(z),\nabla v_k(z))\;\dd z-K_f[u](x_0)\right|\leq 2\omega_{d-1}\varepsilon+\varepsilon.
\end{equation}
Applying Propositions~\ref{propfixedbdary} and~\ref{ctsembedding} to each $v_k$ and using the fact that $|Dv_k|(\Omega)=|Dv|(\Omega)$, we can obtain a sequence $(w_k)_k\subset\Acal_u(x_0)$ with $\norm{w_k}_{\Lp^\infty}\leq\norm{v}_{\Lp^\infty}$, $w_k\wsc u^\pm_{x_0}$ in $\BV(\bB^d;\R^m)$, $w_k\to u^\pm_{x_0}$ in $\Lp^{d/(d-1)}(\bB^d;\R^m)$, and which, by the same reasoning as before, satisfies
\[
\lim_{r\to 0}\lim_{k\to\infty}\left|\int_{\bB^d}f^\infty(x_0+rz,w_k(z),\nabla w_k(z))\;\dd z-\int_{\bB^d}f^\infty(x_0,w_k(z),\nabla w_k(z))\;\dd z\right|=0,
\]
and also
\[
\left|\int_{\bB^d}f^\infty(x_0,v_k(z),\nabla v_k(z))\;\dd z-\int_{\bB^d}f^\infty(x_0,w_k(z),\nabla w_k(z))\;\dd z\right|\leq\frac{1}{k}.
\]
Hence,
\[
\lim_{r\to 0}\lim_{k\to\infty}\left|\int_{\bB^d}f^\infty(x_0+rz,v_k(z),\nabla v_k(z))\;\dd z-\int_{\bB^d}f^\infty(x_0+rz	,w_k(z),\nabla w_k(z))\;\dd z\right|=0
\]
and it follows from~\eqref{eqK_festimate} that
\[
\limsup_{r\to 0}\lim_{k\to\infty}\left|\frac{1}{\omega_{d-1}}\int_{\bB^d}f^\infty(x_0+rz,w_k(z),\nabla w_k(z))\;\dd z-K_f[u](x_0)\right|< 2\omega_{d-1}\varepsilon + \varepsilon.
\]
Since $\varepsilon>0$ was arbitrary, we have obtained the desired conclusion.

The final assertion follows from the existence of an approximate tangent plane to the measure $K_f[u]\Hcal^{d-1}\restrict\Jcal_u$ at $x_0$.
\end{proof}

\subsection{Primitive recovery sequences}
In the following proposition, we explicitly construct $\Lp^1$-recovery sequences in $\BV(\Omega;\R^m)$ for $\Fcalrw$ in the case where $f=f^\infty$.
\begin{proposition}\label{proprecessionrecoveryseq}
Let $f\in\C(\overline{\Omega}\times\R^m\times\R^{m\times d})$ be a positively one-homogeneous integrand, let $u\in\BV(\Omega;\R^m)$ and assume that 
\[
0\leq f(x,y,A)\leq C|A|\qquad\text{ for all }(x,y,A)\in\overline{\Omega}\times\R^m\times\R^{m\times d}
\]
for some $C>0$. Then there exists a sequence $(u_j)_j\subset\BV(\Omega;\R^m)$ such that $u_j\to u$ in $\Lp^1(\Omega;\R^m)$, $u_j$ and $\nabla u_j$ converge pointwise $\lL$-almost everywhere to $u$ and $\nabla u$ respectively, $u_j\to u$ in $\Lp^{d/(d-1)}(\bB^d;\R^m)$, and
\begin{align*}
&\lim_{j\to\infty}\left(\int_\Omega f(x,u_j(x),\nabla u_j(x))\;\dd x+\int_\Omega\int_0^1 f\left(x,u_j^\theta(x),\frac{\dd D^su_j}{\dd|D^s u_j|}(x)\right)\;\dd\theta\;\dd|D^su_j|(x)\right)\\
&\qquad=\int_\Omega f(x,u(x),\nabla u(x))\;\dd x+\int_\Omega f\left(x,u(x),\frac{\dd D^cu}{\dd|D^c u|}(x)\right)\;\dd|D^c u|(x)\\
&\qquad\qquad +\int_{\Jcal_u}K_f[u](x)\;\dd\Hcal^{d-1}(x).
\end{align*}
\end{proposition}

\begin{proof}
For brevity, abbreviate $\Kfrak_f[u]:=K_f[u]\Hcal^{d-1}\restrict\Jcal_u$ so that
\[
\int_{A\cap\Jcal_u} K_f[u](x)\;\dd\Hcal^{d-1}(x)=\Kfrak_f[u](A)\quad\text{ for every Borel set $A\subset\Omega$.}
\]
Now let $L_0$ be the set of points $x\in\mathcal{J}_u$ which are such that
\begin{enumerate}[(I)]
\item $x$ satisfies the conclusions of Lemma~\ref{lemK_frecoverysequence};
\item\label{eqrecoverytangentcond} $\Kfrak_f[u]$ and $|Du|$ possess approximate tangent planes at $x$;
\item\label{eqrecoveryblowupcond} $u(x+r\frarg)$ converges strictly  in $\BV(\bB^d;\R^m)$ to $u^\pm_x$ as $r\to 0$ as discussed after Definition~\ref{defjumpblowup}.
\end{enumerate}
By Lemma~\ref{lemK_frecoverysequence}, the Lebesgue Differentiation Theorem, the definition of $\mathcal{J}_u$, and the fact that $\mathcal{J}_u$ is countably $\Hcal^{d-1}$-rectifiable, we have that $\mathcal{H}^{d-1}(\mathcal{J}_u\setminus L_0)=0$.

For $i=1,2$, let $F_i\colon L_0\times(0,1]\to\R$ be the functions
\[
F_1(x,r):=\dashint_{B(x,r)}\left|u^\pm_x\left(\frac{\overline{x}-x}{r}\right)-u(x)\right|+\left|u^\pm_x\left(\frac{\overline{x}-x}{r}\right)-u(x)\right|^\frac{d}{d-1}\;\dd \overline{x},
\]
\[
 F_2(x,r):=\frac{1}{\omega_{d-1}}r^{1-d}\frac{|Du|(B(x,r))}{|u^+(x)-u^-(x)|}.
\]
It follows from~\eqref{eqrecoverytangentcond} and~\eqref{eqrecoveryblowupcond} combined with Proposition~\ref{ctsembedding} that $\lim_{r\downarrow 0}F_1(x,r)=0$ and $\lim_{r\downarrow 0}F_2(x,r)=1$ for each $x\in L_0$. Since the $F_i$ are $(\Hcal^{d-1}\restrict\Jcal_u)\times(\Lcal^1\restrict(0,1])$-measurable and hence $(\Kfrak_f[u])\times(\Lcal^1\restrict(0,1])$-measurable, we can therefore write $L_0$ as the following countable union of increasing $\Kfrak_f[u]$-measurable sets for any $\varepsilon>0$:
\[
L_0=\bigcup_{k\in\mbN}\left\{x\in L_0\colon F_1(x,r)\leq\varepsilon\frac{|Du|(B(x,r))}{r^{d-1}},\;\; F_2(x,r)\in(1-\varepsilon,1+\varepsilon)\text{ for all }r\leq \frac{1}{k} \right\}.
\]
Hence, for fixed $\varepsilon>0$ we can write $L_0=L_0^{\varepsilon}\cup E^{\varepsilon}$ where $\Kfrak_f[u](E^{\varepsilon})<\varepsilon$ and, for some $k_{\varepsilon}\in\mbN$,
\[
L_0^{\varepsilon}\subset\left\{x\in L_0\colon F_1(x,y)\leq\varepsilon\frac{|Du|(B(x,r))}{r^{d-1}},\quad F_2(x,r)\in(1-\varepsilon,1+\varepsilon)\text{ for all }r\leq \frac{1}{k_\varepsilon} \right\}.
\]	
By the outer regularity of Radon measures, there exists an open set $U_\varepsilon$ with $\Jcal_u\subset U_\varepsilon$ and $\lL(U_\varepsilon)<\varepsilon$. For a fixed $x\in L_0$, the fact that $\Kfrak_f[u]$ possesses an approximate tangent plane at $x$ implies that
\[
\lim_{r\to 0}\frac{\Kfrak_f[u]\left(B(x,r)\right)}{r^{d-1}}=\omega_{d-1}\frac{\dd\Kfrak_f[u]}{\dd\Hcal^{d-1}\restrict\Jcal_u}(x)=\omega_{d-1}K_f[u](x).
\]
Since $|Du|$ also possesses an approximate tangent plane at $x$ and $|u^+(x)-u^-(x)|>0$ for every $x\in\Jcal_u$, we have that
\[
\lim_{r\to 0}\frac{|Du|(B(x,r))}{r^{d-1}}=\omega_{d-1}\frac{\dd|Du|}{\dd\Hcal^{d-1}\restrict\Jcal_u}(x)=\omega_{d-1}|u^+(x)-u^-(x)|>0,
\]
and so we can deduce that, for all $r>0$ sufficiently small,
\[
\left|\frac{\Kfrak_f[u]\left(B(x,r)\right)}{r^{d-1}}-\omega_{d-1}K_f[u](x)\right|\leq\frac{\varepsilon}{2}\frac{|Du|(B(x,r))}{r^{d-1}}.
\]
On the other hand, Lemma~\ref{lemK_frecoverysequence} implies that, for all $r$ sufficiently small,
\[
\inf_{\substack{v\in\Acal_u(x)\\ \norm{v}_{\Lp^{1^*}}\leq 2\norm{u^\pm_{x}}_{\Lp^{1^*}}}}\left|\int_{\bB^d}f\left(x+\tau rz,v\left(z\right),\nabla v(z)\right)\dd z-\omega_{d-1}K_f[u](x)\right|\leq\frac{\varepsilon}{2}\frac{|Du|(B(x,r))}{r^{d-1}}
\]
for any $\tau\in(0,1)$, where $1^*:=d/(d-1)$. It therefore follows that, for $\tau\in(0,1)$ fixed, the collection
\begin{align*}
\Gcal^{\tau, \varepsilon}&:=\Bigg\{\overline{B(x,r)}\colon x\in L_0^{\varepsilon},\;r\leq\frac{1}{k_{\varepsilon}},\;\overline{B(x,r)}\Subset U_\varepsilon\text{ and }\\
&\inf_{\substack{v\in\Acal_u(x)\\ \norm{v}_{\Lp^{1^*}}\leq 2\norm{u^\pm_{x}}_{\Lp^{1^*}}}}\left|\int_{\bB^d}f\left(x+\tau rz,v\left(z\right),\nabla v(z)\right)\dd z-\frac{\Kfrak_f[u]\left(B(x,r)\right)}{r^{d-1}}\right|<\varepsilon\frac{|Du|((B(x,r))}{r^{d-1}}\Bigg\}
\end{align*}
is a fine cover for $L_0^{\varepsilon}$ and so, by the Vitali--Besicovitch Covering Theorem, there exists a countable disjoint set $\Hcal^{\tau,\varepsilon}\subset \Gcal^{\tau,\varepsilon}$ whose union covers $\Kfrak_f[u]$-almost all of $L_0^{\varepsilon}$.

Let $\overline{B(x_1,r_1)},\,\overline{B(x_{2},r_{2})}\ldots$ be a sequence of elements from $\Hcal^{\tau, \varepsilon}$ such that there exists an increasing sequence $N_1,\,N_2\ldots$ in $\mbN$ with
\[
\Hcal^{d-1}\left(L_0^{\varepsilon}\setminus\bigcup_{i=1}^{N_j}\overline{B(x_{i},r_{i})}\right)\leq\frac{1}{j}.
\]
Let $\eta_\tau\in\C_c^\infty(\bB^d;[0,1])$ be such that $\eta_\tau \equiv 1$ on $\tau\bB^d$.

Fix $j$. For $i=1\ldots N_j$, let $v^\tau_i\in\Acal_u(x_i)$ be such that $\norm{v^\tau_i}_{\Lp^{1^*}}\leq 2\norm{u_{x_i}^\pm}_{\Lp^{1^*}}$ and
\[
\left|\int_{\bB^d}f\left(x+\tau r_iz,v^\tau_i\left(z\right),\nabla v^\tau_i(z)\right)\;\dd z-\frac{\Kfrak_f[u]\left(B(x_i,{r_i})\right)}{r_i^{d-1}}\right|<\varepsilon\frac{|Du|((B(x_i,{r_i}))}{r_i^{d-1}}.
\] 
Define
$(w_i^{\tau})_i\subset\BV(\bB^d;\R^m)$ by
\[
w_i^\tau(z):=\begin{cases}
v^\tau_i\left(\frac{z}{\tau}\right)&\text{ if }|z|<\tau,\\
u^\pm_{x_i}(z)&\text{ if }\tau\leq|z|<1.
\end{cases}
\]
We can now set
\[
v_j^{\varepsilon,\tau}(x):=\sum_{i=1}^{N_j}w_i^{\tau}\left(\frac{x-x_i}{r_i}\right)\eta_\tau\left(\frac{x-x_i}{r_i}\right)
\]
and
\[
u_j^{\varepsilon,\tau}(x):=u(x)\left(1-\sum_{i=1}^{N_j}\eta_\tau\left(\frac{x-x_i}{r_i}\right)\right) +v_j^{\varepsilon,\tau}(x).
\]
Invoking the criterion for membership of $L_0^\varepsilon$ involving $F_1$, we have that
\begin{align*}
\norm{v_j^{\varepsilon,\tau}}^{1^*}_{\Lp^{1^*}}&\leq \sum_{i=1}^\infty r_i^{d}\norm{w_i^{\tau}}^{1^*}_{\Lp^{1^*}}\\
&\leq (2^{1^*}+1)\sum_{i=1}^\infty r_i^{d}\norm{u_{x_i}^\pm}^{1^*}_{\Lp^{1^*}}\\
&\leq (2^{1^*}+1)\sum_{i=1}^\infty r_i^{d}\left(\varepsilon r_i^{1-d}|Du|(B(x_i,{r_i}))+\dashint_{B(x_i,{r_i})}|u(x)|^{1^*}\;\dd x\right)\\
&\leq (2^{1^*}+1)\left(\frac{\varepsilon}{k_\varepsilon}|Du|\left(\bigcup\Fcal^\varepsilon\right)+\int_{\bigcup\Fcal^\varepsilon}|u(x)|^{1^*}\;\dd x\right)\\
&\leq (2^{1^*}+1)\left(\frac{\varepsilon}{k_\varepsilon}|Du|(U_\varepsilon)+\int_{U_\varepsilon}|u(x)|^{1^*}\;\dd x\right),
\end{align*}
which implies that $v_j^{\varepsilon,\tau}\to 0$ in $\Lp^{1^*}(\Omega;\R^m)$ as $\varepsilon\to 0$ independently of $\tau$ and $j$ and hence that $u_j^{\varepsilon,\tau}\to u$ in $\Lp^{1^*}(\Omega;\R^m)$ as $\varepsilon\to 0$ independently of $\tau$ and $j$.

Now observe
\begin{align}
\begin{split}\label{eqrecoverdecomp1}
&\int_{\bigcup_{i=1}^{N_j} B(x_i,r_i)}\int_0^1 f\left(x,\left(u_j^{\varepsilon,\tau}\right)^\theta(x),\frac{\dd D u_j^{\varepsilon,\tau}}{\dd|Du_j^{\varepsilon,\tau}|}(x)\right)\;\dd\theta\;\dd|Du_j^{\varepsilon,\tau}|(x)\\
&\qquad =\sum_{i=1}^{N_j}\Bigg\{\int_{B(x_i,\tau r_i)}f\left(x,v^\tau_i\left(\frac{x-x_i}{\tau r_i}\right),\frac{1}{\tau r_i}\nabla v^\tau_i\left(\frac{x-x_i}{\tau r_i}\right)\right)\;\dd x\\
&\qquad\qquad+\int_{\left\{\tau\leq\frac{|x-x_i|}{r_i}<1\right\}}\int_0^1 f\left(x,\left(u_j^{\varepsilon,\tau}(x)\right)^\theta,\frac{\dd D u_j^{\varepsilon,\tau}}{\dd|D u_j^{\varepsilon,\tau}|}(x)\right)\;\dd\theta\;\dd|D u_j^{\varepsilon,\tau}|(x) \Bigg\}.
\end{split}
\end{align}
Changing coordinates, we can manipulate the first term in this expression as follows
\begin{align}
\begin{split}\label{eqrecoverydecomp1.5}
&\sum_{i=1}^{N_j}\int_{B(x_i,\tau r_i)}f\left(x,v^\tau_i\left(\frac{x-x_i}{\tau r_i}\right),\frac{1}{\tau r_i}\nabla v^\tau_i\left(\frac{x-x_i}{\tau r_i}\right)\right)\;\dd x\\
&\qquad =\sum_{i=1}^{N_j}\tau^{d-1} r_i^{d-1}\int_{\bB^d}f\left(x_i+\tau r_i z,v^\tau_i\left(z\right),\nabla v^\tau_i\left(z\right)\right)\;\dd z.
\end{split}
\end{align}
Since $\Hcal^{\tau,\varepsilon}$ is a fine cover for $L_0^\varepsilon$ with respect to $\Kfrak_f[u]$, we can write
\begin{align*}
&\lim_{j\to\infty}\Bigg|\Bigg(\tau^{d-1}\sum_{i=1}^{N_j}r_i^{d-1}\int_{\bB^d}f\left(x_i+\tau r_i z,v^\tau_i\left(z\right),\nabla v^\tau_i(z)\right)\;\dd z\Bigg)-\Kfrak_f[u](\Jcal_u)\Bigg|\\
&\qquad \leq\tau^{d-1}\sum_{B(x_i,r_i)\in\Hcal^{\varepsilon}}r_i^{d-1}\Bigg|\int_{\bB^d}f\left(x_i+\tau r_i z,v_i^\tau\left(z\right),\nabla v_i^\tau(z)\right)\;\dd z-r_i^{1-d}\Kfrak_f[u](\overline{B(x_i,r_i)})\Bigg|\\
&\qquad\qquad+(1-\tau^{d-1})\Kfrak_f[u](L_0^\varepsilon)+\Kfrak_f[u](\Jcal_u\setminus L_0^{\varepsilon}).
\end{align*}
By our choice of $v^\tau_i$ and the fact that $\Kfrak_f[u](\Jcal_u\setminus L_0^\varepsilon)<\varepsilon$, we therefore have that
\begin{align}
\begin{split}\label{eqrecoverydecomp2}
&\lim_{j\to\infty}\Bigg|\Bigg(\tau^{d-1}\sum_{i=1}^{N_j}r_i^{d-1}\int_{\bB^d}f\left(x_i+\tau r_i z,v^\tau_i\left(z\right),\nabla v^\tau_i(z)\right)\;\dd z\Bigg)-\Kfrak_f[u](\Jcal_u)\Bigg|\\
&\qquad \leq\tau^{d-1}\sum_{B(x_i,r_i)\in\Hcal^{\varepsilon}}\varepsilon|Du|(B(x_i,r_i))+(1-\tau^{d-1})\Kfrak_f[u](\Jcal_u)+\varepsilon\\
&\qquad \leq\varepsilon|Du|(\Omega)+(1-\tau^{d-1})\Kfrak_f[u](\Jcal_u)+\varepsilon.
\end{split}
\end{align}

For each $x$ satisfying $\tau\leq\frac{|x-x_i|}{r_i}<1$,
\[
u_j^{\varepsilon,\tau}(x)=u(x)+\eta_\tau\left(\frac{x-x_i}{r_i}\right)\left[u^\pm_{x_i}\left(\frac{x-x_i}{r_i}\right)-u(x)\right], 
\]
and so we can use the product rule to deduce
\begin{align*}
|Du_j^{\varepsilon,\tau}|\left(\left\{x\colon\tau\leq\frac{|x-x_i|}{r_i}<1\right\}\right)\leq &\frac{1}{r_i}\norm{\nabla\eta_\tau}_\infty\int_{\left\{\tau\leq\frac{|x-x_i|}{r_i}<1\right\}}\left|u^\pm_{x_i}\left(\frac{x-x_i}{r_i}\right)-u(x)\right|\;\dd x\\
&\qquad+r_i^{d-1}(1-\tau^{d-1})\omega_{d-1}|u^+(x_i)-u^-(x_i)|\\
&\qquad+|Du|(B(x_i,r_i))-|Du|(B(x_i,\tau r_i)).
\end{align*}
Since $f$ satisfies $0\leq f(x,y,A)\leq C|A|$ for some $C>0$, we can therefore estimate
\begin{align*}
&\int_{\left\{\tau\leq\frac{|x-x_i|}{r_i}<1\right\}}\int_0^1 f\left(x,\left(u_j^{\varepsilon,\tau}(x)\right)^\theta,\frac{D u_j^{\varepsilon,\tau}}{|D u_j^{\varepsilon,\tau}|}(x)\right)\;\dd\theta\;\dd|D u_j^{\varepsilon,\tau}|(x)\\
&\qquad\leq C\int_{\left\{\tau\leq\frac{|x-x_i|}{r_i}<1\right\}}\;\dd|Du_j^{\varepsilon,\tau}|(x)\\
&\qquad\leq C\biggl(r_i^{d-1}\norm{\nabla\eta_\tau}_\infty\dashint_{B(x_i,r_i)}\left|u^\pm_{x_i}\left(\frac{x-x_i}{r_i}\right)-u(x)\right|\;\dd x\\
&\qquad\qquad\qquad+r_i^{d-1}(1-\tau^{d-1})\omega_{d-1}|u^+(x_i)-u^-(x_i)|\\
&\qquad\qquad\qquad+|Du|(B(x_i,r_i))-|Du|(B(x_i,\tau r_i))\biggr).
\end{align*}
The membership criterion of for $L_0^{\varepsilon}$ with respect to $F_2$ implies that
\begin{align*}
|Du|(B(x_i,\tau r_i))\geq(1-\varepsilon)\tau^{d-1}\omega_{d-1}r_i^{d-1}|u^+(x_i)-u^-(x_i)|\geq\tau^{d-1}\frac{1-\varepsilon}{1+\varepsilon}|Du|(B(x_i,r_i)),
\end{align*}
which gives
\[
|Du|(B(x_i,r_i))-|Du|(B(x_i,\tau r_i))\leq |Du|(B(x_i,r_i))\left(1-\tau^{d-1}\frac{1-\varepsilon}{1+\varepsilon}\right).
\]
From the same criterion, we also deduce
\[
r_i^{d-1}(1-\tau^{d-1})\omega_{d-1}|u^+(x_i)-u^-(x_i)|\leq\frac{1-\tau^{d-1}}{1-\varepsilon}|Du|(B(x_i,r_i)).
\]
Thus, using also the membership criterion for $L_0^\varepsilon$ with respect to $F_1$ to bound 
\[
r_i^{d-1}\norm{\nabla\eta_\tau}_\infty\dashint_{B(x_i,r_i)}\left|u^\pm_{x_i}\left(\frac{x-x_i}{r_i}\right)-u(x)\right|\;\dd x\leq\varepsilon\norm{\nabla\eta_\tau}_\infty|Du|(B(x_i,r_i)),
\]
we obtain, for $\varepsilon<1/2$,
\begin{align*}
&\int_{\left\{\tau\leq\frac{|x-x_i|}{r_i}<1\right\}}\int_0^1 f\left(x,\left(u_j^{\varepsilon,\tau}(x)\right)^\theta,\frac{D u_j^{\varepsilon,\tau}}{|D u_j^{\varepsilon,\tau}|}(x)\right)\;\dd\theta\;\dd|D u_j^{\varepsilon,\tau}|(x)\\
&\qquad\leq C\biggl(\varepsilon\norm{\nabla\eta_\tau}_\infty|Du|(B(x_i,{r_i}))+\frac{1-\tau^{d-1}}{1-\varepsilon}|Du|(B(x_i,{r_i}))\\
&\qquad\qquad\qquad+\left(1-\tau^{d-1}\frac{1-\varepsilon}{1+\varepsilon}\right)|Du|(B(x_i,{r_i}))\biggr)\\
&\qquad\leq C\left(\varepsilon\norm{\nabla\eta_\tau}_\infty+3(1-\tau^{d-1})+2\varepsilon\right)|Du|(B(x_i,{r_i})).
\end{align*}
Hence, for $\varepsilon<1/2$,
\begin{align}
\begin{split}\label{eqrecoverydecomp3}
&\left|\sum_{i=1}^{N_j}\int_{\left\{\tau\leq\frac{|x-x_i|}{r_i}<1\right\}}\int_0^1 f\left(x,\left(u_j^{\varepsilon,\tau}(x)\right)^\theta,\frac{D u_j^{\varepsilon,\tau}}{|D u_j^{\varepsilon,\tau}|}(x)\right)\;\dd\theta\;\dd|D u_j^{\varepsilon,\tau}|(x)\right|\\
&\qquad\leq C\left(\varepsilon\norm{\nabla\eta_\tau}_\infty+3(1-\tau^{d-1})+2\varepsilon\right)|Du|(U_\varepsilon).
\end{split}
\end{align}
Combining~\eqref{eqrecoverdecomp1},~\eqref{eqrecoverydecomp1.5},~\eqref{eqrecoverydecomp2}, and~\eqref{eqrecoverydecomp3}, we finally deduce
\[
\lim_{\tau\to 1}\lim_{\varepsilon\to 0}\lim_{j\to\infty}\int_{\bigcup_{i=1}^{N_j}B(x_i,r_i)}\int_0^1 f\left(x,\left(u_j^{\varepsilon,\tau}\right)^\theta(x),\frac{\dd D u_j^{\varepsilon,\tau}}{\dd|Du_j^{\varepsilon,\tau}|}(x)\right)\;\dd\theta\;\dd|Du_j^{\varepsilon,\tau}|(x)=\Kfrak_f[u](\Jcal_u).
\]
Since $u_j^{\varepsilon,\tau}\equiv u$ in $\Omega\setminus U_\varepsilon$, we can use a diagonal argument to obtain a sequence $(u_j)_j\subset\BV(\Omega;\R^m)$ satisfying $u_j\to u$ in $\Lp^1(\Omega;\R^m)$, $u_j(x)=u(x)$, $\nabla u_j(x)=\nabla u(x)$ in $\Omega_j:=\Omega\setminus U_{1/j}$ and which is such that
\begin{align*}
&\lim_{j\to\infty}\int_\Omega f(x,u_j(x),\nabla u_j(x))\;\dd x+\int_\Omega\int_0^1 f\left(x,u^\theta_j(x),\frac{\dd D^su_j}{\dd |D^s u_j|}(x)\right)\;\dd|D^s u_j|(x)\\
&\qquad=\int_\Omega f(x,u(x),\nabla u(x))\;\dd x+\int_\Omega f\left(x,u(x),\frac{\dd D^cu}{\dd |D^c u|}(x)\right)\;\dd|D^c u|(x)+\Kfrak_f[u](\Jcal_u),
\end{align*}
as required.
\end{proof}

\subsection{Recovery sequences}
We are now in a position to construct approximate recovery sequences for $\Fcalrw$, finally allowing us to complete the proof of Theorem~\ref{wsclscthm}.
\begin{theorem}[Approximate recovery sequences]\label{thmepsilonrecoverysequences}
Let $u\in\BV(\Omega;\R^m)$ and $f\in\RBVw(\Omega\times\R^m)$ be such that $f(x,y,\frarg)$ is quasiconvex for every $(x,y)\in\overline{\Omega}\times\R^m$. For any $\varepsilon>0$, there exists a sequence $(u_j)_j\subset\C^\infty(\Omega;\R^m)$ such that $u_j\wsc u$ in $\BV(\Omega;\R^m)$ and
\begin{align*}
\biggl|\lim_{j\to\infty}\Fcal[u_j]&-\Big(\int_\Omega f(x,u(x),\nabla u(x))\;\dd x+\int_\Omega f^\infty\left(x,u(x),\frac{\dd D^c u}{\dd|D^c u|}(x)\right)\;\dd|D^c u|(x)\\
&\qquad\qquad+\int_{\mathcal{J}_u}K_f[u](x)\;\dd\mathcal{H}^{d-1}(x)\Big)\biggr|\leq\varepsilon.
\end{align*}
\end{theorem}
\begin{proof}
\emph{Step 1:} Assume first that $f\in\RBVw(\Omega\times\R^m)$ and that $f^\infty$ satisfies $c|A|\leq f^\infty(x,y,A)\leq C|A|$ for some choice of constants $C>c>0$. Let $(u_j)_j\subset\BV(\Omega;\R^m)$ be the sequence provided by Proposition~\ref{proprecessionrecoveryseq}, such that
\begin{align}
\begin{split}\label{eqfinfinitylimit}
&\lim_{j\to\infty}\int_\Omega f^\infty(x,u_j(x),\nabla u_j(x))\;\dd x+\int_\Omega\int_0^1f^\infty\left(x,u^\theta(x),\frac{\dd D^su_j}{\dd|D^s u_j|}(x)\right)\;\dd\theta\dd|D^su_j|(x)\\
&\qquad=\int_\Omega f^\infty(x,u(x),\nabla u(x))\;\dd x+\int_\Omega f^\infty\left(x,u(x),\frac{\dd D^cu}{\dd|D^c u|}(x)\right)\;\dd|D^c u|(x)\\
&\qquad\qquad +\int_{\Jcal_u} K_{f}[u](x)\;\dd\Hcal^{d-1}(x).
\end{split}
\end{align}
Since $f^\infty$ is assumed to be coercive, $(u_j)_j$ is a bounded sequence in $\BV(\Omega;\R^m)$ and so (upon passing to a nonrelabelled subsequence), we can assume that $\gamma\asc{u_j-(u_j)_\Omega}\toY\bnu$ for some $\bnu\in\AYLift(\Omega\times\R^m)$ with $\llbracket\bnu\rrbracket=u-(u)_\Omega$ which, by virtue of Corollary~\ref{corymlsc} satisfies
\begin{align*}
&\int_\Omega f^\infty(x,u(x),\nabla u(x))\;\dd x+\int_\Omega f^\infty\left(x,u(x),\frac{\dd D^cu}{\dd|D^c u|}(x)\right)\;\dd|D^c u|(x)\\
&\qquad\qquad +\int_{\Jcal_u} K_{f}[u](x)\;\dd\Hcal^{d-1}(x)\\
&\qquad=\lim_{j\to\infty}\int_\Omega f^\infty(x,u_j(x),\nabla u_j(x))\;\dd x+\int_\Omega\int_0^1f^\infty\left(x,u_j^\theta(x),\frac{\dd D^su_j}{\dd|D^s u_j|}(x)\right)\;\dd\theta\dd|D^su_j|(x)\\
&\qquad\geq\ddprb{f^\infty(\frarg,\frarg+(u)_\Omega,\frarg),\bnu\restrict(\overline{\Omega}\times\R^m)}.
\end{align*}
Theorem~\ref{thm:restrictionsofyms} implies that $\bnu\restrict(\overline{\Omega}\times\R^m)\in\AYLift(\Omega\times\R^m)$ and so, since $\llbracket\bnu\restrict(\overline{\Omega}\times\R^m)\rrbracket=u-(u)_\Omega$, from Lemma~\ref{lemapproxymprescribedboundary} we get that there exists a sequence $(v_j)_j\subset(\BV_\#\cap\C^\infty)(\Omega;\R^m)$ converging weakly* to $u-(u)_\Omega$ and which is such that $\gamma\asc{v_j}\toY\bnu\restrict(\overline{\Omega}\times\R^m)$. Proposition~\ref{lemextendedrepresentation}~(i) applied to $f$, $\bnu\restrict(\overline{\Omega}\times\R^m)$, and $(\gamma\asc{v_j})_j$ together with Theorem~\ref{thmwsclsc} applied to $(v_j+(u)_\Omega)_j$ therefore yield
\begin{align*}
&\ddprb{f^\infty(\frarg,\frarg+(u)_\Omega,\frarg),\bnu\restrict(\overline{\Omega}\times\R^m)}\\
&\qquad=\lim_{j\to\infty}\int_\Omega f^\infty(x,v_j(x)+(u)_\Omega,\nabla v_j(x))\;\dd x\\
&\qquad\geq \int_\Omega f^\infty(x,u(x),\nabla u(x))\;\dd x+\int_\Omega f^\infty\left(x,u(x),\frac{\dd D^cu}{\dd|D^c u|}(x)\right)\;\dd|D^c u|(x)\\
&\qquad\qquad +\int_{\Jcal_u} K_{f}[u](x)\;\dd\Hcal^{d-1}(x).
\end{align*}
It follows that
\begin{align*}
&\ddprb{f^\infty(\frarg,\frarg+(u)_\Omega,\frarg),\bnu\restrict(\overline{\Omega}\times\R^m)}\\
&\qquad= \int_\Omega f^\infty(x,u(x),\nabla u(x))\;\dd x+\int_\Omega f^\infty\left(x,u(x),\frac{\dd D^cu}{\dd|D^c u|}(x)\right)\;\dd|D^c u|(x)\\
&\qquad\qquad +\int_{\Jcal_u} K_{f}[u](x)\;\dd\Hcal^{d-1}(x)
\end{align*}
and so we can assume $\lambda_\nu(\overline{\Omega}\times\infty\pd\bB^m)=0$.

By testing with integrands $\varphi\otimes h$ for $\varphi\in\C_c(\Omega\times\R^m)$ and $h\in\C_c(\R^{m\times d})$, we see that the pointwise almost everywhere convergence of $\nabla u_j$ to $\nabla u$ implies that $\nu_{x,u(x)}=\delta_{\nabla u(x)}$ for almost every $x\in\Omega$.

Define $h:=f-f^\infty$ and, repeating the construction used in Step~4 of the proof of Proposition~\ref{lemextendedrepresentation}, define $h_k\in\RL(\Omega\times\R^m)$ by
\[
h_k(x,y,A):=\left\{1+\mathbbm{1}_{\{|y|> k\}}(y)\frac{k^{d/(d-1)}-|y|^{d/(d-1)}}{1+|y|^{d/(d-1)}+|A|}\right\}h(x,y,A).
\]
Proposition~\ref{lemextendedrepresentation}~(i) combined with the fact that we can assume $\lambda_\nu(\overline{\Omega}\times\infty\pd\bB^m)=0$ implies that
\begin{align*}
\lim_{j\to\infty}\int_\Omega h_k(x,u_j(x),\nabla u_j(x))\;\dd x & =\int_{\Omega\times\R^m} \ip{h_k(x,y+(u)_\Omega,\frarg)}{\nu_{x,y}}\;\dd\iota_{\nu}(x,y)\\
&=\int_{\Omega}h_k(x,u(x),\nabla u(x))\;\dd x
\end{align*}
for every $k\in\mbN$.

Since $|h(x,y,A)-h_k(x,y,A)|\leq C\mathbbm{1}_{\{|y|< k\}}(y)(|y|^{d/(d-1)}-k^{d/(d-1)})$ and the sequence $(u_j)_j$ is $d/(d-1)$-uniformly integrable, for any $\varepsilon>0$ there exists $k_\varepsilon\in\mbN$ such that $k	\geq k_\varepsilon$ implies
\begin{align*}
&\biggl|\lim_{j\to\infty}\int_\Omega h(x,u_j(x),\nabla u_j(x))\;\dd x -\lim_{j\to\infty}\int_\Omega h_k(x,u_j(x),\nabla u_j(x))\;\dd x\biggr|\\
&\qquad \leq \lim_{j\to\infty}C\int_{\{|u_j(x)|\geq k\}}|y|^{d/(d-1)}-k^{d/(d-1)}\;\dd x\\
&\qquad \leq\varepsilon.
\end{align*}
Thus,
\begin{align}
\begin{split}\label{eqtildeflimit}
\lim_{j\to\infty}\int_\Omega h(x,u_j(x),\nabla u_j(x))\;\dd x&=\lim_{k\to\infty}\int_{\Omega}h_k(x,u(x),\nabla u(x))\;\dd x\\
&=\int_\Omega h(x,u(x),\nabla u(x))\;\dd x.
\end{split}
\end{align}

Adding equations~\eqref{eqfinfinitylimit} and~\eqref{eqtildeflimit}, we obtain
\begin{align*}
\lim_{j\to\infty}\Fcal[u_j]&=\int_\Omega f(x,u(x),\nabla u(x))\;\dd x+\int_\Omega f^\infty\left(x,u(x),\frac{\dd D^cu}{\dd |D^c u|}(x)\right)\;\dd|D^c u|(x)\\
&\qquad+\int_{\mathcal{J}_u}{K}_f[u](x)\;\dd\mathcal{H}^{d-1}(x).
\end{align*}
An application of Theorem~\ref{thmareastrictcontinuity} together with the area-strict density of $\C^\infty(\Omega;\R^m)$ in $\BV(\Omega;\R^m)$ to each $\Fcal[u_j]$ combined with a diagonal argument now leaves us with the desired result in the case where $f^\infty$ is coercive.

\emph{Step 2:} Now assume just that $f\in\RBVw(\Omega\times\R^m)$ and define $f_\rho\in\RBVw(\Omega\times\R^m)$ by
\[
f_\rho(x,y,A):=f(x,y,A)+\rho|A|.
\]
Clearly, $f_\rho\downarrow f$ pointwise as $\rho\downarrow 0$ and so, by the Monotone Convergence Theorem, we have that
\[
\int_\Omega f_\rho(x,u(x),\nabla u(x))\;\dd x \to\int_\Omega f(x,u(x),\nabla u(x))\;\dd x
\]
and
\[
\int_\Omega f^\infty_\rho\left(x,u(x),\frac{\dd D^cu}{\dd|D^cu|}(x)\right)\;\dd |D^cu|(x) \to\int_\Omega f^\infty\left(x,u(x),\frac{\dd D^cu}{\dd|D^cu|}(x)\right)\;\dd |D^cu|(x)
\]
as $\rho\downarrow 0$. Since $f_\rho\geq f$, we have that $K_{f_\rho}[u]\geq K_f[u]$, and in fact it holds that $K_{f_\rho}[u]\downarrow K_f[u]$ pointwise as $\rho\downarrow 0$. This follows from the estimate
\begin{align*}
\lim_{\rho\to 0}K_{f_\rho}[u](x_0)\leq\lim_{\rho\to 0}\int f^\infty_\rho(x_0,\varphi(y),\nabla\varphi(y))\;\dd y&=\int f^\infty(x_0,\varphi(y),\nabla\varphi(y))\;\dd y\\
&\leq K_f[u](x_0)+\varepsilon
\end{align*}
for any $\varphi\in\Acal_u(x_0)$ such that $\int f^\infty(x_0,\varphi(y),\nabla\varphi(y))\;\dd y\leq K_f[u](x_0)+\varepsilon$. Thus,
\[
\int_{\Jcal_u}K_{f_\rho}[u](x)\;\dd\Hcal^{d-1}(x)\to\int_{\Jcal_u}K_f[u](x)\;\dd\Hcal^{d-1}(x)\text{ as $\rho\to 0$.}
\]
Defining $\Fcal_\rho\colon\BV(\Omega;\R^m)\to[0,\infty)$ by
\[
\Fcal_\rho[u]:=\int_\Omega f_\rho(x,u(x),\nabla u(x))\;\dd x+\int_\Omega\int_0^1f^\infty_\rho\left(x,u^\theta(x),\frac{\dd D^su}{\dd |D^su|}(x)\right)\;\dd\theta\;\dd|D^su|(x),
\]
we therefore have that $\Fcal_\rho[u]\to\Fcal[u]$ as $\rho\to 0$ for every $u\in\BV(\Omega;\R^m)$.

For each fixed $\rho>0$, Step~1 implies that there exists a sequence $(u^\rho_j)_j\subset\C^\infty(\Omega;\R^m)$ with $u^\rho_j\wsc u$ in $\BV(\Omega;\R^m)$ and 
\begin{align*}
\lim_{j\to\infty}\Fcal_\rho[u^\rho_j]&=\int_\Omega f_\rho(x,u(x),\nabla u(x))\;\dd x+\int_\Omega f^\infty_\rho\left(x,u(x),\frac{\dd D^cu}{\dd|D^cu|}(x)\right)\;\dd |D^cu|(x)\\
&\qquad+\int_{\Jcal_u}K_{f_\rho}[u](x)\;\dd\Hcal^{d-1}(x).
\end{align*}
Since $f$ is quasiconvex in the final variable, Theorem~\ref{thmwsclsc} implies that
\begin{align*}
&\int_\Omega f(x,u(x),\nabla u(x))\;\dd x+\int_\Omega f^\infty\left(x,u(x),\frac{\dd D^cu}{\dd|D^cu|}(x)\right)\;\dd |D^cu|(x)\\
&\qquad\qquad+\int_{\Jcal_u}K_{f}[u](x)\;\dd\Hcal^{d-1}(x)\\
&\qquad\leq\liminf_{j\to\infty}\Fcal[u^\rho_j]
\end{align*}
for every $\rho>0$. Thus,
\begin{align*}
&\int_\Omega f(x,u(x),\nabla u(x))\;\dd x+\int_\Omega f^\infty\left(x,u(x),\frac{\dd D^cu}{\dd|D^cu|}(x)\right)\;\dd |D^cu|(x)\\
&\qquad\qquad+\int_{\Jcal_u}K_{f}[u](x)\;\dd\Hcal^{d-1}(x)\\
&\qquad\leq\lim_{\rho\to 0}\liminf_{j\to\infty}\Fcal[u^\rho_j]\\
&\qquad\leq\lim_{\rho\to 0}\lim_{j\to\infty}\Fcal_\rho[u^\rho_j]\\
&\qquad=\lim_{\rho\to 0} \biggl( \int_\Omega f_\rho(x,u(x),\nabla u(x))\;\dd x+\int_\Omega f^\infty_\rho\left(x,u(x),\frac{\dd D^cu}{\dd|D^cu|}(x)\right)\;\dd |D^cu|(x)\\
&\qquad\qquad\quad\qquad+\int_{\Jcal_u}K_{f_\rho}[u](x)\;\dd\Hcal^{d-1}(x) \biggr)\\
&\qquad=\int_\Omega f(x,u(x),\nabla u(x))\;\dd x+\int_\Omega f^\infty\left(x,u(x),\frac{\dd D^cu}{\dd|D^cu|}(x)\right)\;\dd |D^cu|(x)\\
&\qquad\qquad+\int_{\Jcal_u}K_{f}[u](x)\;\dd\Hcal^{d-1}(x),
\end{align*}
which leads us to conclude
\begin{align*}
\lim_{\rho\to 0}\liminf_{j\to\infty}\Fcal[u^\rho_j]&=\int_\Omega f(x,u(x),\nabla u(x))\;\dd x+\int_\Omega f^\infty\left(x,u(x),\frac{\dd D^cu}{\dd|D^cu|}(x)\right)\;\dd |D^cu|(x)\\
&\qquad+\int_{\Jcal_u}K_{f}[u](x)\;\dd\Hcal^{d-1}(x).
\end{align*}
For any $\varepsilon>0$, we can therefore find a fixed $\rho_\varepsilon>0$ such that $u_j^{\rho_\varepsilon}\wsc u$ in $\BV(\Omega;\R^m)$ and
\begin{align*}
\Big|\lim_{j\to\infty}\Fcal[u^{\rho_\varepsilon}_j]&-\int_\Omega f(x,u(x),\nabla u(x))\;\dd x+\int_\Omega f^\infty\left(x,u(x),\frac{\dd D^cu}{\dd|D^cu|}(x)\right)\;\dd |D^cu|(x)\\
&\qquad+\int_{\Jcal_u}K_{f}[u](x)\;\dd\Hcal^{d-1}(x)\Big|\leq\varepsilon,
\end{align*}
which suffices to prove the claim.
\end{proof}

\begin{remark}
Interestingly, Theorem~\ref{thmwsclsc} is used in the proof of Theorem~\ref{thmepsilonrecoverysequences} which implies that we are only able to construct approximate recovery sequences for $\Fcalrw$ when $f$ is quasiconvex in the final variable.
\end{remark}

As a side remark, the existence of $\Lp^1$-convergent recovery sequences now follows directly from Theorem~\ref{thmepsilonrecoverysequences} combined with a diagonal argument.
\begin{corollary}[$\Lp^1$-Recovery Sequences]
Let $u\in\BV(\Omega;\R^m)$ and let $f\in\RBVw(\Omega\times\R^m)$ be such that $f(x,y,\frarg)$ is quasiconvex for every $(x,y)\in\overline{\Omega}\times\R^m$. Then there exists a sequence $(u_j)_j\in\C^\infty(\Omega;\R^m)$ such that $u_j\to u$ in $\Lp^1(\Omega;\R^m)$ and
\begin{align*}
\lim_{j\to\infty}\Fcal[u_j]=&\int_\Omega f(x,u(x),\nabla u(x))\;\dd x+\int_\Omega f^\infty\left(x,u(x),\frac{\dd D^c u}{\dd|D^c u|}(x)\right)\;\dd|D^c u|(x)\\
&\qquad+\int_{\mathcal{J}_u}K_f[u](x)\;\dd\mathcal{H}^{d-1}(x).
\end{align*}
\end{corollary}

\subsection{Relaxation}
Combining the results of Sections~\ref{chaptangentliftingyms} and~\ref{chaprecoveryseqs}, we can finally complete the proof of Theorem~\ref{wsclscthm} stated in Section~\ref{secintro}:

\begin{theorem}[Theorem~\ref{wsclscthm}]\label{thmwscrelaxation}
Let $f\in\RBVw(\Omega\times\R^m)$ be such that $f(x,y,\frarg)$ is quasiconvex for every $(x,y)\in\Omega\times\R^m$. The $\BV$-weak* relaxation of the functional $\Fcal\colon\BV(\Omega;\R^m)\to\R$,
\[
\Fcal[u]:=\int_\Omega f(x,u(x),\nabla u(x))\;\dd x+\int_\Omega f^\infty\left(x,u^\theta(x),\frac{\dd D^su}{\dd|D^su|}(x)\right)\;\dd\theta\;\dd|D^su|(x)
\]
is given by
\begin{align*}
\Fcalrw[u]&=\int_\Omega f(x,u(x),\nabla u(x))\;\dd x+\int_\Omega f^\infty\left(x,u(x),\frac{\dd D^cu}{\dd|D^cu|}(x)\right)\;\dd|D^c u|(x)\\
&\qquad+\int_{\Jcal_u} K_f[u](x)\;\dd\Hcal^{d-1}(x).
\end{align*}
\end{theorem}

\begin{proof}
Recalling the definition of $\Fcalrw$, we see that Theorem~\ref{thmwsclsc} implies
\begin{align*}
\Fcalrw[u]&\geq\int_\Omega f(x,u(x),\nabla u(x))\;\dd x+\int_\Omega f^\infty\left(x,u(x),\frac{\dd D^cu}{\dd|D^cu|}(x)\right)\;\dd|D^c u|(x)\\
&\qquad+\int_{\Jcal_u} K_f[u](x)\;\dd\Hcal^{d-1}(x).
\end{align*}
On the other hand, for $\varepsilon>0$ arbitrary, Theorem~\ref{thmepsilonrecoverysequences} guarantees the existence of a sequence $(u_j)_j\subset\BV(\Omega;\R^m)$ such that $u_j\wsc u$
and
\begin{align*}
\Fcalrw[u]&\leq\limsup_{j\to\infty}\Fcal[u_j]\\
&\leq\int_\Omega f(x,u(x),\nabla u(x))\;\dd x+\int_\Omega f^\infty\left(x,u(x),\frac{\dd D^cu}{\dd|D^cu|}(x)\right)\;\dd|D^c u|(x)\\
&\qquad+\int_{\Jcal_u} K_f[u](x)\;\dd\Hcal^{d-1}(x)+\varepsilon=\Fcalrw[u]+\varepsilon
\end{align*}
and so the desired conclusion follows.
\end{proof}